\documentclass[reqno]{amsart}
\usepackage[margin = 1in]{geometry}
\usepackage{amsmath, amssymb, amsthm, fancyhdr, verbatim, graphicx, relsize, mathtools, xfrac}
\usepackage{enumerate}
\usepackage{enumitem} 
\usepackage[all]{xy}
\usepackage[dvipsnames]{xcolor}
\usepackage{mathrsfs}
\usepackage{tikz}
\usetikzlibrary{arrows.meta, positioning, calc, patterns, decorations.pathmorphing, decorations.pathreplacing}
\usepackage{tikz-cd}
\usetikzlibrary{arrows.meta,backgrounds,calc}
\definecolor{CFink}{HTML}{334958}
\definecolor{CFone}{HTML}{238589}
\definecolor{CFtwo}{HTML}{B17E46}
\definecolor{CFthree}{HTML}{75629D}
\definecolor{CFprimeone}{HTML}{7655AD}
\definecolor{CFprimetwo}{HTML}{3E80BA}
\definecolor{CFprimethree}{HTML}{409D73}
\definecolor{CFprimefour}{HTML}{C2AF39}
\definecolor{CFprimefive}{HTML}{D88638}
\definecolor{CFprimesix}{HTML}{C95655}
\definecolor{CFwire}{HTML}{CAD2D8}
\tikzset{
  cf dot/.style={circle,inner sep=0pt,minimum size=4.2pt,
    outer sep=1pt,draw=none,fill=#1},
  cf group/.style={font=\normalsize,text=CFink,fill=none,inner sep=2pt},
  cf axis/.style={draw=CFink,line width=.65pt,
    -{Stealth[length=4.5pt,width=3.2pt]}},
  cf plane/.style={draw=#1!32,line width=.45pt,fill=#1,fill opacity=.022},
  cf edge/.style={line width=.85pt},
  cf route/.style={line width=1.35pt},
  cf change/.style={draw=CFink!80,line width=1.05pt,
    dash pattern=on 3pt off 2.5pt},
  cf guide/.style={line width=.4pt,densely dotted},
}
\usetikzlibrary{calc}
\usepackage[T1]{fontenc} 
\usepackage{framed}
\usepackage[hypertexnames=false]{hyperref}
\usepackage[OT2, T1]{fontenc}  
\usepackage[titletoc]{appendix}
\usepackage{bbm}
\usepackage{adjustbox}
\usepackage{amssymb} 
\usepackage{float}

\numberwithin{equation}{subsection}

\DeclareSymbolFont{cyrletters}{OT2}{wncyr}{m}{n}
\DeclareMathSymbol{\Sha}{\mathalpha}{cyrletters}{"58}

\usepackage{color}

\newcommand{\F}{\mathbf{F}}
\newcommand{\RR}{\mathbf{R}}
\newcommand{\CC}{\mathbf{C}}
\newcommand{\G}{\mathbf{G}}

\newcommand{\wt}[1]{\widetilde{#1}}

\newcommand{\Q}{\mathbf{Q}}
\newcommand{\Z}{\mathbf{Z}}

\newcommand{\mf}[1]{\mathfrak{#1}}

\newcommand{\sgn}{\operatorname{sgn}}
\newcommand{\Gal}{\operatorname{Gal}}

\newcommand{\cal}[1]{\mathcal{#1}}

\newcommand{\ol}[1]{\overline{#1}}
\newcommand{\wh}[1]{\widehat{#1}}

\newcommand{\co}{\colon}
\newcommand{\mrm}[1]{\mathrm{#1}}

\newcommand{\ld}{{}^L}
\newcommand{\lc}{{}^c}

\newcommand{\inj}{\hookrightarrow}

\newcommand{\io}{\iota}

\newcommand\cB{\mathcal{B}}
\newcommand\cC{\mathcal{C}}

\newcommand\cE{\mathcal{E}}

\newcommand\cG{\mathcal{G}}
\newcommand\cH{\mathcal{H}}

\newcommand\cK{\mathcal{K}}

\newcommand\cO{\mathcal{O}}

\newcommand\cS{\mathcal{S}}
\newcommand\cT{\mathcal{T}}

\newcommand\cW{\mathcal{W}}

\newcommand\cZ{\mathcal{Z}}

\newcommand\sH{\mathscr{H}}

\newcommand\frZ{\mathfrak{Z}}

\newcommand\frg{\mathfrak{g}}

\newcommand\frp{\mathfrak{p}}

\DeclareMathOperator{\GL}{GL}
\DeclareMathOperator{\SL}{SL}

\DeclareMathOperator{\ab}{ab}

\DeclareMathOperator{\cyc}{cyc}

\DeclareMathOperator{\Tr}{Tr}

\DeclareMathOperator{\Sp}{Sp}
\DeclareMathOperator{\SO}{SO}
\DeclareMathOperator{\Hom}{Hom}
\DeclareMathOperator{\Ind}{Ind}
\DeclareMathOperator{\Ima}{Im\,}

\DeclareMathOperator{\ord}{ord}
\DeclareMathOperator{\Aut}{Aut}
\DeclareMathOperator{\Rep}{Rep}

\DeclareMathOperator{\Nm}{Nm}

\DeclareMathOperator{\Lie}{Lie}
\DeclareMathOperator{\End}{End}

\DeclareMathOperator{\ad}{ad}

\DeclareMathOperator{\unr}{unr}
\DeclareMathOperator{\Res}{Res}

\DeclareMathOperator{\Div}{Div}

\DeclareMathOperator{\Id}{Id}
\DeclareMathOperator{\Ad}{Ad}
\DeclareMathOperator{\Gr}{Gr}

\DeclareMathOperator{\ind}{ind}

\DeclareMathOperator{\Sat}{Sat}

\DeclareMathOperator{\Exc}{Exc}

\DeclareMathOperator{\alg}{alg}

\DeclareMathOperator{\supp}{supp}

\newcommand{\cHck}{\mathcal{H}\mrm{ck}}

\DeclareMathOperator{\sm}{sm}
\DeclareMathOperator{\FS}{FS}
\DeclareMathOperator{\cInd}{c-Ind}

\DeclareMathOperator{\CT}{CT}

\DeclareMathOperator{\red}{red}

\DeclareMathOperator{\Kal}{Kal}

\DeclareMathOperator{\der}{der}

\DeclareMathOperator{\rk}{rk}

\DeclareMathOperator{\diag}{diag}

\DeclareMathOperator{\Fr}{Fr}

\RequirePackage{xspace}

\newcommand{\rH}{\ensuremath{\mathrm{H}}\xspace}

\newcommand{\rT}{\ensuremath{\mathrm{T}}\xspace}

\newcommand{\cbr}{br}

\newcommand{\chG}{\check{G}}
\newcommand{\chH}{\check{H}}
\newcommand{\chT}{\check{T}}

\renewcommand{\ss}{\mathrm{ss}}

\newtheorem{thm}{Theorem}[subsection]
\newtheorem{lemma}[thm]{Lemma}
\newtheorem{prop}[thm]{Proposition}
\newtheorem{cor}[thm]{Corollary}

\theoremstyle{remark}
\newtheorem{remark}[thm]{Remark} 
 
\newtheorem{defn}[thm]{Definition}
\newtheorem{ansatz}[thm]{Ansatz}

\newtheorem{example}[thm]{Example}

\makeatletter
\def\th@remark{%
  \thm@headfont{\bfseries}%
  \normalfont 
  \thm@preskip \thm@preskip 
  \thm@postskip\thm@preskip
}
\def\imod#1{\allowbreak\mkern5mu({\operator@font mod}\,\,#1)}
\makeatother

\numberwithin{equation}{subsection}
\numberwithin{figure}{subsection}

\title[Local Langlands functoriality for Yu's supercuspidals II]{Local Langlands functoriality for Yu's supercuspidals II: Fargues--Scholze's parametrization}

\author{Sean Cotner and Tony Feng}

\begin{document}

\begin{abstract}
This is the second of two papers dedicated to the explicit computation of the Fargues--Scholze correspondence. We compute the Tate cohomology of cuspidal representations arising from Yu's construction. Combining this with modular functoriality in the Local Langlands Correspondence, independence of $\ell$, and the partial characterization of the Local Langlands Correspondence established in the first paper, we compare the Fargues--Scholze and Kaletha parametrizations. Among the consequences, we deduce that the Fargues--Scholze correspondence has finite fibers and is surjective onto inertial L-parameters.
\end{abstract}

\maketitle

\tableofcontents

\section{Introduction}
This paper is the second half of a two-paper series which aims to explicitly compute the Fargues--Scholze parameters, at least on inertia, for cuspidal representations arising from Yu's construction \cite{Yu01}. We encourage the reader to read \cite[\S 1.1]{CF26a}, which sets the motivation and context for our work, before proceeding further. 

Let $F$ be a non-archimedean local field of residue characteristic $p$, and let $G$ be a connected reductive $F$-group. Throughout this introduction, we will assume that $G$ splits after a tamely ramified extension of $F$ and $p$ does not divide the order of the absolute Weyl group $\Omega$ of $G$. Let $\ell \neq p$ be a prime number.

\subsection{The main theorems}

In \cite{CF26a} we defined, following and extending work of Kaletha \cite{Kal19,Kal21b}, a partial form of an explicit Local Langlands Correspondence for irreducible cuspidal representations of $G(F)$ over $k \in \{\ol\Q_\ell, \ol\F_\ell\}$. More precisely, to any such representation $\pi$ we constructed an inertial L-parameter \cite[Definition~7.2.12]{CF26a}
\[
\rho_I^{\Kal}(\pi) : I_F \rightarrow \ld G(k).
\]
If $\pi$ is moreover \emph{non-singular} \cite[Definition~7.2.10]{CF26a}, we defined a full L-parameter 
\[
\rho^{\Kal}(\pi) : W_F \rightarrow \ld G(k).
\]
For example, $\rho_I^{\Kal}(\pi)$ is constructed by attaching a torus-character pair $(T, \theta)$ to $\pi$ and then defining
\[
\rho_I^{\Kal}(\pi) = \left(I_F \xrightarrow{\ld\theta} \ld T(k) \xrightarrow{\ld j_{T,G}} \ld G(k) \right),
\]
where $\ld j_{T,G}\co \ld T \to \ld G$ is a certain L-embedding coming from \cite{FKS23}; see \cite[Definition~4.5.1]{CF26a}.

On the other hand, for any such $\pi$ we also have a Fargues--Scholze parameter from \cite{FS},
\[
\rho^{\FS}(\pi) : W_F \rightarrow \ld G(k).
\]

In this paper, we prove the following comparison of the two constructions. 

\begin{thm}[Theorem~\ref{thm:main-tame}]\label{thm:II-intro-main-thm}
Let $\pi$ be an irreducible cuspidal $k$-representation of $G(F)$, and let $T$ be a torus attached to $\pi$ as above.
\begin{enumerate}
    \item If $T$ is maximally unramified, then $\rho^{\FS}(\pi)|_{I_F} \sim \rho_I^{\Kal}(\pi)$. If moreover $\pi$ is non-singular, then 
    \[
    \rho^{\FS}(\pi) \sim \rho^{\Kal}(\pi).
    \]
    \item We have
    \begin{equation}\label{eqn:II-intro-inertial-comparison}
    \rho^{\FS}(\pi)|_{I_F} \sim \rho_I^{\Kal}(\pi).
    \end{equation}
    Let $T_0$ be the maximal unramified $F$-subtorus of $T$. If moreover $\pi$ is non-singular, then there is an explicit integer $N$, depending only on the image of $T$ in $Z_G(T_0)_{\ad}$, such that 
    \begin{equation}\label{eqn:II-intro-full-comparison}
    \rho^{\FS}(\pi)|_{W_{F_N}} \sim \rho^{\Kal}(\pi)|_{W_{F_N}}.
    \end{equation}
    \item The representation $\pi$ is non-singular if and only if $\rho^{\FS}(\pi)$ is irreducible.\footnote{An L-parameter is \emph{irreducible} if and only if it does not normalize any proper parabolic of $\wh G$. Such L-parameters are referred to as \emph{elliptic} in \cite[Definition X.2.1]{FS} and \emph{supercuspidal} in \cite[Definition 4.1.1]{Kal21b}. Our terminology is more standard in the algebraic groups literature (see, e.g., \cite{BMR05}, which uses \emph{$\ld G$-ir} as an abbreviation), and it originates from Serre. It is convenient for our purposes because inertial L-parameters are irreducible in some important cases, but it seems inapt to call these \emph{elliptic} or \emph{supercuspidal}.}
\end{enumerate}
\end{thm}

\begin{remark}The integer $N$ in Theorem~\ref{thm:II-intro-main-thm} is the same as the one appearing in \cite[Theorem~1.3.1]{CF26a}; as mentioned there, it is not difficult to compute in examples, and it is often ``small''; for instance, it is at most $4$ if $G \cong \mathrm{G}_2$. We note that Theorem~\ref{thm:II-intro-main-thm}(3) implies that the Fargues--Scholze L-packet of an irreducible L-parameter consists of non-singular cuspidal representations; this has been expected since at least \cite{Kal21b}.
\end{remark}

The following consequence of Theorem~\ref{thm:II-intro-main-thm} has been expected since \cite{FS}; we note that (1) relies on the fact that we have computed $\rho^{\FS}(\pi)|_{I_F}$ for \emph{all} cuspidal $\pi$, not just the non-singular ones from \cite{Kal21b}.

\begin{cor}\label{cor:II-intro-main-cor}

\begin{enumerate}
    \item (Theorem~\ref{thm:finiteness-of-fargues-scholze}) $\rho^{\FS}$ preserves depth and has finite fibers.
    \item (Theorem~\ref{thm:inertial-surjectivity}) If $\rho\co W_F \to \ld G(k)$ is an irreducible L-parameter, then there exists an irreducible smooth $k$-representation $\pi$ of $G(F)$ such that $\rho^{\FS}(\pi)|_{I_F} \sim \rho|_{I_F}$.
\end{enumerate}
\end{cor}

\subsection{Related work}\label{ssec:related-work}
A broad overview of progress on the Local Langlands Correspondence, thanks to work of many authors, is documented in \cite{CF26a}. Here we discuss some results that are most closely related to ours. 

For depth $0$ cuspidal representations, the restriction of the Fargues--Scholze parameter to inertia has been computed by Eteve \cite[Theorem~2.3.9]{Ete23} for split groups over local function fields, and conditionally by Fu \cite[Theorem~1.0.6]{Fu26} for unramified groups over arbitrary non-archimedean local fields. Zhu \cite[\S 5.3.4]{Zhu25} computes the parameters of regular depth $0$ supercuspidals for his own version of the Local Langlands Correspondence, identifying them with the DeBacker--Reeder parameters for unramified groups of adjoint type with $\ol\Q_\ell$-coefficients. The work of Gleason--Hamann--Ivanov--Louren\c{c}o--Zou \cite{GHILZ26} makes it plausible that one will eventually be able to deduce the same comparison for Fargues--Scholze from Zhu's work.

In another direction, Gan--Harris--Sawin and Beuzart-Plessis \cite[Theorem 1.4]{GHS24} prove a weak form of depth preservation, namely that if $F$ is a function field whose residue field is not too small and $G$ is unramified, then there are many $\pi$ for which $\rho^{\FS}(\pi)$ is wildly ramified.\footnote{Strictly speaking, \cite{GHS24} concerns the Genestier--Lafforgue Local Langlands Correspondence \cite{GL18}, but this was shown to agree with Fargues--Scholze by Li-Huerta \cite{LH23}.} A stronger version of Corollary~\ref{cor:II-intro-main-cor}(2) (namely, the same statement without restriction to inertia) is also proven by Beuzart-Plessis--Harris--Thorne \cite{BPHT25} in the case that $F$ is a function field, $G$ is unramified and semisimple, and $p$ does not divide the order of the absolute Weyl group of $G$, under the hypothesis that a version of the stable twisted trace formula holds over function fields. By work of Li-Huerta \cite{LH24}, this implies a certain surjectivity statement onto not-very-ramified L-parameters when $F$ is $p$-adic of ``large'' ramification degree.

Corollary~\ref{cor:II-intro-main-cor} is of particular interest in view of recent work of Hansen--Mann \cite{HM26} which proves the categorical local Langlands conjecture of \cite{FS} for $\GL_n$ (under a hypothesis on compatibility with Eisenstein series) and provides a strategy for proving the categorical local Langlands conjecture in general. This strategy requires a number of ``soft'' inputs as in Corollary~\ref{cor:II-intro-main-cor}.\footnote{See the definition of ``well-understood'' groups in \cite[Definition 6.1.1]{HM26}. Our results do not establish any of the conditions (i)-(v) of that definition in full, but they do establish the first sentence of (i) and half of (iv) when $G$ splits after a tamely ramified extension and $p$ does not divide the order of the absolute Weyl group and $\pi \mapsto \phi_\pi$ is any map satisfying (v).}

For $k=\ol\Q_\ell$, several comparisons of $\rho^{\FS}$ with classical local Langlands correspondences are known after semisimplification. For $\GL_n$, this follows from Fargues--Scholze \cite[Theorem I.9.6(viii)--(ix)]{FS}; for its inner forms over $p$-adic fields, it is due to Hansen--Kaletha--Weinstein \cite[Theorem 1.0.3]{HKW22}. Hamann \cite[Theorem 1.1 and Corollary 8.3]{Ham25} treated $\mathrm{GSp}_4$, $\mathrm{Sp}_4$, and their inner forms when $F/\Q_p$ is unramified and $p>2$. Bertoloni Meli--Hamann--Nguyen \cite[Theorem 1.1]{BMHN24} treated odd-dimensional unitary and unitary similitude groups over $\Q_p$ attached to its unramified quadratic extension. Peng \cite[Theorem A]{Pen26} treated special orthogonal and unitary groups splitting over an unramified extension, for unramified $F/\Q_p$ and $p>2$. Daniels--van Hoften--Kim--Zhang \cite[Theorem II]{DvHKZ26} removed these restrictions on $p$ and $F$ for unitary groups, odd special orthogonal groups, their inner forms, inner forms of $\mathrm{GSp}_4$, and pure inner forms of quasi-split even special orthogonal groups, partially using \cite{Han26}.

However, the compatibilities established in the preceding paragraph do not imply our results, even for the specific families of classical groups covered there, because compatibility of the classical Local Langlands correspondence with Kaletha's explicit correspondence is not known in general. In fact, our results can be turned around to imply this compatibility in many cases: combining them with Theorem~\ref{thm:II-intro-main-thm}, under its hypotheses and for the groups covered by these comparison theorems, implies that $\rho_I^{\Kal}$ agrees with the restriction to $I_F$ of the semisimplification of the classical correspondences of Arthur \cite{Art13}, Mok and Kaletha--M\'{i}nguez--Shin--White \cite{Mok15,KMSW14}, Chen--Zou \cite{CZ21}, Ishimoto \cite{Ish24}, Gan--Takeda \cite{GT10,GT11}, Gan--Tantono \cite{GT14}, and Choiy \cite{Cho17}.

The following theorem strengthens Theorem~\ref{thm:II-intro-main-thm} in a special case.

\begin{thm}[Theorem~\ref{thm:sln-full-conjugacy}]\label{thm:II-intro-type-A}\emergencystretch=1em
Suppose $G_{F^{\unr}} \cong \SL_n$, and let $\pi$ be a non-singular cuspidal $k$-representation of $G(F)$. Then $\rho^{\FS}(\pi) \sim \rho^{\Kal}(\pi)$.
\end{thm}

We emphasize that our proof of Theorem~\ref{thm:II-intro-type-A} is \emph{purely local}. Theorem~\ref{thm:II-intro-type-A} implies in particular that if $G$ is an inner form of $\SL_n$ or an unramified special unitary group, then $\rho^{\Kal}$ agrees with the correspondence deduced from the classical ones for inner forms of $\GL_n$ or unramified unitary groups.

\begin{remark}\label{rem:type-A} If $G$ is an inner form of $\GL_n$, then \cite[Theorem I.9.6 (ix)]{FS} and \cite[Theorem 1.0.3]{HKW22} show that $\rho^{\FS}$ agrees with the semisimplification of the classical Local Langlands Correspondence $\rho^{\mathrm{cl}}$, while \cite[Theorem 1.1]{OT21} and \cite[Theorem 0.1]{Tok23} show that $\rho^{\Kal}$ agrees with $\rho^{\mathrm{cl}}$; these combine to prove Theorem~\ref{thm:II-intro-type-A} when $G$ is an inner form of $\SL_n$ (or even $\GL_n$). On the other hand, if $G$ is an inner form of an unramified special unitary group, then to our knowledge $\rho^{\Kal}$ has not previously been compared to the classical Local Langlands Correspondence of Mok \cite{Mok15} and Kaletha--M\'{i}nguez--Shin--White \cite{KMSW14}, and thus Theorem~\ref{thm:II-intro-type-A} is new in this case despite \cite[Theorem II]{DvHKZ26}.
\end{remark}

\subsection{Functoriality results}

The conclusion of Theorem~\ref{thm:II-intro-main-thm} is a consequence of \cite[Theorem~1.3.1]{CF26a}, and the proof involves checking the hypotheses of the latter. In view of the results of \cite[Theorem I.9.6]{FS}, the important hypotheses to check are (5) and (4), which say respectively that $\rho^{\FS}$ satisfies functoriality with respect to descent to unramified twisted Levis and certain small degree base change. We will bootstrap these from the two general functoriality results to which we have access, namely compatibility with parabolic induction \cite[Theorem I.9.6(viii)]{FS} and \emph{modular functoriality}, using \cite[Theorem~9.1.1 and Example~9.1.2]{F24} at all primes $\ell\neq p$.

\begin{thm}[Theorem~\ref{thm:base-change-functoriality}]\label{thm:II-intro-large-degree-bc}
There exists a positive integer $C$ (depending on $G$) such that for every prime $\ell > C$, every cyclic extension $E/F$ of degree $\ell$, and every tame irreducible cuspidal $k$-representation $\pi$ of $G(F)$, there is an explicit tame irreducible cuspidal $k$-representation $\pi_E$ of $G(E)$ such that $\rho^{\FS}(\pi_E)|_{I_E} \sim \rho^{\FS}(\pi)|_{I_E}$. Moreover, $C$ can be chosen such that if $\rho^{\FS}(\pi)$ is irreducible, then we have $\rho^{\FS}(\pi_E) \sim \rho^{\FS}(\pi)|_{W_E}$.
\end{thm}

The constant $C$ in Theorem~\ref{thm:II-intro-large-degree-bc} is in principle effective, but we have made no effort to optimize it or make it explicit. Although we expect Theorem~\ref{thm:II-intro-large-degree-bc} to have other applications, for our purposes its main interest comes from the role it plays in the proof of the following theorem. A slight technical strengthening (Theorem~\ref{thm:unramified-twisted-Levi-functoriality}) implies hypothesis (5) of \cite[Theorem~1.3.1]{CF26a}.

\begin{thm}[Theorem~\ref{thm:unramified-twisted-Levi-functoriality}]\label{thm:II-intro-unram-twisted-levi}
Let $\pi$ be an irreducible cuspidal $k$-representation of $G(F)$, let $T \subset G$ be a maximal $F$-torus associated to $\pi$, and suppose that $T \cap G_{\der}$ is not totally ramified. There exists a proper unramified twisted Levi $F$-subgroup $M \subsetneq G$ and an irreducible cuspidal $k$-representation $\pi_M$ of $M(F)$ such that
\begin{equation}\label{eqn:adlr-kaletha}
\rho_I^{\Kal}(\pi) \sim \ld j_{M,G} \circ \rho_I^{\Kal}(\pi_M)
\end{equation}
and
\begin{equation}\label{eqn:adlr-fs}
\rho^{\FS}(\pi)|_{I_F} \sim \ld j_{M,G} \circ \rho^{\FS}(\pi_M)|_{I_F}.
\end{equation}
If $\pi$ is non-singular, then $\pi_M$ is non-singular and
\[
\rho^{\Kal}(\pi) \sim \ld j_{M,G} \circ \rho^{\Kal}(\pi_M).
\]
\end{thm}

The proofs of Theorems~\ref{thm:II-intro-large-degree-bc} and \ref{thm:II-intro-unram-twisted-levi} appear to be quite robust; we expect them to extend to give similar functoriality results in any setting in which one has an explicit construction of cuspidal representations. For instance, the arguments should also apply to the constructions of \cite{FS25} and \cite{Tak26}, which go beyond Yu's construction when $p$ is small (especially $p = 2$), but they should even apply to constructions which exhibit more wildly ramified behavior, such as those for small $p$ in \cite{BK93}, \cite{SS08}, \cite{Ste08}. In \cite[\S 5]{Cot26b}, the first-named author uses this strategy to compute the Fargues--Scholze L-parameters of depth $0$ non-singular cuspidal representations of arbitrary connected reductive $F$-groups.

Finally, the following theorem supplies the small-degree base change used in the proof of Theorem~\ref{thm:II-intro-main-thm}.

\begin{thm}[Corollary~\ref{cor:yu-datum-small-degree-bc-parameter}]\label{thm:II-intro-small-degree-bc}
Let $\pi$ be a toral\footnote{The word ``toral'' is often used to mean that the twisted Levi sequence in any Yu datum defining $\pi$ is of the form $(T \subset G)$ for a torus $T$. Importantly, we only demand the less restrictive condition that the smallest twisted Levi in any Yu datum defining $\pi$ is a torus.} cuspidal $k$-representation of $G(F)$ with finite order central character, let $T \subset G$ be a maximal $F$-torus associated to $\pi$, and let $E/F$ be a cyclic extension of degree $\ell$ such that $T_E$ is elliptic. Then there exists a toral cuspidal $k$-representation $\pi_E$ such that
\[
\rho^{\Kal}(\pi_E) \sim \rho^{\Kal}(\pi)|_{W_E}
\]
and
\[
\rho^{\FS}(\pi_E) \sim \rho^{\FS}(\pi)|_{W_E} \pmod{\ell}.
\]
\end{thm}

In particular, Theorem~\ref{thm:II-intro-small-degree-bc} shows (under its hypotheses) that $\rho^{\FS}(\pi_E)|_{P_F} \sim \rho^{\FS}(\pi)|_{P_F}$; see \cite[Lemma~5.3.2]{Cot26b}. This provides hypothesis (4) in \cite[Theorem~1.3.1]{CF26a} and is thereby enough to complete the proof of Theorem~\ref{thm:II-intro-main-thm}. If one wishes to extend the proof of Theorem~\ref{thm:II-intro-main-thm} to a more wild construction (and, in particular, beyond Yu's construction), at present it appears that the most difficult step would be to generalize Theorem~\ref{thm:II-intro-small-degree-bc} to (especially ramified) cyclic extensions of degree $p$.

\subsection{Discussion of the proofs}\label{ssec:III-intro-proofs}

Our methods are completely different from those of prior works (recalled in \S \ref{ssec:related-work}). Notably, they are purely local and representation-theoretic. The starting point is the theory of \emph{modular functoriality} as developed in \cite{F24}, building on ideas of Treumann--Venkatesh \cite{TV}. This refers to functoriality along a homomorphism $\ld H \rightarrow \ld G$ where $H$ is a connected reductive group arising as the $\sigma$-fixed points of a prime order $\ell$ automorphism $\sigma \in \Aut(G)$, and the L-groups $\ld H$ and $\ld G$ are regarded over $k = \ol \F_\ell$. We emphasize that it is \emph{crucial} that the same $\ell$ appears as both the order of $\sigma$ and the characteristic of $k$. The adjective ``modular'' refers to the phenomenon that this functoriality only exists in the defining characteristic $\ell$, thanks to miracles of modular arithmetic. 

It is highly non-obvious that under the given assumptions defining $H \subset G$, there is a matching ``$\sigma$-dual homomorphism $\ld H_k \rightarrow \ld G_k$'' defined over $k$. This was proven in some cases by \cite{TV} under the (presumably artificial) assumptions that $G$ is simply connected and $H$ is semisimple, and $G$ does not contain factors of type $\mathrm{E}_6$. The construction in \cite[Theorem~7.2.1]{F24} applies whenever $G^\sigma$ is connected, for every prime $\ell\neq p$.

\begin{example}[Spectral endoscopy] An important class of examples of modular functoriality arise when $\sigma$ is an \emph{inner} automorphism, so $H$ is the centralizer of an element of $G(F)$. We refer to this situation as \emph{spectral endoscopy}, and we call $\chH$ a \emph{spectral endoscopic group} for $\chG$. The name comes from the resemblance to endoscopy, in which (roughly) we would say that $H$ is an endoscopic group for $G$ if $\chH$ is the centralizer of a semisimple element of $\chG$. Thus, spectral endoscopy is endoscopy with the roles of the group and its dual group reversed.

We will crucially exploit the miracle that \emph{unlike} in endoscopy, there is a canonical conjugacy class of $\sigma$-dual homomorphisms $\ld H_k \rightarrow \ld G_k$ in the setting of spectral endoscopy. At first glance, this appears to be obviously wrong, as the group theory should work in the same way as in the usual theory of endoscopy. Indeed, a commonly cited example in endoscopy is that $H = \Sp_{2a} \times \Sp_{2b}$ can arise as the centralizer of an order 2 element of $G = \Sp_{2a+2b}$. Then the $\sigma$-dual homomorphism is a finite map 
\[
\SO_{2a+1} \times \SO_{2b+1} \rightarrow \SO_{2a+2b+1}.
\]
The existence of such a finite homomorphism goes against the grain, but the key point is that \emph{it only exists in characteristic $2$}. For more details, see Example \ref{ex:spectral-endoscopy}.
\end{example}

Modular functoriality says that \emph{Tate cohomology realizes descent} (in the sense of Langlands functoriality) from $G$ to $H$. Informally speaking, it supplies a ``descent'' construction which is formally analogous to a Jacquet functor. Such constructions should be regarded as ``exceptional'' features of positive characteristic. More precisely, Theorem~\ref{thm:modular-functoriality} asserts that for any irreducible $\Pi \in \Rep_k^{\sm} G(F)$ whose isomorphism class is fixed by $\sigma$, so that the $G(F)$-action on $\Pi$ extends to a $G(F) \rtimes \langle\sigma\rangle$-action, every irreducible $H(F)$-subquotient $\pi$ of the Tate cohomology group $\rT^0(\sigma, \Pi)$ has the property that
\begin{equation}\label{eq:III-intro-modular-functoriality}
\rho^{\FS}(\Pi^{(\ell)})
\sim \bigl(\ld\psi\circ\rho^{\FS}(\pi)\bigr)^{\ss},
\end{equation}
where $\Pi^{(\ell)}=\Pi\otimes_{k,\Fr_\ell}k$ is the coefficient Frobenius twist and $(\cdot)^{\ss}$ denotes semisimplification in $\ld G$.

In the setting above, since $H$ has smaller dimension, we can hope to understand $\rho^{\FS}(\Pi)$ inductively using \eqref{eq:III-intro-modular-functoriality}, by understanding $\rT^0(\sigma, \Pi)$ and $\ld \psi$. This is the strategy that we will follow, but its execution is subtle. One reason is that modular functoriality forces us to work in positive characteristic, and to use automorphisms with order equal to that characteristic. For a given group and prime $\ell$, there may be few (e.g., no) such automorphisms. Furthermore, even when they exist, the prime $\ell$ must then be ``small'' (i.e., non-banal) relative to $G$, so that reduction modulo $\ell$ loses information. For these reasons, modular functoriality has the feel of an ``exceptional'' phenomenon, and it is not immediately clear how much it can be leveraged. 

The solution to these problems has two key ingredients.
\begin{enumerate}
    \item Scholze's independence of $\ell$ results \cite[Theorem 1.1]{Sch25} allow us to ``glue'' congruences modulo several different primes in order to obtain characteristic $0$ information.
    \item Using degree $\ell$ cyclic base change modulo $\ell$ and another instance of modular functoriality, we can pass to unramified extensions where our groups acquire more torsion elements, and we can thereby be put into the situation of spectral endoscopy by \cite[Lemma~5.3.5]{Cot26b}.
\end{enumerate}
In particular, although the applications described above are to representation theory with characteristic $0$ coefficients, \emph{the proofs make essential use of miracles occurring in positive characteristic}. It is striking that this seemingly transient information, apparently so delicate that it can only exist in specific characteristics, is simultaneously powerful enough to accumulate to highly nontrivial knowledge of the Fargues--Scholze correspondence in characteristic zero. 

To prove Theorem \ref{thm:II-intro-large-degree-bc}, modular functoriality reduces one to a problem of computing Tate cohomology of Yu's construction. In particular, we have to control the Tate cohomology of Deligne--Lusztig representations, which boils down to the link between Tate cohomology and the \emph{Glauberman correspondence}, established in \cite[Corollary~3.5.5, Lemma~3.5.11]{Cot26b}. To prove Theorem \ref{thm:II-intro-unram-twisted-levi}, we first base change to reduce to the setting of spectral endoscopy for a very large prime, and then we apply modular functoriality again. Again we have to control the Tate cohomology of Deligne--Lusztig representations, and now this boils down to the link between Tate cohomology and \emph{Deligne--Lusztig restriction} established in \cite[Proposition~3.4.1]{Cot26b}. Thus, the moral reason that this works is the phenomenon highlighted in \cite{Cot26b} that Tate cohomology realizes a type of modular functoriality for finite groups, unifying the modular reductions of various other known constructions in the spirit of ``Langlands functoriality'' for finite reductive groups. The details of Yu's construction also require us to understand the (remarkably subtle) interaction of Tate cohomology with tensor products, as well as the Tate cohomology of Weil--Heisenberg representations.

\subsection{Sign characters}\label{ss:III-signs} One important theme of the explicit Local Langlands parametrization is the necessity of certain subtle sign characters, e.g., as seen in \cite{FKS23}. These sign characters must arise ``naturally'' in the comparison to the Fargues--Scholze correspondence, and our proof does offer some (partial) conceptual explanation of their provenance. 

Indeed, the FKS sign character itself arises as a product of other sign characters, and this decomposition is at least partly present in our analysis. This is discussed in some detail in \cite[\S 1.5]{CF26a}. From our perspective, part of the FKS sign character arises to correct a discrepancy between the $L$-embeddings defined in \cite[\S 4.2]{FKS23} for $G$ and for various unramified twisted Levi subgroups of $G$. Another part arises from the Weil-Heisenberg representation, as we now describe.

In order to execute the strategy sketched in \S \ref{ssec:III-intro-proofs}, we need to compute the Tate cohomology of cuspidal representations arising from Yu's construction. These are compact inductions of tensor products of depth 0 cuspidals and Weil--Heisenberg representations, and so we eventually need to understand the Tate cohomology of Weil--Heisenberg representations. One factor in the FKS sign character, namely $\epsilon_{\sharp,x}$ (introduced as $\epsilon_{x,r/2}^{\mathrm{ram}}$ in \cite[\S 4.3]{DS18}), is ultimately seen to arise from the Tate cohomology of Weil--Heisenberg representations, as a consequence of certain sign characters appearing in \cite{Ger77}; see Proposition~\ref{prop:tate-cohom-of-weil} and Lemma~\ref{lemma:eta-sign-comparison}.

\subsection{Outline of the paper}

We now describe the various sections of this paper in some detail.

In \S\ref{ssec:notation-fs}, we record the notation that we will use in the body of the paper.

In \S\ref{sec:generalities-on-l-parameters}, we prove a number of technical results concerning L-parameters which will be needed to apply modular functoriality; in particular, we prove a number results giving a prior bounds on the orders of the images of (inertial) L-parameters. In \S\ref{sec:spectral-endoscopy}, we prove that the $\sigma$-dual homomorphisms of \cite{F24} are compatible with a number of natural operations on $G$, and we bootstrap these compatibilities to give (partial) calculations in our situations of interest. Some calculations are relegated to Appendix~\ref{app:l-homs}.

In \S\ref{section:tate-tensor-products}, we prove a certain surprisingly subtle compatibility of Tate cohomology with tensor products. The proofs are mainly combinatorial. In \S\ref{ss:weil-heisenberg}, we compute the Tate cohomology of Weil--Heisenberg representations using the results of \cite{Ger77}, and we describe a certain sign character which intervenes in the calculation.

In \S\ref{section:tate-cohom-yu}, we combine all of the previous calculations to describe the Tate cohomology of cuspidal representations arising from Yu's construction in three important situations: base change of large prime degree, base change of small prime degree for toral cuspidal representations, and descent to unramified twisted Levis under some hypotheses. Using these, we obtain our main explicit results on modular functoriality, including Theorem~\ref{thm:II-intro-small-degree-bc}.

In \S\ref{ssec:modular-cuspidal-parabolic-induction}, we establish a certain compatibility of Yu's construction with parabolic induction; the relevance of this to the overall strategy is explained in detail there. This relies in Appendix~\ref{app:kim-yu-ohara-coefficients}, which extends certain results of \cite{KY17} and \cite{Ohara24} from characteristic $0$ coefficients to positive characteristic coefficients.

The most important part of this paper is \S\ref{section:fs-functoriality}, which combines all of the preceding sections to prove Theorems~\ref{thm:II-intro-large-degree-bc} and \ref{thm:II-intro-unram-twisted-levi}. Finally, in \S\ref{section:fs-calculation}, we use \cite[Theorem 10.2.1]{CF26a} to deduce Theorem~\ref{thm:II-intro-main-thm}, Corollary~\ref{cor:II-intro-main-cor}, and Theorem~\ref{thm:II-intro-type-A}. 

Throughout this paper, we pay special attention to the case that $p$ is small, and thus all results are stated in greater generality than in the introduction.

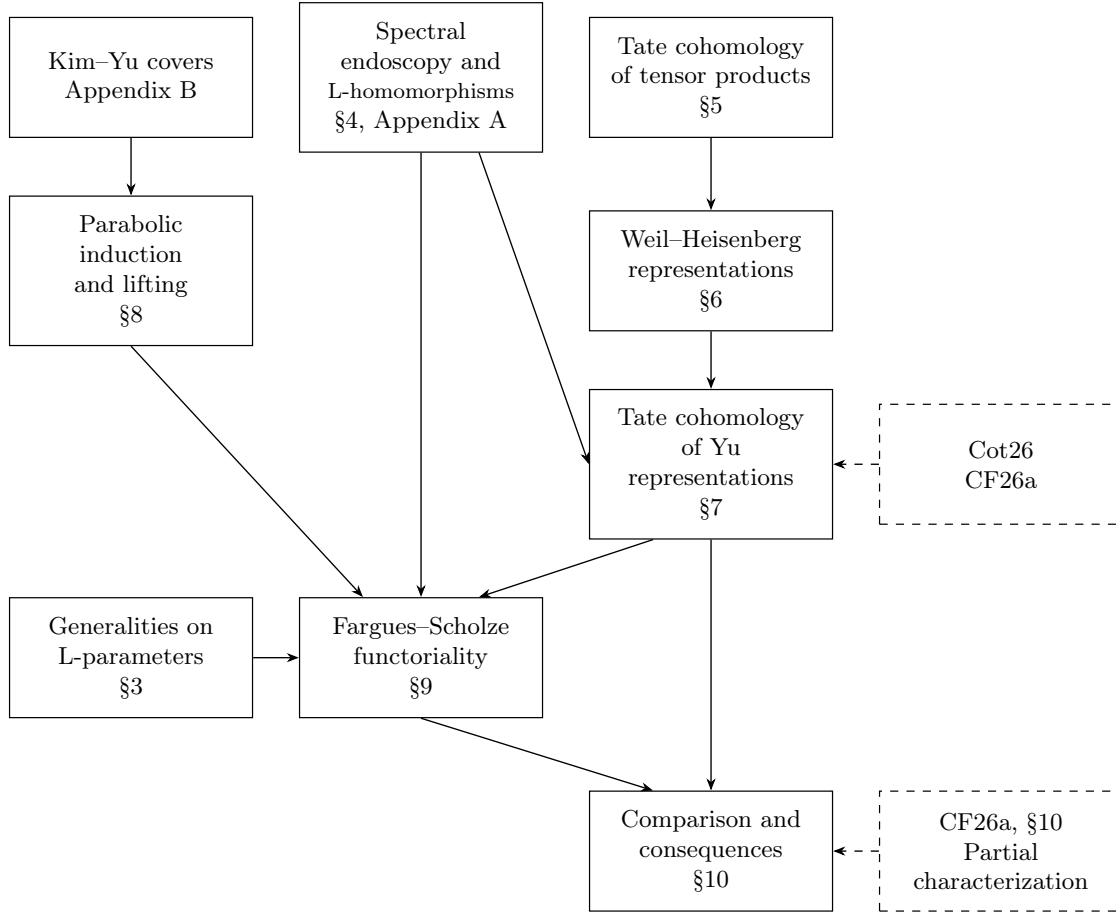
\begin{figure}[!htbp]
  \centering
  \hypersetup{hidelinks}
  \tikzset{
    lf box/.style={draw, line width=.45pt, rectangle, align=center,
      inner xsep=7pt, inner ysep=7pt, font=\small,
      text width=3.35cm, minimum height=1.55cm, fill=white},
    lf wide/.style={lf box, text width=4.65cm, minimum height=1.55cm},
    lf result/.style={lf box, text width=7.1cm},
    lf input/.style={lf box, dashed},
    lf arrow/.style={-{Stealth[length=4.5pt,width=3.4pt]}, line width=.5pt},
    lf external/.style={lf arrow, dashed},
    lf note/.style={font=\scriptsize, align=center, fill=white, inner sep=2pt}
  }
  \resizebox{.90\textwidth}{!}{%
    \begin{tikzpicture}[x=1cm,y=1cm,
      lf box/.append style={text width=2.65cm}]
      \node[lf box] (bounds) at (-7.5,-7.5)
        {Generalities on\\L-parameters\\
         \S\ref{sec:generalities-on-l-parameters}};
      \node[lf box] (tensor) at (0,0)
        {Tate cohomology\\of tensor products\\
         \S\ref{section:tate-tensor-products}};
      \node[lf box] (covers) at (-7.5,0)
        {Kim--Yu covers\\
         Appendix~\ref{app:kim-yu-ohara-coefficients}};
    
      \node[lf box] (spectral) at (-3.75,0)
        {Spectral\\endoscopy and\\{\footnotesize L-homomorphisms}\\
         \S\ref{sec:spectral-endoscopy}, Appendix~\ref{app:l-homs}};
      \node[lf box] (weil) at (0,-2.5)
        {Weil--Heisenberg\\representations\\
         \S\ref{ss:weil-heisenberg}};
      \node[lf box] (parabolic) at (-7.5,-2.5)
        {Parabolic induction\\and lifting\\
         \S\ref{ssec:modular-cuspidal-parabolic-induction}};
    
      \node[lf box] (yu) at (0,-5)
        {Tate cohomology\\of Yu\\representations\\
         \S\ref{section:tate-cohom-yu}};
      \node[lf input] (companions) at (3.75,-5)
        {Cot26\\CF26a};
    
      \node[lf box] (functoriality) at (-3.75,-7.5)
        {Fargues--Scholze\\functoriality\\
         \S\ref{section:fs-functoriality}};
      \node[lf input] (characterization) at (3.75,-10)
        {CF26a, \S10\\Partial\\characterization};
      \node[lf box] (comparison) at (0,-10)
        {Comparison and\\consequences\\
         \S\ref{section:fs-calculation}};
    
      \draw[lf arrow] (tensor.south) -- (weil.north);
      \draw[lf arrow] (weil.south) -- (yu.north);
      \draw[lf arrow] (covers.south) -- (parabolic.north);
      \draw[lf arrow] ([xshift=.75cm]spectral.south) -- (yu.west);
      \draw[lf external] (companions.west) -- (yu.east);
      \draw[lf arrow] ([xshift=-.75cm]yu.south) -- ([xshift=.75cm]functoriality.north);
      \draw[lf arrow] (bounds.east) -- (functoriality.west);
      \draw[lf arrow] (spectral.south) -- (functoriality.north);
      \draw[lf arrow] (parabolic.south) -- ([xshift=-.75cm]functoriality.north);
      \draw[lf arrow] (functoriality.south) -- ([xshift=-.75cm]comparison.north);
      \draw[lf external] (characterization.west) -- (comparison.east);
      \draw[lf arrow] (yu.south) -- (comparison.north);
    \end{tikzpicture}
  }
  \caption{Leitfaden of the paper.}
  \label{fig:cf26b-leitfaden-normal}
\end{figure}

\subsection{AI methodology}

After a draft of this paper was mostly written, AI tools were used (by the second author only) to proofread, reorganize, revise, and create figures. Among the mathematically substantive contributions, GPT 5.2 found and corrected a mistake in the proof of \eqref{eqn:zero-sum}. Later, GPT 5.5 unearthed a gap in the material of \S \ref{ssec:modular-cuspidal-parabolic-induction}, and then filled it by providing the ideas behind Lemmas~\ref{lemma:center-artin-rees} and \ref{lemma:family-to-fiber}. The material of Appendix~\ref{app:kim-yu-ohara-coefficients} consists of a technical extension of existing results in the literature (with roughly the same proofs); a first draft of it was produced (again by the second author) with the help of GPT 5.6. All other statements and proofs in this paper originate from its human authors.

\subsection{Acknowledgements}

We thank Jeff Adler, Ad\`ele Bourgeois, Charlotte Chan, Stephen DeBacker, Jessica Fintzen, Michael Harris, Alex Hazeltine, Alex Ivanov, Josh Lansky, Monica Nevins, Sian Nie, Peter Scholze, David Schwein, Jack Sempliner, Loren Spice, Jay Taylor, and Cheng-Chiang Tsai for helpful conversations. 

S.C.~acknowledges support from the National Science Foundation under Award No.\ 2402231 and the European Research Council (ERC) under the
European Union’s Horizon 2020 research and innovation programme (grant agreement no.\ 950326).

T.F.~was supported by the NSF (grants DMS-2302520 and DMS-2441922), the Simons Foundation, and the Alfred P. Sloan Foundation.

\section{Notation and conventions}\label{ssec:notation-fs}

We will use the terminology of paraductive group schemes from \cite{Cot26b}. We will also use the terminology for Yu's construction from \cite{CF26a}, which differs slightly from the usual conventions; in particular, the twisted Levis in a Yu datum are not required to be pairwise distinct, and the depth $0$ term need not be irreducible. Our notation follows that of \cite{CF26a}, which we repeat below for the reader's convenience. 
\subsection{Local fields}
Let $F$ be a local field with ring of integers $\cO_F$ and residue field $\F_q$ of characteristic $p > 0$. We fix a separable closure $\ol F/F$. A separable algebraic extension of $F$ is always understood as a subfield of $\ol F$. For any finite separable field extension $E/F$ and integer $n \geq 1$, we denote by $E_n$ the unramified subextension of $\ol F/E$ of degree $n$. We write $F^{\unr}\subset \ol F$ for the maximal unramified extension of $F$. The absolute value is normalized by $|a|_F=q^{-v(a)}$, where $v$ is the normalized valuation of $F$.

We write $W_F$ for the Weil group of $F$, $I_F \triangleleft W_F$ for the inertia subgroup, and $P_F \triangleleft W_F$ for the wild inertia subgroup. 

We write $\Fr$ for geometric Frobenius in $W_F/I_F$, and also for a chosen lift to $W_F$.

\subsection{Reductive groups}
We denote by $G$ a connected reductive group over $F$. (Our convention is that reductive groups over fields are not necessarily assumed to be connected.) All representations and characters of $G(F)$ are understood to be smooth unless otherwise stated. We write $\Rep_k^{\sm}G(F)$, or $\Rep_k(G(F))$, for the category of smooth $k$-representations of $G(F)$. If $C \subset G$ is a closed subscheme, then we denote by $Z_G(C)$ (resp.\ $N_G(C)$) the scheme-theoretic centralizer (resp.\ normalizer) of $C$ in $G$.

If $H$ is a group scheme of finite type over a field $k$, then we use $H^\circ$ to denote the identity component of $H$. The notation $H^0$, which is sometimes used instead of $H^\circ$, will be reserved for Yu's construction. We write $\pi_0(H)=H/H^\circ$ for the group scheme of connected components and $H_{\red}$ for the underlying reduced subscheme.

For a connected reductive group $H$, we write $Z(H)$ for its center, $H_{\der}$ for its derived subgroup, $H_{\ad}=H/Z(H)$ for its adjoint quotient, and $H_{\ab}=H/H_{\der}$ for its abelianization. We write $\rk H$ for its absolute rank. For semisimple $H$, $\pi_1(H)$ denotes its algebraic fundamental group. We use $\frg=\Lie(G)$ and $\Lie^*(G)$ for the linear dual of $\Lie(G)$.

If $T$ is an $F$-torus, then $T(F)_{\mathrm{b}}$ denotes the maximal bounded subgroup of $T(F)$.

If $P=MU$ is a parabolic $F$-subgroup of $G$ and $W$ is a smooth representation of $M(F)$ over one of the coefficient rings considered in this paper, then we write
\[
I_P^G(W)\coloneqq \Ind_{P(F)}^{G(F)}(W\otimes\delta_P^{-1/2})
\]
for normalized parabolic induction, where $W$ is inflated to $P(F)$ and $\delta_P$ is the modulus character. We make compatible choices of square roots of $q$ in $\ol\Q_\ell$, $\ol\Z_\ell$, and $\ol\F_\ell$. We use $\ind$ for ordinary induction in finite or finite-index settings and $\cInd$ for compact induction.

\subsection{Dual groups}
The Langlands dual group $\wh{G}$ is regarded over $\Z[1/p]$, though we will often base change to a field of characteristic $\neq p$ without changing the notation. In \S \ref{sec:spectral-endoscopy} we use the notation $\chG$ for the Tannakian dual group, which is the same reductive group as $\wh{G}$ but with a different convention for the Galois action, which can be trivialized upon choosing a square root of the cyclotomic character \cite[\S VI.11]{FS}.

We write $\ld G$ for the L-group of $G$, usually as $\widehat G\rtimes W_0$ for a chosen finite quotient $W_0$ of $W_F$ through which the action on $\widehat G$ factors. If $\rho_1,\rho_2$ are (inertial) L-parameters valued in $\ld G(k)$, then $\rho_1\sim\rho_2$ means that $\rho_1$ and $\rho_2$ are $\wh G(k)$-conjugate. We reserve \emph{equality} of L-parameters for actual equality of chosen representatives by homomorphisms.

\subsection{Bruhat--Tits buildings}
We let $\cB(G)$ denote the extended Bruhat--Tits building of $G$ over $F$. For a point $x \in \cB(G)$, we let $[x] = [x]_G$ denote the corresponding point of $\cB(G_{\der})$. Let $\cG_x$ (resp.\ $\cG_{[x]}$) denote the smooth separated $\cO_F$-group scheme with generic fiber $G$ and $\cG_x(\cO_F) = G(F)_x$ (resp.\ $\cG_{[x]}(\cO_F) = G(F)_{[x]}$), the stabilizer of $x$ (resp.\ $[x]$) in $G(F)$, as in \cite[Remark 8.3.4]{KP}. Note that $\cG_x$ is of finite type, but $\cG_{[x]}$ might not be. Let $\ol G_x$ (resp.\ $\ol G_{[x]}$) denote the quotient of the special fiber of $\cG_x$ (resp.\ $\cG_{[x]}$) by the unipotent radical of its identity component.

For $r\geq 0$, the notations $G(F)_{x,r}$ and $G(F)_{x,r+}$ refer to the Moy--Prasad filtration at $x$; analogous notation is used for the Lie algebra. We implicitly use the normalized valuation $v$ of $F$ in this definition; however, if $E/F$ is a finite separable extension then $G(E)_{x,r}$ and $G(E)_{x,r+}$ will be defined using the unique extension of $v$ to $E$. Thus $G(F)_{x,r} = G(E)_{x,r} \cap G(F)$.

\subsection{Coefficients}\label{sssec:coefficient-conventions-fs}
We use $k$ for coefficient rings, which need not be fields. Our coefficient prime $\ell$ is always distinct from $p$. We write $\ol\Z_\ell$ for the valuation ring of $\ol\Q_\ell$, with residue field $\ol\F_\ell$.

For a $k$-algebra $k'$ and a $k$-module or representation $V$, we write $V_{k'}=V\otimes_k k'$. Subscripts on group schemes and homomorphisms likewise indicate base change.

Let $\Fr_\ell\co\ol\F_\ell\to\ol\F_\ell$ be $a\mapsto a^\ell$. If $V$ is a representation over $\ol\F_\ell$, then $V^{(\ell)} \coloneqq V \otimes_{\ol \F_\ell, \Fr_\ell} \ol \F_\ell$ denotes its Frobenius twist. Thus, for an $\ol\F_\ell$-valued character $\chi$, we have $\chi^{(\ell)}=\Fr_\ell\circ\chi$.

\section{Generalities on L-parameters}\label{sec:generalities-on-l-parameters}

In this section, we establish abstract group-theoretic results which will be applied to L-groups and other groups arising from L-parameters. The main goal is to prove a priori bounds on the prime numbers dividing the orders of the images of irreducible (i.e., supercuspidal or elliptic) L-parameters. This will be necessary later for showing that L-parameters with characteristic $0$ coefficients can be probed effectively using congruences.

If $k$ is a field and $H$ is a smooth affine $k$-group scheme, then we will say that $H$ is reductive provided that its identity component $H^\circ$ is reductive. In other words, we do not require reductive groups to be connected. Similarly, we will say that $H$ is semisimple if $H^\circ$ is semisimple. This generality is meant to accommodate two separate situations: first, in the following sections, we will often wish to consider the ``finite form'' $\ld G = \widehat{G} \rtimes W_0$ of an L-group, where $W_0$ is a finite quotient of the Weil group $W_F$. Second, even when working entirely inside $\widehat{G}$, we will often wish to consider centralizers which may fail to be connected. 

\subsection{Finiteness results in group theory}\label{ssec:group-theory}

We begin with a number of abstract finiteness results which will ultimately be applied to L-parameters. Recall the notion of semisimple homomorphism $\Gamma \to H(k)$ from \cite[\S 10.1]{CF26a}; such a map is semisimple if and only if the image of $\Gamma$ is $H$-cr in the sense of \cite{BMR05}.

\begin{lemma}\label{lemma:finite-subgroups-invertible-order}
Let $k$ be an algebraically closed field, let $H$ be a reductive $k$-group, and let $n$ be a positive integer. Then up to $H^\circ(k)$-conjugacy there are only finitely many finite subgroups of $H(k)$ of order $n$ for which the inclusion map is semisimple. If $n$ is invertible in $k$, then the inclusion map of any finite subgroup of order $n$ is semisimple.
\end{lemma}

\begin{proof}
There are only finitely many isomorphism classes of finite groups of order $n$, so it is enough to show that if $\Gamma$ is one such finite group then there are only finitely many conjugacy classes of semisimple homomorphisms $\Gamma \to H(k)$. Let $\sH$ be the $k$-scheme parameterizing homomorphisms $\Gamma \to H$, so $\sH/\!/H$ is finite by \cite[Theorem 1.1]{Cot24} (taking $H = G$ and $\Gamma' = 1$). This is enough for the first claim since $(\sH/\!/H)(k)$ parameterizes closed orbits in $\sH$, and closed orbits correspond to semisimple homomorphisms by \cite[Corollary 3.7, \S6.3]{BMR05}. 

Now suppose $n$ is invertible in $k$. In this case, the result is \cite[Lemma 2.6]{BMR05}, but we give a slightly different proof for convenience. Every orbit map $H^\circ \to \sH$ is smooth by \cite[Lemma A.1]{DHKM25}, so every $H^\circ$-orbit in $\sH$ is open. Since each orbit is the complement of the union of the other orbits, it follows that every orbit is closed, as desired.
\end{proof}

\begin{lemma}\label{lemma:big-finiteness-claim}
Let $k$ be an algebraically closed field, let $H$ be a reductive $k$-group, and let $n$ be a positive integer. There exists a positive integer $N$ such that for every finite subgroup $\Gamma_0 \subset H(k)$ of order $n$ such that the inclusion map $\Gamma_0 \to H$ is semisimple and every element $h \in H(k)$ normalizing $\Gamma_0$ such that $Z_H(\Gamma_0, h)$ is finite, the group $\langle \Gamma_0, h\rangle$ is finite of order $\leq N$.
\end{lemma}

\begin{proof}
By Lemma~\ref{lemma:finite-subgroups-invertible-order}, there are only finitely many $H^\circ(k)$-conjugacy classes of finite subgroups $\Gamma_0 \subset H(k)$ of order $n$ such that $\Gamma_0 \to H$ is semisimple; thus there is a positive integer $N_0$ such that $N_H(\Gamma_0)/Z_H(\Gamma_0)^\circ_{\mathrm{red}}$ is of order $\leq N_0$ for all such $\Gamma_0$.

Temporarily fix a pair $(\Gamma_0, h)$ as in the lemma statement and let $\Gamma = \langle \Gamma_0, h\rangle$, so $\Gamma$ is a subgroup of $H(k)$ containing $\Gamma_0$ as a normal subgroup. Note that $h$ normalizes the group $Z_H(\Gamma_0)_{\mathrm{red}}^\circ$, which is connected reductive by \cite[Proposition 3.12]{BMR05}, and $(Z_H(\Gamma_0)_{\mathrm{red}}^\circ)^h$ is finite by hypothesis, so \cite[Lemma~10.1.1]{CF26a} implies that $Z_H(\Gamma_0)_{\mathrm{red}}^\circ$ is a torus. In particular, the reduced centralizer $Z_H(\Gamma)_{\red}$ is of order bounded only in terms of $\Gamma_0$ and the image of $h$ in $\pi_0(N_H(\Gamma_0))$ (since the Aut scheme of a torus is \'etale). Since $\pi_0(N_H(\Gamma_0))$ is finite, it follows that there is a positive integer $N_1$ such that the order of $Z_H(\Gamma)$ is $\leq N_1$ for all pairs $(\Gamma_0, h)$.

By the first paragraph, $h^{N_0!}$ centralizes $\Gamma_0$; thus $h^{N_0!} \in Z_H(\Gamma_0, h)$. By the second paragraph, $h^{N_0! N_1!} = 1$. Thus we may take $N = N_0! N_1! n$.
\end{proof}

Lemma~\ref{lemma:big-finiteness-claim} is sufficient for the proof of Theorem~\ref{thm:II-intro-main-thm}, but for the sake of cleaner statements below we will prove a slightly refined finiteness statement when $k$ is of characteristic $0$.

\begin{lemma}\label{lemma:finitely-many-iso-classes}
Let $k$ be an algebraically closed field, and let $d, n \in \Z$. There are only finitely many isomorphism classes of reductive $k$-groups $H$ such that $\dim H \leq d$ and $|\pi_0(H)| \leq n$.
\end{lemma}

\begin{proof}
It is enough to show that there are only finitely many $H$ with $\dim H = d$ and $|\pi_0(H)| = n$. Recall that $H^\circ$ is determined by $(H^\circ)_{\der}$, the maximal central torus $Z_{H^\circ}$, and the intersection $(H^\circ)_{\der} \cap Z_{H^\circ}$. Note that $(H^\circ)_{\der}$ is determined up to isomorphism by its root datum, of which there are only finitely many possibilities; the torus $Z_{H^\circ}$ is determined by its dimension, which is $\leq d$; and the intersection $(H^\circ)_{\der} \cap Z_{H^\circ}$ is a subgroup of the finite multiplicative type $k$-group scheme $Z((H^\circ)_{\der})$, so it admits only finitely many possibilities. This settles the case $n = 1$.

In general, let $(B, T)$ be a Borel-torus pair in $H^\circ$. By \cite[Lemma 3.3]{Cot24} and the conjugacy of Borel-torus pairs in $H^\circ$, there is a subgroup $E \subset N_H(B, T)$ such that $E \cap H^\circ = T[n]$ and $H = H^\circ \cdot E$. If $T[n]$ acts on $H^\circ \times E$ on the right by $(h, x) \cdot t = (ht, t^{-1}x)$, then the multiplication map $(H^\circ \times E)/T[n] \to H$ is a scheme-theoretic isomorphism, under which the multiplication of $H$ is given by
\[
(h_1, x_1) \cdot (h_2, x_2) = (h_1 \cdot {}^{x_1}h_2, x_1x_2).
\]
Thus $H$ is determined by $H^\circ$, $E$, the action of $E$ on $H^\circ$, and the embedding $T[n] \subset E$. There are only finitely many possibilities for $H^\circ$ by the first paragraph. The group $E$ admits only finitely many possible isomorphism classes since its order can be bounded in terms of $\dim H$ and $|\pi_0(H)|$ and its identity component is of multiplicative type. For a given $E$, there are only finitely many actions of $E$ on $H^\circ$ preserving a Borel-torus pair by Lemma~\ref{lemma:finite-subgroups-invertible-order}, since each such action corresponds to a semisimple $k$-homomorphism $E \to \Aut_{H^\circ/k}$. Finally, there are only finitely many embeddings $T[n] \subset E$. This proves the lemma.
\end{proof}

\begin{lemma}\label{lemma:bigger-finiteness-claim}
Let $k$ be a field of characteristic $0$, and let $H$ be a reductive $k$-group. There exists a positive integer $N$ such that for every finite subgroup $\Gamma_0 \subset H(k)$ and every element $h \in H(k)$ normalizing $\Gamma_0$ such that $Z_H(\Gamma_0, h)$ is finite, the group $\langle\Gamma_0, h\rangle$ is finite, and the only prime numbers dividing its order divide $N \cdot |\Gamma_0|$.
\end{lemma}

\begin{proof}
We may and do assume that $k$ is algebraically closed. Since $k$ is of characteristic $0$, Jordan's theorem implies that there exists an integer $N_0$ such that every finite subgroup $\Gamma \subset H(k)$ admits an abelian normal subgroup $\Delta \subset \Gamma$ of index dividing $N_0$. By \cite[Chapter II, Theorem 5.17]{SS70}, there is a maximal $k$-torus $T \subset H$ which is normalized by $\Delta$. If $e$ is the exponent of $(N_H(T)/T)(\ol k)$, then the group of $e$th powers in $\Delta$ is a characteristic subgroup which is contained in $T(k)$. If $N_1 = N_0 e^{\dim T} |(N_H(T)/T)(\ol k)|$, then it follows that each such $\Gamma$ admits an abelian normal subgroup of index dividing $N_1$ which lies in the group of $k$-points of a maximal $k$-torus of $H$.

Next, let $T \subset H$ be a maximal $k$-torus. As $\Delta_0$ ranges over all finite subgroups of $T(k)$, we claim that the isomorphism class of the normalizer $N_H(\Delta_0)/\Delta_0$ ranges over only finitely many reductive $k$-groups. Indeed, $\dim(N_H(\Delta_0)/\Delta_0) \leq \dim H$, and the component group $\pi_0(N_H(\Delta_0)/\Delta_0)$ is a subquotient of the finite group $\pi_0(N_H(Z_H(\Delta_0)^\circ))$. This has cardinality which is bounded independently of $\Delta_0$ because $|\pi_0(N_H(Z_H(\Delta_0)^\circ))| \leq |N_H(T)/T|$, so the claim follows from Lemma~\ref{lemma:finitely-many-iso-classes}.

Finally, let $\Gamma_0 \subset H(k)$ and $h \in H(k)$ be as in the lemma statement, and let $\Gamma = \langle\Gamma_0, h\rangle$. Since $k$ has characteristic $0$, the inclusion $\Gamma_0 \to H$ is semisimple by Lemma~\ref{lemma:finite-subgroups-invertible-order}; hence Lemma~\ref{lemma:big-finiteness-claim}, applied with $n=|\Gamma_0|$, shows that $\Gamma$ is finite. By the first paragraph, there is an abelian normal subgroup $\Delta \subset \Gamma$ of index dividing $N_1$ such that $\Delta$ lies in a maximal torus of $H$. Let $\Delta_0 = \Delta \cap \Gamma_0$, so $\Delta_0$ is an abelian normal subgroup of $\Gamma$, and note that by the second paragraph and Lemma~\ref{lemma:big-finiteness-claim}, there is an integer $N$ depending only on $H$ such that the order of the image of $\Gamma$ in $N_H(\Delta_0)/\Delta_0$ is of order dividing $N$. Thus $\gamma^N \in \Delta_0$ for all $\gamma \in \Gamma$, so the prime numbers dividing the order of $\gamma$ all divide $N \cdot |\Delta_0|$, which divides $N \cdot |\Gamma_0|$.
\end{proof}

\begin{lemma}\label{lemma:finite-subgroup-normalizer-bound}
Let $k$ be a field, and let $H$ be a reductive $k$-group. There exists an integer $N$ such that for every finite subgroup $\Gamma \subset H(k)$ of order prime to $\operatorname{char} k$, we have $|\pi_0(N_H(\Gamma)/\Gamma)| \leq N$.
\end{lemma}

\begin{proof}
We may and do assume that $k$ is algebraically closed, so $H$ is split. Suppose first that $\operatorname{char} k = \ell > 0$, so there is a split reductive group scheme $\cH$ over the ring $W(k)$ of Witt vectors such that $H \cong \cH_k$. The Hom scheme $\underline{\Hom}(\Gamma, \cH)$ is smooth by \cite[Lemma A.1]{DHKM25}, so the inclusion $\Gamma \subset H(k)$ lifts to an inclusion $\Gamma \subset \cH(W(k))$. By \cite[Lemma A.1]{DHKM25} again, each orbit map $\cH \to \underline{\Hom}(\Gamma, \cH)$ is smooth. If $\Omega \subset \underline{\Hom}(\Gamma, \cH)$ is the open subscheme consisting of monic homomorphisms, then $\Aut(\Gamma)$ acts freely on $\Omega$, so the quotient map $\Omega \to \Omega/\Aut(\Gamma)$ is \'etale. It follows that the orbit map $\cH \to \Omega/\Aut(\Gamma)$ through the inclusion $\Gamma \subset \cH(W(k))$ is smooth. Thus $N_{\cH}(\Gamma)$, being a fiber of this map, is smooth over $W(k)$, and hence so is $N_{\cH}(\Gamma)/\Gamma$. Since each fiber of $N_{\cH}(\Gamma)/\Gamma$ has reductive identity component, \cite[Proposition 3.1.3]{Con14} shows that $|\pi_0(N_H(\Gamma)/\Gamma)| \leq |\pi_0(N_{\cH_{W(k)[1/\ell]}}(\Gamma)/\Gamma)|$. Thus we may and do assume $\operatorname{char} k = 0$.

Let $S \subset Z_H(\Gamma)$ be a maximal $k$-torus, and let $M = N_H(S)$, so $\Gamma \subset M(k)$. If $T \subset H$ is a maximal $k$-torus, then $\pi_0(M)$ is a subquotient of $N_H(T)/T$, and in particular its order is bounded independently of $\Gamma$. By Lemma~\ref{lemma:finitely-many-iso-classes}, we may therefore pass from $H$ to $M/S$ and from $\Gamma$ to its image to assume that $Z_H(\Gamma)$, and hence also $N_H(\Gamma)$, is finite.

As in the proof of Lemma~\ref{lemma:bigger-finiteness-claim}, there exists an integer $N_1$ depending only on $H$ such that $N_H(\Gamma)$ admits a normal subgroup $\Delta$ of index $\leq N_1$ which lies in the group of $k$-points of a maximal $k$-torus $T$ of $H$. Observe that the normalizer $N_H(\Delta)$ contains $T$, so its component group is a subquotient of $N_H(T)/T$ and in particular is of order bounded independently of $\Delta$. Since $N_H(\Gamma) \subset N_H(\Delta)$, Lemma~\ref{lemma:finitely-many-iso-classes} and induction on $\dim H$ allow us to pass from $H$ to $N_H(\Delta)$ to assume that $\Delta$ is normal in $H$. It follows that $\Delta$ is central in $H^\circ$, so $N_H(\Delta \cap \Gamma)$ is an open $k$-subgroup scheme of $H$; since $N_H(\Gamma) \subset N_H(\Delta \cap \Gamma)$, we may further assume that $\Delta \cap \Gamma$ is normal in $H$. The map $N_H(\Gamma)/(\Delta \cap \Gamma) \to N_{H/(\Delta \cap \Gamma)}(\Gamma/(\Delta \cap \Gamma))$ is injective, and since $\Gamma/(\Delta \cap \Gamma)$ is of order $\leq N_1$, the result therefore follows from Lemma~\ref{lemma:finite-subgroups-invertible-order} and Lemma~\ref{lemma:finitely-many-iso-classes}.
\end{proof}

\subsection{Irreducible L-parameters} 

Throughout this section, we fix a finite quotient $W_0$ of $W_F$ through which the action on $\widehat{G}$ factors. Let $\ld G$ denote the reductive group $\widehat{G} \rtimes W_0$. Recall the notion of irreducible homomorphism $\Gamma \to H(k)$ from \cite[\S 10.1]{CF26a}; this is slightly less restrictive than the condition that the image of $\Gamma$ in $H(k)$ is $H$-ir in the sense of \cite{BMR05}.

\begin{lemma}\label{lemma:check-irreducibility-mod-ell}If $\rho\colon W_F \to \ld G(\ol{\Z}_\ell)$ is an L-parameter such that $\rho_{\ol{\F}_\ell}$ is irreducible, then $\rho_{\ol{\Q}_\ell}$ is irreducible. Conversely, if $\rho(W_F)$ is finite of order prime to $\ell$ and $\rho_{\ol{\Q}_\ell}$ is irreducible, then $\rho_{\ol{\F}_\ell}$ is irreducible.
\end{lemma}

\begin{proof}
If $\rho_{\ol\Q_\ell}$ is not irreducible, then by definition there is a proper parabolic $\ol\Q_\ell$-subgroup $P_0 \subset \wh G_{\ol\Q_\ell}$ such that $\rho(W_F)$ normalizes $P_0$. Since the scheme of parabolic subgroup schemes of $\wh G$ satisfies the valuative criterion of properness \cite[Corollary 5.2.9]{Con14}, it follows that $\rho(W_F)$ normalizes a parabolic $\ol\Z_\ell$-subgroup scheme $P$ of $\wh G$ and hence also the special fiber $P_{\ol\F_\ell}$, contradicting irreducibility of $\rho_{\ol\F_\ell}$.

For the converse, note that $\rho_{\ol\F_\ell}$ is semisimple by Lemma~\ref{lemma:finite-subgroups-invertible-order}, and $Z_{\wh G}(\rho(W_F))$ is flat by \cite[Lemma A.1]{DHKM25}, so we conclude by \cite[Lemma~10.1.7]{CF26a}.
\end{proof}

\begin{lemma}\label{lemma:semisimple-image}
Let $k \in \{\ol\F_\ell, \ol\Q_\ell\}$. If $\rho\co W_F \to \ld G(k)$ is a semisimple L-parameter, then $\rho(I_F)$ is finite. Moreover, there exists a constant $C$, depending only on $\wh G$ and $|W_0|$, such that if either $k = \ol\Q_\ell$ or $\ell > C$, then $\rho(w)$ is semisimple for all $w \in W_F$.
\end{lemma}

\begin{proof}
Note first that $\rho(P_F)$ is finite and consists of semisimple elements since $\ell \neq p$. If $k = \ol\F_\ell$, then it is clear that $\rho(I_F)$ is finite; by \cite[Lemma 2.2]{DHKM25}, it is enough in the case $k = \ol\Q_\ell$ to show that every element of $\rho(I_F)$ is semisimple, so we reduce to proving the second claim. By Lemma~\ref{lemma:finitely-many-iso-classes} and Lemma~\ref{lemma:finite-subgroup-normalizer-bound}, there is an integer $N$ depending only on $\wh G$ and $|W_0|$ such that
\[
|\pi_0(N_{\ld G}(\rho(P_F))/\rho(P_F))| \leq N.
\]
Let $s \in I_F$ lift a pro-generator of $I_F/P_F$, and let $\rho(s) = xu$ be the Jordan decomposition of $\rho(s)$, where $u$ is unipotent. If $\ell > N$, then since $\ell \neq p$ we have $u \in N_{\wh G}(\rho(P_F))^\circ = Z_{\wh G}(\rho(P_F))^\circ$. By the Hilbert--Mumford criterion, there is a cocharacter $\lambda\co \G_m \to Z_{\wh G}(\rho(P_F))$ such that $\lim_{t \to 0} \lambda(t)u\lambda(t)^{-1} = 1$. Since $\rho|_{I_F}$ is semisimple by \cite[Theorem 3.10]{BMR05}, it follows that $u = 1$, i.e., $\rho(s)$ is semisimple. In particular, $\rho(I_F)$ is finite as above. Precisely the same argument shows that if $\ell > N$ then $\rho(\Fr^n)$ is semisimple if $n \in \Z$ and $\Fr^n \in W_F$ is a lift of the $n$th power of Frobenius in $W_F/I_F$. Since every element of $W_F$ is such a lift, we conclude that $C = N$ works.
\end{proof}

\begin{lemma}\label{lem:semisimple-finite-image}
Let $k \in \{\ol{\F}_\ell, \ol{\Q}_\ell\}$, and let $\rho\colon W_F \to \ld G(k)$ be a semisimple L-parameter. If $G$ is semisimple and $\rho$ is irreducible, then there is a positive integer $B$, depending only on $\widehat{G}$ and $|W_0|$ (not on $F$), such that every prime dividing $|\rho(W_F)|$ either divides $|\rho(I_F)|$ or is at most $B$. In particular, there is a positive integer $N$, depending only on $\widehat{G}$, the set of prime divisors of $|\rho(I_F)|$, and $|W_0|$ (not on $F$), such that every prime dividing $|\rho(W_F)|$ is at most $N$.
\end{lemma}

\begin{proof}
If $k = \ol\F_\ell$, then by flatness of the moduli space of L-parameters \cite[Theorem 4.1 i)]{DHKM25} and Lemma~\ref{lemma:check-irreducibility-mod-ell}, there is some $\wt\rho_0\co W_F \to \ld G(\ol\Z_\ell)$ lifting $\rho$; let $\wt\rho\co W_F \to \ld G(\ol\Q_\ell)$ be a semisimplification of $(\wt\rho_0)_{\ol\Q_\ell}$. By Lemma~\ref{lemma:finitely-many-iso-classes} and Lemma~\ref{lemma:bigger-finiteness-claim}, there is some $N_0$ depending only on $\wh G$ and $|W_0|$ such that the only prime numbers dividing $|\wt\rho(W_F)|$ divide $N_0 \cdot |\wt\rho(I_F)|$. If $\ell > C$, where $C$ is as in Lemma~\ref{lemma:semisimple-image}, then $\ell$ does not divide the order of $|\rho(W_F)|$. Thus one can take $B = \max(N_0, C)$ and $N = \max(N_0, C, D)$, where $D$ is the largest squarefree divisor of $|\rho(I_F)|$.
\end{proof}

The following lemma will allow us to apply Lemma~\ref{lem:semisimple-finite-image} effectively in practice.

\begin{lemma}\label{lemma:inertia-order}
There exists an integer $N$ depending only on $\widehat{G}$, $|W_0|$, and $q$ such that for every prime number $\ell\geq N$, every $k\in\{\ol\Q_\ell,\ol\F_\ell\}$, and every semisimple L-parameter $\rho\colon W_{F_\ell}\to\ld G(k)$, the image $\rho(I_F)$ is of order not divisible by $\ell$.
\end{lemma}

\begin{proof}
Note that given $\widehat{G}$ and $|W_0|$, there are only finitely many possibilities for $\ld G$ up to isomorphism, so we need only show the existence of some $N$ for a fixed $\ld G$. Let $C$ be the constant from Lemma~\ref{lemma:semisimple-image}. Choose an embedding $\iota\colon \ld G \to \GL_n$, and let $N = \max(q^n+1,C+1)$. If $\ell \geq N$, then $q^i \not\equiv 1 \pmod\ell$ for all $1 \leq i \leq n$. This implies $q^{i\ell} \not\equiv 1 \pmod\ell$ for all $1 \leq i \leq n$. If $\rho$ is as above and $\sigma \in I_F$ is a pro-generator of $I_F/P_F$ of pro-order prime to $p$, then $\rho(\sigma)^{q^\ell}$ is conjugate to $\rho(\sigma)$ by the defining relation of $W_{F_\ell}$. If $\alpha_1, \dots, \alpha_n$ are the eigenvalues of $\iota(\rho(\sigma))$, then for each $j$ there is some $1 \leq i \leq n$ such that $\alpha_j^{q^{i\ell}} = \alpha_j$, i.e., $\alpha_j$ is a $(q^{i\ell} - 1)$th root of unity. So the choice of $\ell$ implies that $\ell$ does not divide the order of the semisimple part of $\rho(\sigma)$. Since $N > p$, to show $\ell \nmid |\rho(I_F)|$ it suffices to show that $\rho(\sigma)$ is semisimple. But now this follows from Lemma~\ref{lemma:semisimple-image} after passing from $\ld G$ to $\GL_n$, since $N > n+1$.
\end{proof}

Finally, we record three technical lemmas which will only be used in the proof of Theorem~\ref{thm:unramified-twisted-Levi-functoriality}.

\begin{lemma}\label{lemma:fixed-point-torus-flat}
Let $H$ be a reductive $\ol\Z_\ell$-group scheme, and let $\alpha$ be a $\ol\Z_\ell$-automorphism of $H$ such that $(H^\alpha_{\ol\F_\ell})^\circ_{\red}$ is a torus. Then $(H^\alpha_{\ol\Q_\ell})^\circ$ is a torus, and the schematic closure of $(H^\alpha_{\ol\Q_\ell})^\circ$ in $H$ is a $\ol\Z_\ell$-torus with special fiber $(H^\alpha_{\ol\F_\ell})^\circ_{\red}$.
\end{lemma}

\begin{proof}
Recall from \cite[Theorem 7.1.9]{Con14} that the Aut scheme $\Aut_{H/\ol\Z_\ell}$ has identity component $H^{\ad}$, and the quotient $\Aut_{H/\ol\Z_\ell}/H^{\ad}$ is representable by an \'etale-locally constant $\ol\Z_\ell$-group scheme. Let $C$ be the connected component of $\Aut_{H/\ol\Z_\ell}$ through which the section $\alpha$ factors, so $C$ is a smooth affine $\ol\Z_\ell$-scheme. Let $\alpha_0$ denote a pinning-preserving $\ol\Z_\ell$-automorphism of $H$ which factors through $C$. By \cite[Theorem 1.1(1), (4)]{ALRR22}, the fixed point scheme $H^{\alpha_0}$ is $\ol\Z_\ell$-flat, and for each field $k$ over $\ol\Z_\ell$, the $k$-group scheme $(H^{\alpha_0}_k)^\circ_{\red}$ is connected reductive. By \cite[Proposition 3.4]{PY06}, if $\cZ_0$ is the schematic closure of $(H^{\alpha_0}_{\ol\Q_\ell})^\circ$ in $H$, then the normalization $\wt\cZ_0$ of $\cZ_0$ is smooth. Since the normalization map is finite, the induced map $(\wt\cZ_0)_{\ol\F_\ell}^\circ \to (H^{\alpha_0}_{\ol\F_\ell})^\circ_{\red}$ of smooth $\ol\F_\ell$-group schemes has finite kernel, so it is surjective for dimension reasons, and it follows that $(\wt\cZ_0)_{\ol\F_\ell}^\circ$ is a connected reductive group of the same rank as $(H^{\alpha_0}_{\ol\F_\ell})^\circ_{\red}$. In particular, $\wt\cZ_0$ is a reductive $\ol\Z_\ell$-group scheme by \cite[Proposition 3.1.3]{Con14}, so the rank of $(\wt\cZ_0)_{\ol\F_\ell}$ is equal to the rank of $(\wt\cZ_0)_{\ol\Q_\ell} = (H^{\alpha_0}_{\ol\Q_\ell})^\circ$; call this rank $r$.

By \cite[Remark 5.2.2(1)]{XZ19}, since $(H^\alpha_{\ol\F_\ell})^\circ_{\red}$ is a torus, it follows that $\dim(H^\alpha_{\ol\F_\ell}) = r$. It follows from upper-semicontinuity of fiber dimension that $\dim(H^\alpha_{\ol\Q_\ell}) \leq r$, and the previous paragraph shows $\dim(H^\alpha_{\ol\Q_\ell}) \geq r$ (using \cite[Remark 5.2.2(1)]{XZ19} again). Moreover, $(H^\alpha_{\ol\Q_\ell})^\circ$ is a torus: if $\cZ$ is the schematic closure of $(H^\alpha_{\ol\Q_\ell})^\circ$ in $H$, then $(\cZ_{\ol\F_\ell})^\circ_{\red}$ is an $\ol\F_\ell$-torus since it is a smooth $\ol\F_\ell$-group scheme contained in $(H^\alpha_{\ol\F_\ell})^\circ_{\red}$. By \cite[Expos\'e X, Th\'eor\`eme 8.8]{SGA3II}, it follows that $\cZ$ is a $\ol\Z_\ell$-torus, as desired.
\end{proof}

\begin{lemma}\label{lemma:torus-action-mod-ell}
Let $\rho_1, \rho_2 \co W_F \to \ld G(\ol\Z_\ell)$ be L-parameters, and write $\ol\rho_i$ for the reduction of $\rho_i$ modulo the maximal ideal of $\ol\Z_\ell$. Suppose that
\begin{enumerate}
\item $\rho_1|_{I_F} = \rho_2|_{I_F}$,
\item $Z_{\wh G_{\ol\F_\ell}}(\ol\rho_1(I_F))^\circ_{\red}$ is a torus.
\end{enumerate} 
Then $\ol\rho_1$ and $\ol\rho_2$ induce the same conjugation action of $W_F$ on $Z_{\wh G_{\ol\F_\ell}}(\ol\rho_1(I_F))^\circ_{\red}$ if and only if $\rho_1$ and $\rho_2$ induce the same conjugation action of $W_F$ on $Z_{\wh G_{\ol\Q_\ell}}(\rho_1(I_F))^\circ$.
\end{lemma}

\begin{proof}
By \cite[Corollary 3.7]{BCT24}, the $\ol\Z_\ell$-group scheme $\cZ_P \coloneqq Z_{\wh G}(\rho_1(P_F))$ is smooth affine and its relative identity component $\cZ_P^\circ$ is reductive. Let $s \in I_F$ lift a pro-generator of $I_F/P_F$. Since $I_F$ is topologically generated by $P_F$ and $s$, we have $Z_{\wh G}(\rho_1(I_F)) = (\cZ_P)^{\rho_1(s)}$, and $((\cZ_P)^{\rho_1(s)})^\circ = ((\cZ_P^\circ)^{\rho_1(s)})^\circ$.
Since $Z_{\wh G_{\ol\F_\ell}}(\ol\rho_1(I_F))^\circ_{\red}$ is a torus by hypothesis, Lemma~\ref{lemma:fixed-point-torus-flat} shows that if $\cZ$ denotes the schematic closure of $Z_{\wh G_{\ol\Q_\ell}}(\rho_1(I_F))^\circ$ in $\wh G$ then $\cZ$ is a $\ol\Z_\ell$-torus with special fiber $Z_{\wh G_{\ol\F_\ell}}(\ol\rho_1(I_F))^\circ_{\red}$.
For $w \in W_F$ and $j \in \{1,2\}$, conjugation by $\rho_j(w)$ preserves the generic fiber $Z_{\wh G_{\ol\Q_\ell}}(\rho_1(I_F))^\circ$ and hence also its schematic closure $\cZ$. Since tori have \'etale Aut schemes, the result follows.
\end{proof}

\begin{lemma}\label{lemma:toral-l-parameter-centralizer}
Let $H \subset G$ be an unramified twisted Levi, let $k \in \{\ol\F_\ell, \ol\Q_\ell\}$ for a prime $\ell \neq p$, let $\ld j_{H,G}\co \ld H \to \ld G$ be the canonical L-embedding over $k$ from \cite[Definition~4.5.1]{CF26a}, and let $\rho_H\co W_F \to \ld H(k)$ be a semisimple L-parameter. If $\rho \coloneqq \ld j_{H,G} \circ \rho_H$, then the ranks of $Z_{\wh G}(\rho(I_F))^\circ_{\red}$ and $Z_{\wh H}(\rho_H(I_F))^\circ_{\red}$ are equal. In particular, if $Z_{\wh G}(\rho(I_F))^\circ_{\red}$ is a torus, then
\[
Z_{\wh G}(\rho(I_F))^\circ_{\red} = \ld j_{H,G}(Z_{\wh H}(\rho_H(I_F))^\circ_{\red}).
\]
\end{lemma}

\begin{proof}
We will use $\ld j_{H,G}$ to identify $\wh H$ with a Levi subgroup of $\wh G$. Since the maximal central torus $Z$ of $\wh H$ is contained in $Z_{\wh G}(\rho(P_F))$, we see that $Z_{\wh H}(\rho_H(P_F))^\circ$ is a Levi subgroup of $Z_{\wh G}(\rho(P_F))^\circ$, and thus the ranks of these two groups are the same. By \cite[Lemma~10.1.10]{CF26a}, there is a $\rho_H(I_F)$-stable Borel-torus pair $(\wh B_H, \wh T)$ in $Z_{\wh H}(\rho_H(P_F))^\circ$, and $(\wh T^{\rho_H(I_F)})^\circ_{\red}$ is a maximal torus of $Z_{\wh H}(\rho_H(I_F))^\circ_{\red}$. Since $H$ is an \emph{unramified} twisted Levi, there is some $\rho(I_F)$-stable parabolic $\wh P \subset Z_{\wh G}(\rho(P_F))^\circ$ with Levi $Z_{\wh H}(\rho(P_F))^\circ$; if $\wh U$ is the unipotent radical of $\wh P$, then $\wh B \coloneqq \wh B_H \wh U$ is a $\rho(I_F)$-stable Borel subgroup of $Z_{\wh G}(\rho(P_F))^\circ$ containing $\wh T$, so another application of \cite[Lemma~10.1.10]{CF26a} shows that $(\wh T^{I_F})^\circ_{\red}$ is a maximal torus of $Z_{\wh G}(\rho(I_F))^\circ_{\red}$. Thus $Z_{\wh H}(\rho_H(I_F))^\circ_{\red}$ and $Z_{\wh G}(\rho(I_F))^\circ_{\red}$ have the same rank. The final claim follows immediately from this.
\end{proof}




\section{Spectral endoscopy and L-homomorphisms}\label{sec:spectral-endoscopy}

In this section, we compute the $\sigma$-dual homomorphism $\ld\psi\co \ld H \to \ld G$ constructed in \cite[Theorem~7.2.1]{F24} in some cases. We will use the Tannakian dual notation $\chG,\chH,\chT$ from \S\ref{ssec:notation-fs}. This is the same underlying reductive group as the corresponding $\widehat{G}$-notation, but with the geometric $c$-group convention for the Galois action.

\subsection{Recollections on modular functoriality}\label{ssec:modular-functoriality} Throughout this section, we work under the following ansatz: 

\begin{ansatz}\label{ansatz:modular-functoriality}
Let $G$ be a connected reductive group over a local field $F$ of residue characteristic $p$. Let $\sigma \in \Aut(G)$ be an automorphism of prime order $\ell \neq p$ whose fixed point subgroup $G^\sigma =: H$ is connected reductive.
\end{ansatz}

\subsubsection{The $\sigma$-dual homomorphism}\label{sssec:sigma-dual-homomorphism}

We consider the dual groups $\lc H_k, \lc G_k$ over $k \coloneqq \ol \F_\ell$. In the situation of Ansatz \ref{ansatz:modular-functoriality}, \cite[Theorem~7.2.1]{F24} constructs a canonical conjugacy class of maps of $c$-groups (see also \cite[\S 7.8]{TVarxiv})
\begin{equation}\label{eq:sigma-dual-homomorphism}
\lc \psi \co \lc H_k \rightarrow \lc G_k.
\end{equation}
called the ``$\sigma$-dual homomorphism''. To relate this to the more standard formulations of the Local Langlands Correspondence in terms of the L-group, we \emph{choose a square root of the residue field cardinality}\footnote{This choice will usually not matter, but it does intervene at one point: the compatibility of the formation of $\ld \psi$ with base change, cf.\ \cite[Remark~5.2.3]{CF26a}.} of $F$ in $k$ in order to specify a square root of the mod-$\ell$ cyclotomic character $\chi_{\cyc} \co W_F \rightarrow k^\times$. This choice determines isomorphisms $\lc H_k \cong \ld H_k$ and $\lc G_k \cong \ld G_k$ by \cite[Lemma 5.5.7]{Zhu17a}, allowing us to view the $\sigma$-dual homomorphism instead as a homomorphism
\begin{equation}\label{eq:l-group-sigma-dual-homomorphism}
\ld \psi \co \ld H_k \rightarrow \ld G_k.
\end{equation}



\begin{remark}
Without restriction on $\ell$, but under the assumption that $G$ is simply connected and $H$ is semisimple\footnote{The latter condition rules out the main use case in this paper, namely the case that $H$ is an unramified twisted Levi subgroup of $G$.}, a $\sigma$-dual homomorphism is constructed in \cite[\S 1.3]{TV}, except possibly in cases where $(\Lie G, \Lie H)$ contains a factor of the form $(\mf e_6, \mf{sl}_3^3)$ or $(\mf e_6, \mf{sl}_6 \times \mf{sl}_2)$ or $(\mf e_6, \mf{sp}_8)$. The approach of \cite{TV} is orthogonal to that of \cite{F24}, and it has not yet been checked that in cases of overlap, the two $\sigma$-dual homomorphisms agree. This would be of independent interest to check, but it has no relevance to this paper. 
\end{remark}

The following theorem is a special case of \cite[Theorem~9.1.1]{F24}, as explained in \cite[Example~9.1.2]{F24}. It proves the functoriality conjecture of Treumann--Venkatesh \cite[Conjecture 6.3]{TV} in terms of the Fargues--Scholze correspondence.

\begin{thm}[Modular functoriality]\label{thm:modular-functoriality}\label{hyp:modular-functoriality}
Let $\Pi$ be an irreducible smooth representation of $G(F)$ over $k$ whose isomorphism class is fixed by $\sigma$, so that the $G(F)$-action on $\Pi$ extends uniquely to a $G(F) \rtimes \langle\sigma\rangle$-action \cite[Proposition 6.1]{TV}. Then for each $j \in \Z/2\Z$ and every $H(F)$-irreducible subquotient $\pi$ of $\rT^{j}(\Pi)$, the $\sigma$-dual homomorphism satisfies
\[
\rho^{\FS}(\Pi^{(\ell)})\sim\bigl(\ld\psi\circ\rho^{\FS}(\pi)\bigr)^{\ss}.
\] 
\end{thm}

\subsection{Spectral endoscopic groups} Consider the special case of \S \ref{ssec:modular-functoriality} where $\sigma \in \Aut(G)$ is inner, say of the form $g \mapsto sgs^{-1}$ for $s \in G(F)$. Thus $H = Z_G(s)$. Then we will call $\ld H$ a \emph{spectral endoscopic group} for $\ld G$. The name comes from the resemblance to endoscopy, in which we would say that $H$ is an endoscopic group for $G$ if $\chH$ is the centralizer of a semisimple element of $\chG$. Thus, spectral endoscopy is essentially endoscopy with the roles of the group and its dual group reversed.

But \emph{unlike} in endoscopy, there is \emph{also} a canonical conjugacy class of morphisms from the spectral endoscopic group to $\ld G$, namely the $\sigma$-dual homomorphism of \eqref{eq:sigma-dual-homomorphism}. This homomorphism typically does not lift to characteristic $0$ (or even to Artinian rings); it is a miracle of modular arithmetic.

\begin{example}\label{ex:spectral-endoscopy}
Let $a,b \geq 1$ be integers and assume $p \neq 2$. We define the symplectic group $\Sp_{2n}/F$ by
\[
J_n=\begin{pmatrix}0&\Id_n\\-\Id_n&0\end{pmatrix},\qquad
\Sp_{2n}(R)=\{g\in\GL_{2n}(R):{}^t gJ_ng=J_n\}
\]
for every $F$-algebra $R$. 

Let $V=F^{2a+2b}$, with ordered basis
$e_1,\ldots,e_{a+b},f_1,\ldots,f_{a+b}$ having Gram matrix $J_{a+b}$. Set
$W_a=\langle e_1,\ldots,e_a,f_1,\ldots,f_a\rangle$ and
$W_b=\langle e_{a+1},\ldots,e_{a+b},f_{a+1},\ldots,f_{a+b}\rangle$,
so that $V=W_a \perp W_b$. Let $G=\Sp(V)\cong \Sp_{2a+2b}$, and let $s$ act as $\Id$ on $W_a$ and $-\Id$ on $W_b$. In this basis, we have
\[
s = \begin{pmatrix} \Id_a \\ & - \Id_b & \\ & &\Id_a \\ & & &-\Id_b \end{pmatrix}.
\]
Since $p\neq 2$, the subspaces $W_a$ and $W_b$ are the two eigenspaces of $s$. Consequently, an element of $G$ commutes with $s$ if and only if it preserves both subspaces. 
Thus conjugation by $s$ gives an order $\ell=2$ automorphism $\sigma \in \Aut(G)$ such that
$G^\sigma=H=\Sp_{2a}\times\Sp_{2b}$. In this case, the $\sigma$-dual homomorphism is a map
\[
\SO_{2a+1} \times \SO_{2b+1} \rightarrow \SO_{2a+2b+1}
\]
which is finite onto its image. The traditional wisdom is that no such map exists, but recall that we are only looking for such a map in characteristic $\ell=2$. Here, there is indeed an ``exceptional'' homomorphism of the desired shape.

To see this, recall that $\SO_{2a+1}$ is the group of determinant-$1$ isometries of the standard split quadratic form $Q_a = x_0^2 + x_1x_2 + \ldots + x_{2a-1}x_{2a}$ on a vector space $V_a$ of dimension $2a+1$, while $\SO_{2b+1}$ is the group of determinant-$1$ isometries of the standard split quadratic form $Q_b = y_0^2 + y_1y_2 + \ldots + y_{2b-1}y_{2b}$ on a vector space $V_b$ of dimension $2b+1$. The sum of these quadratic forms has rank $2a+2b+2$ in characteristics other than $2$, but in characteristic $2$ the quadratic form $Q_a + Q_b$ on $V_a \oplus V_b$ has a one-dimensional radical (cut out by the conditions $x_0 = y_0$ and the vanishing of the other coordinates), and the action on this quotient induces an injection
\begin{equation}\label{eq:endoscopy-paradox}
\SO_{2a+1} \times \SO_{2b+1} \inj \SO_{2a+2b+1}.
\end{equation}
Finally, the connected part of the $\sigma$-dual homomorphism $\ld \psi \co \chH \rightarrow \chG$ is not \eqref{eq:endoscopy-paradox} but rather its composition with the $\ell$-power Frobenius $\Fr_\ell$ on $\chH$.
\end{example}

\subsection{Properties of the $\sigma$-dual homomorphism} The $\sigma$-dual homomorphism is constructed in a highly indirect manner in \cite[\S 7.2]{F24}. We will need to establish more explicit information about it.

\subsubsection{Compatibility with field extensions}

\begin{lemma}\label{lemma:dual-hom-base-change}
Let $E/F$ be a finite separable field extension. If $\sigma$ is an $F$-automorphism of $G$ of order $\ell$, $H = G^\sigma$ is connected, and $\lc\psi\co \lc H \to \lc G$ is the $\sigma$-dual homomorphism, then
\begin{equation}\label{eq:lc-psi-base-change}
\lc\psi|_{\chH \rtimes W_E} \sim \lc\psi_E.
\end{equation}
Consequently, $\ld\psi\co \ld H \to \ld G$ satisfies
\begin{equation}\label{eq:ld-psi-base-change}
\ld\psi|_{\wh H \rtimes W_{E_2}} \sim \ld\psi_E|_{\wh H \rtimes W_{E_2}}.
\end{equation}

\end{lemma}

\begin{proof}The compatibility for $\lc \psi$ is \cite[Lemma~7.3.1]{F24}. Next, recall from the discussion of \S \ref{sssec:sigma-dual-homomorphism} that the passage from $\lc\psi$ to $\ld \psi$ depends on the choice of square root of the residue cardinality, used to specify a square root of the mod-$\ell$ cyclotomic character. This choice introduces an ambiguity of sign that is killed by squaring, hence disappears after restricting from $W_E$ to $W_{E_2}$, which implies \eqref{eq:ld-psi-base-change}.
\end{proof}

\begin{remark}
It is \textit{not} generally true that $\ld j_{T,G}|_{W_E}\sim\ld j_{T_E,G_E}$ for a finite extension $E/F$, even if $E/F$ is unramified; see \cite[Remark~5.2.3]{CF26a}.
\end{remark}

\subsubsection{Compatibility with Frobenius}\label{sssec:Brauer-for-trivial-action}
We require a preliminary observation about the \emph{(normalized) Brauer functor},
\[
\wt{\cbr} \co \Sat(\cHck_{G, \Div_X^1};k) \rightarrow \Sat(\cHck_{H, \Div_X^1};k),
\]
from \cite[\S\S 6.3--6.5]{F24}. Here the \emph{Satake category} $\Sat(\cHck_{G, \Div_X^1};k)$ is defined as in \cite[Definition VI.7.8]{FS}.

Let an abstract cyclic group $\Sigma$ of order $\ell$ act trivially on $G$, so any generator $\sigma \in \Sigma$ has trivial image in $\Aut(G)$. This is allowed in the conventions of \cite[Sections 2.2--2.3]{F24}, which require an action of $\Sigma$ but do not require it to be faithful. Indeed, the inclusion of fixed points is nothing but the identity map. The Smith-Treumann localization for the identity map is simply projection to the Tate category, which admits an obvious perverse t-exact lift by the identity functor. Hence the normalized Brauer functor is intertwined under Geometric Satake with the functor
\[
V \mapsto \rT^0(V^{\otimes \ell})\cong V^{(\ell)}
\]
as in \cite[Example~4.5.3]{F24}. In other words, this is pullback along the $L$-homomorphism
\[
\Fr_\ell\co\lc G\longrightarrow\lc G,\qquad (g,w)\longmapsto(\Fr_\ell(g),w).
\]
Transporting from the $c$-group to the L-group using \eqref{eq:c-to-L-coordinate-change} gives the formula of \cite[Lemma~8.5.3 and Remark~8.5.4]{F24}:
\begin{equation}\label{eq:fs-coefficient-frobenius}
\rho^{\FS}(\pi^{(\ell)})
\sim z \cdot\Fr_\ell\circ\rho^{\FS}(\pi)
\end{equation}
where the cocycle $z \co W_F \rightarrow Z(\wh G_{\der})(k)^{W_F}$ is given by 
\[
z(w) =(2\rho_{\wh G}^{\vee})\bigl(\chi_{\cyc}^{1/2} (w)^\ell/\chi_{\cyc}^{1/2} (w)\bigr),
\]
where $\chi_{\cyc}^{1/2} \co W_F\to k^\times$ is the chosen unramified square root of the cyclotomic character, and $\rho_{\wh G}^{\vee}$ is the half sum of the positive coroots of $\wh G$.

\subsubsection{Compatibility with passage to Levis}
The proof of \cite[Lemma~6.5.2]{F24} can be repeated essentially verbatim for more general constant terms associated to a $\sigma$-stable pair of parabolic subgroup and Levi. We record the result for convenient reference.

\begin{lemma}\label{lem:parabolic-CT-compatible-br}
Let $M$ be an $F$-rational Levi subgroup of $G$ and let $P$ be an $F$-rational parabolic subgroup of $G$ with Levi $M$. Assume that $M$ and $P$ are $\sigma$-stable. Then $P^\sigma$ is a parabolic subgroup of $H$ with Levi subgroup $M^\sigma$. There is a commutative square of symmetric monoidal functors, 
\begin{equation}\label{eq:parabolic-CT-compatible-br}
\begin{tikzcd}
\Sat(\cHck_{G, \Div_X^1};k) \ar[r, "\wt{\cbr}"] \ar[d, "{\CT_P[\deg_P]}"']  &  \Sat(\cHck_{H, \Div_X^1};k)  \ar[d, "{\CT_{P^\sigma}[\deg_{P^\sigma}]}"] \\
\Sat(\cHck_{M, \Div_X^1};k) \ar[r, "\wt{\cbr}"]  &  \Sat(\cHck_{M^\sigma, \Div_X^1};k) 
\end{tikzcd}
\end{equation}
Here, if $U_P$ denotes the unipotent radical of $P$, then $\deg_P(\nu)=\langle 2\rho_{U_P},\nu\rangle$, and $\deg_{P^\sigma}$ is defined similarly.
\end{lemma}

\begin{lemma}\label{lem:H-in-Levi}
Let $M$ be an $F$-rational Levi subgroup of $G$ and $P$ be an $F$-rational parabolic subgroup of $G$ with Levi $M$. Suppose $\sigma$ is an inner automorphism of $G$ of order $\ell$ stabilizing $(P,M)$, and put $H=G^\sigma$. Let $\lc \psi_G \co \lc H \rightarrow \lc G$ and $\lc \psi_M \co \lc M^\sigma \rightarrow \lc M$ be the associated $\sigma$-dual homomorphisms. Then
\[
\lc\psi_G\circ\lc j_{M^\sigma,H}\sim\lc j_{M,G}\circ\lc\psi_M.
\]
Here $\lc j_{M,G}$ and $\lc j_{M^\sigma,H}$ are the standard embeddings from Proposition~\ref{prop:constant-term-split-c-map}. In particular, if $H\subset M$, then $M^\sigma=H$ and
\[
\lc\psi_G\sim\lc j_{M,G}\circ\lc\psi_M.
\]
\end{lemma}

\begin{proof}
Apply Tannakian reconstruction directly to \eqref{eq:parabolic-CT-compatible-br}.
By definition,
\begin{itemize}
\item $\wt{\cbr}_G$ is Tannakian dual to $\lc \psi_G^*$,
\item $\wt{\cbr}_M$ is Tannakian dual to $\lc \psi_M^*$, and
\item the two constant term functors are Tannakian dual to $\lc j_{M,G}$ and $\lc j_{M^\sigma,H}$, by Proposition~\ref{prop:constant-term-split-c-map}.
\end{itemize} 
Therefore, passing to the Tannakian dual of \eqref{eq:parabolic-CT-compatible-br} gives the result.
\end{proof}

\begin{example}\label{ex:H=Levi}
Suppose in the setting of Lemma~\ref{lem:H-in-Levi} that $H = M$. Then \S\ref{sssec:Brauer-for-trivial-action} identifies $\wt{\cbr}_M \cong \Fr_\ell^*$, so we obtain
\[
\lc \psi_G \sim \lc j_{M,G} \circ \Fr_\ell = \Fr_\ell \circ \lc j_{M,G}.
\]
\end{example}

\subsection{L-embeddings of unramified twisted Levi subgroups}

Recall that if $H$ is an unramified twisted Levi, then there is a canonical L-embedding
\begin{equation}\label{eqn:L-embedding-of-unram-twisted-Levi}
\ld j_{H,G}\co \ld H \rightarrow \ld G
\end{equation}
defined using the minimally ramified set of $\chi$-data, as in \cite[Definition~4.5.1]{CF26a}.

\begin{prop}\label{prop:sigma-dual-unram-agreement-on-inertia}
Let $\sigma$ be an inner automorphism of $G$ of order $\ell$, and suppose that $H = G^\sigma$ is an unramified twisted Levi $F$-subgroup of $G$. Let $\ld\psi\co \ld H \to \ld G$ be the associated $\sigma$-dual homomorphism. Then
\begin{enumerate}
    \item $\ld\psi|_{\wh H \rtimes I_F} \sim \Fr_\ell \circ \ld j_{H,G}|_{\wh H \rtimes I_F}$.
    \item There exists a cocycle $z\co W_F/I_F \to Z(\wh H \cap \wh G_{\der})(k)^{I_F}$ such that
    \[
    \ld\psi \sim z \cdot \Fr_\ell \circ \ld j_{H,G}.
    \]
\end{enumerate}
\end{prop}

\begin{proof}
If $E/F$ is a finite unramified field extension and $\ld\psi_E$ is the $\sigma$-dual homomorphism corresponding to the $E$-automorphism $\sigma_E$ of $G_E$, then Lemma~\ref{lemma:dual-hom-base-change} shows that
\begin{equation}\label{eqn:unram-levi-l-emb-bc}
    \ld\psi|_{\wh H \rtimes I_F} \sim \ld\psi_E|_{\wh H \rtimes I_F}.
\end{equation}
If $E$ is chosen such that $H_E$ is a Levi $E$-subgroup of $G_E$, then Lemma~\ref{lem:H-in-Levi}, Example~\ref{ex:H=Levi}, and \cite[Lemma~4.5.2(1)]{CF26a} show that\footnote{Below, we implicitly use that the translation between the c-group and L-group inserts an unramified modulus character, which restricts trivially to inertia. See Lemma~\ref{lem:transport-c-to-L}, Remark~\ref{rem:L-group-CT-dual}, and Example~\ref{ex:H=Levi}.}
\begin{equation}\label{eqn:honest-levi-l-emb}
    \ld\psi_E|_{\wh H \rtimes I_F} \sim \Fr_\ell \circ \ld j_{H_E, G_E}|_{\wh H \rtimes I_F}.
\end{equation}
Combining \eqref{eqn:unram-levi-l-emb-bc} and \eqref{eqn:honest-levi-l-emb} with \cite[Lemma~5.2.2]{CF26a} yields (1).

If $G$ is a torus, then (2) holds by \S\ref{sssec:Brauer-for-trivial-action}, since the $c$-group and L-group are canonically isomorphic in this case. By \cite[Lemma~7.4.1]{F24}, the formation of $\lc\psi$ is compatible with central isogenies; the isomorphism between the $c$-group and L-group is functorial for them (independently of the choice of $\sqrt{q}$), so the same holds for $\ld\psi$.

Put $Z=Z(G)^\circ_{\red}$ and consider the $\sigma$-equivariant central isogeny $G_{\mathrm{sc}}\times Z\to G$. The induced action of $\sigma$ on $G_{\mathrm{sc}}$ has connected fixed subgroup by \cite[Theorem 8.1]{St68}, and acts trivially on $Z$, so $G_{\mathrm{sc}}^\sigma$ and $Z^\sigma$ are connected. By compatibility of the $\sigma$-dual homomorphisms with central isogenies and \cite[Lemma~4.5.2(2)]{CF26a}, we reduce to the case $G=G_{\mathrm{sc}}\times T$ for a torus $T$. Similarly, \cite[Lemma~7.3.3]{F24} shows that $\lc\psi$ is compatible with products, and the isomorphism of $c$-groups and L-groups is similarly compatible. \cite[Lemma~4.5.2(3)]{CF26a} gives the same compatibility for $\ld j_{H,G}$. Thus by the settled torus case we may assume that $G$ is semisimple, where (2) follows from (1).
\end{proof}

In order to apply Proposition~\ref{prop:sigma-dual-unram-agreement-on-inertia} in practice, we require the following technical result to guarantee that $G^\sigma$ is connected if the prime $\ell$ is suitably large.

\begin{lemma}\label{lem:centralizer-connected}
Let $t \in G(F)$ be a semisimple element of finite order $n$. If $n$ is coprime to $\#\pi_1(G_{\der})$, then its centralizer is connected reductive. In particular, if $\gcd(n, \#\pi_1(G_{\ad})) = 1$, then $G^\sigma$ surjects onto $(G_{\ad})^\sigma$.
\end{lemma}

\begin{proof}
The first statement is standard and is proven in, for instance, \cite[Lemma 2.14(b), Corollary 2.16(c)]{St75}. The second statement follows for dimension reasons from the first.
\end{proof}

\section{Calculating Tate cohomology of tensor products}\label{section:tate-tensor-products}

Throughout this section, let $\ell$ be a prime, let $k$ be a field of characteristic $\ell$, and let $\sigma$ generate a cyclic group of order $\ell$. Write $k[\sigma]$ for the group algebra of $\langle\sigma\rangle$, and set
\[
N_\sigma \coloneqq 1+\sigma+\cdots+\sigma^{\ell-1}=(\sigma-1)^{\ell-1}.
\]
For a $k[\sigma]$-module $M$, write
\[
\rT^0(\sigma, M)\coloneqq\frac{\ker(\sigma-1\mid M)}{N_\sigma M},
\qquad
\rT^1(\sigma, M)\coloneqq\frac{\ker(N_\sigma\mid M)}{(\sigma-1)M},
\]
where the indices are taken in $\Z/2\Z$. We will write $\rT^i(M)$ in place of $\rT^i(\sigma, M)$ if $\sigma$ is clear from context. All tensor products in this section carry the diagonal $\sigma$-action. If $M$ is a representation of $\Gamma\rtimes\langle\sigma\rangle$, then $\rT^i(M)$ is naturally a representation of $\Gamma^\sigma$.

We are interested in the behavior of Tate cohomology with respect to tensor products. Unfortunately, cup product maps in Tate cohomology are often trivial.

\begin{example}\label{ex:vanishing-cup-product}
Let $\ell = 3$, and let $V = W = k[\sigma]/(\sigma-1)^2 = ke_1 \oplus ke_2$, where $e_2 = 1$ and $e_1 = \sigma - 1$. We claim that the map $\smile\co \rT^0(V) \otimes \rT^0(W) \to \rT^0(V \otimes_k W)$ is trivial. To see this, note first that $\rT^0(V) = \rT^0(W) = ke_1$. Next, note that $V \otimes_k W$ has basis $e_{i,j}$ for $i,j \in \{1,2\}$, where
\[
\begin{aligned}
    (\sigma-1)e_{1,1} &=0,\\
    (\sigma-1)e_{1,2} &=e_{1,1},\\
    (\sigma-1)e_{2,1} &=e_{1,1},\\
    (\sigma-1)e_{2,2} &=e_{1,1} + e_{1,2} + e_{2,1}.
\end{aligned}
\]
From this, we see that $(V \otimes_k W)^\sigma$ has basis $\{e_{1,1}, e_{1,2} - e_{2,1}\}$, while $N_\sigma(V \otimes_k W) = ke_{1,1}$. The map $\smile\co \rT^0(V) \otimes_k \rT^0(W) \to \rT^0(V \otimes_k W)$ is induced by the map $v \otimes w \mapsto v \otimes w$, so indeed it vanishes.
\end{example}

Recall from \cite[Definition~3.5.3]{Cot26b} that a finitely generated $k[\sigma]$-module is \emph{minimal} if it is isomorphic to $k^{\oplus r}\oplus k[\sigma]^{\oplus s}$, and it is \emph{maximal} if it is isomorphic to
\[
\bigl(k[\sigma]/((\sigma-1)^{\ell-1})\bigr)^{\oplus r}\oplus k[\sigma]^{\oplus s}
\]
for some $r,s\geq0$. It is \emph{extremal} if it is minimal or maximal. We will see below (Corollary~\ref{cor:cup-product-extremal-cases}) that cup products interact with extremal $k[\sigma]$-modules in a rather simple way. In particular, because the modules $V$ and $W$ in Example~\ref{ex:vanishing-cup-product} are maximal, the map $\smile\co \rT^1(V) \otimes_k \rT^1(W) \to \rT^0(V \otimes_k W)$ is injective; this can of course be checked directly as well. Since $\rT^0(V)$ and $\rT^1(V)$ have isomorphic semisimplifications \cite[Lemma~3.1.2]{Cot26b}, this is enough in applications.

\subsection{Cup products}

Recall that if $V$ and $W$ are $k[\sigma]$-modules, then there exist natural cup product maps 
\[
\smile_{i,j}\co\rT^i(V) \otimes_k \rT^j(W) \to \rT^{i+j}(V \otimes_k W)
\]
for $i, j \in \Z/2\Z$. If $ij = 0$, then $\smile_{i,j}$ is given by $v \otimes w \mapsto v \otimes w$, while on the other hand $\smile_{1,1}$ is given by $v \otimes w \mapsto \sum_{0 \leq a < b \leq \ell} \sigma^av \otimes \sigma^bw$; see \cite[XII.7, page 252]{CE56}. The following lemma is elementary and combinatorial in nature but the details are surprisingly complicated.

\begin{lemma}\label{lemma:first-tate-cohom-tensor}
Let $V$ and $W$ be finitely generated $k[\sigma]$-modules, and let $1 \leq m_1, \dots, m_a \leq \ell$ and $1 \leq n_1, \dots, n_b \leq \ell$ be the sizes of the Jordan blocks of $\sigma$ on $V$ and $W$, respectively. Then
\begin{enumerate}
    \item $\dim_k \Ima(\smile_{0,1}) = \#\{(i, j) \in [1,a] \times [1,b]\co m_i \leq n_j < \ell\}$,
    \item $\dim_k \Ima(\smile_{0,0}) = \#\{(i, j) \in [1,a] \times [1,b]\co m_i + n_j \leq \ell\}$,
    \item $\dim_k \Ima(\smile_{1,1}) = \#\{(i, j) \in [1,a] \times [1,b]\co m_i + n_j \geq \ell, \max(m_i, n_j) < \ell\}$.
\end{enumerate}
In particular, if $\rT^0(V) \neq 0$ and $\rT^0(W) \neq 0$, then $\rT^0(V \otimes_k W) \neq 0$, $\rT^1(V \otimes_k W) \neq 0$, one of the maps $\smile_{0,1}$ and $\smile_{1,0}$ is nonzero, and one of $\smile_{0,0}$ and $\smile_{1,1}$ is nonzero.
\end{lemma}

\begin{proof}
By passing to direct summands of $V$ and $W$, we may assume that $V$ and $W$ are indecomposable $k[\sigma]$-modules, say of $k$-dimensions $m$ and $n$, respectively. If $m = \ell$ or $n = \ell$, then $\rT^i(V) = 0$ or $\rT^i(W) = 0$ for all $i$ by \cite[Lemma~3.1.3]{Cot26b}, in which case the claims are clear. So assume $m < \ell$ and $n < \ell$. Let $e_1, \dots, e_m$ be a $k$-basis for $V$, and let $f_1, \dots, f_n$ be a $k$-basis for $W$, such that $(\sigma - 1)e_i = e_{i-1}$ and $(\sigma - 1)f_i = f_{i-1}$ for all $i$, where by convention $e_0 = 0$ and $f_0 = 0$. Note that then $\rT^0(V) = ke_1$, $\rT^0(W) = kf_1$, $\rT^1(V) = V/(\bigoplus_{i=1}^{m-1} ke_i)$, and $\rT^1(W) = W/(\bigoplus_{i=1}^{n-1} kf_i)$. If $g_{i,j} \coloneqq e_i \otimes f_j$, then we have
\begin{equation}\label{eqn:g-relation}
(\sigma - 1)g_{i,j} = g_{i-1,j-1} + g_{i-1,j} + g_{i,j-1},
\end{equation}
where again by convention $g_{i,j} = 0$ if either $i\leq 0$ or $j\leq 0$. 



We will analyze the three relevant cup products separately. By the above, the map $\smile_{0,1}$ is nonzero if and only if the class of $g_{1,n}=e_1\otimes f_n$ is nonzero in
\[
\rT^1(V\otimes_k W)=\ker(N_\sigma\mid V\otimes_k W)/(\sigma-1)(V\otimes_k W).
\]
Similarly, $\smile_{0,0}$ is nonzero if and only if $g_{1,1}\notin N_\sigma(V\otimes_k W)$, and $\smile_{1,1}$ is nonzero if and only if $\sum_{0\leq i<j\leq \ell}\sigma^i e_m\otimes \sigma^j f_n \notin N_\sigma(V\otimes_k W)$. We must prove that:
\begin{enumerate}
    \item $g_{1,n}$ has nonzero image in $\rT^1(V \otimes_k W)$ if and only if $m \leq n$,
    \item $g_{1,1}$ has nonzero image in $\rT^0(V \otimes_k W)$ if and only if $m + n \leq \ell$,
    \item $\sum_{0 \leq i < j \leq \ell} \sigma^ie_m \otimes \sigma^j f_n$ has nonzero image in $\rT^0(V \otimes_k W)$ if and only if $m + n \geq \ell$.
\end{enumerate}
Note that the final claims of the lemma follow immediately from (1), (2), (3), and \cite[Lemma~3.1.3]{Cot26b}.

We first prove (1). As a special case of (\ref{eqn:g-relation}), note that
\[
(\sigma - 1)(g_{i, 1}) = g_{i-1,1}
\]
for $1 \leq i \leq m$, so $g_{i,1} \in (\sigma-1)(V \otimes_k W)$ for $1 \leq i \leq m-1$. Similarly, we have
\[
(\sigma - 1)(g_{i,2}) = g_{i-1,1} + g_{i-1,2} + g_{i,1},
\]
so $g_{i,2} \in (\sigma - 1)(V \otimes_k W)$ for $1 \leq i \leq m-2$. Arguing inductively in this way, we find that $g_{i,j} \in (\sigma - 1)(V \otimes_k W)$ for $1 \leq i \leq m-j$. If $m > n$, then it follows that $g_{1,n}$ is trivial in $\rT^1(V \otimes_k W)$, as desired; so suppose for the remainder of the proof that $m \leq n$.

For $1 \leq d \leq m+n-1$, let $S_d$ denote the $k[\sigma]$-submodule of $V \otimes_k W$ spanned as a $k$-vector space by the elements $g_{i,j}$ with $i+j-1 \leq d$. For $d \leq 0$, we set $S_d =0$. It is clear from \eqref{eqn:g-relation} that $(\sigma-1)S_{d+1} \subset S_d$ for all $d$ (see Figure \ref{fig:grid-filtration}).

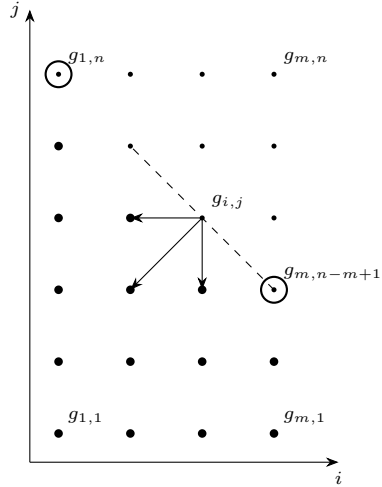
\begin{figure}[H]
\centering
\begin{tikzpicture}[>=Stealth, scale=0.95, every node/.style={font=\scriptsize}]
\def\m{4}
\def\n{6}
\def\d{5} 

\draw[->] (0.6,0.6) -- (\m+0.9,0.6) node[below] {$i$};
\draw[->] (0.6,0.6) -- (0.6,\n+0.9) node[left] {$j$};

\foreach \i in {1,...,\m}{
\foreach \j in {1,...,\n}{
    \pgfmathtruncatemacro{\s}{\i+\j-1}
    \ifnum\s>\d
    \fill (\i,\j) circle (1.0pt);
    \else
    \fill (\i,\j) circle (1.7pt);
    \fi
}
}

\pgfmathtruncatemacro{\diag}{\d+1}
\pgfmathtruncatemacro{\jLeft}{max(1,\diag-\m+1)}
\pgfmathtruncatemacro{\iLeft}{\diag-\jLeft+1}
\pgfmathtruncatemacro{\jRight}{min(\d,\diag)}
\pgfmathtruncatemacro{\iRight}{\diag-\jRight+1}
\draw[dashed] (\iLeft,\jLeft) -- (\iRight,\jRight);

\node[above right] at (1,1) {$g_{1,1}$};
\node[above right] at (1,\n) {$g_{1,n}$};
\node[above right] at (\m,1) {$g_{m,1}$};
\node[above right] at (\m,\n) {$g_{m,n}$};

\draw[thick] (1,\n) circle (0.18);
\pgfmathtruncatemacro{\jB}{\n-\m+1}
\draw[thick] (\m,\jB) circle (0.18);
\node[above right] at (\m,\jB) {$g_{m,n-m+1}$};

\draw[->] (3,4) -- (2,4);
\draw[->] (3,4) -- (3,3);
\draw[->] (3,4) -- (2,3);
\node[above right] at (3,4) {$g_{i,j}$};
\end{tikzpicture}
\caption{The basis $g_{i,j}=e_i\otimes f_j$ arranged on an $(i,j)$-grid. The filtration $S_d$ is the region $i+j-1\le d$ to the ``lower-left'' of the dashed line. Since $(\sigma-1)$ moves down/left/down-left, 
one has $(\sigma-1)S_{d+1}\subset S_d$.}
\label{fig:grid-filtration}
\end{figure}

Conversely, one sees visibly from Figure~\ref{fig:grid-filtration} that if $u \in V \otimes_k W$ is such that $(\sigma-1)(u) \in S_n$, then $u \in S_{n+1}$. 
From \eqref{eqn:g-relation}, we see that if $u \in S_{n+1}$, and $(\sigma-1)u \equiv \sum_{i+j = n+1} d_{i,j}g_{i,j} \pmod{S_{n-1}}$, then $\sum_{i+j=n+1} (-1)^i d_{i,j} = 0$. Thus $g_{1,n} \not\in (1-\sigma)S_{n+1}$, as desired.

Next we establish (2). By \cite[Lemma~3.1.1]{Cot26b}, we have $N_\sigma = (\sigma-1)^{\ell-1}$, so
\begin{equation}\label{eqn:norm-push}
N_\sigma(V \otimes_k W) = (\sigma-1)^{\ell-1}(S_{m+n-1}) \subset S_{m+n-\ell}.
\end{equation}
Thus if $m+n \leq \ell$, then $N_\sigma(V\otimes_k W) = 0$ and in particular $g_{1,1}$ has nonzero image in $\rT^0(V \otimes_k W)$. Conversely, suppose $m+n > \ell$. Then $\ell - n + 1 \leq m$, and \eqref{eqn:norm-push} shows that
\[
N_\sigma(g_{\ell-n+1,n}) \in S_1 = kg_{1,1}.
\]
It remains to show that this element is nonzero.

We first perform the binomial calculation once in the ``universal'' case $V = k[\sigma], W = k[\sigma]$, i.e., working in $k[\sigma]\otimes_k k[\sigma]$. Let $\delta \coloneqq\sigma-1$, and in $k[\sigma]\otimes_k k[\sigma]$ set
\[
h_{i,j}=\delta^{\ell-i}\otimes \delta^{\ell-j}, \qquad 1\leq i,j\leq \ell.
\]
Let $T_d$ be the span of the $h_{i,j}$ with $i+j-1\leq d$. Expanding $\sigma^r=(1+\delta)^r$ gives
\[
N_\sigma(h_{\ell,\ell})
=\sum_{r=0}^{\ell-1}\sum_{0\leq a,b\leq r}{r\choose a}{r\choose b}h_{\ell-a,\ell-b}.
\]
Since $N_\sigma=\delta^{\ell-1}$ and $\delta T_d\subset T_{d-1}$, we have $N_\sigma(h_{\ell,\ell})\in T_\ell$. Define $s_n$ by the congruence
\[
N_\sigma(h_{\ell,\ell})\equiv \sum_{n=1}^{\ell}s_nh_{n,\ell-n+1}\pmod{T_{\ell-1}},
\]
so we have
\[
s_n \coloneqq\sum_{r=1}^{\ell-1}{r\choose \ell-n}{r\choose n-1},
\]
where we interpret ${r\choose a} = 0$ if $r < a$. Applying $\delta$ for the diagonal action and reducing modulo $T_{\ell-2}$ gives
\[
0=\delta N_\sigma(h_{\ell,\ell})
\equiv \sum_{n=1}^{\ell-1}(s_n+s_{n+1})h_{n,\ell-n}
\pmod{T_{\ell-2}}.
\]
Here we used that $\delta(h_{r,\ell-r+1})=h_{r-1,\ell-r+1}+h_{r,\ell-r}+h_{r-1,\ell-r}$, and the last term lies in $T_{\ell-2}$. The classes of the $h_{n,\ell-n}$ for $1\leq n\leq \ell-1$ are linearly independent modulo $T_{\ell-2}$, so $s_{n+1}=-s_n$ for all $n$. Since $s_\ell=1$, it follows that $s_n\neq 0$ for every $n$.

Now return to general $V\otimes_k W$. This admits a quotient map from $k[\sigma]\otimes_k k[\sigma]$, sending
\[
h_{\ell-m+i,\ell-n+j}\longmapsto g_{i,j}.
\]
In particular, $h_{2\ell-m-n+1,\ell}$ maps to $g_{\ell-n+1,n}$ and $h_{\ell-m+1,\ell-n+1}$ maps to $g_{1,1}$. The same expansion as above shows that the coefficient of $h_{\ell-m+1,\ell-n+1}$ in $N_\sigma(h_{2\ell-m-n+1,\ell})$ is $s_n$. Applying the quotient map, we obtain
\[
N_\sigma(g_{\ell-n+1,n})=s_ng_{1,1}.
\]
Since $s_n\neq 0$, this shows that $g_{1,1}\in N_\sigma(V\otimes_k W)$. Hence $g_{1,1}$ has zero image in $\rT^0(V\otimes_k W)$ when $m+n>\ell$, proving (2).

Finally, we prove (3). One computes
\begin{equation}\label{eq:case-3-cup-product}
\begin{aligned}
\smile_{1,1}(e_m \otimes f_n) = \sum_{0 \leq i < j \leq \ell} \sigma^ie_m \otimes \sigma^jf_n &= \sum_{\substack{0 \leq i < j \leq \ell \\ 0 \leq a \leq i \\ 0 \leq b \leq j}} {i \choose a}{j \choose b}g_{m-a,n-b} \\
    &= \sum_{0 \leq a, b \leq \ell-1} \left(\sum_{\substack{a \leq i < j \leq \ell \\ b \leq j}} {i \choose a}{j \choose b}\right)g_{m-a,n-b}.
\end{aligned}
\end{equation}
We first show that $\Ima(\smile_{1,1}) \subset S_{m+n-\ell+1}$; this immediately implies that $\smile_{1,1} = 0$ if $m+n < \ell$.

It is equivalent to show that
\begin{equation}\label{eqn:zero-sum}
    \sum_{0 \leq i < j \leq \ell} {i\choose a}{j\choose b} \equiv 0 \pmod{\ell}
\end{equation}
whenever $a, b \geq 0$ satisfy $a + b < \ell - 2$. Using the ``hockey-stick identity'', we obtain
\[
\sum_{0 \leq i < j \leq \ell} {i \choose a}{j \choose b} = \sum_{j=1}^\ell {j\choose a+1}{j\choose b},
\]
where we interpret ${j\choose c} = 0$ whenever $j < c$. This sum is the coefficient of $x^{a+1}y^b$ in the polynomial $f(x,y) = \sum_{j=1}^\ell (1+x)^j(1+y)^j \in \F_\ell[x,y]$. Let $z \coloneqq (1+x)(1+y) - 1$, so we have
\[
f(x,y) = \sum_{j=1}^\ell (1+z)^j = \frac{(1+z)((1+z)^\ell - 1)}{z} = (x+y+xy)^{\ell-1} + (x+y+xy)^\ell.
\]
The coefficients of this polynomial vanish in degree below $\ell - 1$, and thus the coefficient of $x^{a+1}y^b$ vanishes whenever $a+b < \ell - 2$, as desired.

Now it remains to show that if $m + n \geq \ell$, then \eqref{eq:case-3-cup-product} does not lie in $N_\sigma(V \otimes_k W)$. As before, we have $N_\sigma(V \otimes_k W) \subset S_{m+n-\ell}$, and the quotient $
S_{m+n-\ell+1}/S_{m+n-\ell}$ has basis given by the nonzero $g_{m-a,n-b}$ with $a+b=\ell-2$. Since $m+n\geq \ell$, we can find such $a,b$ with $0\leq a\leq m-1$ and $0\leq b \leq n-1$. It is therefore enough to show that the coefficient of $g_{m-a,n-b}$ is nonzero for such an $a$. This is equivalent to the assertion
\[
    t_a \coloneqq \sum_{\substack{a \leq i < j \leq \ell \\ \ell-a-2 \leq j}} {i\choose a}{j\choose \ell-a-2} \not\equiv 0 \pmod{\ell}
\]
whenever $0 \leq a \leq \ell-2$. But the previous paragraph shows that $t_a$ is the coefficient of $x^{a+1}y^{\ell-a-2}$ in the polynomial $f(x,y)$ above, which is ${\ell-1 \choose a+1} \not\equiv 0 \pmod{\ell}$.
\end{proof}

\subsection{Tensor products of Jordan blocks}

Substantial work has been done to describe tensor products of Jordan blocks in positive characteristic in an algorithmic manner; see for instance \cite{Sri64}. We will only need the following result, which can be proven directly.

\begin{lemma}\label{lemma:tate-cohom-edge-cases}
Let $1 \leq a, b \leq \ell$, and let $V = k[\sigma]/((\sigma-1)^a)$ and $W = k[\sigma]/((\sigma-1)^b)$.
\begin{enumerate}
    \item If $a = 1$, then the map $\smile_{0,0}\co \rT^0(V) \otimes_k \rT^0(W) = V \otimes_k \rT^0(W) \to \rT^0(V \otimes_k W)$ is an isomorphism.
    \item If $b = \ell$, then $V \otimes_k W \cong W^{\oplus a}$ and thus $\rT^0(V \otimes_k W) = \rT^1(V \otimes_k W) = 0$.
    \item If $b = \ell - 1$, then $V \otimes_k W \cong k[\sigma]/((\sigma-1)^{\ell-a}) \oplus k[\sigma]^{\oplus a-1}$ as $k[\sigma]$-modules, and the map $\smile_{1,1}\co \rT^1(V) \otimes_k \rT^1(W) \to \rT^0(V \otimes_k W)$ is an isomorphism.
\end{enumerate}
\end{lemma}

\begin{proof}
The statement (1) is clear. For (2), note that $W$ is the regular representation of $k[\sigma]$, and the tensor product of any representation of a finite group with the regular representation is a direct sum of copies of the regular representation. Hence the result follows from vanishing of Tate cohomology of the regular representation. 

Next we prove (3). If $a = \ell$, then both claims follow from (2) and symmetry, so assume $a < \ell$. We determine the Jordan decomposition of $\sigma$ acting on $V \otimes_k W$ by computing the dimension of $\ker(\sigma-1)$ and the dimension of $\Ima(N_\sigma)$. We will prove that $\dim_k \ker(\sigma-1| V \otimes_k W) \leq a$ and $\dim_k \Ima(N_\sigma|V \otimes_k W) \geq a-1$. By \cite[Lemma~3.1.3]{Cot26b}, we have $\rT^0(V \otimes_k W) \neq 0$, so these two inequalities must actually be equalities, and thus $\dim_k \rT^0(V \otimes_k W) = 1$. Since $\dim_k V \otimes_k W = a(\ell-1) = (a-1)\ell + (\ell-a)$, it follows by applying \cite[Lemma~3.1.3]{Cot26b} again that $V \otimes_k W$ admits $a-1$ Jordan blocks of size $\ell$ and one of size $\ell-a$, which establishes the first point of (3). The second point of (3) follows from the first point and Lemma~\ref{lemma:first-tate-cohom-tensor}(3).

First, we show $\dim_k \ker(\sigma-1| V \otimes_k W) \leq a$. For $1 \leq i \leq a$ and $1 \leq j \leq \ell-1$, let $e_{i,j} = (\sigma-1)^{a-i} \otimes (\sigma-1)^{\ell - 1-j} \in V \otimes_k W$, so the $e_{i,j}$ form a $k$-basis for $V \otimes_k W$ and $(\sigma - 1)e_{i,j} = e_{i-1,j-1} + e_{i-1,j} + e_{i,j-1}$, where by convention $e_{i,j} = 0$ for $i = 0$ or $j = 0$. We claim that the map $\ker(\sigma-1|V \otimes_k W) \to k^{\oplus a}$ given by $v \mapsto (c_{i,1})_{1 \leq i \leq a}$, where $c_{i,j}$ is the coefficient of $e_{i,j}$ in $v$, is injective; this implies the desired inequality. For this, suppose that $c_{i,1} = 0$ for all $i$. We will show by induction on $d$ that $c_{i,d} = 0$ for all $i$, the case $d=1$ being clear by assumption. For $1 \leq d \leq \ell - 1$, let $U_d$ denote the $k$-subspace of $V \otimes_k W$ spanned by those $e_{i,j}$ with $j \geq d$ (see Figure \ref{fig:Ud-Ld}). For $1 \leq d \leq \ell - 2$, the induction hypothesis gives
\[
(\sigma-1)v \equiv \sum_{1 \leq i \leq a} (c_{i+1,d+1} + c_{i,d+1})e_{i,d} \pmod{U_{d+1}},
\]
where by convention $c_{a+1,d+1} = 0$. Taking $i = a$ gives $c_{a,d+1} = 0$, and inducting downwards on $i$ we see $c_{i,d+1} = 0$ for all $i$, as desired. 

Next, we show $\dim_k \Ima(N_\sigma|V \otimes_k W) \geq a-1$. For $0 \leq d \leq a-1$, let $L_d$ denote the $k$-subspace of $V \otimes_k W$ spanned by $e_{i,j}$ with $i \leq d$ (see Figure \ref{fig:Ud-Ld}).

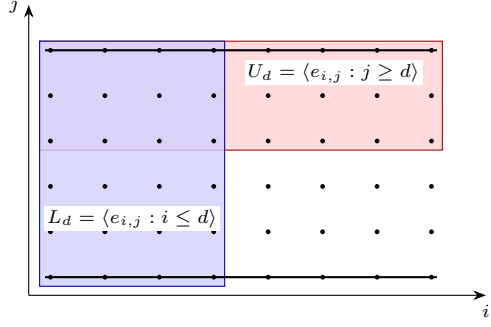
\begin{figure}[!h]
\centering
\begin{tikzpicture}[
>=Stealth,
x=0.72cm,y=0.6cm,
every node/.style={font=\scriptsize},
lab/.style={fill=white, inner sep=1pt}
]
\def\a{8}
\def\b{6} 
\def\d{4}

\fill[red!14] (0.8,\d-0.2) rectangle (\a+0.2,\b+0.2);
\draw[red!65!black] (0.8,\d-0.2) rectangle (\a+0.2,\b+0.2);
\fill[blue!18, fill opacity=0.8] (0.8,0.8) rectangle (\d+0.2,\b+0.2);
\draw[blue!65!black] (0.8,0.8) rectangle (\d+0.2,\b+0.2);

\draw[->] (0.6,0.6) -- (\a+1.0,0.6) node[below] {$i$};
\draw[->] (0.6,0.6) -- (0.6,\b+1.0) node[left] {$j$};

\foreach \i in {1,...,\a}{
\foreach \j in {1,...,\b}{
    \fill (\i,\j) circle (0.9pt);
}
}

\draw[thick] (0.9,1) -- (\a+0.1,1);

\draw[thick] (0.9,\b) -- (\a+0.1,\b);

\node[lab] at (6.2,5.5) {$U_d=\langle e_{i,j}: j\ge d\rangle$};
\node[lab] at (2.5,2.3) {$L_d=\langle e_{i,j}: i\le d\rangle$};

\end{tikzpicture}
\caption{The two filtrations used for the proof of Lemma~\ref{lemma:tate-cohom-edge-cases}(3): $U_d$ (upper rows) and $L_d$ (left-hand columns).}
\label{fig:Ud-Ld}
\end{figure}

One shows easily by induction on $i$ that if $d \leq a-2$ then
\[
\sigma^i(e_{d+2, \ell-1}) \equiv \sum_{j=0}^i {i\choose j} e_{d+2,\ell-1-j} + i\sum_{j=0}^i {i\choose j}e_{d+1,\ell-1-j} \pmod{L_d}.
\]
Thus the coefficient of $e_{d+2,\ell-1-j}$ in $N_\sigma(e_{d+2,\ell-1})$ is $\sum_{i=j}^{\ell-1} {i\choose j} = {\ell \choose j+1} \equiv 0 \pmod{\ell}$ since $0 \leq j < \ell-1$. On the other hand, the coefficient of $e_{d+1,1}$ is
\[
\ell - 2 + (\ell-1)^2 \equiv -1 \pmod{\ell}.
\]
We therefore see that $N_\sigma(e_{2,\ell-1}), \dots, N_\sigma(e_{a,\ell-1})$ are $k$-linearly independent, and it follows that $\dim_k \Ima(N_\sigma|V \otimes_k W) \geq a-1$, as desired.
\end{proof}


\begin{cor}\label{cor:cup-product-extremal-cases}
Let $V$ and $W$ be finitely generated $k[\sigma]$-modules, and for $i, j \in \Z/2\Z$ let $\smile_{i,j}\co \rT^i(V) \otimes_k \rT^j(W) \to \rT^{i+j}(V \otimes_k W)$ be the cup product map.
\begin{enumerate}
    \item If $V$ is minimal, then $\smile_{0,0}$ and $\smile_{0,1}$ are isomorphisms.
    \item If $W$ is maximal, then $\smile_{0,1}$ and $\smile_{1,1}$ are isomorphisms.
\end{enumerate}
\end{cor}

\begin{proof}
Both points are immediate from Lemma~\ref{lemma:first-tate-cohom-tensor} and Lemma~\ref{lemma:tate-cohom-edge-cases} after decomposing $V$ and $W$ into indecomposable $k[\sigma]$-modules.
\end{proof}

\begin{cor}\label{cor:tate-cohom-tensor-product-yu}
Assume that $k$ is algebraically closed. Let $K$ be a group\footnote{In the application below, $K$ will be a compact open subgroup of a reductive $p$-adic group arising from (the twisted version of) Yu's construction of tame cuspidal representations \cite{Yu01}, \cite{FKS23}.}, let $\sigma$ be an order $\ell$ automorphism of $K$, let $V$ be a finite-dimensional $k[K \rtimes \langle\sigma\rangle]$-module, and let $W_1, \dots, W_n$ be finite-dimensional irreducible $k[K \rtimes \langle\sigma\rangle]$-modules. Suppose that $W_i$ is an extremal $k[\sigma]$-module for each $1 \leq i \leq n$. Then there exist $j_1, \dots, j_n \in \Z/2\Z$ such that for each $i \in \Z/2\Z$, we have
\[
\rT^{i + \sum_{m=1}^n j_m}\left(V \otimes \bigotimes_{m=1}^n W_m\right) \cong \rT^i(V) \otimes \bigotimes_{m=1}^n \rT^{j_m}(W_m)
\]
as $K^\sigma$-representations.
\end{cor}

\begin{proof}
Let $j_m = 0$ if $W_m$ is minimal, and let $j_m = 1$ if $W_m$ is maximal and not minimal. It follows from Lemma~\ref{lemma:tate-cohom-edge-cases} and induction on $n$ that if $\sum_{m=1}^n j_m = 0$ then $\bigotimes_{m=1}^n W_m$ is minimal, and otherwise it is maximal. By Corollary~\ref{cor:cup-product-extremal-cases}, it follows that
\[
\smile_{i,\sum_{m=1}^n j_m}\co \rT^i(V) \otimes \rT^{\sum_{m=1}^n j_m}\left(\bigotimes_{m=1}^n W_m\right) \to \rT^{i + \sum_{m=1}^n j_m}\left(V \otimes \bigotimes_{m=1}^n W_m\right)
\]
is an isomorphism of $K^\sigma$-representations (since the cup product maps are evidently $K^\sigma$-equivariant). By the same reasoning, it follows by Corollary~\ref{cor:cup-product-extremal-cases} and induction on $n$ that
\[
\rT^{\sum_{m=1}^n j_m}\left(\bigotimes_{m=1}^n W_m\right) \cong \bigotimes_{m=1}^n \rT^{j_m}(W_m)
\]
as $K^\sigma$-representations. Combining these two isomorphisms yields the claim.
\end{proof}

\section{Tate cohomology of Weil--Heisenberg representations}\label{ss:weil-heisenberg}

For the most part, to apply the results of \cite{F24} we need only show that Tate cohomology is ``large''. Using \cite[Lemma~3.1.3]{Cot26b}, it is easy to see that the Tate cohomology of a Heisenberg representation is nonzero, and since irreducible representations of Heisenberg groups are classified by their central characters it is therefore easy to exhibit an irreducible constituent. However, in the setting of Yu's construction, it is also necessary to understand the Tate cohomology of the \emph{Weil} representation. To control this, we need a more refined understanding of the Tate cohomology of Heisenberg representations, for which we will ultimately use \cite[Theorem~3.5.4]{Cot26b}.

\subsection{Definition of the Weil--Heisenberg representation}
Let $\ell \neq p$ be distinct prime numbers such that $p\neq 2$,\footnote{The assumption $p \neq 2$ can likely be weakened using the formalism of Fintzen--Schwein \cite{FS25} and Takaya \cite{Tak26}, but we do not pursue this here.} let $q$ be a power of $p$, let $(V, B)$ be a finite-dimensional symplectic space over $\F_q$, and let $\sigma$ be a symplectic automorphism of $V$ of order $\ell$. Let $V_0 \coloneqq V^\sigma$, and let $B_0 = B|_{V_0}$. Let $V^\sharp \coloneqq V \times \F_q$ and $V_0^\sharp \coloneqq V_0 \times \F_q$ (with multiplication $(v, a) \cdot (w, b) = (v+w, a+b+\frac{1}{2}B(v,w))$) denote the corresponding Heisenberg groups. Then $\sigma$ induces an automorphism of $V^\sharp$ by $\sigma(v, c) = (\sigma(v), c)$, so that $(V^\sharp)^\sigma = V_0^\sharp$.

In this section, we will use $\mathbf{Sp}(V)$ to denote the symplectic group as a connected reductive $\F_q$-group, and we will use $\Sp(V)$ to denote its set of $\F_q$-points.

Let $\phi\colon \F_q \to \ol\F_\ell^\times$ be a nontrivial character, and let $W_{\phi}$ (resp.\ $W_{0,\phi}$) denote the unique irreducible $\ol\F_\ell$-representation of $V^\sharp$ (resp.\ $V_0^\sharp$) with central character $\phi$. Recall that $W_\phi$ (resp.\ $W_{0,\phi}$) extends canonically to a representation of $\Sp(V) \ltimes V^\sharp$ (resp.\ $\Sp(V_0) \ltimes V_0^\sharp$), called the \emph{Weil--Heisenberg representation} and exposed in great detail in \cite{Ger77}. We remark that \cite{Ger77} works throughout with complex coefficients, but by applying \cite[Part III, \S 15.5, Proposition 43]{Serre77} to representations of the Heisenberg group, and then extending actions to the symplectic group as usual, the results apply equally well over any field of characteristic not equal to $p$.

\subsection{Calculation of Tate cohomology}
Using the element $\sigma \in \Sp(V)$, one can consider the Tate cohomology $\rT^j(W_{\phi})$ for $j \in \Z/2\Z$.

\begin{lemma}\label{lemma:tate-cohom-of-heisenberg}
For $j \in \Z/2\Z$, the Tate cohomology $\rT^j(W_\phi)$ is isomorphic to $W_{0,\phi}$ as an $\Sp(V_0) \ltimes V_0^\sharp$-representation. Moreover, the action of $\sigma$ on $W_\phi$ is extremal in the sense of \cite[Definition~3.5.3]{Cot26b}.
\end{lemma}

\begin{proof}
We will first show that the action of $\sigma$ on $W_\phi$ is extremal and $\rT^j(W_\phi)$ is isomorphic to $W_{0,\phi}$ as a $V_0^\sharp$-representation. Observe that $V_0^\sharp$ acts on $\rT^j(W_\phi)$ with central character $\phi$, so every irreducible constituent of $\rT^j(W_\phi)$ is isomorphic to $W_{0,\phi}$. Thus we are reduced to the claim that $\rT^j(W_\phi)$ is irreducible. Note that $W_\phi$ is the special fiber of a $\ol\Z_\ell$-representation $\cW_\phi$ of $\Sp(V) \ltimes V^\sharp$ by \cite[Lemma 2.5]{Fin22}, where $(\cW_\phi)_{\ol\Q_\ell}$ is the usual Weil--Heisenberg representation as in \cite{Ger77}. By \cite[Corollary~3.5.5]{Cot26b} and the fact that the Glauberman correspondence sends irreducible characters to irreducible characters, the claim follows.

Next we show that the induced action of $\Sp(V_0)$ on $\rT^j(W_\phi)$ realizes the Weil representation. If either $\dim_{\F_q} V_0 \neq 2$ or $q \neq 3$, then there is a unique action of $\Sp(V_0) \ltimes V_0^\sharp$ on $W_{0, \phi}$ extending the Heisenberg representation of $V_0^\sharp$ since $\Sp(V_0)$ is a perfect group, and we conclude that the given action of $\Sp(V_0)$ is the Weil representation.

Hence it remains to handle the case $\dim_{\F_q} V_0 = 2$ and $q = 3$, in which case $\ell \neq 3$. Since the abelianization of $\Sp(V_0) \cong \SL_2(\F_3)$ is cyclic of order $3$, Schur's Lemma implies that any two actions of $\Sp(V_0) \ltimes V_0^\sharp$ on $W_{0,\phi}$ extending the Heisenberg representation of $V_0^\sharp$ can only differ by a character of $\Sp(V_0)$ valued in $\mu_3(\ol\F_\ell)$. Moreover, $\Sp(V_0)$ is of order $24$, and if $u \in \Sp(V_0)$ is of order $3$ then the (trace) character of the Weil representation on $u$ is nonzero by \cite[Theorem 4.4]{Ger77}. Thus it suffices to show that the value of the Brauer character $\lambda$ of $\rT^j(W_\phi)$ on a single such element $u$ agrees with the value of the Brauer character of $W_{0,\phi}$ on $u$. By \cite[Theorem 4.9.1(c)]{Ger77}, the value of the Brauer character of $W_{0,\phi}$ on $u$ is $1+2\zeta$ for a primitive $3$rd root of unity $\zeta$, so it suffices to show $\lambda(u) \in \{\pm (1+2\zeta)\}$. If $\chi$ is the ordinary character of $(\cW_\phi)_{\ol\Q_\ell}$, then by \cite[Theorem~3.5.4]{Cot26b} (applied to $\Gamma = \Sp(V) \ltimes V^\sharp$ and $\Delta = \{1\} \times V^\sharp$) there is some $\epsilon \in \{\pm 1\}$ such that
\begin{equation}\label{eqn:glauberman-weil}
    \chi(\sigma \circ u) = \epsilon \lambda(u).
\end{equation}
If $V_{0,+}$ is the unique line in $V_0$ which is fixed by $u$ and $V_1 = V_{0,+}^\perp/V_{0,+}$ (the perpendicular being taken in $V$), then by \cite[Theorem 4.9.1(c)]{Ger77} we have
\begin{equation}\label{eqn:weil-reduction-to-ss}
    \chi(\sigma \circ u) = (1+2\zeta)\Tr_{W_{1,\phi}}(\sigma),
\end{equation}
where $W_{1,\phi}$ is the Heisenberg representation of $V_1^\sharp$ with central character $\phi$. Since $\sigma$ acts without fixed point on $V_1$, \cite[Theorem 4.9.1(a)]{Ger77} shows that $\Tr_{W_{1,\phi}}(\sigma) \in \{\pm 1\}$. Combining \eqref{eqn:glauberman-weil} and \eqref{eqn:weil-reduction-to-ss}, we see that $\lambda(u) \in \{\pm (1+2\zeta)\}$, as desired.
\end{proof}

\subsection{A sign character}
Let $V_1$ denote the unique $\sigma$-stable complement of $V_0$ in $V$, so there are natural inclusions
\[
\mathbf{Sp}(V_0) \times 1 \subset \mathbf{Sp}(V)^\sigma \subset \mathbf{Sp}(V_0) \times \mathbf{Sp}(V_1) \subset \mathbf{Sp}(V).
\]
Let $r\colon \mathbf{Sp}(V)^\sigma \to \mathbf{Sp}(V_0)$ denote the restriction map, so we have $\mathbf{Sp}(V)^\sigma \cong \mathbf{Sp}(V_0) \times \ker(r)$. This induces an isomorphism 
\begin{equation}\label{eq:Sp-fixed}
\Sp(V)^\sigma \cong \Sp(V_0) \times \ker(r)(\F_q).
\end{equation}
The action of $\Sp(V)$ on $W_\phi$ induces a natural action of $\Sp(V)^\sigma $ on $\rT^j(W_\phi)$. The action of $\Sp(V_0)$ was determined in Lemma~\ref{lemma:tate-cohom-of-heisenberg}, so it remains to compute the action of $\ker(r)(\F_q)$. 

We first recall some notation from \cite[\S 4]{Ger77}. Since $\ell \neq p$, the element $\sigma$ is semisimple, so we may choose a maximal $\F_q$-torus $T\subset\mathbf{Sp}(V)$ containing $\sigma$. Let $P$ denote the set of characters $T_{\ol\F_q} \to \G_m$ which appear as weights for the action of $T_{\ol\F_q}$ on $V_{\ol\F_q}$. Note that $P$ is closed under negation, thanks to the symplectic structure on $V$. Let $\Gamma \coloneqq \Gal(\ol\F_q/\F_q)$, and let $\Sigma \coloneqq \Gamma \times \{\pm 1\}$, so $\Gamma$ and $\Sigma$ act on $P$ in a natural way. If $\omega$ is a $\Gamma$-orbit on $P$, then we will say that $\omega$ is \textit{symmetric} if $\omega = -\omega$, and $\omega$ is \textit{asymmetric} if $\omega \neq -\omega$.

If $\Omega$ is a $\Sigma$-orbit on $P$, let $q_\Omega \coloneqq q^{|\Omega|/2}$, and let $T_\Omega$ denote the subtorus of $T$ which is the intersection of the kernels of the elements $\alpha \in P - \Omega$. Note that the obvious multiplication map induces 
\begin{equation}\label{eq:T-factorize}
T \cong \prod_{\Omega \in P/\Sigma} T_\Omega.
\end{equation}
Let $P_i$, for $i \in \{0,1\}$, denote the set of weights for $T_{\ol\F_q}$ on $(V_i)_{\ol\F_q}$. Then $P = P_0 \sqcup P_1$, and the isomorphism \eqref{eq:T-factorize} restricts to an isomorphism
\[
T \cap \ker(r) \cong \prod_{\Omega \in P_1/\Sigma} T_\Omega.
\]
If $\alpha \in \Omega$, let $\F_\alpha$ denote the extension of $\F_q$ corresponding to the Galois orbit of $\alpha$, and define a character $\chi_\Omega^T\colon T(\F_q) \to \ol\Z_\ell^\times$ by
\[
\chi_\Omega^T(t) = \begin{cases}
\sgn_{\F_\alpha^\times}(\alpha(t)) &\text{if } \alpha \text{ is asymmetric,} \\
\sgn_{\F_\alpha^1}(\alpha(t)) &\text{if } \alpha \text{ is symmetric}
\end{cases}
\]
for all $t \in T(\F_q)$, where $\sgn_{\F_\alpha^\times}$ (resp.\ $\sgn_{\F_\alpha^1}$) is the unique nontrivial sign character of $\F_\alpha^\times$ (resp.\ $\F_\alpha^1$, the unique subgroup of $\F_\alpha^\times$ of order $q_\Omega+1$). It is clear from the definition that $\chi_\Omega^T$ is independent of the choice of $\alpha \in \Omega$, and that $\chi_\Omega^T|_{T_{\Omega'}(\F_q)} = 1$ whenever $\Omega \neq \Omega'$. Hence we may define
\[
\chi^T = \prod_{\Omega \in P/\Sigma} \chi_\Omega^T.
\]

\subsection{The residual action}

We write $\ol \chi$ for the composition of a $\ol \Z_\ell^\times$-valued character $\chi$ with the quotient map $\ol\Z_\ell^\times \to \ol \F_\ell^\times$. Retain the notation of the previous subsection.

\begin{lemma}\label{lemma:weil-rep-character-non-sp}
For $j \in \Z/2\Z$, the $\ker(r)(\F_q)$-action on $\rT^j(W_\phi)$ is through a character $\theta$ of the quotient $\ker(r)_{\ab}(\F_q)$, and we have
\begin{equation}\label{eqn:ker-r-character-restriction}
\theta|_{\prod_{\Omega \in P_1/\Sigma} T_\Omega(\F_q)} = \ol \chi^T.
\end{equation}
\end{lemma}

\begin{proof}
Note first that $\ker(r)(\F_q)$ commutes with $V_0^\sharp$. Lemma~\ref{lemma:tate-cohom-of-heisenberg} shows that $\rT^j(W_\phi)$ is an irreducible $V_0^\sharp$-representation, so $\ker(r)(\F_q)$ acts on $\rT^j(W_\phi)$ through a character $\theta$ by Schur's lemma. Moreover, applying \cite[Theorem~3.5.4]{Cot26b} with $\Gamma = \Sp(V) \ltimes V^\sharp$ and $\Delta = \{1\} \times V^\sharp$, we see that if $\chi$ is the ordinary character of $(\cW_\phi)_{\ol\Q_\ell}$, then there is a sign $\epsilon \in \{\pm 1\}$ such that for any $g \in \ker(r)(\F_q)_{\ell'}$ we have
\begin{equation}\label{eqn:ker-r-character-rel}
|V_0|^{1/2}\cdot\theta(g) = \epsilon\ol\chi(\sigma \circ g)
\end{equation}
since $\dim \rT^j(W_\phi) = |V_0|^{1/2}$.

Since $\ker(r)(\F_q)$ acts through a character, its action formally factors through $\ker(r)(\F_q)_{\ab}$. However, the lemma asserts the a priori stronger statement that the action factors through $\ker(r)_{\ab}(\F_q)$.\footnote{If $q > 3$, then $\ker(r)_{\ab}(\F_q) = \ker(r)(\F_q)_{\ab}$, so the claim is only nontrivial in the case $q = 3$. However, the argument we give is uniform in $q$.} We now prove this claim. Note that $(\ker(r)_{\der})_{\ol\F_q}$ is a product of special linear groups if $\ell$ is odd and a product of symplectic groups if $\ell=2$, hence in particular it is simply connected. Recall (see for instance the references in \cite[Theorem C.1, Proposition C.3]{FS25}) that $\ker(r)_{\der}(\F_q)$ is generated by the unipotent elements of $\ker(r)(\F_q)$. It therefore suffices to show that if $u \in \ker(r)(\F_q)$ is unipotent then $\theta(u) = 1$; we will prove this by induction on $\dim V_0$. Note that any such $u$ commutes with $\sigma$, stabilizes $V_0$ and $V_1$, and acts trivially on $V_0$.

Suppose first that $\dim V_0 = 0$, so $\sigma$ and $\sigma \circ u$ both act without fixed points on $V$. Then \cite[Theorem 4.9.1(a)]{Ger77} shows that
\[
\chi(\sigma \circ u) \in \{\pm 1\}.
\]
By \eqref{eqn:ker-r-character-rel}, it follows that $\theta(u) \in \{\pm 1\}$, so $\theta(u) = 1$ since $p$ is odd and $u$ is of $p$-power order.  Next suppose $\dim V_0 > 0$, and let $V_+ \subset V_0$ be a line. Since $u$ acts trivially on $V_0$, we see that $\sigma \circ u$ also acts trivially on $V_0$. Let $V'$ be a $\sigma$-stable complement of $V_+$ in $V_+^\perp$ containing $V_1$; then $V'$ is equal to $(V' \cap V_0) \oplus V_1$ by semisimplicity of $\sigma$. See Figure \ref{fig:weil-heisenberg-block-decomposition}.

\begin{figure}[H]
\centering
\begin{minipage}{0.96\textwidth}
\small
Set \(U\coloneqq V'\cap V_0\), and choose a line \(V_- \subset V_0\) such that
\(V'^\perp=V_+\oplus V_-\). Then the relevant subspaces are arranged as follows:
\[
\begin{tikzpicture}[x=1cm,y=1cm, every node/.style={font=\small}]
\def\h{0.72}

\draw[fill=blue!10] (0,0) rectangle (1.1,\h);
\node at (0.55,0.36) {$V_+$};

\draw[fill=blue!10] (1.1,0) rectangle (3.4,\h);
\node at (2.25,0.36) {$U=V'\cap V_0$};

\draw[fill=orange!15] (3.4,0) rectangle (6.0,\h);
\node at (4.7,0.36) {$V_1$};

\draw[fill=blue!10] (6.0,0) rectangle (7.1,\h);
\node at (6.55,0.36) {$V_-$};

\draw (0,\h+0.12) -- (0,\h+0.30) -- (6.0,\h+0.30) -- (6.0,\h+0.12);
\node at (3.0,\h+0.56) {$V_+^\perp = V_+\oplus V'$};

\draw (1.1,-0.12) -- (1.1,-0.30) -- (6.0,-0.30) -- (6.0,-0.12);
\node at (3.55,-0.56) {$V'=U\oplus V_1$};

\draw[densely dashed] (0,-0.88) -- (1.1,-0.88);
\draw[densely dashed] (6.0,-0.88) -- (7.1,-0.88);
\node at (3.55,-0.88) {$V'^\perp=V_+\oplus V_- \subset V_0$};

\node[anchor=west] at (7.35,0.52) {blue blocks lie in \(V_0\)};
\node[anchor=west] at (7.35,0.18) {orange block is \(V_1\)};
\end{tikzpicture}
\]

\end{minipage}
\caption{The block decomposition used in the induction step.}
\label{fig:weil-heisenberg-block-decomposition}
\end{figure}
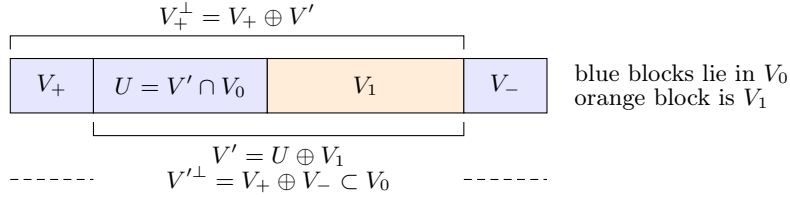

But $V_1$ is $\sigma \circ u$-stable and $\sigma \circ u$ acts trivially on $V_0$, so indeed $V'$ is also $\sigma \circ u$-stable. If $\chi'$ denotes the character for the Weil representation corresponding to $V'$ and for any $g \in \Sp(V)$ stabilizing $V'$ we let $g'$ denote the image of $g$ in $\Sp(V')$, then by \cite[Theorem 4.9.1(c)]{Ger77} we have
\[
\chi(\sigma) = \sum_{\ol y \in V'^\perp/V_+} \phi\left(\frac{B(\sigma y, y)}{2}\right) \chi'(\sigma'),
\]
where $y$ denotes any lift of $\ol y$ to $V'^\perp$. Similarly,
\[
\chi(\sigma \circ u) = \sum_{\ol y \in V'^\perp/V_+} \phi\left(\frac{B((\sigma \circ u)y, y)}{2}\right)\chi'(\sigma' \circ u').
\]
By the induction hypothesis we have $\chi'(\sigma') = \chi'(\sigma' \circ u')$, so it suffices to show $B(\sigma y, y) = B((\sigma \circ u)y, y)$ for all $y \in V'^\perp$. But $V'^\perp \subset V_0$, so $\sigma y = (\sigma \circ u)y$ for any $y \in V'^\perp$ and we have completed the induction.

Now we prove \eqref{eqn:ker-r-character-restriction}. If $\ell = 2$ then both sides are necessarily the trivial character, so the claim is trivial. Therefore we may suppose $\ell \neq 2$. Since $\theta$ is an $\ol\F_\ell^\times$-valued character, it suffices to show that the restriction of $\theta$ to $(T \cap \ker(r))(\F_q)_{\ell'}$ is equal to the restriction of $\ol\chi^T$. From now on, fix $t \in (T \cap \ker(r))(\F_q)_{\ell'}$. By \cite[Corollary 4.8.1]{Ger77}, we have
\begin{equation}\label{eqn:weil-rep-char-t}
\chi(\sigma \circ t) = (-1)^{l(V, T; \sigma \circ t)}q^{N(V; \sigma \circ t)}\chi^T(\sigma \circ t),
\end{equation}
where $l(V, T; \sigma \circ t)$ is the number of orbits $\omega \in P/\Gamma$ such that the action of $\sigma \circ t$ on $V_\omega = \sum_{\alpha \in \omega} V_\alpha$ is nontrivial, and $N(V;\sigma \circ t) = \frac{1}{2}\dim\ker(\sigma \circ t-1)$. Observe first that 
\[
\chi^T(\sigma \circ t) = \chi^T(t)
\]
since $\chi^T$ is an order $2$ character and $\sigma$ is of odd order. Next, note that 
\[
q^{N(V;\sigma \circ t)} = q^{N(V;\sigma)} = |V_0|^{1/2}
\]
since $V^{\sigma \circ t} = V^\sigma = V_0$: indeed, the automorphism $\sigma \circ t$ acts trivially on $V_0$, and conversely, if $\sigma$ acts nontrivially on some $V_\omega$, then since $t$ is of order prime to $\sigma$ it follows that $\sigma \circ t$ also acts nontrivially on $V_\omega$. This also shows that 
\[
l(V, T; \sigma \circ t) = l(V, T; \sigma).
\]
By taking $g = 1$ in \eqref{eqn:ker-r-character-rel}, and comparing with \eqref{eqn:weil-rep-char-t}, we find
\[
\epsilon = (-1)^{l(V, T;\sigma)}.
\]
Combining all of these observations with \eqref{eqn:ker-r-character-rel}, we obtain
\[
|V_0|^{1/2} \cdot \theta(t) = (-1)^{l(V,T;\sigma)}\chi(\sigma \circ t) = |V_0|^{1/2} \cdot \chi^T(t),
\]
which shows $\theta|_{(T \cap \ker(r))(\F_q)_{\ell'}} = \chi^T|_{(T \cap \ker(r))(\F_q)_{\ell'}}$, as desired.
\end{proof}

The decomposition $P = P_0 \sqcup P_1$ induces a factorization 
\[
\chi^T = \prod_{\Omega \in P_0/\Sigma} \chi_\Omega^T \prod_{\Omega \in P_1/\Sigma} \chi_\Omega^T.
\]

\begin{defn}\label{def:chi-sigma}Define a character $\chi_1^T\co T(\F_q) \to \ol\Z_\ell^\times$ by
\begin{equation}
\chi_1^T = \prod_{\Omega \in P_1/\Sigma} \chi_\Omega^T.
\end{equation}
Note that $\chi_1^T|_{\prod_{\Omega \in P_0/\Sigma} T_\Omega(\F_q)} = 1$, while $\chi_1^T$ and $\chi^T$ have the same restriction to $\prod_{\Omega \in P_1/\Sigma} T_\Omega(\F_q)$. 

Using the isomorphism $\Sp(V)^\sigma \cong \Sp(V_0) \times \ker(r)(\F_q)$ from \eqref{eq:Sp-fixed}, we will regard the character $\theta$ of Lemma~\ref{lemma:weil-rep-character-non-sp} as a character of $\Sp(V)^\sigma$. Then Lemma~\ref{lemma:weil-rep-character-non-sp} says that 
\begin{equation}\label{eq:theta-restriction}
\theta|_{T(\F_q)} = \ol \chi_1^T.
\end{equation}
If $T \subset \mathbf{Sp}(V)$ is a maximal $\F_q$-torus with $\sigma \in T(\F_q)$, then the map $T \cap \ker(r) \to \ker(r)_{\ab}$ is surjective with torus kernel, so by Lang's theorem the map $(T \cap \ker(r))(\F_q) \to \ker(r)_{\ab}(\F_q)$ is surjective. Hence $\theta$ is the unique character on $\ker(r)_{\ab}(\F_q)$ satisfying the equation \eqref{eq:theta-restriction} for any (equivalently, all) such $T$. We define $\chi_1 \coloneqq \theta$ as above.
\end{defn}

\begin{prop}\label{prop:tate-cohom-of-weil}
For $j\in \Z/2\Z$, the $\ol\F_\ell$-representation of $\Sp(V)^\sigma \ltimes V_0^\sharp$ on $\rT^j(W_\phi)$ is isomorphic to $W_{0,\phi} \otimes \chi_1$.
\end{prop}

\begin{proof}
This follows from the previous paragraph and Lemmas~\ref{lemma:tate-cohom-of-heisenberg} and \ref{lemma:weil-rep-character-non-sp}.
\end{proof}

\subsection{Restriction of scalars}\label{ss:weil-heisenberg-restriction-of-scalars}

For later considerations with Yu's construction, we need to understand the behavior of $\chi_1$ with respect to a form of restriction of scalars. Maintaining the preceding notation, let $V_p = V$ considered as an $\F_p$-vector space. We define a symplectic form $B_p$ on $V_p$ by sending $x, y \in V_p$ to $B_p(x, y) = \Tr_{\F_q/\F_p}(B(x,y))$. This induces a natural inclusion map $i_p\co\Sp(V) \to \Sp(V_p)$. If $\sigma \in \Sp(V)$ is any element, then $i_p$ induces an inclusion map $\Sp(V)^\sigma \to \Sp(V_p)^{i_p(\sigma)}$.

\begin{lemma}\label{lemma:sign-char-weil-restriction}
Let $\sigma \in \Sp(V)$ be of order $\ell$, and let $\chi_1$ (resp.\ $\chi_{1,p}$) denote the sign character from Definition \ref{def:chi-sigma} of $\Sp(V)^\sigma$ (resp.\ $\Sp(V_p)^{i_p(\sigma)}$) corresponding to $\sigma$ (resp.\ $i_p(\sigma)$). Then $\chi_1 = \chi_{1,p} \circ i_p$.
\end{lemma}

\begin{proof}
Let $\Sigma_p = \Gal(\ol\F_p/\F_p) \times \{\pm 1\}$ and $\Sigma_q = \Gal(\ol\F_q/\F_q) \times \{\pm1\}$. View $\Sp(V,B)$ as the group of $\F_p$-points of $\Res_{\F_q/\F_p}\mathbf{Sp}(V, B)$. Then $i_p$ arises from a monic $\F_p$-homomorphism
\[
j_p\co\Res_{\F_q/\F_p}\mathbf{Sp}(V, B) \to \mathbf{Sp}(V_p,B_p).
\]
If $q = p^r$ and $\dim_{\F_q}(V) = n$, then $\dim_{\F_p}(V_p) = rn$. Thus, if $T_0 \subset \mathbf{Sp}(V,B)$ is a maximal $\F_q$-torus containing $\sigma$, then $j_p(\Res_{\F_q/\F_p}T_0) \subset \mathbf{Sp}(V_p,B_p)$ is a maximal $\F_p$-torus containing $i_p(\sigma)$. Fix such a $T_0$, let $T = \Res_{\F_q/\F_p} T_0$, and identify $T$ with its image under $j_p$.

After base change to $\ol\F_p$, there are natural decompositions
\[
T_{\ol\F_p} \cong \prod_{j\co \F_q \to \ol\F_p} T_0 \otimes_{\F_q, j} \ol\F_p
\quad \text{and} \quad
V_p \otimes_{\F_p} \ol\F_p \cong \prod_{j\co \F_q \to \ol\F_p} V \otimes_{\F_q, j} \ol\F_p.
\]
Under these decompositions, the action of $T$ on $V_p$ is the product of the actions of $T_0 \otimes_{\F_q,j} \ol\F_p$ on $V \otimes_{\F_q,j} \ol\F_p$.

Let $P_{1,q}$ denote the set of weights for $T_{0,\ol\F_q}$ on $V_1 \otimes_{\F_q} \ol\F_q$, and let $P_{1,p}$ denote the set of weights for $T_{\ol\F_p}$ on $V_1 \otimes_{\F_p} \ol\F_p$. Since the preceding decompositions are $\sigma$-equivariant, they identify
\[
P_{1,p} \cong \coprod_{j\co \F_q \to \ol\F_p} P_{1,q}.
\]
If $\omega \subset P_{1,q}$ is a $\Gal(\ol\F_q/\F_q)$-orbit, then $\coprod_j \omega$ is a $\Gal(\ol\F_p/\F_p)$-orbit in $P_{1,p}$. This gives a natural bijection $P_{1,p}/\Sigma_p \cong P_{1,q}/\Sigma_q$ preserving symmetric and asymmetric orbits.

Given $\alpha \in P_{1,p}$ (resp.\ $\beta \in P_{1,q}$), let $\F_{p,\alpha}$ (resp.\ $\F_{q,\beta}$) denote the extension of $\F_p$ (resp.\ $\F_q$) corresponding to the Galois orbit of $\alpha$ (resp.\ $\beta$). Suppose $\alpha$ lies in the copy of the orbit of $\beta$ indexed by an embedding $j\co \F_q \to \ol\F_p$. For the extension of $j$ to $\F_{q,\beta}$ determined by $\alpha$, we have
\[
\F_{p,\alpha}=j(\F_{q,\beta}) \quad \text{and} \quad \alpha(i_p(t)) = j(\beta(t)) \quad \text{for } t \in T(\F_p)=T_0(\F_q).
\]
In the symmetric case, the orders defining $\F_{p,\alpha}^1$ and $\F_{q,\beta}^1$ agree, so $j(\F_{q,\beta}^1)=\F_{p,\alpha}^1$. Since the sign characters are invariant under field automorphisms, we have
\[
\sgn_{\F_{p,\alpha}^\times}(\alpha(i_p(t))) = \sgn_{\F_{q,\beta}^\times}(\beta(t)) \quad \text{and} \quad \sgn_{\F_{p,\alpha}^1}(\alpha(i_p(t))) = \sgn_{\F_{q,\beta}^1}(\beta(t)).
\]
These observations combine to show that for $t \in T(\F_p) = T_0(\F_q)$, we have
\[
\begin{aligned}
\chi_{1,p}(i_p(t)) &= \prod_{\alpha \in (P_{1,p})_{\mathrm{asym}}/\Sigma_p} \sgn_{\F_{p,\alpha}^\times}(\alpha(i_p(t))) \cdot \prod_{\alpha \in (P_{1,p})_{\mathrm{sym}}/\Sigma_p} \sgn_{\F_{p,\alpha}^1}(\alpha(i_p(t))) \\
    &= \prod_{\beta \in (P_{1,q})_{\mathrm{asym}}/\Sigma_q} \sgn_{\F_{q,\beta}^\times}(\beta(t)) \cdot \prod_{\beta \in (P_{1,q})_{\mathrm{sym}}/\Sigma_q} \sgn_{\F_{q,\beta}^1}(\beta(t)) \\
    &= \chi_1(t),
\end{aligned}
\]
as desired.
\end{proof}

\section{Tate cohomology of tame cuspidal representations}\label{section:tate-cohom-yu}

In this section, we use the formalism developed in \cite[\S 3]{Cot26b} and Section~\ref{section:tate-tensor-products}, as well as the calculations of \cite{Cot26b}, to compute the Tate cohomology for some cuspidal $\ol\F_\ell$-representations arising from Yu's construction. Using \cite[Theorem~9.1.1]{F24}, these calculations give rise to explicit modular functoriality results for $\rho^{\FS}$. Namely, we will prove analogues of \cite[Propositions~8.2.3, 8.3.1, and Theorem~9.3.3]{CF26a} modulo certain prime numbers. In \S\ref{section:fs-functoriality}, we will combine these statements (and one more modular functoriality result from the next section) to prove \emph{characteristic} $0$ functoriality for descent to unramified twisted Levis.

We fix a connected reductive $F$-group $G$ which splits after a tamely ramified extension of $F$, and we assume that $p \neq 2$.

\subsection{Large prime degree unramified base change}\label{ss:yu-base-change}

In this section, we prove analogues of the results of \cite[\S 8.2]{CF26a} for $\rho^{\FS}$. Recalling the setting, we take $k = \ol\F_\ell$ or $\ol\Z_\ell$; when $k=\ol\Z_\ell$, we assume that every $\phi_i$ has finite order prime to $\ell$. We let $\ell$ be a banal prime for $G$ such that $\ell > \rk G + 1$. Since $\ell$ is banal for $G$, it is also banal for $G_{F_\ell}$ by \cite[Proposition~4.1.2(2)]{Cot26b}. 
Throughout, let
\[
\Psi = ((G^i)_{0 \leq i \leq d}, x, (r_i)_{0 \leq i \leq d}, \tau, (\phi_i)_{0 \leq i \leq d})
\]
be an irreducible Yu datum for $G$ with coefficients in $k$, and assume that $\Psi$ is normalized in the sense of \cite[Definition 3.7.1]{Kal19}. Recall the base changed Yu datum 
\[
\Psi_\ell = ((G_\ell^i), x_\ell, (r_{\ell,i}), \tau_\ell, (\phi_{\ell,i}))
\]
defined in \cite[\S 8.2]{CF26a}.

\begin{thm}\label{thm:yu-datum-base-change}
Let $\ell$ be a banal prime for $G$ such that $\ell > \rk G + 1$, let $\Psi$ be an irreducible Yu datum for $G$ over $\ol\F_\ell$, and let $\sigma$ be a generator for $\Gal(F_\ell/F)$. Then $\rT^a(\sigma, \pi(\Psi_\ell))$ admits the Frobenius twist $\pi(\Psi)^{(\ell)}$ as a direct summand for each $a \in \Z/2$.
\end{thm}

\begin{proof}
We briefly recall the notation from \cite[\S 6]{CF26a}. Define (cf.\ \cite[(6.2.1)]{CF26a}) 
\[
K_\ell = G^0(F_\ell)_{[x]}G^1(F_\ell)_{x, \frac{r_0}{2}} \cdots G^d(F_\ell)_{x, \frac{r_{d-1}}{2}}.
\]
Let $\epsilon_\ell = \prod_{i=1}^d \epsilon_{x,F_\ell}^{G^i/G^{i-1}}$ be the sign character on $K_\ell$ defined in \cite{FKS23} (cf.\ \cite[\S 6.5]{CF26a}). For each $0\leq i\leq d-1$, we let $\wh \phi_{\ell,i}$ be the $\ol\F_\ell^\times$-valued character of $G^0(F_\ell)_{[x]}G^i(F_\ell)_{x,0}G(F_\ell)_{x,\frac{r_i}{2}+}$ associated to $\phi_{\ell,i}$ (cf.\ \cite[\S 6.2]{CF26a}). For $1\leq i\leq d$, we have groups $J_{\ell,i}$ and $J_{\ell,i}^+$ (cf.\ \cite[Lemma~6.4.4]{CF26a}), and the group $V_{\ell,i} = J_{\ell,i}/J_{\ell,i}^+$ (cf.\ \cite[Lemma~6.3.1]{CF26a}) is an $\F_p$-vector space equipped with a symplectic form 
\[
\langle a, b\rangle_{\ell,i} = \ol\psi^{-1}(\wh\phi_{\ell,i-1}(aba^{-1}b^{-1})) \qquad  \text{for $a, b \in V_{\ell,i}$.}
\]

For $0\leq i\leq d-1$, let $V_{\wh\phi_{\ell,i}}$ denote the Heisenberg representation of $J_{\ell,i+1}/(J_{\ell,i+1}^+ \cap \ker \wh\phi_{\ell,i})$ with central character induced by $\wh\phi_{\ell,i}|_{J_{\ell,i+1}^+}$, as in \cite[\S 6.3]{CF26a}. This can be inflated to a representation of $K_\ell$ as follows: the group $J_{\ell,i+1}$ acts on $V_{\wh\phi_{\ell,i}}$ through the Heisenberg representation; if $j \neq i$, then $J_{\ell,j+1}$ acts on $V_{\wh\phi_{\ell,i}}$ through $\wh\phi_{\ell,i}|_{J_{\ell,j+1}}$; finally, $G^0(F_\ell)_{[x]}$ acts on $V_{\wh\phi_{\ell,i}}$ through $\phi_{\ell,i}$ times the Weil representation arising from the map defined in \cite[\S 6.3]{CF26a}:
\[
G^0(F_\ell)_{[x]}/G^0(F_\ell)_{x,0+} \to \Sp(V_{\ell,i+1}).
\]
We have then
\[
\pi(\Psi_\ell) = \cInd_{K_\ell}^{G(F_\ell)}(\epsilon_\ell \otimes \tau_\ell \otimes \bigotimes_{i=0}^{d-1} V_{\wh\phi_{\ell,i}} \otimes \phi_{\ell,d}).
\]

We now begin the calculation. Fix $a \in \Z/2\Z$. By \cite[Lemma~5.2.1]{Cot26b}, the Tate cohomology $\rT^a(\sigma, \pi(\Psi_\ell))$ admits $\cInd_{K_\ell^\sigma}^{G(F)}(\rT^a(\sigma, \epsilon_\ell \otimes \tau_\ell \otimes \bigotimes_{i=0}^{d-1} V_{\wh\phi_{\ell,i}} \otimes \phi_{\ell,d}))$ as a direct summand. By \cite[Lemma~6.4.2(2)]{CF26a}, we have
\[
K_\ell^\sigma = G^0(F)_{[x]}G^1(F)_{x,\frac{r_0}{2}} \cdots G^d(F)_{x,\frac{r_{d-1}}{2}},
\]
which is the group $K$ (cf.\ \cite[(6.2.1)]{CF26a}) appearing in Yu's construction for $\Psi$. Similarly, \cite[Lemma~6.4.4(2)]{CF26a} shows that $J_i = J_{\ell,i}^\sigma$ and $J_i^+ = (J_{\ell,i}^+)^\sigma$ are the groups appearing in Yu's construction for $\Psi$. If $V_i = J_i/J_i^+$ (as in \cite[Lemma~6.3.1]{CF26a}), then the special isomorphism for $J_{\ell,i}/(J_{\ell,i}^+ \cap \ker \wh\phi_{\ell,i-1})$ restricts to the special isomorphism for $J_i/(J_i^+ \cap \ker \wh\phi_{i-1})$ by \cite[Lemma~6.4.5(2)]{CF26a}.

By Corollary~\ref{cor:tate-cohom-tensor-product-yu}, applicable because of Lemma~\ref{lemma:tate-cohom-of-heisenberg}, there exist elements $m, n_0,\dots,n_{d-1} \in \Z/2\Z$ such that $m + \sum_{i=0}^{d-1} n_i = a$ and such that the cup product map defines an isomorphism
\[
\rT^a(\epsilon_\ell \otimes \tau_\ell \otimes \bigotimes_{i=0}^{d-1} V_{\wh\phi_{\ell,i}} \otimes \phi_{\ell,d}) \cong \epsilon_\ell|_K\cdot \phi_{\ell,d}|_K \rT^m(\tau_\ell) \otimes \bigotimes_{i=0}^{d-1} \rT^{n_i}(V_{\wh\phi_{\ell,i}}).
\]
Note that $\epsilon_\ell|_K = \prod_{i=1}^d \epsilon_{x,F}^{G^i/G^{i-1}}$ by \cite[Lemma~5.2.1]{CF26a}.
If $\ell = 2$, then both sign characters are trivial; if $\ell$ is odd, then passage to the $\ell$-Frobenius twist affects neither character. We have $\phi_{\ell,d}|_K = \phi_d^\ell$ by \cite[Proposition~7.1.2]{CF26a}, and because of the construction in \cite[\S 8.2]{CF26a}, \cite[Lemma~3.5.11]{Cot26b} identifies $\rT^m(\tau_\ell)$ with $\tau^{(\ell)}$.

Observe that $V_{\ell,i+1}^\sigma = V_{i+1}$, so Proposition~\ref{prop:tate-cohom-of-weil} identifies $\rT^{n_i}(V_{\wh\phi_{\ell,i}})$ with the Weil--Heisenberg representation attached to $V_{i+1}$, twisted by a residual character $\chi_{1,i}$ of order at most two. It remains to check that this twist is trivial under the $G^0(F)_{[x]}$-action. \cite[Lemma~6.3.2]{CF26a} shows that the action of $G^0(F)_{[x]}$ on $V_{\ell,i+1}$ factors through $\mathbf{Sp}(V_{i+1},B_{i+1})(\F_q)$, which is generated by unipotent elements (which have order $p \neq 2$). Therefore, the restriction of the quadratic twist to $G^0(F)_{[x]}$ is trivial. Thus $K$ acts on $\rT^{n_i}(V_{\wh\phi_{\ell,i}})$ in the following way: $J_{i+1}$ acts through the Heisenberg representation with central character induced by $\wh\phi_i^\ell|_{J_{i+1}^+}$ by \cite[Lemma~6.4.3]{CF26a}; for $j \neq i$, the group $J_{j+1}$ acts through $\wh\phi_i^\ell|_{J_{j+1}}$; and $G^0(F)_{[x]}$ acts through $\phi_{\ell,i}|_{G^0(F)_{[x]}} = \phi_i^\ell|_{G^0(F)_{[x]}}$ times the Weil representation of $G^0(F)_{[x]}$ through the map
\[
G^0(F)_{[x]} \to \Sp(V_{i+1}).
\]
These observations combine to show
\[
\epsilon_\ell|_K \cdot \rT^m(\tau_\ell) \otimes \bigotimes_{i=0}^{d-1} \rT^{n_i}(V_{\wh\phi_{\ell,i}}) \otimes \phi_{\ell,d}|_K = \epsilon^{(\ell)} \otimes \tau^{(\ell)} \otimes \bigotimes_{i=0}^{d-1} V_{\wh\phi_i}^{(\ell)} \otimes \phi_d^{(\ell)}.
\]
Since the compact induction of the left side is a direct summand of $\rT^a(\pi(\Psi_\ell))$, the theorem follows.
\end{proof}

\begin{cor}\label{cor:yu-datum-base-change-parameter}
Let $\ell$ be a banal prime for $G$ such that $\ell > \rk G + 1$, and let $\Psi$ be a Yu datum for $G$ over $\ol\F_\ell$. Then we have 
\[
\rho^{\FS}(\pi(\Psi))|_{W_{F_\ell}} \sim \rho^{\FS}(\pi(\Psi_\ell)).
\]
\end{cor}

\begin{proof}
Combining Theorem~\ref{thm:modular-functoriality} with Theorem~\ref{thm:yu-datum-base-change}, compatibility with Weil restriction \cite[Theorem I.9.6(vii)]{FS}, and \eqref{eq:cyclic-base-change-L-map} gives the following conjugacy; \cite[Lemma~10.1.6(1)]{CF26a} removes the semisimplification:
\[
\rho^{\FS}(\pi(\Psi)^{(\ell)})|_{W_{F_\ell}}
\sim\rho^{\FS}(\pi(\Psi_\ell)^{(\ell)}).
\]
Apply \eqref{eq:fs-coefficient-frobenius} and the inverse of Frobenius to obtain the corollary. 
\end{proof}

\subsection{Small prime degree base change}\label{ss:yu-low-deg-bc}

In this section, we prove analogues of the results of \cite[\S 8.3]{CF26a} for $\rho^{\FS}$. We take $k = \ol\F_\ell$, assume that $G^0 = T$ is a torus, and assume that $\Psi$ is normalized in the sense of \cite[Definition 3.7.1]{Kal19}. Let $E/F$ be a Galois extension of prime degree $\ell \neq p$ such that $T_E$ is an elliptic $E$-torus of $G_E$, and let $\sigma$ be a generator of $\Gal(E/F)$.

Recall the Yu datum $\Psi_E=((G_E^i),x_E,(r_{E,i}),\tau_E,(\phi_{E,i}))$ defined in \cite[\S 8.3]{CF26a}.

\begin{thm}\label{thm:yu-datum-small-degree-bc}
In the setting above, $\rT^a(\sigma,\pi(\Psi_E))$ admits $\pi(\Psi^{(\ell)})$ as a direct summand for each $a\in\Z/2$.
\end{thm}

\begin{proof}
Let $\epsilon_F$, $\epsilon_{F,\sharp,x}$, $\epsilon_E$, and $\epsilon_{E,\sharp,x}$ be the sign characters as in the definition of $\Psi_E$. With notation as in \cite[\S 6]{CF26a}, we have (cf.\ \cite[(6.5.1)]{CF26a})
\[
\pi(\Psi_E) = \cInd_{K_E}^{G(E)}(\epsilon_E \otimes \tau_E \otimes \bigotimes_{i=0}^{d-1} V_{\wh\phi_{E,i}} \otimes \phi_{E,d}),
\]
where (cf.\ \cite[(6.2.1)]{CF26a})
\[
K_E = G^0(E)_{[x]}G^1(E)_{x,\frac{r_0}{2}} \cdots G^d(E)_{x,\frac{r_{d-1}}{2}}.
\]
By \cite[Lemma~6.4.2]{CF26a}, we have $K_E^\sigma = K$, where $K$ is the compact open subgroup of $G(F)$ constructed in precisely the same way as $K_E$ but with $F$ in place of $E$. Thus by \cite[Lemma~5.2.1]{Cot26b} and Corollary~\ref{cor:tate-cohom-tensor-product-yu}, for a given $a \in \Z/2\Z$ there exist $m, n_0,\dots,n_{d-1} \in \Z/2\Z$ such that $m + \sum_{i=0}^{d-1} n_i = a$ and the $G(F)$-representation $\rT^a(\pi(\Psi_E))$ admits
\[
\cInd_K^{G(F)}(\epsilon_E \otimes \tau_E \otimes \bigotimes_{i=0}^{d-1} \rT^{n_i}(V_{\wh\phi_{E,i}}) \otimes \phi_{E,d})
\]
as a direct summand. By construction of $\Psi_E$ and \cite[Proposition~7.1.2]{CF26a}, we have $(\epsilon_E\epsilon_{E,\sharp,x}\tau_E)|_K = \epsilon_F\epsilon_{F,\sharp,x}\tau^\ell$ and $\phi_{E,d}|_K = \phi_d^\ell|_K$. Thus the argument in the proof of Theorem~\ref{thm:yu-datum-base-change} shows that $\rT^{n_i}(V_{\wh\phi_{E,i}})$ is isomorphic to $V_{\wh\phi_i}^{(\ell)} \otimes \epsilon_{E,\sharp,x}\epsilon_{F,\sharp,x}$ as a $K$-representation, with action as in Yu's construction applied to $\Psi$ (cf.\ \cite[\S 6.3 and Lemma~6.4.5]{CF26a}). Combining these observations yields the theorem.
\end{proof}

\begin{cor}\label{cor:yu-datum-small-degree-bc-parameter}
In the setting of Theorem~\ref{thm:yu-datum-small-degree-bc}, we have
\[
\rho^{\FS}(\pi(\Psi))|_{W_E} \sim \rho^{\FS}(\pi(\Psi_E)).
\]
\end{cor}

\begin{proof}
By Theorem~\ref{thm:yu-datum-small-degree-bc}, the representation $\pi(\Psi)^{(\ell)}$ is an irreducible subquotient of $\rT^a(\sigma, \pi(\Psi_E))$. Apply Theorem~\ref{thm:modular-functoriality} to $\Res_{E/F}(G_E)$ with its $\sigma$-action and connected fixed subgroup $G$. Then \cite[Lemma~10.1.6(1)]{CF26a} removes the semisimplification and gives
\[
\ld \psi \circ \rho^{\FS}(\pi(\Psi)^{(\ell)})\ \sim \rho^{\FS}(\pi(\Psi_E)^{(\ell)}),
\]
where $\pi(\Psi_E)^{(\ell)}$ is regarded as a representation of $G'(F) = G(E)$. Compatibility of $\rho^{\FS}$ with Weil restriction \cite[Theorem I.9.6(vii)]{FS} and the cyclic base change formula \eqref{eq:cyclic-base-change-L-map} imply that
\[
\rho^{\FS}(\pi(\Psi)^{(\ell)})|_{W_E} \sim \rho^{\FS}(\pi(\Psi_E)^{(\ell)}),
\]
where now $\pi(\Psi_E)^{(\ell)}$ is regarded as a representation of $G(E)$. The corollary then follows from compatibility of $\rho^{\FS}$ with Frobenius twists.
\end{proof}

\subsection{Descent to unramified twisted Levis}\label{ss:yu-descent-levi}

In this subsection, we prove analogues of the results of \cite[\S 9]{CF26a} for $\rho^{\FS}$.

\subsubsection{The setting}\label{sss:levi-descent-setting}

Throughout, we take $k = \ol\Z_\ell$ and we assume that the Yu datum $\Psi \otimes_{\ol\Z_\ell} \ol\Q_\ell$ is irreducible. We fix the following situation, which will remain in force through Corollary~\ref{cor:yu-datum-dl-res-parameter}. Let $H \subset G$ be an unramified twisted Levi, and suppose that there exists $[(T, \theta)] \in \cT_0(\Psi \otimes_{\ol\Z_\ell} \ol\Q_\ell)$ (in the notation of \cite[Definition~7.2.1]{CF26a}) such that $T$ is elliptic in $G^0$ and $T \subset H \cap G^0$. Fix an element $s \in H(F)$ of order $\ell$ such that $H = Z_G(s)$\footnote{For a given $H$, such an element $s$ might not exist, so this includes an extra hypothesis on the data.}, and let $\sigma$ denote the automorphism of $G$ induced by conjugation by $s$. For $0 \leq i \leq d$, set $H^i = G^i \cap H$ and $\phi_{H,i} = \phi_i|_{H^i(F)}$. We record the following observations.
\begin{enumerate}
\item Each $H^i$ is a twisted Levi $F$-subgroup of $H$, and $Z_{H^0}/Z_H$ is anisotropic: this follows from the assumption that $T$ is an elliptic maximal torus of $G$.
\item Via any choice of embedding $\cB(H^0) \subset \cB(G^0)$, the point $x$ lies in $\cB(H^0)$. The image $[x]$ of $x$ in $\cB(H^0_{\der})$ is a vertex, since $[x]$ is the point associated to the elliptic maximally unramified maximal $F$-torus $T$: see \cite[Remark 16.7]{KP}.
\item $\phi_{H,i}$ is $H^{i+1}$-generic of depth $r_i$ for all $0\leq i < d$ by \cite[Lemma~9.3.1]{CF26a}.
\end{enumerate}
For simplicity, we will use $\epsilon_G$ (resp.\ $\epsilon_H$) to refer to the sign character constructed from $((G^i), x, (r_i))$ (resp.\ $((H^i), x, (r_i))$); see \cite[\S\S 3.3 and 6.5]{CF26a}.

\subsubsection{The twisting characters}\label{sss:levi-descent-twisting}

For $0 \leq i \leq d-1$, recall from \cite[\S 6.3]{CF26a} the symplectic $\F_p$-vector space $V_{i+1}$: by the Moy--Prasad isomorphism (cf.\ \cite[(6.3.1)]{CF26a}) we have
\[
V_{i+1} \cong \frg^{i+1}(F)_{x,\frac{r_i}{2}}/\left(\frg^{i+1}(F)_{x,\frac{r_i}{2}+} + \frg^i(F)_{x,\frac{r_i}{2}}\right),
\]
and if $X_i^*\in \Lie^*(G^i)^{G^i}(F)_{x, -r_i}$ denotes the generic functional associated to $\phi_i$ then the symplectic form on $V_{i+1}$ is equal to the map $(\ol v, \ol w) \mapsto \Tr_{\F_q/\F_p}\overline{X_i^*([v,w])}$ by \cite[Lemma~6.3.1]{CF26a}. Note that $V_{i+1}$ is an $\F_q$-vector space in a natural way, the map $(\ol v, \ol w) \mapsto \overline{X_i^*([v,w])}$ is a symplectic form, and there is a natural inclusion $\Sp_{\F_q}(V_{i+1}) \subset \Sp_{\F_p}(V_{i+1})$, where the subscripts $\F_q$ and $\F_p$ are used to distinguish the two symplectic forms. It is clear that $\sigma \in \Sp_{\F_q}(V_{i+1})$.

\begin{defn}\label{defn:levi-descent-twisting}
For $0 \leq i \leq d-1$, let $\wt\chi_{1,i}$ denote the sign character of $\Sp_{\F_q}(V_{i+1})^\sigma$ from Definition~\ref{def:chi-sigma}, applied to the $\F_q$-symplectic space $V_{i+1}$ and the element $\sigma \in \Sp_{\F_q}(V_{i+1})$. Let $a_i\co H^0(F)_{[x]} \to \Sp_{\F_q}(V_{i+1})^\sigma$ be the natural map, and define the characters
\[
\eta_i = \wt\chi_{1,i}\circ a_i, \qquad \xi_H = \epsilon_G|_{H^0(F)_{[x]}}\,\epsilon_H\prod_{i=0}^{d-1}\eta_i
\]
of $H^0(F)_{[x]}$.
\end{defn}

The following lemma is the key computation of this subsection. We will use $\mathbf{Sp}_{\F_q}(V_{i+1})$ (resp.\ $\mathbf{Sp}_{\F_p}(V_{i+1})$) to denote the symplectic group of $V_{i+1}$ (resp.\ $V_{i+1, p}$, in the notation of \S\ref{ss:weil-heisenberg-restriction-of-scalars}), considered as an algebraic group over $\F_q$ (resp.\ $\F_p$).

\begin{lemma}\label{lemma:eta-sign-comparison}
If $\epsilon_i = \epsilon_{\sharp,x}^{G^{i+1}/G^i,H}$ as in \cite[Lemma~3.2.2]{CF26a}, then
\begin{equation}\label{eq:eta-sign-comparison}
\eta_i \equiv \epsilon_i|_{H^0(F)_{[x]}} \pmod{\ell}.
\end{equation}
Thus, if $T_H \subset H^0$ is a tamely ramified maximal $F$-torus with $x \in \cB(T_H)$, then
\[
\left.\xi_H\right|_{T_H(F)_{[x]}} \equiv \left.(\epsilon_G\epsilon_H\epsilon_{G,\sharp,x}\epsilon_{H,\sharp,x})\right|_{T_H(F)_{[x]}} \pmod{\ell}.
\]
\end{lemma}

\begin{proof}We may and do assume that $G$ is of adjoint type, so $[x] = x$. First, we claim that if $f_{i+1}\co \ol G_x^0 \to \mathbf{Sp}_{\F_q}(V_{i+1})$ is the $\F_q$-homomorphism defined in \cite[Lemma~6.3.2]{CF26a}, then for every tamely ramified maximal $F$-torus $S \subset G^0$ with $x \in \cB(S)$, the image of $\ol S_x$ under $f_{i+1}$ lies in a maximal $\F_q$-torus of $\mathbf{Sp}_{\F_q}(V_{i+1})$. For this, note that the Kottwitz isomorphism \cite[Corollary 11.7.2]{KP} and the fact that $S$ is tame imply that $\pi_0(\ol S_x)(\ol\F_q)$ is of order prime to $p$. Thus $f_{i+1}(\ol S_x(\ol\F_q))$ is a subgroup of $\mathbf{Sp}_{\F_q}(V_{i+1})(\ol\F_q)$ consisting of semisimple elements, so it is a classical fact that $f_{i+1}(\ol S_x)$ lies in a maximal $\F_q$-torus of $\mathbf{Sp}_{\F_q}(V_{i+1})$. If $S \subset H^0$, then this torus can be chosen to lie in $\mathbf{Sp}_{\F_q}(V_{i+1})^\sigma$.

Now one can argue as in the proof of \cite[Lemma~3.2.2]{CF26a}. If $\ell=2$, then both sides of \eqref{eq:eta-sign-comparison} are trivial, so assume that $\ell$ is odd. Choose a tamely ramified maximal $F$-torus $T_H \subset H^0$ such that $x \in \cB(T_H)$, and let $\ol S_i$ be a maximal $\F_q$-torus of $\mathbf{Sp}_{\F_q}(V_{i+1})^\sigma$ such that $f_{i+1}(\ol T_H) \subset \ol S_i$. Let $Z_i \subset \mathbf{Sp}_{\F_q}(V_{i+1})^\sigma$ be the maximal central $\F_q$-torus.

Let $V_{i+1,0}\coloneqq V_{i+1}^\sigma$ and $V_{i+1,1}\coloneqq V_{i+1,0}^\perp$, so $V_{i+1} = V_{i+1, 0} \oplus V_{i+1, 1}$. Let $\mathfrak{S}_i$ be the set of weights for the action of $(Z_i)_{\ol\F_q}$ on $(V_{i+1,1})_{\ol\F_q}$, equipped with its natural action of $\Gamma=\Gal(\ol\F_q/\F_q)$. Let
\[
(V_{i+1,1})_{\ol\F_q}=\bigoplus_{\cO\in\mathfrak{S}_i}(V_{i+1,1})_\cO
\]
be the weight space decomposition for $Z_i$. Since $Z_i$ is central in $\mathbf{Sp}_{\F_q}(V_{i+1})^\sigma$, every summand is $(\ol H_x^0)_{\ol\F_q}$-stable; define $\chi_\cO=\det(-|(V_{i+1,1})_\cO)$.
Then \cite[Proposition 5.1.13]{FKS23} describes a character $\epsilon_{\mathfrak{S}_i}\co\ol H_x^0(\F_q)\to\{\pm1\}$ associated to $\mathfrak{S}_i$ and the natural map $\mathfrak{S}_i\to X^*((\ol H_x^0)_{\ol\F_q})$ given by $\cO\mapsto\chi_\cO$.

Since $\ell\ne p$, taking $\sigma$-invariants is exact, so \cite[Lemmas~6.4.4(1) and 6.3.1]{CF26a} identify $V_{i+1,1}\simeq V_{i+1}/V_{i+1}^\sigma$ with the quotient $V$ in \cite[Lemma~3.2.2]{CF26a}. Applying \cite[Remark 5.1.12]{FKS23} to
the restriction maps from the weights of $\ol S_i$ on $V_{i+1,1}$ to the weights of $Z_i$ and $\ol T_H$
as in the proof of \cite[Lemma~3.2.2]{CF26a}, we see that the restriction of $\epsilon_{\mathfrak{S}_i}$ to $T_H(F)_{[x]}$ is equal to the restrictions of the left and right sides of \eqref{eq:eta-sign-comparison}, respectively. Since $T_H$ is arbitrary, \cite[Lemma 3.5]{FKS23} shows that this property uniquely determines $\eta_i$, proving \eqref{eq:eta-sign-comparison}.
\end{proof}

\subsubsection{The descent theorem}

Let $\tau_{H,0}$ be a cuspidal irreducible $\ol\F_\ell$-representation of $\ol H^0_{[x]}(\F_q)$, and set $\tau_H \coloneqq \tau_{H,0}\otimes\xi_H$. Define
\begin{equation}\label{eq:levi-Psi-H}
\Psi_{H, \tau_H} = ((H^i), x, (r_i), \tau_H, (\ol \phi_{H,i})),
\end{equation}
where $\ol\phi_{H,i}$ is the mod $\ell$ reduction of $\phi_{H,i}$. It is easy to check that $\Psi_{H,\tau_H}$ is a Yu datum for $H$ over $\ol\F_\ell$, and if $\Psi$ is normalized then so is $\Psi_{H,\tau_H}$. Recall the pair $[(T,\theta)] \in \cT_0(\Psi \otimes_{\ol\Z_\ell} \ol\Q_\ell)$ fixed at the beginning of this section.

\begin{thm}\label{thm:yu-datum-dl-res}
Keep the setting of \S\ref{sss:levi-descent-setting}. Suppose $[\ol G^0_{[x]}(\F_q): (\ol G^0_{[x]})^\circ(\F_q) \cdot Z(\ol G_{[x]}^0)(\F_q)]$ is prime to $\ell$ and $\ell$ is good for $G$ and $\theta$ is of finite order prime to $\ell$ and the $\ol\Q_\ell$-representation $\tau \otimes_{\ol\Z_\ell} \ol\Q_\ell$ is defined over a finite extension of $\Q_\ell^{\unr}$ of degree prime to $\ell-1$.
Then there exists a cuspidal irreducible $\ol\F_\ell$-representation $\tau_{H,0}$ of $\ol H^0_{[x]}(\F_q)$, whose Brauer character occurs with nonzero coefficient in the $\ell$-modular reduction of the Lusztig restriction ${}^*R^{\ol G^0_{[x]}}_{\ol H^0_{[x]}}(\tau\otimes_{\ol\Z_\ell}\ol\Q_\ell)$, such that for every $a \in \Z/2\Z$ the representation $\pi(\Psi_{H,\tau_H})$ is an irreducible subquotient of $\rT^a(\sigma, \pi(\Psi) \otimes_{\ol\Z_\ell} \ol\F_\ell)$.
\end{thm}

\begin{proof}
We use all of the notation introduced in \cite[\S 6]{CF26a} freely, and we write $K_G$ (resp.\ $K_H$) for the compact open subgroup (cf.\ \cite[(6.2.1)]{CF26a}) of $G$ (resp.\ of $H$) constructed from $\Psi$ (resp.\ from $\Psi_{H,\tau_H}$; this subgroup does not depend on the choice of $\tau_{H,0}$). By definition (cf.\ \cite[(6.5.1)]{CF26a}, with $\wh\phi_i$ as in \cite[\S 6.2]{CF26a} and $V_{\wh\phi_i}$ as in \cite[\S 6.3]{CF26a}), we have
\[
\pi(\Psi) = \cInd_{K_G}^{G(F)}(\epsilon_G \otimes \tau \otimes \bigotimes_{i=0}^{d-1} V_{\wh\phi_i} \otimes \phi_d).
\]
Since $T \subset H^0$, the point $x$ lies in $\cB(H^0)$ and thus the finite order central element $s$ fixes $x$, so $s \in G^0(F)_{[x]} \subset K_G$; we regard each tensor factor of the inducing representation as a $K_G \rtimes \langle\sigma\rangle$-module by letting $\sigma$ act through $s$. In particular, Corollary~\ref{cor:tate-cohom-tensor-product-yu} applies below with first factor the possibly reducible module $\epsilon_G \otimes \tau_{\ol\F_\ell}$.

Fix $a \in \Z/2\Z$. By \cite[Lemma~5.2.1]{Cot26b} and Corollary~\ref{cor:tate-cohom-tensor-product-yu} (applicable because of Lemma~\ref{lemma:tate-cohom-of-heisenberg}), there exist $m, n_0,\dots,n_{d-1} \in \Z/2\Z$ such that $m + \sum_{i=0}^{d-1} n_i = a$ and such that the Tate cohomology $\rT^a(\sigma, \pi(\Psi)\otimes_{\ol\Z_\ell}\ol\F_\ell)$ admits the following representation as a direct summand:
\begin{equation}\label{eq:levi-descent-step1-summand}
\cInd_{K_G^\sigma}^{H(F)}(\epsilon_G \otimes \rT^m(\tau \otimes_{\ol\Z_\ell} \ol\F_\ell) \otimes \bigotimes_{i=0}^{d-1} \rT^{n_i}(V_{\wh\phi_i} \otimes_{\ol\Z_\ell} \ol\F_\ell) \otimes \phi_d)
\end{equation}
By \cite[Lemma~6.4.2(1)]{CF26a}, we have $K_G^\sigma = K_H$.

By assumption on $\tau \otimes_{\ol\Z_\ell} \ol\Q_\ell$, we may apply \cite[Proposition~3.4.1(2)]{Cot26b} to the character of the generic fiber $\tau\otimes_{\ol\Z_\ell}\ol\Q_\ell$ and to the same lattice $\tau$ used to form Tate cohomology to obtain the inequality of characters (with notation as in \cite[\S 3.2]{Cot26b})
\[
\chi_{\rT^m(\sigma,\tau\otimes_{\ol\Z_\ell}\ol\F_\ell)}
\geq 
\left|\ol{{}^*R^{\ol G^0_{[x]}}_{\ol H^0_{[x]}}
(\tau\otimes_{\ol\Z_\ell}\ol\Q_\ell)}\right|
\]
for both $m \in \Z/2$. Note that the assumption on $\tau \otimes_{\ol\Z_\ell} \ol\Q_\ell$ implies that the depth $0$ character of $T(F)_{[x]}$ associated to $\Psi \otimes_{\ol\Z_\ell} \ol\Q_\ell$ in \cite[\S 7.2]{CF26a} is of finite order prime to $\ell$. By the assumption $[(T, \theta)] \in \cT_0(\Psi)$, the representation $\tau \otimes_{\ol\Z_\ell} \ol\Q_\ell$ has nonzero pairing with $R_{\ol T_{[x]}}^{\ol G_{[x]}}(\theta_0)$. By \cite[Proposition~3.4.1]{Cot26b} (using that $\ell$ is good), it follows that there is an irreducible $\ol\F_\ell$-representation $\tau_{H,0}$ of $\ol H^0_{[x]}(\F_q)$ whose Brauer character occurs with nonzero coefficient in $\ol{{}^*R^{\ol G^0_{[x]}}_{\ol H^0_{[x]}}(\tau \otimes_{\ol\Z_\ell}\ol\Q_\ell)}$. Thus $\rT^m(\tau\otimes_{\ol\Z_\ell}\ol\F_\ell)$ admits $\tau_{H,0}$ as an irreducible constituent. The representation $\tau\otimes_{\ol\Z_\ell}\ol\F_\ell$ is cuspidal since $\tau \otimes_{\ol\Z_\ell} \ol\Q_\ell$ is cuspidal, so $\tau_{H,0}$ is cuspidal by \cite[Proposition~3.4.2]{Cot26b}.

Next, fix $0\leq i\leq d-1$; we will compute $\rT^{n_i}(V_{\wh\phi_i})$. By Proposition~\ref{prop:tate-cohom-of-weil}, the $\Sp(V_{i+1})^\sigma \ltimes (V_{i+1}^\sigma)^\sharp$-representation $\rT^{n_i}(V_{\wh\phi_i})$ is isomorphic to the tensor product of the Weil--Heisenberg representation of $\Sp(V_{i+1}^\sigma) \ltimes (V_{i+1}^\sigma)^\sharp$ (under the restriction map $\Sp(V_{i+1})^\sigma \to \Sp(V_{i+1}^\sigma)$) and a certain $\{\pm1\}$-valued character $\chi_{1,i}$ of $\Sp(V_{i+1})^\sigma$. By Lemma~\ref{lemma:sign-char-weil-restriction}, applied to the $\F_q$-symplectic structure on $V_{i+1}$ recalled in \S\ref{sss:levi-descent-twisting}, we have $\wt\chi_{1,i} = \chi_{1,i}|_{\Sp_{\F_q}(V_{i+1})^\sigma}$, with $\wt\chi_{1,i}$ as in Definition~\ref{defn:levi-descent-twisting}. In particular, with notation from Definition~\ref{defn:levi-descent-twisting}, the action of $H^0(F)_{[x]}$ on $\rT^{n_i}(V_{\wh\phi_i})$ through $a_i$ is the twist by $\eta_i$ of its action on the Weil--Heisenberg representation of $\Sp(V_{i+1}^\sigma) \ltimes (V_{i+1}^\sigma)^\sharp$. This Weil--Heisenberg representation is isomorphic to $V_{\wh\phi_{H,i}}$ by \cite[Lemmas~6.4.4(1), 6.4.3(1), and 6.4.5(1)]{CF26a}: the first shows $J_{i+1}(\vec G)^\sigma = J_{i+1}(\vec H)$ and $J_{i+1}^+(\vec G)^\sigma = J_{i+1}^+(\vec H)$; the second shows that $\wh\phi_{H,i}$ is the restriction of $\wh\phi_i$; and the third relates the special isomorphisms used in the construction. Thus, as a $K_H$-representation, $\rT^{n_i}(V_{\wh\phi_i})$ has an irreducible constituent isomorphic to $V_{\wh\phi_{H,i}}\otimes\eta_i$.

We have now seen that the $H(F)$-representation in \eqref{eq:levi-descent-step1-summand} has as an irreducible subquotient
\[
\cInd_{K_H}^{H(F)}(\epsilon_G|_{K_H}\otimes \tau_{H,0}\otimes \bigotimes_{i=0}^{d-1}(V_{\wh\phi_{H,i}}\otimes\eta_i)\otimes \phi_d|_{K_H}).
\]
Since $\tau_H = \tau_{H,0}\otimes\xi_H$ and $\epsilon_H^2=1$, the factor $\epsilon_G|_{H^0(F)_{[x]}}\otimes \tau_{H,0}\otimes\prod_{i=0}^{d-1}\eta_i$ is equal to $\epsilon_H\otimes\tau_H$. Thus indeed $\pi(\Psi_{H,\tau_H})$ is an irreducible subquotient of $\rT^a(\pi(\Psi) \otimes_{\ol\Z_\ell} \ol\F_\ell)$.
\end{proof}

\begin{cor}\label{cor:yu-datum-dl-res-parameter}
Under the hypotheses of Theorem~\ref{thm:yu-datum-dl-res}, there exists some $\tau_{H,0}$ such that, for some irreducible subquotient $\ol\pi_0$ of $\pi(\Psi) \otimes_{\ol\Z_\ell} \ol\F_\ell$, there is a cocycle $z\co W_F/I_F \to Z(\wh H \cap \wh G_{\der})(\ol\F_\ell)^{I_F}$ such that 
\begin{equation}\label{eqn:fs-unram-modular-ss}
\rho^{\FS}(\ol\pi_0) \sim \bigl(z \cdot \ld j_{H,G} \circ \rho^{\FS}(\pi(\Psi_{H,\tau_H}))\bigr)^{\ss},
\end{equation}
where $\ld j_{H,G}\co \ld H \to \ld G$ is the canonical L-embedding as in \cite[Definition~4.5.1]{CF26a}. If $\ell>\rk G+1$, the superscript $\ss$ can be omitted.
\end{cor}

\begin{proof}
Since $\Psi_{H,\tau_H}$ is an irreducible Yu datum, the representation $\pi(\Psi_{H,\tau_H})$ is irreducible, so Theorem~\ref{thm:yu-datum-dl-res} says that $\pi(\Psi_{H,\tau_H})$ itself occurs as an irreducible subquotient of $\rT^a(\sigma,\pi(\Psi)\otimes_{\ol\Z_\ell}\ol\F_\ell)$ for both $a \in \Z/2\Z$. The representation $\pi(\Psi)\otimes_{\ol\Z_\ell}\ol\F_\ell$ is admissible of finite length: by \cite[Lemma~7.3.1(1) and (2)]{CF26a}, its irreducible subquotients are the (irreducible by \cite[Theorem 3.1]{Fin22}) representations attached by Yu's construction to the finitely many irreducible subquotients of $\tau \otimes_{\ol\Z_\ell} \ol\F_\ell$. By \cite[Lemma~3.1.4]{Cot26b}, there is an irreducible subquotient $\ol\pi_0$ of $\pi(\Psi)\otimes_{\ol\Z_\ell}\ol\F_\ell$, with $\sigma$-stable isomorphism class, such that $\pi(\Psi_{H,\tau_H})$ is an irreducible subquotient of $\rT^a(\sigma, \ol\pi_0)$. Apply Theorem~\ref{thm:modular-functoriality} to the irreducible representation $\ol\pi_0$, giving
\begin{equation}\label{eq:dl-descent-frobenius}
\bigl(\ld\psi\circ\rho^{\FS}(\pi(\Psi_{H,\tau_H}))\bigr)^{\ss}
\sim\rho^{\FS}(\ol\pi_0^{(\ell)}) \sim z \cdot\Fr_\ell\circ\rho^{\FS}(\ol\pi_0),
\end{equation}
for some cocycle $z \co W_F \rightarrow Z(\wh G_{\der})(\ol\F_\ell)^{W_F}$,
where the second conjugacy is \eqref{eq:fs-coefficient-frobenius}. Proposition~\ref{prop:sigma-dual-unram-agreement-on-inertia}(2) gives a cocycle $z_0\co W_F/I_F\to Z(\wh H\cap\wh G_{\der})(\ol\F_\ell)^{I_F}$ such that $\ld\psi\sim z_0\cdot\Fr_\ell\circ\ld j_{H,G}$. Comparing with \eqref{eq:dl-descent-frobenius} and applying the inverse of Frobenius yields \eqref{eqn:fs-unram-modular-ss}. (The cocycle appearing in \eqref{eq:fs-coefficient-frobenius} is unramified and valued in $Z(\wh G_{\der})\subset Z(\wh H\cap\wh G_{\der})$, so it may be absorbed in $z$.) Finally, \cite[Lemma~10.1.6(3)]{CF26a} removes the semisimplification when $\ell>\rk G+1$.
\end{proof}

\section{Compatibility with parabolic induction}\label{ssec:modular-cuspidal-parabolic-induction}

The main technical result of this paper (Theorem~\ref{thm:unramified-twisted-Levi-functoriality}) is a functoriality result for $\ol\Q_\ell$-representations $\pi$ of $G(F)$ with respect to descent to unramified twisted Levis $H(F)$ of $G(F)$. The rough goal of the proof is to use the modular functoriality results of \S\ref{section:tate-cohom-yu} and independence of $\ell$ \cite[Theorem 1.1]{Sch25} to reduce to the case that $\ell$ is ``large'', then apply Corollary~\ref{cor:yu-datum-dl-res-parameter}. From $\pi$, we use Tate cohomology to extract an explicit $\ol\F_\ell$-representation $\ol\pi_H$ of $H(F)$ such that (for instance)
\[
\rho^{\FS}(\pi)|_{I_F} \equiv \rho^{\FS}(\ol\pi_H)|_{I_F} \pmod{\ell}.
\]
This is nearly enough to obtain a functoriality result with $\ol\Q_\ell$-coefficients, but this argument does not produce a $\ol\Q_\ell$-representation $\pi_H$ such that $\ol\pi_H$ occurs in the mod $\ell$ reduction of some $\ol\Z_\ell$-lattice in $\pi_H$. If $\ol\pi_H$ is non-singular, then it is not difficult to construct such a $\pi_H$, but in general we do not know how to do this.

To explain the situation, suppose $\ol\Psi_H = ((H^i), x, (r_i), \ol\tau_H, (\ol\phi_{H,i}))$ is a Yu datum for $H(F)$ over $\ol\F_\ell$ such that $\ol\pi_H \cong \pi(\ol\Psi_H)$. The natural way to obtain a lift $\pi_H$ would be to find a cuspidal $\ol\Q_\ell$-representation $\tau_H$ of $\ol H_{[x]}(\F_q)$ such that $\ol\tau_H$ occurs in the mod $\ell$ reduction of a lattice in $\tau_H$ and let $\pi_H = \pi(\Psi_H)$, where $\Psi_H = ((H^i), x, (r_i), \tau_H, (\phi_{H,i}))$ and $\phi_{H,i}$ is the Teichm\"uller lift of $\ol\phi_{H,i}$. We do not know whether such a $\tau_H$ exists; see \cite[Question~2.10.4]{Cot26b}. (What is true is that there often exist non-cuspidal such $\tau_H$.) Thus we need a way around this problem.

First, by \cite[Part III, no.\ 16.1, Theorem 33]{Serre77}, there is \emph{some} (not necessarily cuspidal) $\ol\Q_\ell$-representation $\tau_H$ of $\ol H_{[x]}(\F_q)$ such that $\ol\tau_H$ occurs in the mod $\ell$ reduction of $\tau_H$. Using uniqueness of cuspidal supports for paraductive $\F_q$-group schemes (\cite[Lemma~2.10.2]{Cot26b}), one shows that there is a Levi $F$-subgroup $L \subset H$ such that $x \in \cB(L)$ and a cuspidal $\ol\Q_\ell$-representation $\tau_L$ of $\ol L_{[x]_G}(\F_q)$ such that $\tau_H$ occurs in the parabolic induction of $\tau_L$, and from this one obtains a Yu datum $\Psi_L$ for $L$ over $\ol\Q_\ell$.\footnote{This is slightly incorrect for technical reasons, both because $\ol L_{[x]_G} \neq \ol L_{[x]_L}$ and because one must twist by a sign character to get the correct Yu datum for functoriality.} This is the same construction as in \cite[\S 9]{CF26a}, and \cite[Theorem~9.3.3]{CF26a} relates $\rho^{\Kal}(\ol\Psi_L)$ to $\rho^{\Kal}(\ol\Psi_H)$. It remains to relate $\rho^{\FS}(\pi(\ol\Psi_L))$ to $\rho^{\FS}(\pi(\ol\Psi_H))$, for which the obvious tool is compatibility with parabolic induction \cite[Theorem I.9.6(viii)]{FS}. The main result for this purpose is Proposition~\ref{prop:tame-cuspidal-parabolic-induction}, which shows that $\pi(\ol\Psi_H)$ occurs in the parabolic induction of $\pi(\ol\Psi_L) \otimes \chi$ for some unramified character $\chi$, and from this we deduce the crucial lifting result Corollary~\ref{cor:lift-tame-reps-to-char-0}.

In this section, we retain the standing hypothesis $p \neq 2$.

\subsection{The setting}\label{sss:parabolic-descent-setting}

We fix the following situation, which will remain in force through the proof of Proposition~\ref{prop:tame-cuspidal-parabolic-induction} below. Let $H$ be a connected reductive group over $F$, and let $L$ be a Levi $F$-subgroup of a parabolic $F$-subgroup $P$ of $H$. Let
\[
\Psi_H = ((H^i)_{0 \leq i \leq d}, x, (r_i), \tau_H, (\phi_i))
\]
be an irreducible Yu datum for $H$ over $\ol\F_\ell$. Let $P = LU$ be an $F$-parabolic subgroup of $H$, let $\ol P$ be the parabolic opposite to $P$ with respect to $L$, and let $\ol U$ be its unipotent radical. Fix a commutative diagram of embeddings
\[
\begin{tikzcd}
\cB(L^0) \arrow[r] \arrow[d]
&\cB(L^1) \arrow[r] \arrow[d]
&\cdots \arrow[r]
&\cB(L) \arrow[d] \\
\cB(H^0) \arrow[r]
&\cB(H^1) \arrow[r]
&\cdots \arrow[r]
&\cB(H)
\end{tikzcd}
\]
and suppose that $x \in \cB(L^0)$ under the embeddings $\cB(L^0) \subset \cB(H^0)$.\footnote{We emphasize that $x$ is a point of $\cB(H^0)$ chosen in advance, and since $Z(H^0)/Z(H)$ is anisotropic, no group defined below using the Moy--Prasad filtration (for $H^i$ or $L^i$) below depends on the choice of these embeddings.} We will use $\iota_0$ to denote any of these chosen embeddings. For $0 \leq i \leq d$, set
\[
L^i = L \cap H^i, \qquad P^i = P \cap H^i, \qquad U^i = U \cap H^i, \qquad \ol U^i = \ol U \cap H^i.
\]
We assume that  
\begin{itemize}
\item the image $[x]_L$ in $\cB((L^0)_{\der})$ is a vertex, and
\item $Z(L^0)/Z(L)$ is anisotropic, i.e., condition (1) in the definition of a Yu datum (\cite[\S 6.1]{CF26a}) for the sequence $(L^i)$. 
\end{itemize}
In the applications in this paper, the two displayed hypotheses will be verified using \cite[Lemmas~9.1.2 and 9.1.1]{CF26a}.

Choose an irreducible cuspidal $\ol\F_\ell$-representation $\tau_L$ of $\ol L^0_{[x]_L}(\F_q)$, and let
\[
\Psi_L = ((L^i), x, (r_i), \tau_L, (\phi_i|_{L^i(F)})).
\]
Observe that $\Psi_L$ is a Yu datum for $L$ over $\ol\F_\ell$ by \cite[Lemma~9.3.1]{CF26a} and the above assumptions.

\subsection{Some technical preliminaries}
The following two technical lemmas will be used to pass from a family of representations to a single member. Let $k$ be an algebraically closed field of characteristic $\neq p$.

\begin{lemma}\label{lemma:center-artin-rees}
Let $A$ be a commutative noetherian $k$-algebra, and let $X$ be a smooth $A$-representation of $H(F)$. Assume that $X$ is admissible, i.e., $X^J$ is a finitely generated $A$-module for every compact open subgroup $J \subset H(F)$. Let $V$ be an irreducible subquotient of $X$ on which $A$ acts through a character $\chi\co A \to k$, and let $\mathfrak m = \ker\chi$. Then there is an integer $n \geq 1$ such that $V$ is an irreducible subquotient of $X/\mathfrak m^nX$.
\end{lemma}

\begin{proof}
Choose $A$-stable $H(F)$-subrepresentations $V'' \subset V' \subset X$ with $V'/V'' \simeq V$, and choose a compact open pro-$p$ subgroup $J \subset H(F)$ with $V^J \neq 0$.

Since $J$ is pro-$p$ and $\operatorname{char}k\neq p$, averaging over $J$ is an $A$-linear projection $X\to X^J$. Hence $(\mathfrak bX)^J=\mathfrak bX^J$ for every ideal $\mathfrak b\subset A$.

The module $X^J$ is finitely generated over the noetherian ring $A$, and $(V')^J \subset X^J$ is a submodule with $\mathfrak m\,(V')^J \subset (V'')^J$, because $A$ acts on $V$ through $\chi$. By the Artin--Rees lemma, there is $n \geq 1$ with
\begin{equation}\label{eq:J-invariants}
(\mathfrak m^nX)^J \cap (V')^J = \mathfrak m^n(X^J) \cap (V')^J \subset \mathfrak m\,(V')^J \subset (V'')^J.
\end{equation}
Consider the image of $V'$ in $X/\mathfrak m^nX$. Its quotient by the image of $V''$ is 
\begin{equation}\label{eq:V-quotient}
V'/(V'' + V'\cap\mathfrak m^nX),
\end{equation}
which is a quotient of $V'/V'' \simeq V$, hence either $0$ or $V$. Since $k$ is of characteristic $\neq p$, the formation of $J$-invariants is exact, so \eqref{eq:J-invariants} shows that the image of $V'\cap\mathfrak m^nX$ in $V$ has trivial $J$-invariants. Since $V^J \neq 0$ and $V$ is irreducible, it follows that $V'' + V'\cap\mathfrak m^nX \neq V'$, so \eqref{eq:V-quotient} is a nonzero quotient of $V$, hence isomorphic to $V$. Thus $V$ is an irreducible subquotient of $X/\mathfrak m^nX$.
\end{proof}

\begin{lemma}\label{lemma:family-to-fiber}
Let $K_0 \trianglelefteq K_1 \subset L(F)$ be open subgroups such that 
\begin{enumerate}
\item $K_0$ is compact modulo $Z(L)(F)\cap K_0$,
\item $K_0' \coloneqq (Z(L)(F)\cap K_1)\,K_0$ is of finite index in $K_1$.
\end{enumerate} 
Let $M_0$ be a smooth $k$-representation of $K_0$ of finite length, and let $V$ be an irreducible subquotient of $I_P^H\bigl(\cInd_{K_0}^{L(F)}M_0\bigr)$. Then there is an irreducible subquotient $M_1$ of $\cInd_{K_0}^{K_1}M_0$ such that $V$ is an irreducible subquotient of $I_P^H\bigl(\cInd_{K_1}^{L(F)}M_1\bigr)$.
\end{lemma}

\begin{proof}
Note that $\cInd_{K_0}^{L(F)} M_0$ is a finitely generated $k$-representation of $L(F)$ since $K_0 \subset L(F)$ is open. Let $A$ be the image of the Bernstein center $\frZ_L$ of $\Rep_k(L(F))$ in $\End_k(\cInd_{K_0}^{L(F)} M_0)$, so $A$ is a finitely generated $k$-algebra and $\cInd_{K_0}^{L(F)} M_0$ is an admissible $A$-representation of $L(F)$ by \cite[Theorem 1.2]{DHKM}. Thus $I_P^H(\cInd_{K_0}^{L(F)} M_0)$ is an admissible $A$-representation by \cite[Chapitre II, no.\ 2.1 (i)]{Vig96}. Note that $A$ acts on $V$ through a character $\chi\co A \to k$ by Schur's lemma, so if $\mathfrak{m} = \ker \chi$ then Lemma~\ref{lemma:center-artin-rees} shows that there is some $n$ such that $V$ is an irreducible subquotient of
\begin{equation}\label{eqn:parabolic-induction-tensor-product}
I_P^H(\cInd_{K_0}^{L(F)} M_0) \otimes_A A/\mathfrak{m}^n \cong I_P^H((\cInd_{K_0}^{L(F)} M_0) \otimes_A A/\mathfrak{m}^n).
\end{equation}

We may and do assume that $M_0$ is an irreducible $k$-representation of $K_0$. By (1), the $k$-vector space $M_0$ is finite-dimensional. Choose a $k$-representation $M_0'$ of $K_0'$ extending $M_0$, and let $\zeta\co Z(L)(F) \cap K_1 \to k^\times$ be the central character of $M_0'$ (which exists by Schur's lemma). There is a homomorphism $k[Z(L)(F) \cap K_1] \to A$ whose image is a finitely generated $k$-algebra $B$ (since $\zeta$ is smooth); let $\frp \subset B$ be the preimage of $\mathfrak{m}$. Note that
\[
\cInd_{K_0}^{K_1} M_0 \cong \cInd_{K_0'}^{K_1}(M_0' \otimes_k k[K_0'/K_0]),
\]
so $\cInd_{K_0}^{K_1} M_0$ is a finitely generated $B$-module by (2). Thus $M_0' \otimes_k (k[K_0'/K_0] \otimes_B B/\frp^n)$ is finite-dimensional over $k$. Since
\[
\cInd_{K_0}^{L(F)}(M_0) \otimes_B B/\frp^n \cong \cInd_{K_1}^{L(F)}(\cInd_{K_0'}^{K_1}(M_0' \otimes_k (k[K_0'/K_0] \otimes_B B/\frp^n)))
\]
surjects onto $\cInd_{K_0}^{L(F)}(M_0) \otimes_A A/\mathfrak{m}^n$, the result follows by a filtration argument from \eqref{eqn:parabolic-induction-tensor-product} and exactness of $I_P^H$ \cite[Chapitre II, no.\ 2.1 (v)]{Vig96} and the fact that $K_0'$ is of finite index in $K_1$.
\end{proof}

\subsection{The main result}

Write $\epsilon_H$ and $\epsilon_L$ for the FKS characters \cite[(3.3.1)]{CF26a} attached to the towers $(H^i)$ and $(L^i)$, respectively. The following proposition is the main result of this section. 

\begin{prop}\label{prop:tame-cuspidal-parabolic-induction}
Keep the setting of \S\ref{sss:parabolic-descent-setting}. Let $\tau_{L,0}$ be an irreducible cuspidal $\ol\F_\ell$-representation of $\ol L^0_{[x]_H}(\F_q)$, and suppose that $\tau_H$ is an irreducible subquotient of the parabolic induction $\ind_{\ol P^0_{[x]_H}(\F_q)}^{\ol H^0_{[x]_H}(\F_q)}(\tau_{L,0})$. Then there exists an irreducible cuspidal $\ol\F_\ell$-representation $\tau_L$ of $\ol L^0_{[x]_L}(\F_q)$ such that, writing 
\[
\Psi_L = ((L^i), x, (r_i), \tau_L, (\phi_i|_{L^i(F)}))
\]
for the resulting Yu datum for $L$ over $\ol\F_\ell$, the following conditions hold.
\begin{enumerate}
    \item The restriction of $\tau_L$ to $L^0(F)_{[x]_H}$ admits
    $\tau_{L,0}\otimes \epsilon_H\epsilon_L\epsilon_{\sharp,x}^{\vec H,L}$
    as an irreducible subrepresentation, where $\epsilon_{\sharp,x}^{\vec H, L}$ is the sign character from \cite[(3.3.4)]{CF26a}.
    \item $\pi(\Psi_H)$ is an irreducible subquotient of $I_P^H(\pi(\Psi_L))$.
\end{enumerate}
\end{prop}

\begin{proof}
As mentioned above, any irreducible cuspidal $\ol\F_\ell$-representation of $\ol L^0_{[x]_L}(\F_q)$ completes $((L^i), x, (r_i), (\phi_i|_{L^i(F)}))$ to a Yu datum for $L$, so the content of the proposition is the construction of a particular $\tau_L$ and the proof that it satisfies conditions (1) and (2). The proof proceeds in four major steps.
\begin{enumerate}
\item[(Step 1)] We define a compact-mod-center open subgroup $K_{L,P} \subset H(F)$ which interpolates between the compact-mod-center subgroup defined in Yu's construction for $\Psi_H$ and the corresponding one for $L$. 

\item[(Step 2)] We use Appendix~\ref{app:kim-yu-ohara-coefficients} to embed a certain representation $\pi$ of $H(F)$ compactly induced from $K_{L,P}$ into the normalized parabolic induction of a certain finitely generated representation of $L(F)$, and we relate the latter to representations of the form $\pi(\Psi_L)$. 

\item[(Step 3)] We relate $\pi$ to $\pi(\Psi_H)$, which requires wrestling with some sign issues. 

\item[(Step 4)] We conclude using Lemma~\ref{lemma:family-to-fiber}.
\end{enumerate}
Now we start the proof in earnest. 

\emph{Step 1. The subgroups $K_{L,P}$ and $K_{L,P}^{\mathrm{KY},+}$.} For $1 \leq i\leq d$, define
\[
K_{L,P}^i=\ol U^i(F)_{x, \frac{r_{i-1}}{2}+}L^i(F)_{x,\frac{r_{i-1}}{2}}U^i(F)_{x,\frac{r_{i-1}}{2}} \subset H^i(F),
\]
and set
\begin{equation}\label{eq:KLP-depth-zero-par}
K_{L,P}^0\coloneqq\ol U^0(F)_{x,0+}L^0(F)_{[x]_H}U^0(F)_x \subset H^0(F)_{[x]_H}.
\end{equation}
It is easy to see that $K^i_{L,P}$ is a group for all $0 \leq i \leq d$.
For $i>0$, we have 
\begin{equation}\label{eq:KLP-root-decomposition}
K_{L,P}^i=(K_{L,P}^i\cap\ol U^i(F))(K_{L,P}^i\cap L^i(F))(K_{L,P}^i\cap U^i(F))
\end{equation}
and $K_{L,P}^i \cap L(F) = L^i(F)_{x,r_{i-1}/2}$.

Next, we define the compact-mod-center subgroups
\begin{equation}\label{eq:KLP-definition}
K_{L,P}\coloneqq K_{L,P}^0K_{L,P}^1\cdots K_{L,P}^d
\end{equation}
and
\begin{equation}\label{eq:KLP-KY-plus}
\begin{aligned}
K_{L,P}^{+}\coloneqq {}&
\ol U^0(F)_{x,0+}L^0(F)_{x,0+}U^0(F)_x \\
&\cdot
\prod_{i=1}^d
\ol U^i(F)_{x,\frac{r_{i-1}}{2}+}
L^i(F)_{x,\frac{r_{i-1}}{2}+}
U^i(F)_{x,\frac{r_{i-1}}{2}} .
\end{aligned}
\end{equation}
See Figure~\ref{fig:KLP-KY-containment} for a cartoon of these subgroups; note $K_{L,P}^+ \subset K_{L,P}$.

\begin{figure}[!h]
\centering
\begin{tikzpicture}[
    line cap=round,
    line join=round,
    x=0.47cm,
    y=0.47cm,
    outline/.style={draw=black, line width=0.6pt},
    closededge/.style={draw=black, line width=0.75pt},
    openedge/.style={draw=red, dash pattern=on 1.5pt off 1.0pt, line width=0.9pt},
    brace/.style={decorate, decoration={brace, amplitude=3.5pt}, draw=black!65, line width=0.6pt},
    every node/.style={font=\scriptsize}
]

\definecolor{klpfill}{HTML}{D8E9FF}

\def\base{0}
\def\yzero{4.10}
\def\yone{2.88}
\def\ytwo{1.62}

\begin{scope}[shift={(0,6.25)}]
    \node[anchor=west] at (6.75,2.85) {$K_{L,P}$};
    \foreach \xa/\xb/\ytop in
    {-6.20/-5.00/\ytwo,-5.00/-3.80/\yone,-3.80/-2.60/\yzero,
        -1.80/-1.20/\ytwo,-1.20/-0.60/\yone,-0.60/0.60/\yzero,
        0.60/1.20/\yone,1.20/1.80/\ytwo,
        2.60/3.80/\yzero,3.80/5.00/\yone,5.00/6.20/\ytwo}
    {
        \fill[klpfill] (\xa,\base) rectangle (\xb,\ytop);
        \draw[outline] (\xa,\base) -- (\xa,\ytop);
        \draw[outline] (\xb,\base) -- (\xb,\ytop);
        \draw[outline] (\xa,\base) -- (\xb,\base);
    }
    \foreach \xa/\xb/\ytop in {-6.20/-5.00/\ytwo,-5.00/-3.80/\yone,-3.80/-2.60/\yzero}
    \draw[openedge] (\xa,\ytop) -- (\xb,\ytop);
    \foreach \xa/\xb/\ytop in
    {-1.80/-1.20/\ytwo,-1.20/-0.60/\yone,-0.60/0.60/\yzero,
        0.60/1.20/\yone,1.20/1.80/\ytwo,
        2.60/3.80/\yzero,3.80/5.00/\yone,5.00/6.20/\ytwo}
    \draw[closededge] (\xa,\ytop) -- (\xb,\ytop);
    \draw[outline] (-7.25,\base) -- (-7.25,\yzero);
    \draw[outline] (-7.40,\yzero) -- (-7.10,\yzero);
    \draw[outline] (-7.40,\yone) -- (-7.10,\yone);
    \draw[outline] (-7.40,\ytwo) -- (-7.10,\ytwo);
    \node[anchor=east] at (-7.55,\yzero) {$0$};
    \node[anchor=east] at (-7.55,\yone) {$r_0/2$};
    \node[anchor=east] at (-7.55,\ytwo) {$r_1/2$};
    \node[anchor=west] at (6.70,4.60) {$\Phi(H^0,T)$};
    \node[anchor=west] at (6.70,5.35) {$\Phi(H^1,T)$};
    \node[anchor=west] at (6.70,6.10) {$\Phi(H^2,T)$};
    \draw[brace] (-3.80,4.60) -- (-2.60,4.60);
    \draw[brace] (-0.60,4.60) -- (0.60,4.60);
    \draw[brace] (2.60,4.60) -- (3.80,4.60);
    \draw[brace] (-5.00,5.35) -- (-2.60,5.35);
    \draw[brace] (-1.20,5.35) -- (1.20,5.35);
    \draw[brace] (2.60,5.35) -- (5.00,5.35);
    \draw[brace] (-6.20,6.10) -- (-2.60,6.10);
    \draw[brace] (-1.80,6.10) -- (1.80,6.10);
    \draw[brace] (2.60,6.10) -- (6.20,6.10);
        \node at (-4.40,-0.85) {$\ol U$ directions};
        \node at (0,-0.85) {$L$ directions};
        \node at (4.40,-0.85) {$U$ directions};
\end{scope}

    \begin{scope}
        \node[anchor=west] at (6.75,2.85) {$K_{L,P}^+$};
    \foreach \xa/\xb/\ytop in
    {-6.20/-5.00/\ytwo,-5.00/-3.80/\yone,-3.80/-2.60/\yzero,
        -1.80/-1.20/\ytwo,-1.20/-0.60/\yone,-0.60/0.60/\yzero,
        0.60/1.20/\yone,1.20/1.80/\ytwo,
        2.60/3.80/\yzero,3.80/5.00/\yone,5.00/6.20/\ytwo}
    {
        \fill[klpfill] (\xa,\base) rectangle (\xb,\ytop);
        \draw[outline] (\xa,\base) -- (\xa,\ytop);
        \draw[outline] (\xb,\base) -- (\xb,\ytop);
        \draw[outline] (\xa,\base) -- (\xb,\base);
    }
    \foreach \xa/\xb/\ytop in
    {-6.20/-5.00/\ytwo,-5.00/-3.80/\yone,-3.80/-2.60/\yzero,
        -1.80/-1.20/\ytwo,-1.20/-0.60/\yone,-0.60/0.60/\yzero,
        0.60/1.20/\yone,1.20/1.80/\ytwo}
    \draw[openedge] (\xa,\ytop) -- (\xb,\ytop);
    \foreach \xa/\xb/\ytop in
    {2.60/3.80/\yzero,3.80/5.00/\yone,5.00/6.20/\ytwo}
    \draw[closededge] (\xa,\ytop) -- (\xb,\ytop);
    \draw[outline] (-7.25,\base) -- (-7.25,\yzero);
    \draw[outline] (-7.40,\yzero) -- (-7.10,\yzero);
    \draw[outline] (-7.40,\yone) -- (-7.10,\yone);
    \draw[outline] (-7.40,\ytwo) -- (-7.10,\ytwo);
    \node[anchor=east] at (-7.55,\yzero) {$0$};
    \node[anchor=east] at (-7.55,\yone) {$r_0/2$};
    \node[anchor=east] at (-7.55,\ytwo) {$r_1/2$};
        \node at (-4.40,-0.85) {$\ol U$ directions};
        \node at (0,-0.85) {$L$ directions};
            \node at (4.40,-0.85) {$U$ directions};
        \end{scope}

        \fill[klpfill] (-1.95,-1.75) rectangle (-1.58,-1.50);
        \draw[closededge] (-1.95,-1.50) -- (-1.58,-1.50);
        \node[anchor=west] at (-1.45,-1.63) {boundary included};
        \fill[klpfill] (-1.95,-2.28) rectangle (-1.58,-2.03);
        \draw[openedge] (-1.95,-2.03) -- (-1.58,-2.03);
        \node[anchor=west] at (-1.45,-2.16) {boundary omitted};
    \end{tikzpicture}
    \caption[Comparison of KLP and KY-plus]{Comparison of $K_{L,P}$ and
    $K_{L,P}^+$.}
\label{fig:KLP-KY-containment}
\end{figure}
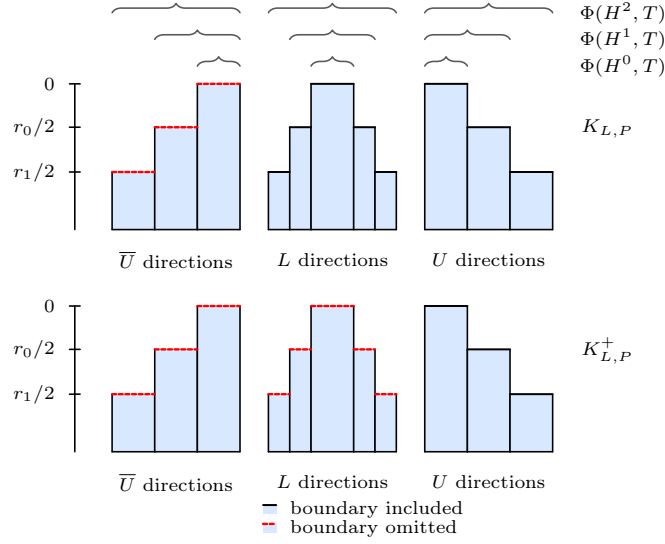

We need to compare $K_{L,P}^+$ to the group appearing in Proposition~\ref{prop:KY-cover}, which is defined in terms of the \emph{generic embeddings} of \cite{KY17}. Choose a rational cocharacter $\nu$ valued in $Z(L^0)$ such that $\langle\alpha, \nu\rangle > 0$ for every root $\alpha$ occurring in $\Lie(U)$, let $x_L$ be the point of $\cB(L^0)$ such that $\iota_0(x_L) = x$, and let $\iota$ be the shift of $\iota_0$ determined by the property that $\iota(x_L) = x + \delta \nu$ for some very small $\delta > 0$, so $\iota$ is generic for all of the depths $0, r_0/2, \dots, r_{d-1}/2$. For each real number $s$, we have $U_\alpha(F)_{x+\delta\nu,s} = U_\alpha(F)_{x,s-\delta\langle\alpha, \nu\rangle}$, so $\delta$ can be chosen small enough such that for all $s$ among $0, r_0/2, \dots, r_{d-1}/2$ we have 
\[
U_\alpha(F)_{x+\delta\nu, s} = \begin{cases}
U_\alpha(F)_{x,s+} &\text{if } \alpha \in \Lie(\ol U), \\
U_\alpha(F)_{x,s} &\text{if } \alpha \in \Lie(U).
\end{cases}
\]
Comparing \eqref{eq:KLP-KY-plus} and \eqref{eq:KY-plus-explicit} therefore shows $K^+_{L,P} = K^+$.

Finally, we define\footnote{When \(d=0\), the convention is that \(\cK_L^+=L^0(F)_{x,0+}\).}
\[
\cK_L \coloneqq L^0(F)_{[x]_H}L^1(F)_{x,\frac{r_0}{2}} \cdots L^d(F)_{x,\frac{r_{d-1}}{2}},
\qquad
\cK_L^+\coloneqq L^0(F)_{x,0+}L^1(F)_{x,\frac{r_0}{2}+}\cdots
L^d(F)_{x,\frac{r_{d-1}}{2}+}.
\]
Multiplying \eqref{eq:KLP-root-decomposition} over \(i>0\), together with
\eqref{eq:KLP-depth-zero-par}, yields
\begin{equation}\label{eq:KLP-parabolic-decomposition}
K_{L,P}= (K_{L,P}\cap\ol U(F))\,\cK_L\,(K_{L,P}\cap U(F)),
\qquad
K_{L,P}\cap L(F)=\cK_L.
\end{equation}
Similarly, we have
\begin{equation}\label{eq:KLP-KY-plus-parabolic-decomposition}
K_{L,P}^+
=
\bigl(K_{L,P}^+\cap\ol U(F)\bigr)\,
\cK_L^+\,
\bigl(K_{L,P}^+\cap U(F)\bigr),
\qquad
K_{L,P}^+\cap L(F)=\cK_L^+.
\end{equation}

\emph{Step 2. Covers and the injection into parabolic induction.} Let $\phi_{L,i} = \phi_i|_{L^i(F)}$ for all $i$, and let $\chi\co L^0(F)_{[x]_H}/L^0(F)_{x,0+} \to \ol\F_\ell^\times$ be a character. Consider the tensor product
\[
\rho_{L,P}(\chi) \coloneqq \epsilon_L \otimes \tau_{L,0}\chi \otimes \bigotimes_{i=0}^{d-1} V_{\wh \phi_{L,i}} \otimes \phi_{L,d},
\]
with notation as in Yu's construction. By \eqref{eq:KLP-parabolic-decomposition}, the representation $\rho_{L,P}(\chi)$ extends to $K_{L,P}$ by inflation; we use the same notation for this inflated representation. 

By \eqref{eq:KLP-parabolic-decomposition} and the definition of $\rho_{L,P}(\chi)$, the pair $(K_{L,P}, \rho_{L,P}(\chi))$ satisfies conditions \ref{itm:decomp} and \ref{itm:inflation} of Definition~\ref{def:blondel-cover} for $P=LU$. 
By Corollary~\ref{cor:KY-64}, Corollary~\ref{cor:KY-dual-cover}, and Proposition~\ref{prop:ohara-33}, applied with $G=H$, $M=L$, $K=K_{L,P}$, and $\rho_M=\rho_{L,P}(\chi)|_{\cK_L}$, there is an $H(F)$-equivariant injection
\begin{equation}\label{eqn:ohara-injection}
\cInd_{K_{L,P}}^{H(F)}\rho_{L,P}(\chi) \hookrightarrow I_P^H\!\left(\cInd_{\cK_L}^{L(F)}\rho_{L,P}(\chi)|_{\cK_L}\right)
\end{equation}
of $H(F)$-representations.

\emph{Step 3. Calculation of the compact induction.} We now want to show that the left-hand side of \eqref{eqn:ohara-injection} admits $\pi(\Psi_H)$ as an irreducible subquotient for some choice of $\chi$. Let
\begin{equation}\label{eq:KLP-filtration-zero}
K_{L,P}(0)\coloneqq H^0(F)_{[x]_H}H^1(F)_{x,\frac{r_0}{2}}\cdots
H^d(F)_{x,\frac{r_{d-1}}{2}}.
\end{equation}
This is the compact-mod-center subgroup that appears in Yu's construction for $\Psi_H$. Define a decreasing filtration of groups $K_{L,P}(0) \supset \ldots \supset K_{L,P}(d+1)=K_{L,P}$ via
\begin{equation}\label{eq:KLP-filtration-positive}
K_{L,P}(i)\coloneqq K_{L,P}^0K_{L,P}^1\cdots K_{L,P}^{i-1}
H^i(F)_{x,\frac{r_{i-1}}{2}}\cdots H^d(F)_{x,\frac{r_{d-1}}{2}} 
\end{equation}
for $1 \leq i \leq d$.

By transitivity of induction, we have
\[
\cInd_{K_{L,P}}^{H(F)}(\rho_{L,P}(\chi)) = \cInd_{K_{L,P}(d)}^{H(F)}(\cInd_{K_{L,P}}^{K_{L,P}(d)}(\rho_{L,P}(\chi))).
\]
If $d > 0$ then all terms in the tensor product on the right side extend to $K_{L,P}(d)$ except $V_{\wh\phi_{L,d-1}}$, so by the projection formula we have
\[
\cInd_{K_{L,P}}^{K_{L,P}(d)}(\rho_{L,P}(\chi)) \cong \epsilon_L \otimes \tau_{L,0}\chi \otimes \bigotimes_{i=0}^{d-2} V_{\wh\phi_{L,i}} \otimes \ind_{K_{L,P}}^{K_{L,P}(d)}(V_{\wh\phi_{L,d-1}}) \otimes \phi_{L,d}.
\]
Using the same reasoning inductively along with the relation 
\[
\cInd_{K_{L,P}(i+1)}^{H(F)} = \cInd_{K_{L,P}(i)}^{H(F)} \circ \cInd_{K_{L,P}(i+1)}^{K_{L,P}(i)},
\]
we conclude that
\begin{align}\label{eq:KP1-induction}
\cInd_{K_{L,P}}^{H(F)}(&\rho_{L,P}(\chi)) \nonumber \\
&\cong \cInd_{K_{L,P}(1)}^{H(F)}\left(\epsilon_L \otimes \tau_{L,0}\chi \otimes \bigotimes_{i=1}^{d} \ind_{K_{L,P}(i+1)}^{K_{L,P}(i)}(V_{\wh\phi_{L,i-1}}) \otimes \phi_{L,d}\right).
\end{align}
We now explicate the action of $K_{L,P}(1)$ on each term $\ind_{K_{L,P}(i+1)}^{K_{L,P}(i)}(V_{\wh\phi_{L,i-1}})$ above. For $j \neq i$, $j > 0$, the action of $H^j(F)_{x,\frac{r_{j-1}}{2}}$ is evidently given by $\wh\phi_{i-1}|_{H^j(F)_{x,\frac{r_{j-1}}{2}}}$. Let
\[
\Delta_{P,H}\coloneqq \epsilon_H\epsilon_L\epsilon_{\sharp,x}^{\vec H, L} =
\epsilon_H\epsilon_L\prod_{i=0}^{d-1}\epsilon_i
\]
as a sign character of $L^0(F)_{[x]_H}$, where $\epsilon_i = \epsilon_{\sharp,x}^{H^{i+1}/H^i,L}$ is the character defined by \cite[Lemma~3.2.2]{CF26a}. We view $\Delta_{P,H}$ as a character of $K_{L,P}(i)$ by inflation.
By \cite[Theorem 2.4(b)]{Ger77}\footnote{There is a famous (series of) typo(s) in the printed statement of this theorem; see \cite[Footnote 1 on page 2]{Fin21b} for the correct statement as given here, which is easily extracted from the proof of \cite[Theorem 2.4(b)]{Ger77}. Although \cite[Theorem 2.4(b)]{Ger77} is phrased for complex coefficients, it is valid over $k = \ol \F_\ell$: both sides of the isomorphism restrict irreducibly to the underlying finite Heisenberg $p$-group, and since $\ell\neq p$, reduction modulo $\ell$ preserves both this irreducibility and the isomorphism class of the restriction.} and the description of $\epsilon_{i-1}$ given in \cite[Lemma~3.2.2]{CF26a} (see also Lemma~\ref{lemma:eta-sign-comparison}), the action of $L^0(F)_{[x]_H}U^0(F)_x H^i(F)_{x,\frac{r_{i-1}}{2}}$ on $\ind_{K_{L,P}(i+1)}^{K_{L,P}(i)}(V_{\wh\phi_{L,i-1}})$ is isomorphic to the twist by $\epsilon_{i-1}$ of the restriction of $V_{\wh\phi_{i-1}}$. Consequently, we have
\begin{align}\label{eq:L-yu-datum-identity}
\cInd_{K_{L,P}(1)}^{H(F)}(&\epsilon_L \otimes \tau_{L,0}\Delta_{P,H} \otimes \bigotimes_{i=1}^{d} \ind_{K_{L,P}(i+1)}^{K_{L,P}(i)}(V_{\wh\phi_{L,i-1}}) \otimes \phi_{L,d}) \nonumber \\
    &\cong\cInd_{K_{L,P}(0)}^{H(F)}(\epsilon_H \otimes \ind_{\ol P^0_{[x]_H}(\F_q)}^{\ol H^0_{[x]_H}(\F_q)}(\tau_{L,0}) \otimes \bigotimes_{i=0}^{d-1} V_{\wh\phi_i} \otimes \phi_d).
\end{align}
Combining \eqref{eqn:ohara-injection}, \eqref{eq:KP1-induction}, and \eqref{eq:L-yu-datum-identity}, we find that $\pi(\Psi_H)$ is an irreducible subquotient of $I_P^H(\cInd_{\cK_L}^{L(F)}(\rho_{L,P}(\Delta_{P,H})))$.

\emph{Step 4: conclusion.} Let $\tau_{L,1}$ be an irreducible representation of $Z_{L^0}(F)L^0(F)_{[x]_H}$ which extends $\tau_{L,0}\Delta_{P,H}$, and observe that 
\begin{align}\label{eqn:ohara-2}
\cInd_{\cK_L}^{L(F)}(&\rho_{L,P}(\Delta_{P,H})) \nonumber\\
&\cong \cInd_{Z_{L^0}(F)\cK_L}^{L(F)}(\epsilon_L \otimes \tau_{L,1} \otimes \bigotimes_{i=0}^{d-1} V_{\wh \phi_{L,i}} \otimes \phi_{L,d} \otimes \ol\F_\ell[Z_{L^0}(F)L^0(F)_{[x]_H}/L^0(F)_{[x]_H}]).
\end{align}
By \eqref{eqn:ohara-2}, Step 3, and Lemma~\ref{lemma:family-to-fiber}, there is a character $\theta\co Z_{L^0}(F)L^0(F)_{[x]_H}/L^0(F)_{[x]_H} \to \ol\F_\ell^\times$ (necessarily of depth $0$) such that $\pi(\Psi_H)$ is an irreducible subquotient of
\[
I_P^H(\cInd_{Z_{L^0}(F)\cK_L}^{L(F)}(\epsilon_L \otimes \tau_{L,1}\theta \otimes \bigotimes_{i=0}^{d-1} V_{\wh \phi_{L,i}} \otimes \phi_{L,d})).
\]
Note that $\epsilon_L$, $V_{\wh\phi_{L,i}}$, and $\phi_{L,d}$ all extend to $K(\vec{L},x,\vec{r})$, so it follows that there is some irreducible $\ol L^0_{[x]_L}(\F_q)$-representation $\tau_L$ whose restriction to $\ol L^0_{[x]_H}(\F_q) \cdot Z(\ol L^0_{[x]_L})(\F_q)$ admits $\tau_{L,1}\theta$ as an irreducible subquotient. If $\Psi_L=((L^i),x,(r_i),\tau_L,(\phi_{L,i}))$ is the resulting irreducible Yu datum for $L$, then $\Psi_L$ satisfies the desiderata.
\end{proof}

\subsection{Lifting to characteristic zero} 
The following corollary is the main consequence of Proposition~\ref{prop:tame-cuspidal-parabolic-induction} that will be used later. 

\begin{cor}\label{cor:lift-tame-reps-to-char-0}
Let $\ol\Psi$ be a normalized irreducible Yu datum for $G$ over $\ol\F_\ell$. There exists a Levi factor $L$ of a parabolic $F$-subgroup $P \subset G$, a normalized Yu datum $\Psi_L$ for $L$ over $\ol\Z_\ell$ with irreducible generic fiber, and a normalized irreducible Yu datum $\ol\Psi_L$ for $L$ over $\ol\F_\ell$, such that
\begin{enumerate}
    \item $\pi(\ol\Psi_L)$ is an irreducible subquotient of $\pi(\Psi_L \otimes_{\ol\Z_\ell}\ol\F_\ell)$,
    \item $\pi(\ol\Psi)$ is an irreducible subquotient of $I_P^G(\pi(\ol\Psi_L))$,
    \item $\rho_I^{\Kal}(\ol\Psi) \sim \ld j_{L,G} \circ \rho_I^{\Kal}(\ol\Psi_L)$,
    \item $\rho^{\FS}(\pi(\ol\Psi)) \sim \ld j_{L,G} \circ \rho^{\FS}(\pi(\ol\Psi_L))$.
\end{enumerate} 
If $\ol\Psi$ is $F$-non-singular, then we can take $L=P=G$ and $\ol\Psi_L = \ol\Psi$.
\end{cor}

\begin{proof}
We write $\ol\Psi = ((G^i), x, (r_i), \ol\tau, (\ol\phi_i))$ as usual. Let $\wt\tau$ be an irreducible $\ol\Q_\ell$-representation of $\ol G^0_{[x]}(\F_q)$ such that $\ol\tau$ is an irreducible constituent of the $\ell$-modular reduction of $\wt\tau$; this exists by \cite[Part III, no.\ 16.1, Theorem 33]{Serre77}. Let $\phi_i\co G^i(F)\to\ol\Z_\ell^\times$ be the Teichm\"uller lift of $\ol\phi_i$, so each $\phi_i$ is normalized. By \cite[Lemma~2.10.2]{Cot26b}, there exists a split $\F_q$-torus $\ol S^\circ \subset \ol G^0_{[x]}$ which is the maximal split central $\F_q$-torus in $\ol L \coloneqq Z_{\ol G^0_{[x]}}(\ol S^\circ)$, a parabolic $\F_q$-subgroup $\ol P^\circ$ of $(\ol G^0_{[x]})^\circ$ with Levi $\ol L^\circ$, and an irreducible cuspidal $\ol\Q_\ell$-representation $\wt\tau_{L,0}$ of $\ol L(\F_q)$ such that $\wt\tau$ is an irreducible constituent of $\ind_{\ol P(\F_q)}^{\ol G^0_{[x]}(\F_q)}(\wt\tau_{L,0})$, where $\ol P = \ol L \cdot \ol P^\circ$. Let $\ol\tau_{L,0}$ be an irreducible constituent of the $\ell$-modular reduction of $\wt\tau_{L,0}$ such that $\ol\tau$ is an irreducible subquotient of $\ind_{\ol P(\F_q)}^{\ol G^0_{[x]}(\F_q)}(\ol\tau_{L,0})$.

Let $\cal S$ be an $\cO_F$-subtorus of $\cal G^0_{[x]}$ lifting $\ol S^\circ$, as in \cite[\S 7.2]{CF26a}, and let $S$ be the generic fiber of $\cal S$, so $S$ is an $F$-split subtorus of $G^0$. Let $L = Z_G(S)$, and let $P$ be a parabolic $F$-subgroup of $G$ with Levi $L$. Note that the bulleted assumptions of \S\ref{sss:parabolic-descent-setting} hold by \cite[Lemmas~9.1.2 and 9.1.1]{CF26a}. By Proposition~\ref{prop:tame-cuspidal-parabolic-induction}, applied with $H = G$ and the Levi subgroup $L$, there is an irreducible Yu datum $\ol\Psi_{L,\ol\tau_L} = ((L^i),x,(r_i),\ol\tau_L,(\ol\phi_{i,L}))$ for $L$ over $\ol\F_\ell$, where $\ol\phi_{i,L} = \ol\phi_i|_{L^i(F)}$ for all $0\leq i\leq d$, such that
\begin{enumerate}[label=(\Alph*)]
\item $\ol\tau_L|_{L^0(F)_{[x]}}$ admits $\ol\tau_{L,0} \otimes \epsilon_G\epsilon_L\epsilon_{\sharp,x}^{\vec G, L}$ as an irreducible subrepresentation, where $\epsilon_{\sharp, x}^{\vec G, L}$ is as in \cite[(3.3.4)]{CF26a},
\item $\pi(\ol\Psi)$ is an irreducible subquotient of $I_P^G(\pi(\ol\Psi_{L,\ol\tau_L}))$.
\end{enumerate} 
By \cite[Lemma~2.10.3]{Cot26b}, if $\wt\tau_L$ is an irreducible cuspidal $\ol\Q_\ell$-representation of $\ol L^0_{[x]_L}(\F_q)$ with a stable $\ol\Z_\ell$-lattice $\Lambda_L$ whose reduction admits $\ol\tau_L$ as an irreducible subquotient, then there is a generalized maximal torus-character pair $(\ol T_{L,1},\eta_1)$ as in \cite[(9.2.3)--(9.2.4)]{CF26a}, i.e., $\ol\tau_L\in\cE(\ol L^0_{[x]_L},[\ol T_{L,1},\eta_1])$ and $\eta_1|_{\ol T_L(\F_q)}  = \eta_0\epsilon_G\epsilon_L\epsilon_{\sharp,x}^{\vec G, L}$. Let
\[
\Psi_{L,\Lambda_L} = ((L^i),x,(r_i),\Lambda_L,(\phi_{i,L})),
\]
where $\phi_{i,L} = \phi_i|_{L^i(F)}$ as above. The inclusions $L^i_{\der}\to G^i_{\der}$ lift to the simply connected covers, so $\Psi_{L,\Lambda_L}$ and $\ol\Psi_{L,\ol\tau_L}$ are normalized. We have now proven (1) and (2). Now \cite[Theorem~9.3.3]{CF26a} gives
\[
\rho_I^{\Kal}(\ol\Psi) \sim \ld j_{L,G} \circ \rho_I^{\Kal}(\ol\Psi_L),
\]
proving (3). Finally, (4) follows from (2) and \cite[Theorem I.9.6(viii)]{FS}.
\end{proof}

\section{Functoriality for the Fargues--Scholze correspondence}\label{section:fs-functoriality}

In this section, we will establish certain characteristic $0$ functoriality statements for ``large'' prime degree cyclic field extensions and for descent to unramified twisted Levis. The rough idea is that the results of \S\ref{section:tate-cohom-yu} give rise to functoriality statements modulo ``large'' primes, and \cite[Lemma~5.3.2]{Cot26b} shows that such congruences can be lifted to characteristic $0$.

Recall the notation $\rho_I^{\Kal}$ and $\cT(\Psi)$ from \cite[\S 7.2]{CF26a}; also recall from \S\ref{ssec:notation-fs} that $F_\ell/F$ denotes the unramified extension of degree $\ell$.

\subsection{Modular reduction}\label{ssec:modular-reduction} We will need to know that the formation of Fargues--Scholze parameter is compatible with reduction modulo $\ell$ (even when said reduction is not irreducible).

\begin{lemma}\label{lem:L-parameter-of-reduction}
Let $\Pi$ be an irreducible smooth representation of $G(F)$ over $\ol \Q_\ell$, and let $\Lambda\subset \Pi$ be a $G(F)$-stable $\ol \Z_\ell$-lattice. For every irreducible $G(F)$-subquotient $\Pi'$ of $\Lambda\otimes_{\ol\Z_\ell}\ol\F_\ell$, we have $\rho^{\FS}(\Pi')\sim \ol{\rho^{\FS}(\Pi)}^{\ss}$.
\end{lemma}

\begin{proof}
The L-parameter $\rho^{\FS}(\Pi) \in \rH^1(W_F, \wh{G}(\ol \Q_\ell))$ corresponds to a character $\chi_{\Pi} \co \Exc_{\Q_\ell}(W_F, \wh{G}) \rightarrow \ol \Q_\ell$. Since integral excursion operators preserve $\Lambda$ (the Fargues--Scholze construction is defined with $\Z_\ell[\sqrt q]$-coefficients and maps to the integral Bernstein center; cf.\ \cite[Chapter IX]{FS}), it follows that $\chi_{\Pi}$ restricts to a $\Z_\ell$-algebra homomorphism $\Exc_{\Z_\ell}(W_F,\wh G)\to \ol\Z_\ell$. We therefore obtain a character
\[
\ol \chi_{\Pi} \coloneqq \chi_{\Pi} \otimes_{\Z_\ell} \F_\ell  \co \Exc_{\F_\ell}(W_F, \wh{G}) \rightarrow \ol \F_\ell.
\]
This defines the semisimplified mod $\ell$ reduction $\ol \rho^{\FS}(\Pi)^{\ss} \in \rH^1(W_F, \wh{G}(\ol \F_\ell))$. If $\ol{\Pi}\coloneqq \Lambda \otimes_{\ol\Z_\ell}\ol\F_\ell$, then it follows that every modular excursion operator acts on $\ol\Pi$ through $\ol\chi_\Pi$. Hence, for every irreducible constituent $\Pi'$ of $\ol\Pi^{\ss}$, the parameter $\rho^{\FS}(\Pi')$ satisfies $\rho^{\FS}(\Pi')\sim \ol{\rho^{\FS}(\Pi)}^{\ss}$, as desired.
\end{proof}




\subsection{Independence of $\ell$}\label{ssec:independence-of-ell}
 Our argument will require ``sewing'' information from different primes.  To do this, we fix isomorphisms $\ol \Q_\ell \cong \CC$. For primes $\ell, \ell' \neq p$, this immediately allows us to convert between smooth representations of $G(F)$ over $\ol\Q_\ell$ and $\ol\Q_{\ell'}$.
Similarly, we can convert between semisimple L-parameters $W_F \rightarrow \ld G(k)$ for $k \in \{\CC\} \cup \{\ol\Q_\ell : \ell \neq p\}$.
It is not obvious that the Fargues--Scholze correspondence is compatible with these conversions, but that is established in \cite{Sch25}. The more precise statement is as follows.

\begin{thm}[{\cite[Theorem 1.2]{Sch25}}]
For a coefficient field $k$, let $\mf{Z}(G;k)$ be the Bernstein center of $G(F)$ with coefficients in $k$. Then there is a map
\[
\FS_G \co \Exc(W_F, \wh G_{\Q(\sqrt{q})}) \rightarrow \mf{Z}(G; \Q(\sqrt{q}))
\]
which after base extension to $\Q_\ell$ for any $\ell \neq p$ recovers the map 
\[
\Exc(W_F, \wh G_{\Q_\ell(\sqrt{q})}) \rightarrow \mf{Z}(G; \Q_\ell(\sqrt{q})).
\]
\end{thm}

Fixing $\sqrt{q} \in \ol \Q$ for simplicity of notation, a character $\chi \co \Exc(W_F, \wh G_{\ol \Q}) \rightarrow \ol \Q$ thus induces a character $\chi_{\Q_\ell} \co \Exc(W_F, \wh G_{\ol \Q_\ell}) \rightarrow \ol \Q_\ell$ for each $\ell \neq p$. This in turn induces a compatible system of $(\rho_{\chi})_{\ol \Q_\ell} \in \rH^1(W_F, \wh G(\ol \Q_\ell))$ for each $\ell \neq p$.
In particular, irreducibility of a single $(\rho_{\chi})_{\ol \Q_\ell}$ is equivalent to irreducibility of all $(\rho_{\chi})_{\ol \Q_\ell}$.

If $\chi$ restricts to a $\ol \Z_\ell$-valued character on $\Exc(W_F, \wh G_{\ol \Z_\ell})$, then this integral character can be base changed to $\ol \F_\ell$. This corresponds to a semisimple L-parameter which we denote $(\rho_{\chi})_{\ol \F_\ell}$; informally, it is the ``semisimplification of the mod $\ell$ reduction of $\chi_{\Q_\ell}$'' (cf.\ \S \ref{ssec:modular-reduction}).

\subsection{Functoriality for base change of large prime degree}

First we show the existence of a cyclic base change lifting along extensions of sufficiently large prime degree.

\begin{thm}\label{thm:base-change-functoriality}
Let $G$ be a connected reductive group over $F$, and assume that $p\neq2$. There exists a positive integer $C$, depending only on $G/F$, such that for every prime $\ell\neq p$ with $\ell>C$, every $k\in\{\ol\Q_\ell,\ol\F_\ell\}$, and every normalized irreducible Yu datum $\Psi$ for $G$ over $k$, there is an explicit normalized irreducible Yu datum $\Psi_\ell$ for $G_{F_\ell}$ over $k$ such that
\[
\rho^{\FS}(\pi(\Psi_\ell))|_{I_F} \sim \rho^{\FS}(\pi(\Psi))|_{I_F}\quad \text{and} \quad \rho_I^{\Kal}(\Psi_\ell) \sim \rho_I^{\Kal}(\Psi).
\]
Moreover:
\begin{enumerate}
\item If $\Psi$ is non-singular, then so is $\Psi_\ell$, and $\rho^{\Kal}(\Psi_\ell) \sim \rho^{\Kal}(\Psi)|_{W_{F_\ell}}$.

\item If $S_\pi\coloneqq Z_{\wh G}(\rho^{\FS}(\pi(\Psi))|_{I_F})^\circ_{\red}$ is a torus, then there exist representatives such that $\rho^{\FS}(\pi(\Psi_\ell))|_{I_F} = \rho^{\FS}(\pi(\Psi))|_{I_F}$ and such that $\rho^{\FS}(\pi(\Psi))$ and $\rho^{\FS}(\pi(\Psi_\ell))$ induce the same conjugation action of $W_{F_\ell}$ on $S_\pi$.

\item $\rho^{\FS}(\pi(\Psi))$ is irreducible if and only if $\rho^{\FS}(\pi(\Psi_\ell))$ is irreducible, and in this case $\rho^{\FS}(\pi(\Psi_\ell)) \sim \rho^{\FS}(\pi(\Psi))|_{W_{F_\ell}}$. 
\end{enumerate}
\end{thm}

\begin{proof}
Let
\[
\Psi = ((G^i)_{0 \leq i \leq d}, x, (r_i)_{0 \leq i \leq d}, \tau, (\phi_i)_{0 \leq i \leq d})
\]
be the Yu datum chosen in the statement. We begin with some preparations to simplify the form of $\Psi$. Recall that $\rho^{\FS}$, $\rho_I^{\Kal}$, and $\rho^{\Kal}$ are all compatible with twists by characters, so we may assume that $\pi(\Psi)$ has finite order central character $\zeta$; we may further choose $C$ large enough (independently of $\Psi$) such that the prime numbers dividing the order of $\zeta$ are $\leq C$. By \cite[Proposition~6.6.3 and Lemma~6.6.5]{CF26a}, if $C$ is large enough then we may replace $\Psi$ by a $G$-equivalent Yu datum to assume that each $\phi_i$ has finite order prime to $\ell$ and the central character of $\tau$ has finite order prime to $\ell$. In particular, if $k = \ol\Q_\ell$ then $\Psi$ may be assumed to be the generic fiber of a normalized Yu datum $\Psi_0$ for $G$ over $\ol\Z_\ell$. Similarly, if $k = \ol\F_\ell$ then $\Psi$ may be assumed to be the special fiber of some such $\Psi_0$. Set $\ol\Psi \coloneqq \Psi_0 \otimes_{\ol\Z_\ell} \ol\F_\ell$, and let $\ol\tau$ denote its depth $0$ representation.  We may further assume that any $\ell > C$ is banal; note that $\ell$ is also banal for $G(F_\ell)$ by \cite[Proposition~4.1.2(2)]{Cot26b}.

Let $(\Psi_0)_\ell$ be the ``base change'' Yu datum constructed from $\Psi_0$ in \cite[\S 8.2]{CF26a}. Set $\ol\Psi_\ell = (\Psi_0)_\ell \otimes_{\ol\Z_\ell} \ol\F_\ell$, and let $\ol\tau_\ell$ denote its depth $0$ representation. If $k = \ol\Q_\ell$, we set $\Psi_\ell = (\Psi_0)_\ell \otimes_{\ol\Z_\ell} \ol\Q_\ell$, while if $k = \ol\F_\ell$ we set $\Psi_\ell = \ol\Psi_\ell$. Note that $\Psi_\ell$ is normalized. Let $\Omega$ denote the absolute Weyl group of $G$. If $C$ is chosen to be larger than $\rk G + 1$ and $|\Omega|^2$, then \cite[Proposition~8.2.3]{CF26a} yields the claims for $\rho_I^{\Kal}$ and $\rho^{\Kal}$.

Next we analyze $\rho^{\FS}$. If $\ol\tau$ is non-singular and $\ell > \rk G + 1$ and $\ell$ is larger than the constant $A$ from \cite[Proposition~4.3.4(2)]{Cot26b}, then $\ol\tau_\ell$ is also non-singular by \cite[Proposition~4.3.4(1)--(2)]{Cot26b} and the fact that composition with $\Fr_\ell$ preserves non-singularity. By Corollary~\ref{cor:yu-datum-base-change-parameter}, we have
\begin{equation}\label{eq:bc-fs-param-identity-modular}
\rho^{\FS}(\pi(\ol\Psi_\ell)) \sim \rho^{\FS}(\pi(\ol\Psi))|_{W_{F_\ell}}.
\end{equation} 
For (2), note first that the set of prime numbers dividing the order of $\rho^{\FS}(\pi(\Psi))(I_F)$ is bounded independently of $\Psi$ by \cite[Lemma 2.2]{DHKM25} and Lemma~\ref{lemma:semisimple-image}. By Lemma~\ref{lemma:inertia-order}, we may choose $C$ large enough that for each $\ell > C$, the order of $\rho^{\FS}(\pi(\Psi_\ell))(I_F)$ is prime to $\ell$. By \cite[Lemma~5.3.2]{Cot26b} and \eqref{eq:bc-fs-param-identity-modular}, it follows that if $\ell > C$ then $Z_{\wh G}(\rho^{\FS}(\pi(\Psi))|_{I_F})^\circ_{\red}$ is a torus if and only if $Z_{\wh G}(\rho^{\FS}(\pi(\Psi_\ell))|_{I_F})^\circ_{\red}$ is a torus, and similarly for $\rho^{\FS}(\pi(\ol\Psi))$ and $\rho^{\FS}(\pi(\ol\Psi_\ell))$. We may assume that \eqref{eq:bc-fs-param-identity-modular} is an equality, so (2) follows from Lemma~\ref{lemma:torus-action-mod-ell}, and the irreducibility assertion in (3) follows from \cite[Lemma~10.1.7]{CF26a}.

Now suppose that $\rho^{\FS}(\pi(\Psi))$ is irreducible. By Lemma~\ref{lem:semisimple-finite-image}, if $f\co \ld G \to \ld G/Z(\wh G)$ is the quotient map then we may choose $C$ large enough that for each $\ell > C$, the only prime numbers dividing the order of $f(\rho^{\FS}(\pi(\Psi_\ell))(W_{F_\ell}))$ either divide the order of $f(\rho^{\FS}(\pi(\Psi))(I_F))$ or are $\leq C$. If $h\co \ld G \to \ld G/\wh G_{\der}$ is the quotient map, then by compatibility with central characters and the Kummer sequence, every prime dividing the order of either $h(\rho^{\FS}(\pi(\Psi))(W_F))$ or $h(\rho^{\FS}(\pi(\Psi_\ell))(W_{F_\ell}))$ divides $\ord(\zeta) \cdot |W_0|$, so every prime dividing either order is at most $C$ if $C > |W_0|$. If $C$ is chosen larger than the order of $Z(\wh G_{\der})$, then (3) follows.
\end{proof}

\subsection{Descent to unramified twisted Levis}

Throughout this section, let
\[
\Psi = ((G^i)_{0\leq i\leq d}, x, (r_i)_{0\leq i\leq d}, \tau, (\phi_i)_{0\leq i\leq d})
\]
be a normalized irreducible Yu datum for $G$ with coefficients in a field $k \in \{\ol\Q_\ell, \ol\F_\ell\}$.

\begin{thm}\label{thm:unramified-twisted-Levi-functoriality}
Assume $p\neq 2$. Let $\pi = \pi(\Psi)$. Suppose that $G^0 \cap G_{\der}$ is not a totally ramified torus. Then there exists a proper unramified twisted Levi $F$-subgroup $M \subset G$ and a normalized irreducible Yu datum $\Psi_M$ for $M$ over $k$ such that
\begin{enumerate}
\item The following two $\wh G(k)$-conjugacy relations hold:
\begin{equation}\label{eqn:FS-inertial-unram-functoriality}
\rho^{\FS}(\pi)|_{I_F}
\sim
\ld j_{M,G}\circ \rho^{\FS}(\pi(\Psi_M))|_{I_F},
\end{equation}
\[
\rho_I^{\Kal}(\Psi)
\sim
\ld j_{M,G}\circ \rho_I^{\Kal}(\Psi_M).
\]
\item Assuming in addition that $S_\pi \coloneqq Z_{\wh G}(\rho^{\FS}(\pi)(I_F))^\circ_{\red}$ is a torus, then there are representatives of $\rho^{\FS}(\pi)$ and $\rho^{\FS}(\pi_H)$ such that equality holds in \eqref{eqn:FS-inertial-unram-functoriality} and such that $\rho^{\FS}(\pi)$ and $\ld j_{M,G}\circ \rho^{\FS}(\pi(\Psi_M))$ induce the same conjugation action of $W_F$ on $S_\pi$.
\item If $\Psi$ is non-singular, then $\Psi_M$ is non-singular and
\[
\rho^{\Kal}(\Psi) \sim \ld j_{M, G} \circ \rho^{\Kal}(\Psi_M).
\]
\item If $k$ is a field, then $\rho^{\FS}(\pi)$ is irreducible if and only if $\ld j_{M,G} \circ \rho^{\FS}(\pi(\Psi_M))$ is irreducible, and in this case there is a cocycle $z\co W_F/I_F \to Z(\wh M \cap \wh G_{\der})(k)^{I_F}$ such that 
\[
\rho^{\FS}(\pi) \sim z \cdot \ld j_{M,G} \circ \rho^{\FS}(\pi(\Psi_M)).
\]
\end{enumerate}
Moreover, if $k = \ol\Q_\ell$ and $\pi$ has central character valued in $\ol\Z_\ell^\times$, then we may choose $\Psi_M$ such that it extends to a Yu datum for $M$ over $\ol\Z_\ell$.
\end{thm}

\begin{remark}
We expect that the cocycle in Theorem~\ref{thm:unramified-twisted-Levi-functoriality}(4) can be chosen to be a coboundary (and thus removed from the statement), but this would require a more precise calculation of $\sigma$-dual homomorphisms than is given in Proposition~\ref{prop:sigma-dual-unram-agreement-on-inertia}. Fortunately, this is not needed for the upcoming calculation of L-parameters in Theorem~\ref{thm:main-tame}.
\end{remark}

\subsubsection{First reductions}\label{sss:levi-functoriality-setting}

The proof of Theorem~\ref{thm:unramified-twisted-Levi-functoriality} occupies the remainder of this subsection; we will retain its hypotheses and notation throughout.

\begin{lemma}\label{lemma:unram-twisted-levi-first-reductions}
    It suffices to prove Theorem~\ref{thm:unramified-twisted-Levi-functoriality} under the additional assumptions that $k = \ol\Q_\ell$, that each $\phi_i$ is of finite order, and that $\pi$ has finite order central character.
\end{lemma}

\begin{proof}
    First, we show that if the result holds for $k = \ol\Q_\ell$, then it holds for $k = \ol\F_\ell$. Thus assume for the moment that $k = \ol\F_\ell$. By Corollary~\ref{cor:lift-tame-reps-to-char-0}, there is a parabolic $F$-subgroup $P \subset G$ with Levi factor $L$ and normalized irreducible Yu data $\ol\Psi_L$ and $\Psi_L$ for $L$ over $\ol\F_\ell$ and $\ol\Z_\ell$, respectively, such that $\pi(\ol\Psi_L)$ is an irreducible subquotient of $\pi(\Psi_L) \otimes_{\ol\Z_\ell} \ol\F_\ell$ and $\pi$ is an irreducible subquotient of $I_P^G(\pi(\ol\Psi_L))$ and we have
    \[
    \rho^{\FS}(\pi)|_{I_F} \sim \ld j_{L,G} \circ \rho^{\FS}(\pi(\ol\Psi_L))|_{I_F}
    \qquad\text{and}\qquad
    \rho_I^{\Kal}(\Psi) \sim \ld j_{L,G} \circ \rho_I^{\Kal}(\ol\Psi_L).
    \]
    By \cite[Chapitre II, no.\ 2.1]{Vig96} and \cite[Corollary 1.5(1)]{DHKM}, the normalized parabolic inductions $I_P^G(\pi(\Psi_L) \otimes_{\ol\Z_\ell} \ol\Q_\ell)$ and $I_P^G(\pi(\Psi_L) \otimes_{\ol\Z_\ell} \ol\F_\ell)$ are of finite length. Moreover, $I_P^G$ is exact (by \cite[Chapitre II, no.\ 2.1]{Vig96}) and sends flat $\ol\Z_\ell$-modules to flat $\ol\Z_\ell$-modules (by definition), so there is some smooth $\ol\Z_\ell$-representation $\wt\pi$ of $G(F)$ which occurs as a subquotient of $I_P^G(\pi(\Psi_L))$ such that $\wt\pi \otimes_{\ol\Z_\ell} \ol\Q_\ell$ is irreducible and $\pi$ is an irreducible subquotient of $\wt\pi \otimes_{\ol\Z_\ell} \ol\F_\ell$. By Lemma~\ref{lem:L-parameter-of-reduction} and induction on the semisimple rank of $G$, we may therefore pass from $(G, \Psi)$ to $(L, \Psi_L)$ in (1). For (2), note that if $S_\pi$ is a torus, then so is $S_{\wt\pi}$ by Lemma~\ref{lemma:fixed-point-torus-flat}. Thus we can reduce to $k = \ol\Q_\ell$ by Lemma~\ref{lemma:torus-action-mod-ell}. For (3) and (4), note that $L = G$ by Corollary~\ref{cor:lift-tame-reps-to-char-0}, so the reduction to $k = \ol\Q_\ell$ is clear.

    By twisting by a character of $G(F)$, we may assume that $\pi$ has finite order central character. By \cite[Lemmas~6.6.1, Proposition~6.6.3, and 6.6.5]{CF26a}, we may then pass to an equivalent Yu datum to assume that each $\phi_i$ is of finite order.
\end{proof}

Throughout the remainder of this section, we will assume that the hypotheses of Lemma~\ref{lemma:unram-twisted-levi-first-reductions} hold.

Call an $F$-torus $S'$ \emph{simple} if $X^*(S'_{\ol F}) \otimes_\Z \Q$ is a simple $\Q[\Gal(\ol F/F)]$-module, or equivalently if $S'$ admits no nontrivial proper $F$-subtorus. Fix a class $[(T, \theta)] \in \cT_0(\Psi)$, with notation as in \cite[Definition~7.2.1]{CF26a}. By hypothesis, the intersection $G^0 \cap G_{\der}$ is not a totally ramified torus. Since $T$ is a maximally unramified $F$-torus of $G^0$ by construction, it follows that the maximal unramified $F$-torus $T_0$ of $T \cap G_{\der}$ is nontrivial. Choose a simple $F$-subtorus $S \subset T_0$, and let $H = Z_G(S)$, so $H$ is a proper unramified twisted Levi $F$-subgroup of $G$. The unramified twisted Levi $M$ of Theorem~\ref{thm:unramified-twisted-Levi-functoriality} will be found as an $F$-Levi subgroup of $H$.

Note that
\begin{itemize}
\item $T$ is a maximally unramified \emph{elliptic}\footnote{Here we use that $k = \ol\Q_\ell$ and that $[(T, \theta)] \in \cT_0(\Psi)$, not just $\cT(\Psi)$.} maximal $F$-torus of $G^0$ with $x \in \cB(T)$, and there is a character $\theta_0 \co \ol T_{[x]}(\F_q) \to \ol\Q_\ell^\times$ such that $\tau$ is an irreducible constituent of $R_{\ol T_{[x]}}^{\ol G^0_{[x]}}(\theta_0)$ and $\theta = \theta_0 \cdot \prod_{i=0}^d \phi_i|_{T(F)_{[x]}}$. Since each $\phi_i$ is of finite order, so is $\theta_0$.
\item The point $x$ lies in $\cB(T) \subset \cB(H)$, and $Z(H^0)/Z(H)$ is anisotropic because $T$ is elliptic.
\end{itemize}
Let $\ol Z$ be the center of $\ol G_{[x]}^0$. Below, we will regard $\ld G$ as $\wh G \rtimes W_0$, where $W_0$ is a finite quotient of $W_F$ such that the map $\ld j_{H,G}$ factors through a map $\wh H \rtimes W_0 \to \wh G \rtimes W_0$.

\subsubsection{The candidate data}\label{sss:candidate-data}

For each maximal $\F_q$-torus $\ol S^\circ \subset \ol H^0_{[x]}$, set $\ol S = \ol S^\circ \cdot \ol Z$. For a character $\eta_0 \co \ol S(\F_q) \to \ol\Q_\ell^\times$, let $\cE(\ol H^0_{[x]}, [\ol S, \eta_0])$ denote the corresponding semi-rational Lusztig series from \cite[Definition~2.8.2]{Cot26b}. Consider the set
\begin{equation}\label{eqn:sigma-G-H-indexing-set}
\mathfrak X_H \coloneqq \bigcup_{\cE(\ol G^0_{[x]}, [\ol S,\eta_0]) = \cE(\ol G^0_{[x]}, [\ol T_{[x]},\theta_0])} \cE(\ol H^0_{[x]}, [\ol S,\eta_0]),
\end{equation}
where the union is over $\ol H^0_{[x]}(\F_q)$-conjugacy classes of torus-character pairs $(\ol S, \eta_0)$. The set $\mathfrak X_H$ is finite by \cite[Lemma~2.8.3]{Cot26b}.

We define a set $\mathfrak Y_H$ of Yu data for $H$ constructed from $\Psi$ via the procedure in \cite[\S 9]{CF26a}, which we now recall concisely.
The construction has the following shape:
\[
\begin{tikzcd}[column sep=4.8em, row sep=3em]
\text{(1) }\tau_{H,0}\in\mathfrak X_H
\ar[r, "\text{cuspidal}", "\text{support}"'] &
\text{(2) }(\ol M^0,\tau_{M,0})
\ar[d, "\text{lift to an }F\text{-Levi}"'] \\
&
\text{(3) }(M,Q)
\ar[r, "\text{sign}", "\text{twisting}"'] &
\text{(4) }\tau_M
\ar[r, "\text{extension}"'] &
\text{(5) }\Psi_M\in\mathfrak Y_H .
\end{tikzcd}
\]
where the steps are detailed as follows.
\begin{enumerate}
\item Choose $\tau_{H,0} \in \mathfrak X_H$.
\item Choose a split $\F_q$-torus $\ol T_0 \subset (\ol H^0_{[x]})^\circ$, let $\ol M^0 = Z_{\ol H^0_{[x]}}(\ol T_0)$, let $\ol Q^\circ$ be a parabolic $\F_q$-subgroup of $(\ol H^0_{[x]})^\circ$ with Levi $(\ol M^0)^\circ$, let $\ol Q = \ol M^0 \cdot \ol Q^\circ$, and let $\tau_{M,0}$ be a cuspidal $\ol\Q_\ell$-representation of $\ol M^0(\F_q)$ such that $\tau_{H,0}$ is an irreducible constituent of $\ind_{\ol Q(\F_q)}^{\ol H^0_{[x]}(\F_q)}(\tau_{M,0})$. Such data exist, and the pair $(\ol M^0, \tau_{M,0})$ is unique up to $\ol H^0_{[x]}(\F_q)$-conjugacy, by \cite[Lemma~2.10.2]{Cot26b}.
\item Next, choose a split $F$-torus $T_0 \subset H^0$ lifting $\ol T_0$, let $M = Z_H(T_0)$, set $M^0\coloneqq M\cap G^0$, and let $Q$ be a parabolic $F$-subgroup of $H$ with Levi factor $M$ lifting $\ol Q^\circ$. \cite[Lemmas~9.1.1 and 9.1.2]{CF26a}, with $(\ol S^\circ,\cS,L)$ replaced by $(\ol T_0,\cT_0,M)$, show that $T_0$ is the maximal $F$-split central torus of $M^0$, that $Z(M^0)/Z(M)$ is anisotropic, and that $[x]_{M^0}$ is a vertex of $\cB(M^0_{\der})$. The torus $T_0$, and thereby also the subgroup $M$, is unique up to $G(F)$-conjugacy by deformation theory for tori \cite[Exp.\ IX, Corollaire 7.3]{SGA3II}.

\item Choose a generalized maximal torus-character pair $(\ol T_M, \theta_{M,0})$ with $\tau_{M,0} \in \cE(\ol M^0, [\ol T_M, \theta_{M,0}])$ and an irreducible $\ol\Q_\ell$-representation $\tau_M$ of $\ol M^0_{[x]_M}(\F_q)$ whose restriction to $\ol M^0_{[x]}(\F_q) = \ol M^0(\F_q)$ admits $\tau_{M,0} \otimes \epsilon_G\epsilon_M\epsilon_{\sharp,x}^{\vec{G},M}$, with notation as in \cite[\S\S 3.3 and 6.5]{CF26a}. Thus \cite[Lemma~2.10.3]{Cot26b} shows that if $\ol T_{M,1} = Z_{\ol M^0_{[x]_M}}(\ol T_M)$ then there is a character
\[
\theta_{M,1}\co \ol T_{M,1}(\F_q) \to \ol\Q_\ell^\times \quad \text{extending} \quad \theta_{M,0} \cdot (\epsilon_G\epsilon_M\epsilon_{G,\sharp,x}\epsilon_{M,\sharp,x})|_{\ol T_M(\F_q)}
\]
such that
\begin{equation}\label{eq:sigma-M-depth-zero-condition}
\tau_M \in \cE\left(\ol M^0_{[x]_M},
\left[
\ol T_{M,1},
\theta_{M,1}\right]\right),
\end{equation}
\item We set
\[
\Psi_M = ((M^i), x, (r_i), \tau_M, (\phi_i|_{M^i(F)})), \qquad M^i \coloneqq M \cap G^i.
\]
Observe that $\Psi_M$ is a normalized Yu datum by \cite[Proposition~9.3.2]{CF26a}.
\end{enumerate}
Let $\mathfrak Y_H$ be the set of all Yu data $\Psi_M$ obtained by carrying out this construction, as $\tau_{H,0}$ runs through $\mathfrak X_H$. We declare two elements $\Psi_M$ and $\Psi_{M'}$ of $\mathfrak Y_H$ to be \emph{inertially equivalent} if there is some $g \in G(F)$ such that $M = gM'g^{-1}$ and there is an unramified character $\chi\co M(F)\to\ol\Q_\ell^\times$ such that $\pi(\Psi_M)\cong{}^g\pi(\Psi_{M'})\otimes\chi$.

\begin{lemma}\label{lemma:finite-adapted-inertial-classes}
The set $\mathfrak{Y}_H$ consists of finitely many inertial equivalence classes, and each inertial equivalence class contains some Yu datum $\Psi_M$ such that $\pi(\Psi_M)$ has finite order central character.
\end{lemma}

\begin{proof}
The set $\mathfrak X_H$ is finite by \cite[Lemma~2.8.3]{Cot26b}, so fix $\tau_{H,0}$. By construction, the subgroup $M$ constructed in (3) is unique up to $G(F)$-conjugacy. Note that the semi-rational Lusztig series \eqref{eq:sigma-M-depth-zero-condition} is finite by definition, and there are only finitely many choices of $\theta_{M,1}$ up to twist by an unramified character of $M(F)$. For the final claim, note that for any choice of $\theta_{M,1}$ as above, there is a character $\chi_0$ of $\ol M^0_{[x]_M}(\F_q)/\ol M^0_{[x]}(\F_q)$ such that $\theta_{M,1}\chi_0|_{\ol T_{M,1}(\F_q)}$ is of finite order. Moreover, for each such $\chi_0$ there is an unramified character $\chi$ of $M(F)$ such that $\chi_0 \chi|_{\ol M^0_{[x]_M}(\F_q)}$ is of finite order; the claim follows.
\end{proof}

\begin{lemma}\label{lemma:candidate-data-properties}
Let $\Psi_M \in \mathfrak{Y}_H$.
\begin{enumerate}
\item We have
\[
    \rho_I^{\Kal}(\Psi) \sim \ld j_{M,G} \circ \rho_I^{\Kal}(\Psi_M).
\]
\item If $\Psi$ is non-singular, then $\Psi_M$ is non-singular, $M = H$, and
\[
    \rho^{\Kal}(\Psi) \sim \ld j_{H,G} \circ \rho^{\Kal}(\Psi_M).
\]
\end{enumerate}
\end{lemma}

\begin{proof}
Note that the construction above is an instance of the construction of the descended Yu datum preceding \cite[Theorem~9.3.3]{CF26a}, so that theorem gives the claims.
\end{proof}

By Lemma~\ref{lemma:candidate-data-properties}, it suffices to prove the following proposition. 

\begin{prop}\label{prop:unram-twisted-levi-heart}
    If the hypotheses of Lemma~\ref{lemma:unram-twisted-levi-first-reductions} hold and $H = Z_G(S)$ is as above, then there is some $\Psi_M \in \mathfrak Y_H$ such that \eqref{eqn:FS-inertial-unram-functoriality} holds, as do conclusions (2) and (4) of Theorem~\ref{thm:unramified-twisted-Levi-functoriality}.
\end{prop}

To prove this, we make a few further reductions.

\subsubsection{Base change to a spectral endoscopic situation}

We want to use Theorem~\ref{thm:yu-datum-dl-res} to prove Theorem~\ref{thm:unramified-twisted-Levi-functoriality}, so we must reduce to a setting in which its hypotheses hold. The following lemma shows that our construction of $H$ is well-suited for this.

\begin{lemma}\label{lemma:simple-tori-centralizers}
Let $S_0$ be a non-central simple $F$-subtorus of $G$. Then $S_0 \cap Z(G)$ is finite, and there is an integer $N_0$ with the following property: for every field extension $E/F$ such that $(S_0)_E$ is a simple $E$-torus, and every $t \in S_0(E)$ of finite order greater than $N_0$ and coprime to $\#\pi_1(G_{\der})$, we have $Z_{G_E}(t) = H_E$.
\end{lemma}

\begin{proof}
The identity component of $(S_0 \cap Z(G))_{\red}$ is an $F$-subtorus of $S_0$, proper because $S_0 \not\subset Z(G)$, hence trivial because $S_0$ is simple. Therefore $S_0 \cap Z(G)$ is finite, proving the first part of the assertion.

Let $E/F$ be an extension such that $(S_0)_E$ is a simple $E$-torus. Fix a maximal torus $T \supset (S_0)_{\ol F}$ of $G_{\ol F}$, and let $\Phi = \Phi(G_{\ol F}, T_{\ol F})$. For a subset $\Phi_0 \subset \Phi$, let $K_{\Phi_0} = \bigcap_{\alpha \in \Phi_0} \ker \alpha|_{(S_0)_{\ol F}} \subset (S_0)_{\ol F}$, and let $N_0$ be the maximum of $\#(S_0 \cap Z(G))(\ol F)$ and of $\#K_{\Phi_0}(\ol F)$ over the subsets $\Phi_0$ for which $K_{\Phi_0}$ is finite.

If $t \in S_0(E)$ is of finite order larger than $N_0$, then we claim that $\alpha(t) \neq 1$ for all $\alpha \in \Phi$ with $\alpha|_{S_0} \neq 0$. If not, then $t$ lies in $K_{\Phi_0}(E)$ for $\Phi_0$ the $\Gal(\ol E/E)$-orbit of some nonzero $\alpha$. By construction $K_{\Phi_0}$ is stable under $\Gal(\ol E/E)$, so $(K_{\Phi_0})^\circ_{\red}$ is an $E$-subtorus of $(S_0)_E$, proper because $\alpha|_{S_0} \neq 0$, hence trivial because $(S_0)_E$ is simple. Thus $K_{\Phi_0}$ is finite, hence of order at most $N_0$ by definition, contradicting the assumption about the order of $t$. Now we have established that the roots of $Z_{G_E}(t)^\circ$ are exactly the roots vanishing on $S_0$, which are the roots of $H$ (as $H = Z_G(S_0)$); therefore $Z_{G_E}(t)^\circ = H_E$. Since the order of $t$ is coprime to $\#\pi_1(G_{\der})$, the centralizer $Z_{G_E}(t)$ is connected by Lemma~\ref{lem:centralizer-connected}, so $Z_{G_E}(t) = H_E$.
\end{proof}

The following lemma reduces Theorem~\ref{thm:unramified-twisted-Levi-functoriality} to the case where $\chH$ is a spectral endoscopic group of $\chG$.

\begin{lemma}\label{lemma:endoscopic-reduction}
To prove Proposition~\ref{prop:unram-twisted-levi-heart}, it suffices to prove it whenever, in addition to its hypotheses, the following conditions hold:
\begin{enumerate}
\item there exists an element $t \in S(F)$ of order $\ell$ such that $H = Z_G(t)$;
\item $\tau$ is defined over $\Q_\ell^{\unr}$;
\item $\ell$ is large enough so that the following are satisfied:
\begin{itemize}
\item $\ell$ is larger than $\rk G + 1$,
\item $\ell$ divides neither $\#\pi_1(G_{\der})$ nor $\#\pi_1(M'_{\ad})$ for any Levi subgroup $M' \subset G_{\ol F}$, and
\item $\ell$ does not divide the order of $\theta_0$, of any $\phi_i$, of any $\eta_0$ appearing in \eqref{eqn:sigma-G-H-indexing-set}, or of any character $\theta_0'$ of a generalized maximal torus-character pair $(\ol T', \theta_0')$ in $\ol G^0_{[x]}$ with $\tau \in \cE(\ol G^0_{[x]}, [\ol T', \theta_0'])$, or of $\pi_0(\ol T'\cap \ol Z)$ for any such pair.
\item $\ell$ does not divide $[\ol G^0_{[x]}(\F_{q^n}): (\ol G^0_{[x]})^\circ(\F_{q^n}) \cdot Z(\ol G^0_{[x]})(\F_{q^n})]$ for any $n \geq 1$.
\end{itemize}
\item $\ell$ does not divide the order of $\rho(I_F)$, where $\rho$ is any of the maps $\rho^{\FS}(\pi(\Psi))$, $\rho_I^{\Kal}(\Psi)$, $\ld j_{M,G}\circ\rho^{\FS}(\pi(\Psi_M))$, and $\ld j_{M,G}\circ\rho_I^{\Kal}(\Psi_M)$ for any $\Psi_M \in \mathfrak Y_H$. Moreover, whenever a map $\rho$ in the preceding list is irreducible, $\ell$ is larger than the corresponding integer $N$ of Lemma~\ref{lem:semisimple-finite-image} attached to $|(f\circ\rho)(I_F)|$, where $f \co \ld G \to \ld G/Z(\wh G)$ is the quotient map.
\end{enumerate}
\end{lemma}

\begin{proof}
For each inertial equivalence class in $\mathfrak{Y}_H$, choose a representative $\Psi_M$ such that $\tau_M$ factors through a finite quotient; such representatives exist by Lemma~\ref{lemma:finite-adapted-inertial-classes}. Since there are only finitely many such classes, we may choose an integer $a$ such that $\tau$ and every chosen $\tau_M$ factor through quotient groups of order at most $a$. Let $N > a$ be a positive integer large enough that each prime $\ell > N$ satisfies the conditions in (3) and (4)\footnote{For the fourth bulleted condition in (3), note that the desired index can be bounded in terms of $[\ol G^0_{[x]}: (\ol G^0_{[x]})^\circ \cdot Z(\ol G^0_{[x]})]$ and $\rH^1(\F_{q^n}, Z(\ol G^0_{[x]}) \cap ((\ol G^0_{[x]})^\circ)_{\der})$, both of which are bounded independently of $n$.}, and such that $N$ is larger than
\begin{itemize}
\item the integer $N_0$ of Lemma~\ref{lemma:simple-tori-centralizers}(2) and $\rk G + 1$,
\item the constants $C$ appearing in \cite[Corollary~4.4.1]{Cot26b} and Theorem~\ref{thm:base-change-functoriality} for all twisted Levi $F$-subgroups of $G$,
\item the constant $A$ appearing in \cite[Proposition~4.3.4]{Cot26b}.
\end{itemize}

By \cite[Lemma~5.3.5]{Cot26b}, there is a prime $\ell_1 > N$ different from $p$ such that $S(F_{\ell_1})$ contains an element $t$ of prime order $\ell_0 > N$ different from $p$. These two primes play different roles:
\begin{itemize}
\item $\ell_1$ is the degree of the auxiliary unramified base change $F_{\ell_1}/F$, and is large enough for the base change results to apply (in particular $\ell_1$ is banal for $G$, being larger than the constant $C$);
\item the prime $\ell_0$ is the order of the torsion element $t$ and will become the characteristic of the coefficient field to which we apply Tate cohomology; it is \emph{not} banal for $G(F_{\ell_1})$, as it divides $\#S(F_{\ell_1})_{\mathrm{tors}}$.
\end{itemize}
Since $\ell_1 > \rk G + 1$, the argument in the proof of \cite[Lemma~4.2.1(1)]{Cot26b} shows that the Frobenius action on $X^*(S_{\ol F})\otimes_\Z\Q$ has order prime to
$\ell_1$, so the torus $S_{F_{\ell_1}}$ is simple. Since $\ell_0 > N_0$ and $\ell_0 \nmid \#\pi_1(G_{\der})$ by construction, Lemma~\ref{lemma:simple-tori-centralizers}(2) implies that $H_{F_{\ell_1}} = Z_{G_{F_{\ell_1}}}(t)$.

Passing from $\ol\Q_\ell$ to $\ol\Q_{\ell_1}$ by independence of $\ell$ for $\rho^{\FS}$ and applying Theorem~\ref{thm:base-change-functoriality} with $E = F_{\ell_1}$ (permissible since $\ell_1 > N \geq C$), we obtain a Yu datum
\[
\Psi_{\ell_1} = ((G^i_{\ell_1}), x, (r_i), \tau_{\ell_1}, (\phi_{\ell_1,i}))
\]
for $G_E$ over $\ol\Q_{\ell_1}$ such that, using that $I_E = I_F$ because $E/F$ is unramified,
\[
\rho^{\FS}(\pi(\Psi_{\ell_1}))|_{I_F} \sim \rho^{\FS}(\pi(\Psi))|_{I_F}
\qquad\text{and}\qquad
\rho_I^{\Kal}(\Psi_{\ell_1}) \sim \rho_I^{\Kal}(\Psi).
\]
Note that the base change $[(T_{F_{\ell_1}}, \theta \circ \Nm_{F_{\ell_1}/F})]$ lies in $\cT_0(\Psi_{\ell_1})$; this follows from \cite[Proposition~4.3.4(2)]{Cot26b}, the definition of $\cT_0$, and the assumpion $\ell_1 > N \geq A$. Now let $\Psi_M = ((M^i), x, (r_i), \tau_M, (\phi_i)) \in \mathfrak Y_H$ be one of the representatives chosen above. By Theorem~\ref{thm:base-change-functoriality} again, the Yu datum
\[
\Psi_{M,\ell_1} = ((M^i_{\ell_1}),x,(r_i),\tau_{M,\ell_1}, (\phi_{i,\ell_1})),
\]
from \cite[\S 8.2]{CF26a} has the property that
\[
\rho^{\FS}(\pi(\Psi_{M,\ell_1}))|_{I_F} \sim \rho^{\FS}(\pi(\Psi_M))|_{I_F}
\qquad\text{and}\qquad
\rho_I^{\Kal}(\Psi_{M,\ell_1}) \sim \rho_I^{\Kal}(\Psi_M).
\]

Recall from \cite[Definition~8.2.2]{CF26a} that $\tau_{\ell_1}$ is constructed as follows (since by choice of $a$ the representation $\tau$ factors through a quotient group $Q$ of $\ol G_{[x]}^0(\F_q)$ of order prime to $\ell_1$):
\begin{enumerate}
\item[(i)] Let $\ol\tau$ denote the $\ell_1$-modular reduction of $\tau$,
\item[(ii)] let $\ol\tau'_{\ell_1}$ denote the representation of $\ol G^0_{[x]}(\F_{q^{\ell_1}})$ obtained from $\ol\tau$ by the Glauberman correspondence of \cite[\S 3.5.3]{Cot26b},
\item[(iii)] let $(\ol\tau'_{\ell_1})^{(\ell_1)}$ be the $\ell_1$-Frobenius twist of $\ol\tau'_{\ell_1}$, and
\item[(iv)] let $\tau_{\ell_1}$ denote the unique irreducible lift of $(\ol\tau'_{\ell_1})^{(\ell_1)}$ which factors through a finite quotient of $\ol G^0_{[x]}(\F_{q^{\ell_1}})$ of order prime to $\ell_1$.
\end{enumerate}
A completely similar description applies to each $\tau_{M,\ell_1}$ arising from these chosen representatives. Since $\ell_0>N$ is prime to the orders of the fixed finite quotients supporting $\tau$ and the chosen $\tau_M$, \cite[Lemmas~3.5.14 and 3.5.15]{Cot26b} show that $\tau_{\ell_1}$ and every resulting $\tau_{M,\ell_1}$ are defined over $\Q_{\ell_0}^{\unr}$.

Fix a pair $(\ol S, \eta_0)$ as above, let $\eta'_{0,\ell_1}$ be the unique $\Gal(\F_{q^{\ell_1}}/\F_q)$-stable character of $\ol S(\F_{q^{\ell_1}})$ extending $\eta_0$ (which exists by \cite[Proposition~4.3.4(1)]{Cot26b}), and let $\eta_{0,\ell_1} = \eta_{0,\ell_1}'^{\ell_1}$. By \cite[Corollary~4.4.1]{Cot26b} and \cite[Part III, no.\ 15.5, Proposition 43]{Serre77}, the map
\[
\cE(\ol H^0_{[x]}, [\ol S, \eta_0]) \to \cE((\ol H^0_{[x]})_{\F_{q^{\ell_1}}}, [\ol S_{\F_{q^{\ell_1}}}, \eta_{0,\ell_1}])
\]
induced by the Glauberman correspondence and the above procedure of reducing modulo $\ell_1$, passing to a Frobenius twist, and lifting, is bijective. By \cite[Corollary~4.4.2]{Cot26b}, the $\ol H^0_{[x]}(\F_q)$-conjugacy classes of pairs $(\ol S,\eta_0)$ appearing in \eqref{eqn:sigma-G-H-indexing-set} are unchanged upon passage to $F_{\ell_1}$. By the uniqueness of cuspidal supports in \cite[Lemma~2.10.2]{Cot26b}, this identifies the set $\mathfrak{X}_H$ with the corresponding set $\mathfrak{X}_{H,F_{\ell_1}}$ over $F_{\ell_1}$, and similarly it identifies the inertial equivalence classes in $\mathfrak{Y}_H$ with those in the corresponding set $\mathfrak{Y}_{H,F_{\ell_1}}$ over $F_{\ell_1}$. Thus, up to unramified twist, the corresponding data may be taken to be the base changes of the representatives chosen above. By assumption, it follows that there exists some $\Psi_{M,\ell_1}$ such that
\[
\rho^{\FS}(\pi(\Psi_{\ell_1}))|_{I_F} \sim \ld j_{M,G} \circ \rho^{\FS}(\pi(\Psi_{M,\ell_1}))|_{I_F}.
\]
We have thereby reduced Theorem~\ref{thm:unramified-twisted-Levi-functoriality}(1) to the corresponding statement over $F_{\ell_1}$.

Now suppose that $S_\Psi\coloneqq Z_{\wh G}(\rho^{\FS}(\pi(\Psi))(I_F))^\circ_{\red}$ is a torus and Theorem~\ref{thm:unramified-twisted-Levi-functoriality}(2) holds for $\Psi_{\ell_1}$ and $\Psi_{M,\ell_1}$ in place of $\Psi$ and $\Psi_M$. Passing again from $\ol\Q_\ell$ to $\ol\Q_{\ell_1}$, we may apply Theorem~\ref{thm:base-change-functoriality}(2) to pass to conjugates such that $\rho^{\FS}(\pi(\Psi))|_{I_F} = \rho^{\FS}(\pi(\Psi_{\ell_1}))|_{I_F}$ and the actions of $W_{F_{\ell_1}}$ on $S_\Psi$ through $\rho^{\FS}(\pi(\Psi))$ and $\rho^{\FS}(\pi(\Psi_{\ell_1}))$ are equal, and similarly for $\ld j_{M,G} \circ \rho^{\FS}(\pi(\Psi_M))$ and $\ld j_{M,G} \circ \rho^{\FS}(\pi(\Psi_{M,\ell_1}))$. But the action of $W_F$ on $S_\Psi$ is of order prime to $\ell_1$ since $\ell_1 > \rk G + 1$ by assumption, so if the actions of $\rho^{\FS}(\pi(\Psi_{\ell_1}))$ and $\ld j_{M,G} \circ \rho^{\FS}(\pi(\Psi_{M,\ell_1}))$ on $S_\Psi$ agree then the same is true of the actions of $W_F$ through $\rho^{\FS}(\pi(\Psi))$ and $\ld j_{M,G} \circ \rho^{\FS}(\pi(\Psi_M))$. This reduces Theorem~\ref{thm:unramified-twisted-Levi-functoriality}(1) and (2) from $F$ to $F_{\ell_1}$. By passing from $\ol\Q_\ell$ to $\ol\Q_{\ell_0}$, using independence of $\ell$ again, we have now arranged that $\ell$ satisfies conditions (1)--(4). Reducing Theorem~\ref{thm:unramified-twisted-Levi-functoriality}(4) to this setting is completely similar; the one additional ingredient is Lemma~\ref{lem:semisimple-finite-image}.
\end{proof}

\subsubsection{Modular descent to $H$}

Recall that we have now reduced to the case that $k = \ol\Q_\ell$, and each $\phi_i$ and the central character of $\tau$ have finite order. Thus we may fix an integral model $\Psi_{\ol\Z_\ell}=((G^i),x,(r_i),\tau_{\ol\Z_\ell},(\phi_i))$ of $\Psi$, which will be used to define the mod $\ell$ reduction. Let $H^i\coloneqq G^i\cap H$ and $\phi_{H,i}\coloneqq\phi_i|_{H^i(F)}$, and let $\xi_H$ be the twisting character of Definition~\ref{defn:levi-descent-twisting}. Write $\ol\phi_{H,i}$ for the reduction of $\phi_{H,i}$ and continue to write $\xi_H$ for its reduction. For each cuspidal irreducible $\ol\F_\ell$-representation $\ol\tau_{H,0}$ of $\ol H^0_{[x]}(\F_q)$, set $\ol\tau_H\coloneqq\ol\tau_{H,0}\otimes\xi_H$ and $\ol\Psi_H\coloneqq\Psi_{H,\ol\tau_H}$, as in \eqref{eq:levi-Psi-H}.

\begin{lemma}\label{lemma:tate-descent-to-H}
With notation as above, assume that conditions (1)--(4) of Lemma~\ref{lemma:endoscopic-reduction} hold; write $t \in S(F)$ for the element in condition (1). Then there exists a cuspidal irreducible $\ol\F_\ell$-representation $\ol\tau_{H,0}$ of $\ol H^0_{[x]}(\F_q)$ with the following two properties:
\begin{enumerate}
\item there is an irreducible $\ol\Q_\ell$-representation $\tau_{H,0} \in \mathfrak X_H$, with $\mathfrak X_H$ as defined in \eqref{eqn:sigma-G-H-indexing-set}, such that $\ol \tau_{H,0}$ is an irreducible constituent of the $\ell$-modular reduction of $\tau_{H,0}$;
\item for every irreducible subquotient $\ol\pi$ of the $\ell$-modular reduction of $\pi(\Psi)$, there exists a cocycle $z\co W_F/I_F \to Z(\wh H \cap \wh G_{\der})(\ol\F_\ell)^{I_F}$ such that
\[
\rho^{\FS}(\ol\pi) \sim z\cdot \ld j_{H,G} \circ \rho^{\FS}(\pi(\ol\Psi_H)).
\]
\end{enumerate}
\end{lemma}

\begin{proof}
The integral model $\Psi_{\ol\Z_\ell}$ places us in the setting of \S\ref{sss:levi-descent-setting} with $s = t$: the element $t$ has order $\ell$ with $H = Z_G(t)$ by condition (1) of Lemma~\ref{lemma:endoscopic-reduction}, the generic fiber of $\Psi_{\ol\Z_\ell}$ is the irreducible datum $\Psi$, $\ell$ is odd (condition (3)), and $H \cap G^0$ contains the elliptic maximally unramified maximal $F$-torus $T$.

We verify the hypotheses of Theorem~\ref{thm:yu-datum-dl-res}. The quotient $Z(H^0)/Z(H)$ is anisotropic since $T$ is elliptic in $G$. The generic fiber $\tau$ of $\tau_{\ol\Z_\ell}$ is defined over $\Q_\ell^{\unr}$ by condition (2) of Lemma~\ref{lemma:endoscopic-reduction}. Moreover, we have $\cT_0(\Psi_{\ol\Z_\ell}\otimes_{\ol\Z_\ell}\ol\Q_\ell)=\cT_0(\Psi)$, and every representative character of every class in this set has the form $\theta_0' \cdot \prod_i \phi_i$ for a generalized maximal torus-character pair $(\ol T', \theta_0')$ in $\ol G^0_{[x]}$ whose semi-rational Lusztig series contains $\tau$; all these characters, including $\theta_0'$, have finite order not divisible by $\ell$ by condition (3). The index $[\ol G^0_{[x]}(\F_q): (\ol G^0_{[x]})^\circ(\F_q) \cdot Z(\ol G^0_{[x]})(\F_q)]$ is prime to $\ell$ by the fourth bullet in condition (3). Finally, $\ell$ is good for $G$ by condition (3).

All hypotheses of Theorem~\ref{thm:yu-datum-dl-res} are now verified. Moreover, $\ell$ is good for $\chG$ and $\chH$, and $\ell \nmid \#\pi_1(M'_{\ad})$ for every Levi subgroup $M' \subset G$ containing $H$, and $\ell > \rk G + 1$ by condition (3) in Lemma~\ref{lemma:endoscopic-reduction}. Corollary~\ref{cor:yu-datum-dl-res-parameter} therefore shows that there exists a cuspidal irreducible $\ol\F_\ell$-representation $\ol\tau_{H,0}$ of $\ol H^0_{[x]}(\F_q)$ as in Theorem~\ref{thm:yu-datum-dl-res}, so that its Brauer character occurs with nonzero coefficient in $\ol{{}^*R^{\ol G^0_{[x]}}_{\ol H^0_{[x]}}(\tau)}$, and such that conclusion (2) of the present lemma holds for one irreducible subquotient of the reduction of $\pi(\Psi)$. Lemma~\ref{lem:L-parameter-of-reduction} shows that all such subquotients have the same Fargues--Scholze parameter, so conclusion (2) holds for every irreducible subquotient.

It remains to prove (1). Since the Brauer character of $\ol\tau_{H,0}$ occurs in $\ol{{}^*R^{\ol G^0_{[x]}}_{\ol H^0_{[x]}}(\tau)}$, there are an irreducible $\ol\Q_\ell$-representation $\tau_{H,0}$ and an $\ol H^0_{[x]}(\F_q)$-stable lattice $\Lambda_{H,0}\subset\tau_{H,0}$ such that $\ol \tau_{H,0}$ is a Jordan--H\"older constituent of $\Lambda_{H,0}\otimes_{\ol\Z_\ell}\ol\F_\ell$ and
\[
\left\langle
\tau_{H,0},
{}^*R^{\ol G^0_{[x]}}_{\ol H^0_{[x]}}(\tau)
\right\rangle\neq 0.
\]
By \cite[Corollary~2.8.6]{Cot26b}, this implies that $\tau_{H,0}\in\mathfrak X_H$.
\end{proof}

\subsubsection{Conclusion of the proof}

Finally, we prove Theorem~\ref{thm:unramified-twisted-Levi-functoriality}. As observed above, we have already reduced to proving Proposition~\ref{prop:unram-twisted-levi-heart}.

\begin{proof}[Proof of Proposition~\ref{prop:unram-twisted-levi-heart}]
By Lemma~\ref{lemma:endoscopic-reduction}, we may and do assume that conditions (1)--(4) of that lemma hold; write $t \in S(F)$ for the element in condition (1). Let $\ol\tau_{H,0}$, $\tau_{H,0}$, and $\ol\Psi_H$ be as in Lemma~\ref{lemma:tate-descent-to-H}, with $\tau_{H,0} \in \mathfrak X_H$, and let $\Psi_M = ((M^i), x, (r_i), \tau_M, (\phi_i)) \in \mathfrak{Y}_H$ and $\tau_{M,0} \in \cE(\ol M^0, [\ol T_M, \theta_{M,0}])$ be as constructed in \S\ref{sss:candidate-data}. Let $Q$ be a parabolic $F$-subgroup of $H$ with Levi $M$ as in \textit{loc.\ cit.}\ and let $(\ol T_{M,1}, \theta_{M,1})$ be a generalized maximal torus-character pair in $\ol M^0_{[x]_M}$ constructed in \textit{loc.\ cit.}\ satisfying \eqref{eq:sigma-M-depth-zero-condition}. Recall that $\ol\tau_H = \ol\tau_{H,0} \otimes \xi_H$ is the depth $0$ representation of $\ol\Psi_H$, where $\xi_H = \epsilon_G\epsilon_H\epsilon_{\sharp,x}^{\vec{G},H}$ after reduction modulo $\ell$ by Lemma~\ref{lemma:eta-sign-comparison}; the sign characters $\epsilon$ are as in \cite[\S\S 3.3 and 6.5]{CF26a}.

By the construction of \S\ref{sss:candidate-data}, $\tau_{H,0}$ is an irreducible constituent of $\ind_{\ol Q(\F_q)}^{\ol H^0_{[x]}(\F_q)}
(\tau_{M,0})$.
Choose an $\ol M^0(\F_q)$-stable lattice $\Lambda_{M,0}\subset\tau_{M,0}$. Since parabolic induction is exact and commutes with reduction modulo $\ell$, there is an irreducible constituent $\ol\tau_{M,0}$ of $\left(\Lambda_{M,0}\otimes_{\ol\Z_\ell}\ol\F_\ell\right)^{\mathrm{ss}}$ such that
\begin{equation}\label{eq:H0-M0-depth-zero-subquotient}
\ol\tau_{H,0}
\text{ is an irreducible subquotient of }
\ind_{\ol Q(\F_q)}^{\ol H^0_{[x]}(\F_q)}
(\ol\tau_{M,0}).
\end{equation}
Exactness of finite Jacquet functors modulo $\ell\neq p$ and \cite[Lemma~2.10.1]{Cot26b} imply that $\ol\tau_{M,0}$ is cuspidal. Set
\begin{equation}\label{eq:twisted-M-cuspidal-input}
\ol\sigma_{M,0}\coloneqq
\ol\tau_{M,0}\otimes\xi_H|_{\ol M^0(\F_q)} .
\end{equation}
Since $\ol\tau_{M,0}$ is cuspidal and $\xi_H$ is a character, the twist $\ol\sigma_{M,0}$ is also cuspidal, and twisting \eqref{eq:H0-M0-depth-zero-subquotient} by $\xi_H$ shows that $\ol\tau_H$ is an irreducible subquotient of $\ind_{\ol Q(\F_q)}^{\ol H^0_{[x]}(\F_q)}(\ol\sigma_{M,0})$.

For the pair $M \subset H$, the bulleted hypotheses of \S\ref{sss:parabolic-descent-setting}, required for Proposition~\ref{prop:tame-cuspidal-parabolic-induction}, are the content of \cite[Lemmas~9.1.2 and 9.1.1]{CF26a}. Proposition~\ref{prop:tame-cuspidal-parabolic-induction}, applied with ambient group $H$, Levi $M$, parabolic $Q$, and the representations $\ol\tau_H$ and $\ol\sigma_{M,0}$, shows that there is a Yu datum $\ol\Psi_M = ((M^i),x,(r_i),\ol\tau_M,(\ol\phi_{H,i}|_{M^i(F)}))$ for $M$ over $\ol\F_\ell$ such that
\begin{align}
\ol\sigma_{M,0}\otimes\epsilon_H\epsilon_M\epsilon_{\sharp,x}^{\vec{H},M}
&\subset
\ol\tau_M|_{M^0(F)_{[x]_H}} \text{, and}
\label{eq:M-depth-zero-output}\\
\pi(\ol\Psi_H)
&\text{ is an irreducible subquotient of }
I_Q^H(\pi(\ol\Psi_M)).
\label{eq:H-M-parabolic-subquotient}
\end{align}
Observe that
\[
\xi_H\epsilon_H\epsilon_M\epsilon_{\sharp,x}^{\vec{H},M} = \epsilon_G\epsilon_M\epsilon_{\sharp,x}^{\vec{G},M}
\]
by definition, so by \eqref{eq:M-depth-zero-output} and the construction of $\Psi_M$, we may and do assume that $\ol\tau_M$ is an irreducible constituent of the $\ell$-modular reduction of $\tau_M$. Thus \cite[Lemma~7.3.1(1)]{CF26a} and exactness properties of Yu's construction give
\begin{equation}\label{eq:M-Yu-lift-reduction}
\pi(\ol\Psi_M)
\text{ is an irreducible subquotient of the reduction of }
\pi(\Psi_M).
\end{equation}

Let $\ol\pi$ be an irreducible subquotient of the $\ell$-modular reduction of $\pi(\Psi)$; since $\pi(\Psi)$ is irreducible, Lemma~\ref{lem:L-parameter-of-reduction} identifies $\rho^{\FS}(\ol\pi)$ with the semisimplified reduction of $\rho^{\FS}(\pi(\Psi))$. By Lemma~\ref{lemma:tate-descent-to-H}(2), by \eqref{eq:H-M-parabolic-subquotient} and the compatibility of $\rho^{\FS}$ with parabolic induction \cite[Theorem I.9.6(viii)]{FS}, and by \eqref{eq:M-Yu-lift-reduction} and Lemma~\ref{lem:L-parameter-of-reduction} again (applied to the irreducible representation $\pi(\Psi_M)$), there is a cocycle $z\co W_F/I_F \to Z(\wh H \cap \wh G_{\der})(\ol\F_\ell)^{I_F}$ such that
\begin{equation}\label{eq:full-modular-WF-comparison}
\ol{\rho^{\FS}(\pi(\Psi))}^{\ss}
\sim z \cdot \ld j_{H,G}\circ\ld j_{M,H} \circ
\ol{\rho^{\FS}(\pi(\Psi_M))}^{\ss}.
\end{equation}
Restricting \eqref{eq:full-modular-WF-comparison} to $I_F$ and applying \cite[Proposition~5.3.2]{CF26a} therefore gives
\begin{equation}\label{eq:modular-inertial-comparison}
\ol{\rho^{\FS}(\pi(\Psi))}^{\ss}|_{I_F}
\sim \ld j_{M,G} \circ \ol{\rho^{\FS}(\pi(\Psi_M))}^{\ss}|_{I_F}.
\end{equation}
By condition (4) of Lemma~\ref{lemma:endoscopic-reduction}, the groups $\rho^{\FS}(\pi(\Psi))(I_F)$ and $\rho^{\FS}(\pi(\Psi_M))(I_F)$ are finite of order prime to $\ell$. In view of \eqref{eq:modular-inertial-comparison}, \cite[Lemma~5.3.2]{Cot26b} implies (1).

For conclusion (2) of the theorem, suppose that $S_\Psi$ is a torus. Note that by Lemma~\ref{lemma:toral-l-parameter-centralizer}, we have $S_\Psi \subset \wh H$. Since $\ell$ does not divide the order of $\rho^{\FS}(\pi(\Psi))(I_F)$ nor $\rho^{\FS}(\pi(\Psi_M))(I_F)$ as above, restricting \eqref{eq:full-modular-WF-comparison} to $I_F$ and applying \cite[Lemma~5.3.2]{Cot26b}, we may pass to conjugates to assume
\[
\rho^{\FS}(\pi(\Psi))|_{I_F} = \ld j_{H,G}\circ\ld j_{M,H} \circ \rho^{\FS}(\pi(\Psi_M))|_{I_F}
\]
as inertial L-parameters valued in $\ld G(\ol\Z_\ell)$. By \cite[Corollary 3.7]{BCT24}, their common inertial centralizer in $\wh G_{\ol\Z_\ell}$ is smooth over $\ol\Z_\ell$. The conjugating element in \eqref{eq:full-modular-WF-comparison} therefore lifts from the special fiber of this centralizer, so after conjugating $\rho^{\FS}(\pi(\Psi))$ by such a lift we may assume that \eqref{eq:full-modular-WF-comparison} is an equality. The cocycle $z$ appearing in \eqref{eq:full-modular-WF-comparison} acts trivially by conjugation on $\ol S_\Psi$, so \eqref{eq:full-modular-WF-comparison} identifies the actions of the two reduced parameters on $\ol S_\Psi$, and (2) follows from Lemma~\ref{lemma:torus-action-mod-ell}.

Finally, we prove (4). The first clause of (4) follows immediately from (2) and \cite[Lemma~10.1.9]{CF26a}. Thus $\ld j_{M,G}\circ\rho^{\FS}(\pi(\Psi_M))$ is irreducible, so the containment $M\subset H$ must be an equality. The displayed conjugacy follows from precisely the same argument which showed (1), using Proposition~\ref{prop:sigma-dual-unram-agreement-on-inertia}(2) and the fact that the orders of $\rho^{\FS}(\pi(\Psi))(W_F)$ and $\rho^{\FS}(\pi(\Psi_M))(W_F)$ are prime to $\ell$ by conditions (3)-(4) of Lemma~\ref{lemma:endoscopic-reduction} (arguing as in the last paragraph of the proof of Theorem~\ref{thm:base-change-functoriality}).
\end{proof}

\section{Calculation of Fargues--Scholze parameters}\label{section:fs-calculation}

This section is the culmination of the paper. We will use \cite[Theorem~10.2.1]{CF26a} to pin down $\rho^{\FS}$ on inertia. We then deduce a number of ``soft'' corollaries, including that $\rho^{\FS}$ has finite fibers when $G$ splits after a tamely ramified extension of $F$ and $p$ does not divide the order of the absolute Weyl group of $G$.

We continue to use the notation $\cT(\Psi)$ and $\rho_I^{\Kal}(\Psi)$ from \cite[\S 7.2]{CF26a}, as well as the embeddings $\ld j_{T,G}$ associated to a Yu datum as in \cite[Definition~4.5.1]{CF26a}. As in \S\ref{ssec:notation-fs}, if $n$ is a positive integer then $F_n$ denotes the unramified degree $n$ extension of $F$. 

\subsection{The main calculation}

The following theorem is the main result of this paper.

\begin{thm}\label{thm:main-tame}

Assume $p \neq 2$, and let $G$ be a connected reductive group over $F$. Let $\Psi$ be a normalized irreducible Yu datum for $G$ over a ring $k \in \{\ol \Q_\ell,\ol\F_\ell\}$, and let $\pi \coloneqq \pi(\Psi)$ be the corresponding tame cuspidal representation of $G(F)$ over $k$.

Let $[(T, \theta)] \in \cT(\Psi)$, let $T_0$ be the maximal unramified subtorus of $T$, and let $n$ be the ramification degree of the splitting field of $T\cap Z_G(T_0)_{\der}$. Suppose that the extension $F(\mu_n)/F$ is of degree prime to $p$\footnote{For instance, this holds if $p$ does not divide the order of the absolute Weyl group $\Omega$ of $Z_G(T_0)$: indeed, $n$ divides the order of $\Omega$, and \cite[\S 2, Table 1]{Fin21} shows that if $\ell$ is a prime number dividing $|\Omega|$, then every smaller prime also divides $|\Omega|$. Thus the degree of $F(\mu_n)/F$ is only divisible by primes dividing $|\Omega|$.}.
\begin{enumerate}
    \item Recalling that $\sim$ denotes $\wh G(k)$-conjugacy, we have
    \[
    \rho^{\FS}(\pi)|_{I_F} \sim \rho^{\Kal}_I(\Psi).
    \]
    Set $S_\Psi \coloneqq Z_{\wh G}(\rho^{\FS}(\pi)(I_F))^\circ_{\red}$. Therefore, by \cite[Lemma~10.1.11(2)]{CF26a}, $\Psi$ is non-singular if and only if $S_\Psi$ is a torus.
    \item Suppose that $p$ is good for $Z_G(T_0)$ and $\Psi$ is $F$-non-singular. Let
    \[
    X = X_*((T/Z(Z_G(T_0))^\circ_{\red})_{\ol F})_{I_F} \quad \text{and} \quad X_{\ad} = X_*((T/Z(Z_G(T_0)))_{\ol F})_{I_F}. 
    \]
    If $N$ is the least integer such that the map
    \[
    \rH^1(W_F/I_F, \Hom(X_{\ad}, k^\times)) \to \rH^1(W_{F_N}/I_F, \Hom(X, k^\times))
    \]
    is trivial, then
    \[
    \rho^{\FS}(\pi)|_{W_{F_N}} \sim \rho^{\Kal}(\Psi)|_{W_{F_N}}.
    \]
    Moreover, there exist representatives such that $\rho^{\FS}(\pi)|_{I_F} = \rho^{\Kal}(\Psi)|_{I_F}$ and the actions of $W_F$ on $S_\Psi$ through $\rho^{\FS}(\pi)$ and $\rho^{\Kal}(\Psi)$ are the same. In particular, if $\Psi_0$ is a normalized irreducible Yu datum for $G$ over $k$ then $\rho^{\FS}(\pi(\Psi_0))$ is irreducible if and only if $\Psi_0$ is non-singular.
\end{enumerate}
\end{thm}

\begin{proof}
The conclusion of this theorem is the same as the conclusion of \cite[Theorem~10.2.1]{CF26a} in the case $\rho = \rho^{\FS}$. Corollary~\ref{cor:yu-datum-small-degree-bc-parameter}, using modular functoriality at every prime $\ell\neq p$, allows us to take $\cS=\{p\}$. Thus we need only verify the hypotheses of that theorem:
\begin{enumerate}
    \item Compatibility with local Langlands for tori is \cite[Theorem I.9.6(i)]{FS}.
    \item Compatibility with character twisting and central characters is \cite[Theorem I.9.6(ii), (iii)]{FS}.
    \item Compatibility with homomorphisms which induce an isomorphism on adjoint groups is \cite[Theorem I.9.6(v)]{FS}.
    \item The analogue of \cite[Proposition~8.3.1]{CF26a} is Corollary~\ref{cor:yu-datum-small-degree-bc-parameter}.
    \item The analogue of \cite[Theorem~9.3.3]{CF26a} is Theorem~\ref{thm:unramified-twisted-Levi-functoriality}.
\end{enumerate}
This completes the proof.
\end{proof}

\subsection{Finiteness of the Fargues--Scholze correspondence}

In this section, we show that the calculation of $\rho^{\FS}(\pi)|_{P_F}$ for $k$-representations $\pi$ arising from Yu's construction of $G(F)$ is already sufficient to show that $\rho^{\FS}$ has finite fibers in significant generality.

\begin{thm}\label{thm:finiteness-of-fargues-scholze}

Suppose that $G$ splits after a tamely ramified extension of $F$ and $p$ does not divide the order of the absolute Weyl group $\Omega$ of $G$. Then $\rho^{\FS}$ preserves depth and has finite fibers.
\end{thm}

\begin{proof}
Let $k$ be a field among $\ol\Q_\ell$ and $\ol\F_\ell$, and fix a semisimple L-parameter $\varphi\co W_F \to \ld G(k)$; we want to show that $(\rho^{\FS})^{-1}(\varphi)$ is finite. Let $P_1, \dots, P_n \subset G$ be representatives for the $G(F)$-conjugacy classes of parabolic $F$-subgroups of $G$ such that $\varphi$ normalizes $\wh P_i \subset \wh G$, and for each $i$ choose a Levi factor $L_i \subset P_i$. If $\pi$ is an irreducible $k$-representation of $G(F)$ such that $\rho^{\FS}(\pi) \sim \varphi$, then since $\rho^{\FS}$ is compatible with parabolic induction, there is some $i$ and some irreducible cuspidal $k$-representation $\tau_i$ of $L_i(F)$ such that $\pi$ is an irreducible subquotient of the normalized parabolic induction $I_{P_i}^G(\tau_i)$. By \cite[Corollary 2.4]{DHKM} (see also \cite[Theorem 1.1]{Cot24} and \cite[Lemma 3.13]{Cot25}), for each $i$ there are only finitely many semisimple L-parameters $\psi\co W_F \to \ld L_i(k)$ up to $\wh L_i(k)$-conjugacy such that $\varphi\sim \ld j_{L_i,G} \circ \psi$. Thus by passing from $G$ to each $L_i$ separately, it suffices to show that $(\rho^{\FS})^{-1}(\varphi)$ contains only finitely many irreducible cuspidal $k$-representations of $G(F)$. Since $\rho^{\FS}$ is compatible with twists by characters, we may and do assume that the composition of $\varphi$ with $\ld G \to \ld G/\wh G_{\der}$ has finite image. 

If $G$ is a torus, then the result is clear; otherwise note that the classification of root systems shows that $p\nmid|\Omega|$ implies that $p\neq2$, that $p$ is good for $G$, and that $p\nmid|\pi_1(G_{\der})|$. Then \cite[Theorem 8.1]{Fin21} and \cite[Theorem 4.1]{Fin22} imply that every irreducible cuspidal $k$-representation of $G(F)$ arises from Yu's construction. Suppose that
\[
\Psi = ((G^i)_{0\leq i\leq d}, x, (r_i)_{0\leq i\leq d}, \tau, (\phi_i)_{0\leq i\leq d})
\]
is a Yu datum for $G$ over $k$ such that $\rho^{\FS}(\pi(\Psi)) \sim \varphi$.
By \cite[Lemma~6.6.1]{CF26a}, we may replace $\Psi$ by an irredundant datum defining the same representation, and thus assume $G^i\neq G^{i+1}$ for all $i$. Since $p\nmid|\pi_1(G_{\der})|$, \cite[Lemma~3.7.2]{Kal19} allows us to replace $\Psi$ with a normalized datum $\Psi'$ which is $G$-equivalent to $\Psi$, and \cite[Proposition~6.6.3]{CF26a} gives $\pi(\Psi')\cong\pi(\Psi)$. Thus we may assume that $\Psi$ is both irredundant and normalized.

Let $[(T, \theta)] \in \cT(\Psi)$. By Theorem~\ref{thm:main-tame}(1), we have $\varphi|_{P_F}\sim \ld j_{T,G} \circ \ld(\theta|_{T(F)_{\mathrm{b}}})|_{P_F}$. Let $r$ be the depth of $\varphi$, so $r$ is the least real number such that $\varphi|_{W_F^s}$ is trivial for all $s > r$, where $(W_F^s)_{s \in \RR_{\geq 0}}$ is the upper numbering filtration of $W_F$. Note that $\ld\theta$ is also of depth $r$ by \cite[Lemma~7.2.4]{CF26a}. By \cite[Theorem 7.10]{Yu09}, this implies that $\theta$ is also of depth $r$, i.e., $r$ is minimal such that $\theta|_{T(F)_{r+}}$ is trivial. But $\theta$ is of depth $r_d$ by construction, so $r_d = r$. This already establishes depth preservation.

It is enough to show that, up to equivalence, there are only finitely many $\Psi$ as above. We start with the following sequence of observations.
\begin{enumerate}
    \item There are only finitely many tamely ramified twisted Levi sequences $G^0 \subsetneq G^1 \subsetneq \cdots \subsetneq G^d$ up to $G(F)$-conjugacy since there are only finitely many tamely ramified maximal $F$-tori in $G$ up to $G(F)$-conjugacy, and each twisted Levi subgroup is determined by a tamely ramified maximal torus and a subset of the absolute root system.
    \item If $G^0 \subset G$ is a twisted Levi $F$-subgroup, then there are only finitely many vertices $x \in \cB(G^0_{\der})$ up to $G^0(F)$-conjugacy since each vertex is $G^0(F)$-conjugate to a vertex in a chosen alcove.
    \item If $(G^i)_{0\leq i\leq d}$ is fixed, then there are only finitely many sequences $(r_i)_{0\leq i\leq d}$ such that $r_d = r$, since the jumps in the Moy--Prasad filtration for $G^i(F)_{[x]}$ are discrete.
    \item If $G^0 \subset G$ is a twisted Levi $F$-subgroup and $x \in \cB(G^0)$ and $N$ is an integer, then there are only finitely many irreducible cuspidal $k$-representations of $G^0(F)_{[x]}/G^0(F)_{x,0+}$ with fixed restriction to $Z(G)(F)$ since $G^0(F)_{[x]}/(G^0(F)_{x,0+} \cdot Z(G)(F))$ is finite.
    \item If $(G^i)_{0\leq i\leq d}$ and $(r_i)_{0\leq i\leq d}$ are given and $(N_i)_{0\leq i\leq d}$ is a sequence of integers, then there are only finitely many choices of $(\phi_i)_{0\leq i\leq d}$ such that each $\phi_i$ is of order $\leq N_i$.
\end{enumerate}
By (1), (2), and (3), it is enough to prove that there are only finitely many $\Psi$ as above with prescribed $(G^i)$, $x$, and $(r_i)$. By \cite[Lemma~6.6.5]{CF26a}, there is some sequence $(N_i)$ of integers, depending only on $(G^i)$, $x$, and $(r_i)$, such that any such $\Psi$ is equivalent to a Yu datum with $\phi_i$ of order $\leq N_i$ for each $i$. By (5), it is therefore enough to show that there are only finitely many $\Psi$ such that $(\phi_i)$ is also prescribed. For any such $\Psi$, the construction of the pairs $[(T, \theta)] \in \cT(\Psi)$ and the compatibility of $\rho^{\FS}$ with central characters shows that the central character of $\tau$ is of bounded order, and thus the theorem follows from (4).
\end{proof}

\subsection{Surjectivity on inertial L-parameters}

\begin{thm}\label{thm:inertial-surjectivity}

Suppose that $G$ splits after a tamely ramified extension of $F$, and $p$ does not divide the order of the absolute Weyl group of $G$. Let $k \in \{\ol\Q_\ell, \ol\F_\ell\}$, and let $\varphi\co W_F \to \ld G(k)$ be a semisimple L-parameter. If $\varphi$ is irreducible, then there exists an irreducible $k$-representation $\pi$ of $G(F)$ such that
\begin{equation}\label{eq:inertial-param}
\rho^{\FS}(\pi)|_{I_F} \sim \varphi|_{I_F}.
\end{equation}
If $G$ is quasi-split, then the same conclusion holds without assuming that $\varphi$ is irreducible.
\end{thm}

\begin{proof}
The second statement follows from the first and compatibility of $\rho^{\FS}$ with parabolic induction, so we will assume that $\varphi$ is irreducible. If $G$ is a torus, then the result is clear. Otherwise, the assumption on the absolute Weyl group implies that $p\neq2$, that $p$ is good for $G$, and that $p$ does not divide $|\pi_1(G_{\der})|$.

If $k=\ol\Q_\ell$, set $\varphi_0=\varphi$. If $k=\ol\F_\ell$, then \cite[Theorem 1.3, Theorem 1.6]{DHKM25} gives a lift $\wt\varphi\co W_F\to\ld G(\ol\Z_\ell)$ of $\varphi$; set $\varphi_0=\wt\varphi_{\ol\Q_\ell}$. In the latter case, Lemma~\ref{lemma:check-irreducibility-mod-ell} shows that $\varphi_0$ is irreducible. Using compatibility with character twists, we may assume that the composition of $\varphi_0$ with $\ld G\to\ld G/\wh G_{\der}$ has finite image.

By \cite[Lemma 4.1.3, Proposition 4.1.8, \S 4.2]{Kal21b}, there is a non-singular irreducible Yu datum $\Psi$ for $G$ over $\ol\Q_\ell$ such that $\rho^{\Kal}(\Psi)\sim\varphi_0$. After applying \cite[Lemma~6.6.1]{CF26a} to reduce to the case that $\Psi$ is irredundant, \cite[Lemma 3.7.2]{Kal19} (the statement of which assumes, but the proof of which does not use, that $k$ is of characteristic $0$) allows us to replace $\Psi$ by a normalized $G$-equivalent datum; \cite[Proposition~6.6.3 and Lemma~7.2.13]{CF26a} show that this does not change $\pi(\Psi)$ or $\rho^{\Kal}(\Psi)$. By the hypothesis on the absolute Weyl group, Theorem~\ref{thm:main-tame} gives
\[
\rho^{\FS}(\pi(\Psi))|_{W_{F_N}}\sim\varphi_0|_{W_{F_N}}
\]
for some $N$. Compatibility with central characters shows that $\pi(\Psi)$ has finite-order central character. \cite[Lemma~6.6.5]{CF26a} and the construction of $\pi(\Psi)$ therefore show that there is a $G(F)$-stable $\ol\Z_\ell$-lattice $\Lambda \subset \pi(\Psi)$. If $k=\ol\Q_\ell$, take $\pi=\pi(\Psi)$. If $k=\ol\F_\ell$, let $\pi$ be an irreducible subquotient of $\Lambda \otimes_{\ol\Z_\ell} \ol\F_\ell$. Since $I_F\subset W_{F_N}$, Lemma~\ref{lem:L-parameter-of-reduction} and the choice of $\wt\varphi$ show that $\pi$ satisfies \eqref{eq:inertial-param}.
\end{proof}

\begin{remark}
Theorem~\ref{thm:inertial-surjectivity} has not been seriously optimized; for certain groups which split under unramified extensions \cite[Examples~10.2.6 and 10.2.8]{CF26a} show that one can find $\pi$ such that $\rho^{\FS}(\pi)|_{W_{F_N}} \sim \varphi|_{W_{F_N}}$ for relatively small $N$.

One can also obtain much sharper results for particular parameters. For instance, if $\varphi|_{I_F}$ factors through a maximal torus of $\wh G$, then one can find $\pi$ such that
\[
\rho^{\FS}(\pi) \sim \varphi.
\]
This follows from the same argument as the proof of Theorem~\ref{thm:inertial-surjectivity}, noting that the integer $N$ in Theorem~\ref{thm:main-tame} is equal to $1$ for any $\Psi$ constructed in the proof in this case.
\end{remark}

\begin{remark}
    \cite[Theorem 1.2]{Cot24b} and Theorem~\ref{thm:inertial-surjectivity} imply that if $G$ splits after a tamely ramified extension of $F$ and $p$ does not divide the order of the absolute Weyl group of $G$, then every connected component of the scheme of cocycles $\underline{Z}^1(W_F, \wh G)$ admits a cocycle lying in the image of $\rho^{\FS}$.\footnote{The invariant $\alpha$ in \cite[Theorem 1.2]{Cot24b} (which is taken from \cite[\S 4]{DHKM25}) is trivial in this case: indeed, $Z_{\wh G}(\rho(P_F))$ is connected under the hypotheses on $p$.}
\end{remark}

\subsection{Groups of type A}

For groups of absolute type A, we can strengthen the comparison in Theorem~\ref{thm:main-tame}. We will show that if $G_{F^{\unr}} \cong \SL_n$ for some $n$, then (under very weak hypotheses on $p$) in fact $\rho^{\FS}(\pi(\Psi)) \sim \rho^{\Kal}(\Psi)$ for all non-singular irreducible Yu data $\Psi$, extending \cite[Example~10.2.3]{CF26a}. While this is already known by other methods for some such groups $G$, our proof is purely local and also covers new cases. See Remark~\ref{rem:type-A} for a discussion of related prior results.

\begin{lemma}\label{lemma:l-parameters-for-gln}
Let $G$ be a connected reductive $F$-group such that $G_{F^{\unr}} \cong \GL_n$ for some $n$, let $f\co \ld G \to \ld G/\wh G_{\der}$ denote the quotient map, and let $k \in \{\ol\Q_\ell, \ol\F_\ell\}$. Let $\varphi_1, \varphi_2\co W_F \to \ld G(k)$ be two L-parameters such that
\begin{enumerate}
    \item $\varphi_1|_{I_F} \sim \varphi_2|_{I_F}$,
    \item $f \circ \varphi_1 \sim f \circ \varphi_2$,
    \item $\varphi_1$ is irreducible.
\end{enumerate}
Then there exists a cocycle $z\co W_F/I_F \to Z(\wh G_{\der})(k)$ such that $\varphi_2 \sim z\cdot\varphi_1$.
\end{lemma}

\begin{proof}
By (1), we may conjugate to assume that $\varphi_1|_{I_F} = \varphi_2|_{I_F}$. By (2) and the fact that $Z(\wh G) \to \wh G/\wh G_{\der}$ is a surjection on whose source and target $I_F$ acts trivially, we may and do further assume $f \circ \varphi_1 = f \circ \varphi_2$. By (3) and \cite[Lemma~10.1.9]{CF26a}, the group $\wh S \coloneqq Z_{\wh G}(\varphi_1(I_F))^\circ_{\red}$ is a torus and, if $\Fr \in W_F$ is a lift of Frobenius, then $\wh S^{\varphi_1(\Fr)}/Z(\wh G)^{\varphi_1(\Fr)}$ is finite. But \cite[III, 3.22]{SS70} shows that the centralizer of any subset of $\GL_n$ is connected, and the fact that $\GL_n$ is an open subscheme of its Lie algebra shows that every such centralizer is smooth. So in fact $\wh S = Z_{\wh G}(\varphi_1(I_F))$.

Since $\varphi_1|_{I_F} = \varphi_2|_{I_F}$ and $f \circ \varphi_1 = f \circ \varphi_2$, there is a cocycle $z\co W_F/I_F \to (\wh S \cap \wh G_{\der})(k)$ such that $\varphi_2 = z \cdot \varphi_1$. Note that $\wh S \cap \wh G_{\der} = (\wh S \cap \wh G_{\der})^\circ_{\red} \cdot Z(\wh G_{\der})$ since $\wh S$ is a torus. Since $(\wh S \cap \wh G_{\der})^\circ_{\red}(k)^{\varphi_1(\Fr)}$ is finite, \cite[Lemma~5.3.4]{Cot26b} shows that there is some $t_0 \in (\wh S \cap \wh G_{\der})^\circ_{\red}(k)$ such that $t_0 \varphi_1(\Fr) t_0^{-1} \in \varphi_2(\Fr) \cdot Z(\wh G_{\der})$. So we may conjugate to assume $z$ is valued in $Z(\wh G_{\der})(k)$, as desired.
\end{proof}

\begin{lemma}\label{lemma:embed-sln-to-gln}
Let $G$ be a connected reductive $F$-group such that $G_{F^{\unr}} \cong \SL_n$ for some $n$. Then there exists an embedding $G \subset \wt G$, where $\wt G$ is a connected reductive $F$-group such that $\wt G_{F^{\unr}} \cong \GL_n$.
\end{lemma}

\begin{proof}
Recall that $\GL_n = \SL_n \times^{\mu_n} \G_m$, where the embedding $\mu_n \to \SL_n \times \G_m$ is given by $z \mapsto (z, z^{-1})$. Thus the claim is that there is a $1$-dimensional unramified $F$-torus $T$ into which $Z(G)$ embeds. Note that $\Gal(\ol F/F)$ acts on $X^*(Z(G)_{\ol F})$ through $\Gal(F_2/F)$ since the group of outer automorphisms of $\SL_n$ is either trivial (if $n = 2$) or $\Z/2$ (otherwise), and the nontrivial element either acts trivially or by inversion. If the action is trivial, we may take $T = \G_m$; if the action is by inversion, we may take $T = \Res^1_{F_2/F}\G_m$.
\end{proof}

\begin{thm}\label{thm:sln-full-conjugacy}
Suppose $p \neq 2$, let $G$ be a connected reductive $F$-group such that $G_{F^{\unr}} \cong \SL_n$ for some $n$, let $k \in \{\ol\Q_\ell, \ol\F_\ell\}$, and let $\Psi$ be a non-singular normalized irreducible Yu datum for $G$ over $k$. If $p \nmid [F(\mu_m):F]$ for all $m \leq n$ (e.g., $p > n$), then $\rho^{\FS}(\pi(\Psi)) \sim \rho^{\Kal}(\Psi)$.
\end{thm}

\begin{proof}
Let $\Psi = ((G^i), x, (r_i), \tau, (\phi_i))$ be the given Yu datum. Choose an embedding $G \subset \wt G$ as in Lemma~\ref{lemma:embed-sln-to-gln}. For each $i$, let $\wt G^i = G^i \cdot Z(\wt G)$, let $\wt x \in \cB(\wt G)$ be a point mapping to $x \in \cB(G)$, let $\wt r_i = r_i$ for all $i$, let $\wt\tau$ be a $k$-representation of $\wt G^0(F)_{[\wt x]}$ which is trivial on $\wt G^0(F)_{\wt x,0+}$ and whose restriction to $G^0(F)_{[x]}$ admits $\tau$ as an irreducible constituent, and let $\wt\phi_i$ be a character of $\wt G^i(F)$ extending $\phi_i$; this exists since $\Psi$ is normalized. Let $\wt\Psi = ((\wt G^i),\wt x,(\wt r_i), \wt\tau, (\wt\phi_i))$. Theorem~\ref{thm:main-tame} applies to $\wt\Psi$ by our assumptions on $p$ and gives
\[
\rho^{\FS}(\pi(\wt\Psi))|_{I_F} \sim \rho^{\Kal}(\wt\Psi)|_{I_F}
\]
and $\rho^{\FS}(\pi(\wt\Psi))$ is irreducible. By compatibility of $\rho^{\FS}$ and $\rho^{\Kal}$ with maps which induce isomorphisms of adjoint groups (\cite[Lemma~7.4.1]{CF26a} and \cite[Theorem I.9.6(v)]{FS}) and Lemma~\ref{lemma:l-parameters-for-gln}, it follows that if $h\co \ld \wt G \to \ld G$ is the natural map, then 
\[
\rho^{\FS}(\pi(\Psi)) \sim h \circ \rho^{\FS}(\pi(\wt\Psi)) \sim h \circ \rho^{\Kal}(\wt\Psi) \sim \rho^{\Kal}(\Psi).
\]
This proves the claim.
\end{proof}

\appendix
\section{Generalities on L-homomorphisms}\label{app:l-homs}

In this section we discuss some generalities on the Tannaka dual perspective of L-homomorphisms. Let $k$ be a commutative ring. 

\subsection{Admissible homomorphisms}
\label{sssec:split-extensions-and-cocycles}

We begin with a general lemma.  Let $\Gamma$ be a topological group, and let
$M$ and $G$ be affine group schemes over $k$ equipped with continuous
$\Gamma$-actions.  We write
\[
M\rtimes\Gamma \longrightarrow \Gamma,
\qquad
G\rtimes\Gamma \longrightarrow \Gamma
\]
for the corresponding semidirect products, with distinguished sections
\[
s_M(w)=(1,w),\qquad s_G(w)=(1,w).
\]

We say that a homomorphism $M \rtimes \Gamma \rightarrow G \rtimes \Gamma$ is \emph{admissible} if it lies over the identity on $\Gamma$. 

\begin{lemma}\label{lem:homomorphisms-of-split-extensions}
Let
\[
\Phi\co M\rtimes\Gamma\longrightarrow G\rtimes\Gamma
\]
be an admissible homomorphism. There exist unique maps $i\co M \to G$, $c\co \Gamma \to G$ such that
\begin{equation}\label{eq:general-cocycle-formula}
\Phi(m,w)=\bigl(i(m)c(w),w\bigr)
\end{equation}
for all $m \in M$ and $w \in \Gamma$. The map $i$ is a homomorphism, and $c$ is a $1$-cocycle.
Moreover, the following conditions are equivalent:
\begin{enumerate}[label=(\roman*)]
\item $\Phi$ preserves the distinguished sections;
\item $c(w)=1$ for every $w\in\Gamma$;
\item $i$ is $\Gamma$-equivariant and
    $\Phi(m,w)=(i(m),w)$.
\end{enumerate}
If $i$ is already $\Gamma$-equivariant, then every value of $c$ centralizes
$i(M)$.
\end{lemma}

\begin{proof}
Set
\[
\Phi(m,1)=(i(m),1),
\qquad
\Phi(1,w)=(c(w),w).
\]
Since $(m,w)=(m,1)(1,w)$, formula
\eqref{eq:general-cocycle-formula} follows.  Applying $\Phi$ to
$(1,w)(1,w')=(1,ww')$ shows that $c$ is a $1$-cocycle.
Finally,
\[
\Phi(s_M(w))=(c(w),w),
\]
so (i) is equivalent to (ii). Clearly (ii) is equivalent to (iii). The final claim is clear. 
\end{proof}

Let us say that $\Phi$ is \emph{untwisted} if the equivalent conditions (i), (ii), (iii) are satisfied.

\begin{lemma}\label{lem:conjugating-split-homomorphism}
Let $\Phi$ be an admissible homomorphism. The $G$-conjugation class of $\Phi$ has an untwisted representative if and only if $c$ is a coboundary. 
\end{lemma}

\begin{proof}
Let $g \in G(k)$. We have
\[
(g,1)(i(m)c(w),w)(g^{-1},1)
=\bigl(g i(m)c(w){}^w g^{-1},w\bigr),
\]
which implies that the maps $i^g$ and $c^g$ associated to $\Phi^g = (g,1)\Phi(g,1)^{-1}$ satisfy
\[
i^g=\Ad(g)\circ i,
\qquad
c^g(w)=g\,c(w)\,{}^w g^{-1},
\]
proving the claim.
\end{proof}

\subsection{Tannaka duality}
\label{sssec:relative-tannaka-general}

Next we study the nature of admissible homomorphisms arising from (relative) Tannaka duality over a topological group $\Gamma$. Let $\Rep_k(\Gamma)$ be the category of finite-dimensional continuous representations of $\Gamma$ over $k$.

Let $\cC_G$ be a $k$-linear tensor category over $\Rep_k(\Gamma)$ with relative fiber functor
\[
\omega_G\co\mathcal C_G\longrightarrow\Rep_k(\Gamma).
\]
The relative Tannaka duality formalism of
\cite[Proposition VI.10.2]{FS} associates an affine group scheme internal to $\Rep_k(\Gamma)$. The usual forgetful fiber functor gives a group scheme $\chG$ over $k$ with a $\Gamma$-action. This action defines the semidirect product $\chG(k)\rtimes\Gamma$, with canonical section $s_G\co\Gamma\to\chG(k)\rtimes\Gamma$, $s_G(\gamma)=(1,\gamma)$. Under relative Tannaka duality, restriction along $s_G$ identifies with $\omega_G$.

Now suppose that $\cC_H$ is another tensor category over $\Rep_k(\Gamma)$ with relative fiber functor $\omega_H\co\mathcal C_H\longrightarrow\Rep_k(\Gamma)$.
Suppose that $F\co\mathcal C_G\longrightarrow\mathcal C_H$
is a symmetric monoidal functor equipped with a specified tensor
isomorphism
\begin{equation}\label{eq:relative-fiber-rigidification}
\eta\co\omega_G\xrightarrow{\sim}\omega_H\circ F
\end{equation}
in $\Rep_k(\Gamma)$. Then Tannaka duality associates to $(\cC_H, \omega_H)$ a group scheme $\chH$ internal to $\Rep_k(\Gamma)$ whose semidirect product $\chH(k)\rtimes\Gamma$ has canonical section $s_H(\gamma)=(1,\gamma)$, and to $(F,\eta)$ a
morphism of affine group schemes $i_F\co\check H\longrightarrow\check G$ internal to
$\Rep_k(\Gamma)$.
For an algebra $R$ and
$a\in\Aut^\otimes(R\otimes\omega_H)$, it is given by
\begin{equation}\label{eq:tannakian-map-explicit}
(i_F(a))_X
=\eta_X^{-1}\,a_{F(X)}\,\eta_X,
\qquad X\in\mathcal C_G.
\end{equation}
Because $\eta$ is a morphism of $\Gamma$-representations,
\eqref{eq:tannakian-map-explicit} is $\Gamma$-equivariant. It therefore induces a homomorphism of semidirect products given by
\begin{equation}\label{eq:split-relative-tannakian-map}
\Phi(h,\gamma)=\bigl(i_F(h),\gamma\bigr).
\end{equation}

In summary, a homomorphism $\Phi\co\chH(k)\rtimes\Gamma\to\chG(k)\rtimes\Gamma$ arising in this way is always admissible and untwisted. Note however that in some cases we may want to conjugate $i_F$ (for example when $H$ is a torus, we may want to conjugate $i_F$ to become the canonical embedding of the abstract Cartan), in which case Lemma
\ref{lem:conjugating-split-homomorphism} shows that the corresponding homomorphism can become twisted by a coboundary with respect to the original $\Phi$. 

\subsubsection{Passage from $c$-groups to L-groups}
\label{sssec:c-groups-to-L-groups}

We now specialize the preceding discussion to $\Gamma=W_F$ and to the
geometric $c$-groups occurring in Fargues--Scholze.  In this setting
\eqref{eq:split-relative-tannakian-map} is the
$c$-group homomorphism
\begin{equation}\label{eq:split-c-group-map}
\lc\psi(h,w)=\bigl(i_F(h),w\bigr).
\end{equation}
The usual pinned L-group is isomorphic to the $c$-group, but there is no canonical isomorphism; it depends on a square root of the cyclotomic character
\cite[Lemma 5.5.7]{Zhu17a}.  We now record the resulting cocycle for the L-group explicitly.

For a reductive $F$-group $J$, choose an isomorphism over $W_F$
\[
\alpha_J\co\ld J\xrightarrow{\sim}\lc J
\]
coming from a fixed square root of the cyclotomic character.  After
identifying the identity components, there is a uniquely determined function
$a_J\co W_F\to\wh J$ such that
\begin{equation}\label{eq:c-to-L-coordinate-change}
\alpha_J(j,w)=\bigl(j a_J(w),w\bigr).
\end{equation}
The precise description of $a_J$ depends on the convention used for the
$c$-group, but the following transport formula is convention-independent
once \eqref{eq:c-to-L-coordinate-change} is fixed.

\begin{lemma}\label{lem:transport-c-to-L}
Assume that the $c$-group homomorphism is given by
\eqref{eq:split-c-group-map}, and let
\[
\ld\psi=\alpha_G^{-1}\circ\lc\psi\circ\alpha_H.
\]
Then, in the corresponding pinned L-group coordinates,
\begin{equation}\label{eq:transported-L-map}
\ld\psi(h,w)
=\bigl(i_F(h)c_\psi(w),w\bigr),
\qquad
c_\psi(w)=i_F(a_H(w))a_G(w)^{-1}.
\end{equation}
In particular, the transported map has trivial cocycle if and only if
\begin{equation}\label{eq:normalization-compatibility-criterion}
i_F(a_H(w))=a_G(w)
\qquad\text{for every }w\in W_F.
\end{equation}
\end{lemma}

\begin{proof}
By \eqref{eq:c-to-L-coordinate-change},
\[
\alpha_H(h,w)=(h a_H(w),w).
\]
Applying $\lc\psi$ and then $\alpha_G^{-1}$ gives
\[
(h,w)\longmapsto
\bigl(i_F(h)i_F(a_H(w)),w\bigr)
\longmapsto
\bigl(i_F(h)i_F(a_H(w))a_G(w)^{-1},w\bigr),
\]
which is \eqref{eq:transported-L-map}.
\end{proof}

Combining Lemmas \ref{lem:conjugating-split-homomorphism} and
\ref{lem:transport-c-to-L}, if one replaces $i_F$ by $\Ad(g)\circ i_F$, then the cocycle becomes
\begin{equation}\label{eq:coordinate-change-total-cocycle}
c_{\psi,g}(w)
=g\,i_F(a_H(w))a_G(w)^{-1}\,{}^w g^{-1}.
\end{equation}

\subsection{Example calculations of the L-homomorphism}
\label{ssec:explicit-L-group-calculations}

\subsubsection{The constant term functor}
\label{sssec:constant-term-L-group-calculation}

Let $P\subset G$ be a parabolic subgroup defined over $F$, with Levi quotient
$M$, and let
\begin{equation}\label{eq:CT}
\CT_P[\deg_P]\co
\Sat(\Gr_{G,\Div_X^1};k)
\longrightarrow
\Sat(\Gr_{M,\Div_X^1};k)
\end{equation}
be the normalized constant term functor, so $\CT_P[\deg_P]$ is exact and symmetric monoidal.  Geometric Satake
identifies its restriction to a geometric fiber with restriction of
representations along the standard dual Levi homomorphism
\[
j_{M,G}\co\check M\longrightarrow\check G.
\]
Moreover, the cohomological comparison defining constant terms gives a
specified tensor isomorphism, as in the constant term compatibility of
\cite[\S VI.3 and \S VI.11]{FS},
\begin{equation}\label{eq:constant-term-fiber-rigidification}
\eta_P\co\omega_G
\xrightarrow{\sim}
\omega_M\circ\CT_P[\deg_P]
\end{equation}
in $\Rep_k(W_F)$.

\begin{prop}\label{prop:constant-term-split-c-map}
The $c$-homomorphism dual to \eqref {eq:CT} is
\begin{equation}\label{eq:constant-term-split-formula}
\lc j_{M,G}\co\lc M\longrightarrow \lc G,
\qquad
(m,w)\longmapsto\bigl(j_{M,G}(m),w\bigr).
\end{equation}
\end{prop}

\begin{proof}
Apply the relative Tannakian construction to the pair
$(\CT_P[\deg_P],\eta_P)$.  Formula
\eqref{eq:tannakian-map-explicit} identifies the induced connected
homomorphism with $j_{M,G}$, and the fact that \eqref{eq:constant-term-fiber-rigidification}
is a morphism in $\Rep_k(W_F)$ implies that the map $j_{M,G}$ is $W_F$-equivariant for the actions defining the $c$-groups.  Lemma \ref{lem:homomorphisms-of-split-extensions} then
gives \eqref{eq:constant-term-split-formula}.
\end{proof}

\begin{remark}\label{rem:L-group-CT-dual}
It is important that Proposition \ref{prop:constant-term-split-c-map} is a
statement about $c$-groups.  If one transports the map to the L-groups, Lemma
\ref{lem:transport-c-to-L} gives
\begin{equation}\label{eq:constant-term-usual-L-formula}
\ld j_{M,G}(m,w)
=\bigl(j_{M,G}(m)c_P(w),w\bigr),
\qquad
c_P(w)=j_{M,G}(a_M(w))a_G(w)^{-1}.
\end{equation}
Equivalently, \(c_P\) is dual to the modulus character, as follows. If \(U_P\) denotes
the unipotent radical of \(P\), set
\[
\delta_{P,\alg}(m)
=\det\!\bigl(\Ad(m)\mid\Lie U_P\bigr)\in X^*(M),
\qquad
\delta_P(m)=|\delta_{P,\alg}(m)|_F.
\]
Let \(\wh\delta_{P,\alg}\in X_*(Z(\wh M))\) be the cocharacter dual to
\(\delta_{P,\alg}\).  With the square root convention used in
\eqref{eq:c-to-L-coordinate-change},
\[
c_P(w)=\wh\delta_{P,\alg}\bigl(\chi_{\cyc}^{1/2}(w)\bigr).
\]
\end{remark}

\subsubsection{Cyclic base change and the diagonal homomorphism}
\label{sssec:cyclic-base-change-diagonal}

Let $E/F$ be a cyclic extension of degree $\ell$, let
$\Gamma=\Gal(E/F)=\langle\sigma\rangle$, and let $H$ be a connected
reductive group over $F$.  Put
\[
G=\Res_{E/F}(H_E).
\]
Note that there is a natural action of $\Gamma$ on $G$ under which $G^\Gamma \cong H$.

Recall that $E$ is understood to be a subfield of a fixed separable closure of $F$. The product pinning gives an
isomorphism $\wh G\xrightarrow{\sim}
\prod_{\gamma\in\Gamma}\wh H_\gamma$.
If $\bar w$ denotes the image of $w\in W_F$ in $\Gamma$, the standard
$W_F$-action is
\begin{equation}\label{eq:Weil-action-on-induced-dual}
{}^w(h_\gamma)_{\gamma\in\Gamma}
=\bigl({}^w h_{\bar w^{-1}\gamma}\bigr)_{\gamma\in\Gamma}.
\end{equation}
Let $\Delta\co\wh H\longrightarrow\wh G$ be the diagonal map. 

It is computed in \cite[Proposition~7.5.2]{F24} that the $\sigma$-dual homomorphism has the form
\begin{equation}\label{eq:cyclic-base-change-c-map}
\lc\psi(h,w)=\bigl(\Delta(h),w\bigr).
\end{equation}
We claim that after choosing the same square root of the cyclotomic character for $H$ and
$G$, and using the product pinning on $\wh G$, the transported usual
L-homomorphism is also
\begin{equation}\label{eq:cyclic-base-change-L-map}
\ld\psi(h,w)=\bigl(\Delta(h),w\bigr).
\end{equation}
To see this, use Lemma \ref{lem:transport-c-to-L}: note that we have $a_G(w)=\Delta(a_H(w))$ for all $w \in W_F$,
so the $c_\psi(w)$ from Lemma \ref{lem:transport-c-to-L} satisfies
\[
c_\psi(w)=\Delta(a_H(w))a_G(w)^{-1}=1.
\]
Putting this into Lemma \ref{lem:transport-c-to-L} yields \eqref{eq:cyclic-base-change-L-map}.

\section{Kim--Yu covers with modular coefficients}
\label{app:kim-yu-ohara-coefficients}

The purpose of this appendix is to adapt certain technical results of Kim--Yu \cite{KY17} and Ohara \cite{Ohara24}, which are stated with characteristic zero coefficients, to modular coefficients. 

\subsection{Coefficient conventions}
Let \(\ell\neq p\) and \(k=\ol\F_\ell\), and fix a choice of $q^{1/2} \in k$. 
Fix a parabolic \(F\)-subgroup \(P=MU\) of a connected reductive group \(G\) and a smooth $k$-representation $V$ of $P(F)$. We
write \(V_U\) for the space of $U(F)$-coinvariants of $V$. Following the convention
in \cite[\S 2]{Ohara24}, normalized Jacquet functor and normalized induction
are defined as
\[
r_P^G(V)=V_U\otimes\delta_P^{1/2},
\qquad
I_P^G(W)=\Ind_{P(F)}^{G(F)}(W\otimes\delta_P^{-1/2}),
\]
where $\delta_P$ is the modulus character of $P(F)$.

\begin{lemma}\label{lem:heisenberg-step}
Let $H$ be a finite Heisenberg $p$-group with center $C$. Let $A$ and $\ol A$ be subgroups such that $A\cap C=\ol A\cap C=\{1\}$. Let $\psi\co C\to k^\times$ be a character for which
\[
(xC,yC)\longmapsto \psi([x,y])
\]
is a nondegenerate alternating pairing on $H/C$, and assume that the images of $A$ and $\ol A$ are complementary Lagrangians for this pairing. Let $X$ be a $k[H]$-module on which $C$ acts by $\psi$. If $v\in X^A$ is
nonzero, then $\sum_{\bar a\in\ol A}\bar a v \neq 0$.
\end{lemma}

\begin{proof}
Same as the proof of \cite[Lemma 5.5]{KY17}.
\end{proof}

\subsection{The Kim--Yu cover criterion over \texorpdfstring{\(k\)}{k}}

We use the notation of \cite[\S 6]{KY17}, so $\ol P=M\ol U$ is the opposite parabolic for $P$ with Levi $M$. Fix a Yu datum
\[
\Psi=\bigl((G^j)_{0\leq j\leq d},y,(r_j)_{0\leq j\leq d},
\tau,(\phi_j)_{0\leq j\leq d}\bigr)
\]
for $G$ over $k$. We will assume that the maximal split central subtorus $Z_s(M)$ of $M$ is contained in $G^0$. Define a sequence of twisted Levis $(M^j)_{0 \leq j \leq d}$ in $M = M^d$ by 
\[
M^j\coloneqq Z_{G^j}\bigl(Z_s(M)\bigr) = M\cap G^j.
\]
Note that $Z(M^0)/Z(M)$ is anisotropic by the assumption $Z_s(M) \subset G^0$. Let $\{\io\}$ be an $\vec s$-generic diagram of building embeddings as in \cite[\S 3.5]{KY17} for\footnote{When $d=0$, this means $\vec s=(0)$.}
\[
\vec s=\left(0,\frac{r_0}{2},\ldots,\frac{r_{d-1}}{2}\right).
\]
If $H$ is one of the groups $G^j$, then $H(F)^{\{\io\}}_{y,a}$ denotes the Moy--Prasad subgroup formed at the image of $y$ under the diagram $\{\io\}$, and $H(F)^{\{\io\}}_{[y]}$ denotes the stabilizer of the same image in the reduced building of $H$. Define\footnote{With the convention that $K^+=G^0(F)^{\{\io\}}_{y,0+}$ when $d=0$.}
\begin{equation}\label{eq:KY-plus-explicit}
K^+
\coloneqq
G^0(F)^{\{\io\}}_{y,0+}
G^1(F)^{\{\io\}}_{y,\frac{r_0}{2}+}\cdots
G^d(F)^{\{\io\}}_{y,\frac{r_{d-1}}{2}+}.
\end{equation}
For $0\leq j<d$, let
\[
\widehat\phi_j^{\{\io\}}\co
G^0(F)^{\{\io\}}_{[y]}\,G^j(F)^{\{\io\}}_{y,0}\,
G(F)^{\{\io\}}_{y,\frac{r_j}{2}+}\longrightarrow k^\times
\]
be the extension of $\phi_j$ defined in \cite[\S 6.2]{CF26a} (following \cite[\S 4]{Yu01}), formed using the diagram $\{\io\}$, and set $\widehat\phi_d^{\{\io\}}=\phi_d$. Define the character
\begin{equation}\label{eq:KY-theta-explicit}
\theta = \prod_{j=0}^{d}
\left.\widehat\phi_j^{\{\io\}}\right|_{K^+} \co K^+\longrightarrow k^\times.
\end{equation}

\begin{defn}[Blondel cover condition over $k$]\label{def:blondel-cover}
Let $K\subset G(F)$ be an open subgroup which is compact mod center, let $K_M=K\cap M(F)$, and let $\rho$ and $\rho_M$ be finite-dimensional smooth $k$-representations of $K$ and $K_M$ on the same underlying vector space. We say that $(K,\rho)$ is a \emph{Blondel $G$-cover} of $(K_M,\rho_M)$ over $k$ for the fixed parabolic $P=MU$ if, writing $\ol P=M\ol U$ for the opposite parabolic, the following conditions hold:
\begin{enumerate}[label=(\roman*)]
\item\label{itm:decomp} $K=(K\cap\ol U(F))(K\cap M(F))(K\cap U(F))$. 
\item\label{itm:inflation} $\rho|_{K_M}=\rho_M$, and $K\cap U(F)$ and $K\cap\ol U(F)$ act trivially on $\rho$.
\item\label{itm:injective-jacquet} For every smooth $k$-representation $V$ of $G(F)$, the Jacquet map $j_U\co V\to V_U$ induces an injection
\begin{equation}\label{eq:injective-cover-hom}
\Hom_K(\rho,V)\hookrightarrow \Hom_{K_M}(\rho_M,V_U).
\end{equation}
\end{enumerate}
\end{defn}

\begin{remark}\label{rem:injective-vs-strong}
With characteristic zero coefficients, Definition~\ref{def:blondel-cover} is made in \cite[Definition 4.2]{KY17} under the name ``$G$-cover''. It is equivalent to Bushnell--Kutzko's notion of $G$-cover \cite[\S 8]{BK} in characteristic zero, by \cite[Th\'eor\`eme~1]{Blondel97}, but not in positive characteristic. 
\end{remark}

The following is the analogue of \cite[Theorem 6.3]{KY17} over $k$. Note that \textit{loc.\ cit.}\ uses the character $\theta \cdot \phi_d^{-1}$ in place of $\theta$, but it is obvious that these formulations are equivalent. 

\begin{prop}[Kim--Yu]\label{prop:KY-cover}
The pair $(K^+,\theta)$ is a Blondel $G$-cover of $(K_M^+,\theta_M)$ over $k$ in the sense of Definition~\ref{def:blondel-cover} for the fixed parabolic $P=MU$.
\end{prop}

\begin{proof}
Conditions \ref{itm:decomp} and \ref{itm:inflation} are \cite[Lemma 6.2]{KY17}; the proof of the latter makes no mention of the coefficients. It remains to prove \ref{itm:injective-jacquet}; we will sketch the argument from \cite{KY17}, indicating the necessary ingredients. The proof of \cite[Theorem 6.3]{KY17} is by induction on the length $d$ of the Yu datum. For $d=0$, \cite[Theorem 6.3]{KY17} reduces to \cite[Proposition 3.3]{KY17}, which is a reformulation of \cite[Proposition 6.7]{MP96}: for a $0$-generic embedding of buildings, the Jacquet map gives a bijection
\begin{equation}\label{eq:mp-base}
V^{G_{y,0+}}\xrightarrow{\sim} (V_U)^{M_{y,0+}} .
\end{equation}
The proof of \eqref{eq:mp-base} uses only averaging over compact open subgroups of the unipotent radical, hence it carries over to $k$ because $\ell \neq p$.

Before proceeding to the inductive step, we define some notation. Choose $\gamma\in X_*(Z_s(M^0))\otimes\RR$ with $\langle\alpha,\gamma\rangle>0$ for every root $\alpha$ of $Z_s(M^0)$ on $\Lie U$, where $Z_s(M^0)$ denotes the maximal $F$-split subtorus of the center of $M^0$. For $t\in\RR$, let $\{\io\}_{t\gamma}$ be the shifted diagram of embeddings of \cite[\S 3.6]{KY17}, and define $K^+(t)$ and $\theta(t)$ by the formulas \eqref{eq:KY-plus-explicit} and \eqref{eq:KY-theta-explicit} with $\{\io\}$ replaced by $\{\io\}_{t\gamma}$.  Let
\[
\cdots<t_{-1}<t_0<t_1<t_2<\cdots
\]
be a sequence of real numbers such that $\lim_{n\to\infty} t_n = \infty$ and $\lim_{n\to-\infty} t_n = -\infty$ and $K^+(t)$ is constant in the intervals $(t_i, t_{i+1})$, chosen so that $0\in(t_0,t_1)$. Set
\[
K^+(t_i,t_{i+1}) \coloneqq K^+(t),
\qquad
\theta(t_i,t_{i+1}) \coloneqq \theta(t)
\quad \text{for any }t\in(t_i,t_{i+1}).
\]

For the inductive step $d\geq 1$, Kim--Yu define $N_i = K^+(t_{i-1},t_i) \cap U$ and $\ol N_i = K^+(t_{i-1},t_i) \cap \ol U$, and form the groups $K^+(t_{i-1},t_i)=N_iK_M^+\ol N_i$. If $v$ is a nonzero vector on which $K^+$ acts through $\theta$, then Kim--Yu prove that $\int_{N_i} xvdx$ is a nonzero vector on which $K^+(t_{i-1},t_i)$ acts by $\theta(t_{i-1},t_i)$; since $\bigcup_{i \geq 0} N_i = U(F)$, this is enough to show that $v$ has nonzero image in $V_U$. Their proof uses only averages over pro-$p$ groups and Lemma~\ref{lem:heisenberg-step}, which still make sense over $k$ because $\ell\neq p$. So the canonical map $V\to V_U$ is injective on the $(K^+,\theta)$-isotypic vectors. Equivalently, the induced map
\[
\Hom_{K^+}(\theta,V)\longrightarrow
\Hom_{K_M^+}(\theta_M,V_U)
\]
is injective, which is \eqref{eq:injective-cover-hom} for $(K^+,\theta)$.
\end{proof}

The next corollary is the obvious analogue of \cite[Corollary 6.4]{KY17} over $k$. 

\begin{cor}\label{cor:KY-64}
Let $K$ be an open subgroup of $G(F)$ which is compact mod center, and let $\rho$ be a finite-dimensional smooth $k$-representation of $K$. Suppose that for the fixed parabolic $P=MU$, conditions \ref{itm:decomp} and \ref{itm:inflation} of Definition~\ref{def:blondel-cover} hold; suppose also that $K\supset K^+$ and that $\rho|_{K^+}$ is $\theta$-isotypic. Then $(K,\rho)$ is a Blondel $G$-cover of $(K_M,\rho|_{K_M})$ over $k$.
\end{cor}

\begin{proof}
Let $V$ be a smooth $k$-representation of $G(F)$. To prove that \eqref{eq:injective-cover-hom} is injective, let $T\in\Hom_K(\rho,V)$ and suppose that $j_U\circ T=0$. Since $K^+$ acts on $\rho$ through $\theta$, the image of $T$ lies in the $(K^+,\theta)$-eigenspace of $V$. Proposition~\ref{prop:KY-cover} says that $j_U$ is injective on this eigenspace, so $T=0$.
\end{proof}

\begin{cor}\label{cor:KY-dual-cover}
In the situation of Corollary~\ref{cor:KY-64}, the dual pair $(K,\rho^\vee)$ is a Blondel $G$-cover of $(K_M,(\rho|_{K_M})^{\vee})$ over $k$.
\end{cor}

\begin{proof}
Conditions (i) and (ii) in Definition~\ref{def:blondel-cover} are clear. Note that $K^+$ acts on $\rho^\vee$ by $\theta^{-1}$, so Corollary~\ref{cor:KY-64} proves the claim.
\end{proof}

\subsection{A substitute for type theory over \texorpdfstring{\(k\)}{k}}

For a central closed subgroup $Z \subset G(F)$, a closed subgroup $N \subset G(F)$, and a $k$-vector space $V$, let $C_c^\infty(N\backslash G(F),Z,V)$ denote the $k$-vector space of locally constant $V$-valued functions $\varphi\co N\backslash G(F) \to V$ which are compactly supported modulo $Z$ and are smooth in the sense that there is some compact open subgroup $K \subset G(F)$ such that $\varphi(Nh) = \varphi(Nhg)$ for all $h \in G(F)$ and all $g \in K$. Equip $C_c^\infty(N\backslash G(F),Z,V)$ with the right $G(F)$-action given by
\[
(\varphi \cdot g)(Nh) = \varphi(Nhg^{-1}),
\]
so $C_c^\infty(N\backslash G(F), Z, V)$ is a smooth $k$-representation of $G(F)$. If $M \subset G(F)$ is another closed subgroup normalizing $N$, then there is a left $M$-action on $C_c^\infty(N\backslash G(F), Z, V)$ given by
\[
(m \cdot \varphi)(Nh) = \varphi(Nm^{-1}h).
\]
In particular, $C_c^\infty(G(F), Z, V)$ admits commuting left and right $G(F)$-actions.

Observe that if $K$ is a compact-mod-center open subgroup of $G(F)$ and $Z = K \cap Z(G)(F)$ and $(\rho, V_\rho)$ is a smooth representation of $K$, then
\[
\cInd_K^{G(F)}(\rho) \cong C_c^\infty(G(F), Z, V_\rho)^K,
\]
where the $K$-action is given by
\[
(k \cdot \varphi)(g) = \rho(k)\varphi(k^{-1}g),
\]
and the isomorphism intertwines the right $G(F)$-actions.

For a given $Z$, let
\[
\lambda_U\co C_c^\infty(G(F),Z,k)
\longrightarrow C_c^\infty(U(F)\backslash G(F),Z,k)
\]
denote the map
\begin{equation}\label{eq:lambda-u}
\lambda_U(h)(g)=\int_{U(F)} h(ug)\,du,
\end{equation}
where $du$ is a fixed $k$-valued Haar measure. We equip $C_c^\infty(U(F)\backslash G(F),Z,k)$ with the $M(F)$-action
\begin{equation}\label{eq:star-action}
(m\star \varphi)(g)=\delta_P(m)^{-1}\varphi(m^{-1}g).
\end{equation}

The following result is recorded in the case $Z = 1$ in \cite[Remarque~3.4]{Dat09}.

\begin{lemma}\label{lem:function-coinvariants}
The map $\lambda_U$ is $U(F)$-invariant and induces an $M(F)$-equivariant isomorphism (with respect to the usual left action on the left side and the action \eqref{eq:star-action} on the right side)
\begin{equation}\label{eq:function-coinvariant-iso}
C_c^\infty(G(F),Z,k)_U
\xrightarrow{\sim}
C_c^\infty(U(F)\backslash G(F),Z,k).
\end{equation}
\end{lemma}

\begin{proof}
Since $U(F)$ is unimodular, it is clear from the definition \eqref{eq:lambda-u} that $\lambda_U$ is $P(F)$-equivariant and thus factors through $C_c^\infty(G(F),Z,k)_U$. Injectivity follows directly from the definition of coinvariants. For surjectivity, choose a compact open subgroup $U_0 \subset U(F)$, and rescale $du$ to assume $\operatorname{vol}(U_0) = 1$. The key claim is that there exists a continuous $Z$-equivariant section $\sigma\co U(F)\backslash G(F) \to G(F)$ to the natural projection. Granting this, we define $\wt f\co G(F) \to k$ by
\[
\wt f(u\sigma(\ol g)) = \begin{cases}
    f(\ol g) &\text{if } u \in U_0, \\
    0 &\text{otherwise}
\end{cases}
\]
for $u \in U(F)$ and $\ol g \in U(F)\backslash G(F)$. Since $\sigma$ is $Z$-equivariant, it is clear that $\wt f$ has compact support modulo $Z$. To see that $\wt f$ is smooth, let $K$ be a compact open subgroup of $G(F)$ and note that the map $c\co \supp f \times K \to U(F)$ given by $(g, h) \mapsto \sigma(\ol g)h\sigma(\ol g h)^{-1}$ has compact image: this follows again from $Z$-equivariance of $\sigma$. By shrinking $K$, we may therefore assume that $c$ factors through $U_0$, and we may further assume that $f$ is right $K$-invariant. But then for $g \in G(F)$ and $h \in K$ we have
\[
\sigma(\ol g)h = u_0\sigma(\ol g\ol h)
\]
for some $u_0 \in U_0$, and thus
\[
\wt f(u\sigma(\ol g)h) = \wt f(uu_0\sigma(\ol g h)) = \begin{cases}
    f(\ol g h) = f(\ol g) &\text{if } uu_0 \in U_0, \\
    0 &\text{otherwise.}
\end{cases}
\]
It follows that $\wt f$ is $K$-invariant, as desired.

Now we check that $\sigma$ exists. First, recall that the multiplication map $U(F) \times M(F) \times \ol U(F) \to G(F)$ is a homeomorphism to an open subset $\Omega$ of $G(F)$. For each $g \in G(F)$, there is a $Z$-equivariant continuous section $t_g\co U(F)\backslash \Omega g \to \Omega g$ given by
\[
t_g(um\ol u g) = m\ol u g
\]
for $u \in U(F)$, $m \in M(F)$, and $\ol u \in \ol U(F)$. It is an exercise in point-set topology (using the fact that $U(F)Z\backslash G(F)$ is locally compact, Hausdorff, and totally disconnected) to check that $G(F)$ is the disjoint union of $U(F)Z$-stable open subsets, each of which is contained in some $\Omega g$. Using this, one can patch together the various $t_g$ to obtain a section $\sigma$ as claimed.
\end{proof}

\begin{lemma}\label{lem:ohara-32}
Let \(K_M\subset M(F)\) be an open subgroup such that \(K_M/(K_M\cap Z(G)(F))\) is compact, and let \(\rho_M\) be a finite-dimensional smooth \(k\)-representation of \(K_M\). Let \(\operatorname{Inf}(\rho_M)\) denote the inflation of \(\rho_M\) from \(K_M\) to \(U(F)K_M\). The map
\begin{equation}\label{eq:iu1}
I_{U,1}\co F\longmapsto
\left[g\longmapsto
\left[m\longmapsto \delta_P(m)^{1/2}F(mg)\right]\right]
\end{equation}
is a \(G(F)\)-equivariant isomorphism
\[
\cInd_{U(F)K_M}^{G(F)}\operatorname{Inf}(\rho_M)
\xrightarrow{\sim}
I_P^G\left(\cInd_{K_M}^{M(F)}\rho_M\right).
\]
\end{lemma}

\begin{proof}
As mentioned in the proof of \cite[Lemma~3.2]{Ohara24}, the map $F \mapsto [g \mapsto F(g)(1)]$ is easily checked to define the inverse isomorphism.
\end{proof}

\begin{prop}
\label{prop:ohara-33}
Let $K\subset G(F)$ be a compact-mod-center open subgroup of $G(F)$, and let $\rho$ be a finite-dimensional smooth $k$-representation of $K$. Assume that $(K,\rho)$ satisfies conditions \ref{itm:decomp} and \ref{itm:inflation} of Definition~\ref{def:blondel-cover} for $P=MU$, where $\rho_M=\rho|_{K_M}$. Assume furthermore that the dual pair $(K,\rho^\vee)$ is a Blondel $G$-cover of $(K_M,\rho_M^\vee)$ over $k$ for $P$. Then the map
\begin{equation}
I_{U,2}\co 
\cInd_K^{G(F)}\rho
\rightarrow
\cInd_{U(F)K_M}^{G(F)}\operatorname{Inf}(\rho_M)
\end{equation}
given by 
\begin{equation}\label{eq:iu2}
f\longmapsto
\left[g\longmapsto \int_{U(F)} f(ug)\,du\right]
\end{equation}
is injective. Consequently, $I_U=I_{U,1}\circ I_{U,2}$ gives an injection
\begin{equation}\label{eq:ohara-composition}
\cInd_K^{G(F)}\rho
\hookrightarrow
I_P^G\left(\cInd_{K_M}^{M(F)}\rho_M\right).
\end{equation}
\end{prop}

\begin{remark}
This is an analogue of \cite[Proposition 3.3]{Ohara24}, which says that over $\CC$ the analogous map \eqref{eq:ohara-composition} is actually an \emph{isomorphism} when $K$ is compact. The proof of \emph{loc.\ cit.}\ uses the fact that smooth representations of compact groups are semisimple, which is only true with characteristic $0$ coefficients. However, it turns out that our applications will only require injectivity, and this follows from Ohara's argument. The result over $\CC$ was also proven abstractly in \cite[Lemma B.3]{BS20} using type theory, which is the reason for the title of this subsection.
\end{remark}

\begin{proof}
First, we will show that the integral in \eqref{eq:iu2} is well-defined. Put $Z=K\cap Z(G)(F)$. Since $K/Z$ is compact, $\operatorname{supp}(f)$ is compact modulo $Z$. The closed embedding $U(F)\hookrightarrow G(F)/Z$ implies that $U(F)g\cap\operatorname{supp}(f)$ is compact for each $g$, so local constancy of $f$ gives the claim. Assumptions (i) and (ii) in Definition~\ref{def:blondel-cover} imply
\[
I_{U,2}(f)(u m g)=\rho_M(m) I_{U,2}(f)(g)
\qquad
(u\in U(F),\,m\in K_M),
\]
where $\delta_P(m)=1$ because $K_M$ is compact mod center. The function $I_{U,2}(f)$ has compact support modulo $U(F)K_M$: indeed, $\operatorname{supp}(I_{U,2}(f))\subset U(F)\operatorname{supp}(f)$, and $\operatorname{supp}(f)$ is compact modulo $Z\subset K_M$. Thus \eqref{eq:iu2} lands in $\cInd_{U(F)K_M}^{G(F)}\operatorname{Inf}(\rho_M)$.

Since $\rho$ is finite-dimensional, there is a natural $G(F)$-equivariant isomorphism
\begin{equation}\label{eq:compact-ind-hom}
\cInd_K^{G(F)}\rho
\xrightarrow{\sim}
\Hom_K(\rho^\vee,C_c^\infty(G(F),Z,k)),
\qquad
f\longmapsto
\bigl[\lambda\mapsto(g\mapsto \lambda(f(g)))\bigr].
\end{equation}
Likewise, since the action of $U(F)$ on $\rho_M$ is trivial, we have 
\begin{equation}\label{eq:target-ind-hom}
\cInd_{U(F)K_M}^{G(F)}\operatorname{Inf}(\rho_M)
\xrightarrow{\sim}
\Hom_{K_M}(\rho_M^\vee,
C_c^\infty(U(F)\backslash G(F),Z,k)).
\end{equation}
Under \eqref{eq:compact-ind-hom} and \eqref{eq:target-ind-hom}, the map \(I_{U,2}\) is the map induced by
\[
\lambda_U\co \Hom_K(\rho^\vee,C_c^\infty(G(F),Z,k))
\longrightarrow
\Hom_{K_M}(\rho_M^\vee,C_c^\infty(U(F)\backslash G(F),Z,k)).
\]
By Lemma~\ref{lem:function-coinvariants}, this is exactly the instance \(V=C_c^\infty(G(F),Z,k)\) of Definition~\ref{def:blondel-cover} for $(K, \rho^\vee)$. Hence $I_{U,2}$ is injective, and we conclude by Lemma~\ref{lem:ohara-32}.
\end{proof}

\bibliographystyle{amsalpha-with-labels}
\bibliography{Bibliography}

@misc{ALRR22,
      title={Fixed points under pinning-preserving automorphisms of reductive group schemes}, 
      author={Pramod N. Achar and João Lourenço and Timo Richarz and Simon Riche},
      year={2022},
      eprint={2212.10182},
      archivePrefix={arXiv},
      primaryClass={math.AG},
      url={https://arxiv.org/abs/2212.10182}, 
}

@misc{BS20,
      title={The {G}elfand--{G}raev representation of {S}{O}$(2n+1)$ in terms of {H}ecke algebras}, 
      author={Petar Bakic and Gordan Savin},
      year={2020},
      eprint={2011.02456},
      archivePrefix={arXiv},
      primaryClass={math.RT},
      url={https://arxiv.org/abs/2011.02456}, 
}

@article {BBM,
    AUTHOR = {Beilinson, A. and Bezrukavnikov, R. and Mirkovi\'{c}, I.},
     TITLE = {Tilting exercises},
   JOURNAL = {Mosc. Math. J.},
  FJOURNAL = {Moscow Mathematical Journal},
    VOLUME = {4},
      YEAR = {2004},
    NUMBER = {3},
     PAGES = {547--557, 782},
      ISSN = {1609-3321},
   MRCLASS = {14F05 (18E30)},
 MRNUMBER = {2119139},
MRREVIEWER = {Toshiyuki Tanisaki},
       DOI = {10.17323/1609-4514-2004-4-3-547-557},
       URL = {https://doi.org/10.17323/1609-4514-2004-4-3-547-557},
}

@article {BCT24,
    AUTHOR = {Booher, Jeremy and Cotner, Sean and Tang, Shiang},
     TITLE = {Lifting {$G$}-valued {G}alois representations when {$\ell\neq
              p$}},
   JOURNAL = {Forum Math. Sigma},
  FJOURNAL = {Forum of Mathematics. Sigma},
    VOLUME = {12},
      YEAR = {2024},
     PAGES = {Paper No. e109, 41},
      ISSN = {2050-5094},
   MRCLASS = {11F80 (14L17 20G07)},
  MRNUMBER = {4831145},
       DOI = {10.1017/fms.2024.113},
       URL = {https://doi.org/10.1017/fms.2024.113},
}

@misc{BPHT25,
      title={Inductive construction of supercuspidal $L$-packets}, 
      author={Raphaël Beuzart-Plessis and Michael Harris and Jack Thorne},
      year={2025},
      eprint={2502.20611},
      archivePrefix={arXiv},
      primaryClass={math.NT},
      url={https://arxiv.org/abs/2502.20611}, 
}

@book {BK93,
    AUTHOR = {Bushnell, Colin J. and Kutzko, Philip C.},
     TITLE = {The admissible dual of {${\rm GL}(N)$} via compact open
              subgroups},
    SERIES = {Annals of Mathematics Studies},
    VOLUME = {129},
 PUBLISHER = {Princeton University Press, Princeton, NJ},
      YEAR = {1993},
     PAGES = {xii+313},
      ISBN = {0-691-03256-4; 0-691-02114-7},
   MRCLASS = {22E50 (22-02)},
  MRNUMBER = {1204652},
MRREVIEWER = {Mark\ Reeder},
       DOI = {10.1515/9781400882496},
       URL = {https://doi.org/10.1515/9781400882496},
}

@article {Blondel97,
    AUTHOR = {Blondel, Corinne},
     TITLE = {Crit\`ere d'injectivit\'e pour l'application de {J}acquet},
   JOURNAL = {C. R. Acad. Sci. Paris S\'er. I Math.},
  FJOURNAL = {Comptes Rendus de l'Acad\'emie des Sciences. S\'erie I.
              Math\'ematique},
    VOLUME = {325},
      YEAR = {1997},
    NUMBER = {11},
     PAGES = {1149--1152},
      ISSN = {0764-4442},
}

@article {BK,
    AUTHOR = {Bushnell, Colin J. and Kutzko, Philip C.},
     TITLE = {Smooth representations of reductive {$p$}-adic groups:
              structure theory via types},
   JOURNAL = {Proc. London Math. Soc. (3)},
    VOLUME = {77},
      YEAR = {1998},
     PAGES = {582--634},
}

@article {BMHN24,
    AUTHOR = {Bertoloni Meli, Alexander and Hamann, Linus and Nguyen, Kieu
              Hieu},
     TITLE = {Compatibility of the {F}argues-{S}cholze correspondence for
              unitary groups},
   JOURNAL = {Math. Ann.},
  FJOURNAL = {Mathematische Annalen},
    VOLUME = {390},
      YEAR = {2024},
    NUMBER = {3},
     PAGES = {4729--4787},
      ISSN = {0025-5831,1432-1807},
   MRCLASS = {11S37},
  MRNUMBER = {4803488},
MRREVIEWER = {Nicolas\ Arancibia Robert},
       DOI = {10.1007/s00208-024-02877-x},
       URL = {https://doi.org/10.1007/s00208-024-02877-x},
}

@article {BMR05,
    AUTHOR = {Bate, Michael and Martin, Benjamin and R\"ohrle, Gerhard},
     TITLE = {A geometric approach to complete reducibility},
   JOURNAL = {Invent. Math.},
  FJOURNAL = {Inventiones Mathematicae},
    VOLUME = {161},
      YEAR = {2005},
    NUMBER = {1},
     PAGES = {177--218},
      ISSN = {0020-9910,1432-1297},
   MRCLASS = {20G15 (20E42)},
  MRNUMBER = {2178661},
MRREVIEWER = {Andy\ R.\ Magid},
       DOI = {10.1007/s00222-004-0425-9},
       URL = {https://doi.org/10.1007/s00222-004-0425-9},
}

@book {CE56,
    AUTHOR = {Cartan, Henri and Eilenberg, Samuel},
     TITLE = {Homological {A}lgebra},
 PUBLISHER = {Princeton University Press, Princeton, NJ},
      YEAR = {1956},
     PAGES = {xv+390},
   MRCLASS = {09.0X},
  MRNUMBER = {77480},
MRREVIEWER = {G.\ Hochschild},
}

@incollection {Con14,
    AUTHOR = {Conrad, Brian},
     TITLE = {Reductive group schemes},
 BOOKTITLE = {Autour des sch\'emas en groupes. {V}ol. {I}},
    SERIES = {Panor. Synth\`eses},
    VOLUME = {42/43},
     PAGES = {93--444},
 PUBLISHER = {Soc. Math. France, Paris},
      YEAR = {2014},
      ISBN = {978-2-85629-794-0},
   MRCLASS = {14L15},
  MRNUMBER = {3362641},
}

@article {Cot24,
    AUTHOR = {Sean Cotner},
     TITLE = {Morphisms of character varieties},
   JOURNAL = {Int. Math. Res. Not. IMRN},
  FJOURNAL = {International Mathematics Research Notices. IMRN},
      YEAR = {2024},
    NUMBER = {16},
     PAGES = {11540--11548},
      ISSN = {1073-7928,1687-0247},
   MRCLASS = {14M35 (14L15 20G35)},
  MRNUMBER = {4789091},
MRREVIEWER = {Wen-Wei\ Li},
       DOI = {10.1093/imrn/rnae124},
       URL = {https://doi.org/10.1093/imrn/rnae124},
}

@misc{Cot24b,
      title={Connected components of the moduli space of {L}-parameters}, 
      author={Sean Cotner},
      year={2024},
      eprint={2404.16716},
      archivePrefix={arXiv},
      primaryClass={math.NT},
      url={https://arxiv.org/abs/2404.16716}, 
}

@misc{Cot25,
      title={Hom schemes for algebraic groups}, 
      author={Sean Cotner},
      year={2025},
      eprint={2309.16458},
      archivePrefix={arXiv},
      primaryClass={math.AG},
      url={https://arxiv.org/abs/2309.16458},
      note={To appear in Algebra \& Number Theory}
}

@misc{Cot26b,
      title={Modular functoriality for finite groups}, 
      author={Sean Cotner},
      year={2026}, 
}

@misc{CF26a,
      key={CF26a},
      title={Local {L}anglands functoriality for {Y}u's supercuspidals {I}: {K}aletha's parametrization}, 
      author={Sean Cotner and Tony Feng},
      year={2026}, 
}

@article {Dat09,
    AUTHOR = {Dat, Jean-Fran\c{c}ois},
     TITLE = {Finitude pour les repr\'{e}sentations lisses de groupes
              {$p$}-adiques},
   JOURNAL = {J. Inst. Math. Jussieu},
  FJOURNAL = {Journal of the Institute of Mathematics of Jussieu. JIMJ.
              Journal de l'Institut de Math\'{e}matiques de Jussieu},
    VOLUME = {8},
      YEAR = {2009},
    NUMBER = {2},
     PAGES = {261--333},
      ISSN = {1474-7480},
   MRCLASS = {22E50},
MRNUMBER = {2485794},
       DOI = {10.1017/S1474748008000054},
       URL = {https://doi.org/10.1017/S1474748008000054},
}

@article {DS18,
    AUTHOR = {DeBacker, Stephen and Spice, Loren},
     TITLE = {Stability of character sums for positive-depth, supercuspidal
              representations},
   JOURNAL = {J. Reine Angew. Math.},
  FJOURNAL = {Journal f\"ur die Reine und Angewandte Mathematik. [Crelle's
              Journal]},
    VOLUME = {742},
      YEAR = {2018},
     PAGES = {47--78},
      ISSN = {0075-4102,1435-5345},
   MRCLASS = {22E40 (43A15)},
  MRNUMBER = {3849622},
MRREVIEWER = {Adrien\ Boyer},
       DOI = {10.1515/crelle-2015-0094},
       URL = {https://doi-org.proxy.lib.umich.edu/10.1515/crelle-2015-0094},
}

@misc{DvHKZ26,
      title={Igusa {S}tacks and the {C}ohomology of {S}himura {V}arieties {I}{I}}, 
      author={Patrick Daniels and Pol van Hoften and Dongryul Kim and Mingjia Zhang},
      year={2026},
      eprint={2603.24921},
      archivePrefix={arXiv},
      primaryClass={math.NT},
      url={https://arxiv.org/abs/2603.24921}, 
}

@article {DHKM,
    AUTHOR = {Dat, Jean-Fran\c{c}ois and Helm, David and Kurinczuk, Robert and
              Moss, Gilbert},
     TITLE = {Finiteness for {H}ecke algebras of {$p$}-adic groups},
   JOURNAL = {J. Amer. Math. Soc.},
  FJOURNAL = {Journal of the American Mathematical Society},
    VOLUME = {37},
      YEAR = {2024},
    NUMBER = {3},
     PAGES = {929--949},
      ISSN = {0894-0347},
   MRCLASS = {22E50 (11F70 11F80)},
  MRNUMBER = {4736530},
       DOI = {10.1090/jams/1034},
       URL = {https://doi.org/10.1090/jams/1034},
}

@article {DHKM25,
    AUTHOR = {Dat, Jean-Fran\c{c}ois and Helm, David and Kurinczuk, Robert
              and Moss, Gilbert},
     TITLE = {Moduli of {L}anglands parameters},
   JOURNAL = {J. Eur. Math. Soc. (JEMS)},
  FJOURNAL = {Journal of the European Mathematical Society (JEMS)},
    VOLUME = {27},
      YEAR = {2025},
    NUMBER = {5},
     PAGES = {1827--1927},
      ISSN = {1435-9855,1435-9863},
   MRCLASS = {11F80 (11F70 22E50)},
  MRNUMBER = {4889237},
       DOI = {10.4171/jems/1599},
       URL = {https://doi.org/10.4171/jems/1599},
}

@misc{Ete23,
      title={Faisceaux monodromiques, th\'eorie de {D}eligne--{L}usztig, et cohomologie des champs de chtoucas en profondeur 0}, 
      author={Arnaud Eteve},
      year={2023}, 
}

@misc{F24,
      title={Modular functoriality in the {L}ocal {L}anglands {C}orrespondence}, 
      author={Feng, Tony},
      year={2024},
      eprint={2312.12542},
      archivePrefix={arXiv},
      primaryClass={math.NT},
      url={https://arxiv.org/abs/2312.12542}, 
}

@misc{FS,
      title={Geometrization of the local {L}anglands correspondence}, 
      author={Laurent Fargues and Peter Scholze},
      year={2021},
      eprint={2102.13459},
      archivePrefix={arXiv},
      primaryClass={math.RT},
      note = {To appear in Asterisque.}
}

@article {Fin21b,
    AUTHOR = {Fintzen, Jessica},
     TITLE = {On the construction of tame supercuspidal representations},
   JOURNAL = {Compos. Math.},
  FJOURNAL = {Compositio Mathematica},
    VOLUME = {157},
      YEAR = {2021},
    NUMBER = {12},
     PAGES = {2733--2746},
      ISSN = {0010-437X,1570-5846},
   MRCLASS = {22E50 (20G25)},
  MRNUMBER = {4357723},
MRREVIEWER = {Yongqi\ Feng},
       DOI = {10.1112/S0010437X21007636},
       URL = {https://doi.org/10.1112/S0010437X21007636},
}

@article {Fin22,
    AUTHOR = {Fintzen, Jessica},
     TITLE = {Tame cuspidal representations in non-defining characteristics},
   JOURNAL = {Michigan Math. J.},
  FJOURNAL = {Michigan Mathematical Journal},
    VOLUME = {72},
      YEAR = {2022},
     PAGES = {331--342},
      ISSN = {0026-2285},
   MRCLASS = {22E50 (11F70)},
  MRNUMBER = {4460255},
MRREVIEWER = {Corina Ciobotaru},
       DOI = {10.1307/mmj/20217217},
       URL = {https://doi.org/10.1307/mmj/20217217},
}

@article {Fin21,
    AUTHOR = {Fintzen, Jessica},
     TITLE = {Types for tame {$p$}-adic groups},
   JOURNAL = {Ann. of Math. (2)},
  FJOURNAL = {Annals of Mathematics. Second Series},
    VOLUME = {193},
      YEAR = {2021},
    NUMBER = {1},
     PAGES = {303--346},
      ISSN = {0003-486X},
   MRCLASS = {22E50},
MRNUMBER = {4199732},
MRREVIEWER = {Corina Ciobotaru},
       DOI = {10.4007/annals.2021.193.1.4},
       URL = {https://doi.org/10.4007/annals.2021.193.1.4},
}

@article {FKS23,
    AUTHOR = {Fintzen, Jessica and Kaletha, Tasho and Spice, Loren},
     TITLE = {A twisted {Y}u construction, {H}arish-{C}handra characters,
              and endoscopy},
   JOURNAL = {Duke Math. J.},
  FJOURNAL = {Duke Mathematical Journal},
    VOLUME = {172},
      YEAR = {2023},
    NUMBER = {12},
     PAGES = {2241--2301},
      ISSN = {0012-7094,1547-7398},
   MRCLASS = {22E50 (11S37)},
  MRNUMBER = {4654051},
MRREVIEWER = {Yongqi\ Feng},
       DOI = {10.1215/00127094-2022-0080},
       URL = {https://doi.org/10.1215/00127094-2022-0080},
}

@misc{FS25,
      title={Construction of tame supercuspidal representations in arbitrary residue characteristic}, 
      author={Jessica Fintzen and David Schwein},
      year={2025},
      eprint={2501.18553},
      archivePrefix={arXiv},
      primaryClass={math.RT},
      url={https://arxiv.org/abs/2501.18553}, 
}

@misc{Fu26,
      title={Depth-zero local {L}anglands via shtukas in mixed characteristic}, 
      author={Chenji Fu},
      year={2026}, 
}

@article {GHS24,
    AUTHOR = {Gan, Wee Teck and Harris, Michael and Sawin, Will and
              Beuzart-Plessis, Rapha\"el},
     TITLE = {Local parameters of supercuspidal representations},
   JOURNAL = {Forum Math. Pi},
  FJOURNAL = {Forum of Mathematics. Pi},
    VOLUME = {12},
      YEAR = {2024},
     PAGES = {Paper No. e13, 41},
      ISSN = {2050-5086},
   MRCLASS = {22E50 (11F70 11F80 11S37)},
  MRNUMBER = {4794807},
MRREVIEWER = {Zhengyu\ Mao},
       DOI = {10.1017/fmp.2024.10},
       URL = {https://doi.org/10.1017/fmp.2024.10},
}

@article {Ger77,
    AUTHOR = {G{\'e}rardin, Paul},
     TITLE = {Weil representations associated to finite fields},
   JOURNAL = {J. Algebra},
  FJOURNAL = {Journal of Algebra},
    VOLUME = {46},
      YEAR = {1977},
    NUMBER = {1},
     PAGES = {54--101},
      ISSN = {0021-8693},
   MRCLASS = {20G05},
  MRNUMBER = {460477},
MRREVIEWER = {G.\ I.\ Lehrer},
       DOI = {10.1016/0021-8693(77)90394-5},
       URL = {https://doi-org.proxy.lib.umich.edu/10.1016/0021-8693(77)90394-5},
}

@misc{GL18,
      title={Chtoucas restreints pour les groupes r\'eductifs et param\'etrisation de Langlands locale}, 
      author={Alain Genestier and Vincent Lafforgue},
      year={2018},
      eprint={1709.00978},
      archivePrefix={arXiv},
      primaryClass={math.AG}
}

@misc{GHILZ26,
      title={On the {S}chematic and {A}nalytic {C}onstructions of the {L}ocal {L}anglands {C}ategory}, 
      author={Ian Gleason and Linus Hamann and Alexander B. Ivanov and João Lourenço and Konrad Zou},
      year={2026},
      eprint={2606.02799},
      archivePrefix={arXiv},
      primaryClass={math.NT},
      url={https://arxiv.org/abs/2606.02799}, 
}

@article {Ham25,
    AUTHOR = {Hamann, Linus},
     TITLE = {Compatibility of the {F}argues-{S}cholze and {G}an-{T}akeda
              local {L}anglands},
   JOURNAL = {Compos. Math.},
  FJOURNAL = {Compositio Mathematica},
    VOLUME = {161},
      YEAR = {2025},
    NUMBER = {9},
     PAGES = {2202--2271},
      ISSN = {0010-437X,1570-5846},
   MRCLASS = {11S37 (14G35 14G45)},
  MRNUMBER = {4978375},
       DOI = {10.1112/S0010437X25102455},
       URL = {https://doi.org/10.1112/S0010437X25102455},
}

@misc{HM26,
      title={The categorical local Langlands conjecture}, 
      author={David Hansen and Lucas Mann},
      year={2026},
      eprint={2606.00983},
      archivePrefix={arXiv},
      primaryClass={math.NT},
      url={https://arxiv.org/abs/2606.00983}, 
}

@article {HKW22,
    AUTHOR = {Hansen, David and Kaletha, Tasho and Weinstein, Jared},
     TITLE = {On the {K}ottwitz conjecture for local shtuka spaces},
   JOURNAL = {Forum Math. Pi},
  FJOURNAL = {Forum of Mathematics. Pi},
    VOLUME = {10},
      YEAR = {2022},
     PAGES = {Paper No. e13, 79},
   MRCLASS = {14G45 (11S37)},
  MRNUMBER = {4430954},
MRREVIEWER = {Kazuhiro Ito},
       DOI = {10.1017/fmp.2022.7},
       URL = {https://doi.org/10.1017/fmp.2022.7},
}

@article {Han26,
    AUTHOR = {Hansen, David},
     TITLE = {Excursion operators and the stable {B}ernstein center},
   JOURNAL = {Forum Math. Pi},
  FJOURNAL = {Forum of Mathematics. Pi},
    VOLUME = {14},
      YEAR = {2026},
     PAGES = {Paper No. e10, 14},
      ISSN = {2050-5086},
   MRCLASS = {22E50 (11F77)},
  MRNUMBER = {5042225},
       DOI = {10.1017/fmp.2026.10028},
       URL = {https://doi.org/10.1017/fmp.2026.10028},
}

@article {Kal19,
    AUTHOR = {Kaletha, Tasho},
     TITLE = {Regular supercuspidal representations},
   JOURNAL = {J. Amer. Math. Soc.},
  FJOURNAL = {Journal of the American Mathematical Society},
    VOLUME = {32},
      YEAR = {2019},
    NUMBER = {4},
     PAGES = {1071--1170},
      ISSN = {0894-0347},
   MRCLASS = {22E50 (11F70 11S37)},
  MRNUMBER = {4013740},
MRREVIEWER = {Amiya Kumar Mondal},
       DOI = {10.1090/jams/925},
       URL = {https://doi.org/10.1090/jams/925},
}

@misc{Kal21b,
      title={Supercuspidal {L}-packets}, 
      author={Tasho Kaletha},
      year={2021},
      eprint={1912.03274},
      archivePrefix={arXiv},
      primaryClass={math.RT},
      url={https://arxiv.org/abs/1912.03274}, 
}

@book {KP,
    AUTHOR = {Kaletha, Tasho and Prasad, Gopal},
     TITLE = {Bruhat-{T}its theory---a new approach},
    SERIES = {New Mathematical Monographs},
    VOLUME = {44},
 PUBLISHER = {Cambridge University Press, Cambridge},
      YEAR = {2023},
     PAGES = {xxx+718},
      ISBN = {978-1-108-83196-3},
   MRCLASS = {20E42 (11F70 20G25 22E50)},
  MRNUMBER = {4520154},
}

@incollection {KY17,
    AUTHOR = {Kim, Ju-Lee and Yu, Jiu-Kang},
     TITLE = {Construction of tame types},
 BOOKTITLE = {Representation theory, number theory, and invariant theory},
    SERIES = {Progr. Math.},
    VOLUME = {323},
     PAGES = {337--357},
 PUBLISHER = {Birkh\"auser/Springer, Cham},
      YEAR = {2017},
      ISBN = {978-3-319-59727-0; 978-3-319-59728-7},
   MRCLASS = {22E50},
  MRNUMBER = {3753917},
MRREVIEWER = {Alan\ Roche},
       DOI = {10.1007/978-3-319-59728-7\_12},
       URL = {https://doi-org.proxy.lib.umich.edu/10.1007/978-3-319-59728-7_12},
}

@misc{LH24,
      title={On close fields and the local Langlands correspondence}, 
      author={Siyan Daniel Li-Huerta},
      year={2024},
      eprint={2407.07063},
      archivePrefix={arXiv},
      primaryClass={math.NT},
      url={https://arxiv.org/abs/2407.07063}, 
}

@article {Mok15,
    AUTHOR = {Mok, Chung Pang},
     TITLE = {Endoscopic classification of representations of quasi-split
              unitary groups},
   JOURNAL = {Mem. Amer. Math. Soc.},
  FJOURNAL = {Memoirs of the American Mathematical Society},
    VOLUME = {235},
      YEAR = {2015},
    NUMBER = {1108},
     PAGES = {vi+248},
      ISSN = {0065-9266,1947-6221},
      ISBN = {978-1-4704-1041-4; 978-1-4704-2226-4},
   MRCLASS = {22E55 (11R42 22E50)},
  MRNUMBER = {3338302},
MRREVIEWER = {Neven\ Grbac},
       DOI = {10.1090/memo/1108},
       URL = {https://doi.org/10.1090/memo/1108},
}

@article {MP96,
    AUTHOR = {Moy, Allen and Prasad, Gopal},
     TITLE = {Jacquet functors and unrefined minimal {$K$}-types},
   JOURNAL = {Comment. Math. Helv.},
  FJOURNAL = {Commentarii Mathematici Helvetici},
    VOLUME = {71},
      YEAR = {1996},
    NUMBER = {1},
     PAGES = {98--121},
      ISSN = {0010-2571},
   MRCLASS = {22E50 (22E35)},
  MRNUMBER = {1371680},
MRREVIEWER = {Mark Reeder},
       DOI = {10.1007/BF02566411},
       URL = {https://doi.org/10.1007/BF02566411},
}

@article {Ohara24,
    AUTHOR = {Ohara, Kazuma},
     TITLE = {A comparison of endomorphism algebras},
   JOURNAL = {J. Algebra},
  FJOURNAL = {Journal of Algebra},
    VOLUME = {659},
      YEAR = {2024},
     PAGES = {183--343},
      ISSN = {0021-8693,1090-266X},
   MRCLASS = {22E50 (20C08 20G25)},
  MRNUMBER = {4775177},
MRREVIEWER = {Arnaud\ Mayeux},
       DOI = {10.1016/j.jalgebra.2024.07.002},
       URL = {https://doi-org.proxy.lib.umich.edu/10.1016/j.jalgebra.2024.07.002},
}

@article {OT21,
    AUTHOR = {Oi, Masao and Tokimoto, Kazuki},
     TITLE = {Local {L}anglands correspondence for regular supercuspidal
              representations of {$GL(n)$}},
   JOURNAL = {Int. Math. Res. Not. IMRN},
  FJOURNAL = {International Mathematics Research Notices. IMRN},
      YEAR = {2021},
    NUMBER = {3},
     PAGES = {2007--2073},
      ISSN = {1073-7928,1687-0247},
   MRCLASS = {22E50 (11F70 11S37)},
  MRNUMBER = {4206603},
MRREVIEWER = {Corina\ Ciobotaru},
       DOI = {10.1093/imrn/rnaa197},
       URL = {https://doi.org/10.1093/imrn/rnaa197},
}

@misc{Pen26,
      title={Fargues-{S}cholze correspondence and endoscopic classification for special orthogonal and unitary groups}, 
      author={Hao Peng},
      year={2026},
      eprint={2503.04623},
      archivePrefix={arXiv},
      primaryClass={math.NT},
      url={https://arxiv.org/abs/2503.04623}, 
}

@article {PY06,
    AUTHOR = {Prasad, Gopal and Yu, Jiu-Kang},
     TITLE = {On quasi-reductive group schemes},
      NOTE = {With an appendix by Brian Conrad},
   JOURNAL = {J. Algebraic Geom.},
  FJOURNAL = {Journal of Algebraic Geometry},
    VOLUME = {15},
      YEAR = {2006},
    NUMBER = {3},
     PAGES = {507--549},
      ISSN = {1056-3911,1534-7486},
   MRCLASS = {14L17 (14L15)},
  MRNUMBER = {2219847},
MRREVIEWER = {Alan\ Koch},
       DOI = {10.1090/S1056-3911-06-00422-X},
       URL = {https://doi-org.proxy.lib.umich.edu/10.1090/S1056-3911-06-00422-X},
}

@article {Sch25,
    AUTHOR = {Scholze, Peter},
     TITLE = {Geometrization of the local {L}anglands correspondence,
              motivically},
   JOURNAL = {J. Reine Angew. Math.},
  FJOURNAL = {Journal f\"ur die Reine und Angewandte Mathematik. [Crelle's
              Journal]},
    VOLUME = {835},
      YEAR = {2026},
     PAGES = {301--327},
      ISSN = {0075-4102,1435-5345},
  MRNUMBER = {5080051},
       DOI = {10.1515/crelle-2026-0011},
       URL = {https://doi.org/10.1515/crelle-2026-0011},
}

@article {SS08,
    AUTHOR = {S\'echerre, V. and Stevens, S.},
     TITLE = {Repr\'esentations lisses de {${\rm GL}_m(D)$}. {IV}.
              {R}epr\'esentations supercuspidales},
   JOURNAL = {J. Inst. Math. Jussieu},
  FJOURNAL = {Journal of the Institute of Mathematics of Jussieu. JIMJ.
              Journal de l'Institut de Math\'ematiques de Jussieu},
    VOLUME = {7},
      YEAR = {2008},
    NUMBER = {3},
     PAGES = {527--574},
      ISSN = {1474-7480,1475-3030},
   MRCLASS = {22E50},
  MRNUMBER = {2427423},
       DOI = {10.1017/S1474748008000078},
       URL = {https://doi.org/10.1017/S1474748008000078},
}

@book {SGA3II,
    KEY = {SGA3\textsubscript{II}},
   EDITOR = {Demazure, Michel and Grothendieck, Alexandre},
     TITLE = {Sch\'{e}mas en groupes. {II}: {G}roupes de type multiplicatif, et structure des sch\'{e}mas en groupes g\'{e}n\'{e}raux},
    SERIES = {Lecture Notes in Mathematics},
    VOLUME = {Vol. 152},
      NOTE = {S\'{e}minaire de G\'{e}om\'{e}trie Alg\'{e}brique du Bois
              Marie 1962/64 (SGA 3),
              Dirig\'{e} par M. Demazure et A. Grothendieck},
 PUBLISHER = {Springer-Verlag, Berlin-New York},
      YEAR = {1970},
     PAGES = {ix+654},
   MRCLASS = {14.50},
  MRNUMBER = {274459},
}

@book {Serre77,
    AUTHOR = {Serre, Jean-Pierre},
     TITLE = {Linear representations of finite groups},
    SERIES = {Graduate Texts in Mathematics},
    VOLUME = {Vol. 42},
      NOTE = {Translated from the French by Leonard L. Scott},
 PUBLISHER = {Springer-Verlag, New York-Heidelberg},
      YEAR = {1977},
     PAGES = {x+170},
      ISBN = {0-387-90190-6},
   MRCLASS = {20CXX},
  MRNUMBER = {450380},
MRREVIEWER = {W.\ Feit},
}

@incollection {SS70,
    AUTHOR = {Springer, T. A. and Steinberg, R.},
     TITLE = {Conjugacy classes},
 BOOKTITLE = {Seminar on {A}lgebraic {G}roups and {R}elated {F}inite
              {G}roups ({T}he {I}nstitute for {A}dvanced {S}tudy,
              {P}rinceton, {N}.{J}., 1968/69)},
    SERIES = {Lecture Notes in Math.},
    VOLUME = {Vol. 131},
     PAGES = {167--266},
 PUBLISHER = {Springer, Berlin-New York},
      YEAR = {1970},
   MRCLASS = {14.50 (20.00)},
  MRNUMBER = {268192},
}

@article {Sri64,
    AUTHOR = {Srinivasan, Bhama},
     TITLE = {The modular representation ring of a cyclic {$p$}-group},
   JOURNAL = {Proc. London Math. Soc. (3)},
  FJOURNAL = {Proceedings of the London Mathematical Society. Third Series},
    VOLUME = {14},
      YEAR = {1964},
     PAGES = {677--688},
      ISSN = {0024-6115,1460-244X},
   MRCLASS = {20.80},
  MRNUMBER = {168666},
MRREVIEWER = {C.\ W.\ Curtis},
       DOI = {10.1112/plms/s3-14.4.677},
       URL = {https://doi-org.proxy.lib.umich.edu/10.1112/plms/s3-14.4.677},
}

@book {St68,
    AUTHOR = {Steinberg, Robert},
     TITLE = {Endomorphisms of linear algebraic groups},
    SERIES = {Memoirs of the American Mathematical Society, No. 80},
 PUBLISHER = {American Mathematical Society, Providence, R.I.},
      YEAR = {1968},
     PAGES = {108},
   MRCLASS = {14.50 (22.00)},
  MRNUMBER = {0230728},
MRREVIEWER = {E. Abe},
}

@article {St75,
    AUTHOR = {Steinberg, Robert},
     TITLE = {Torsion in reductive groups},
   JOURNAL = {Advances in Math.},
  FJOURNAL = {Advances in Mathematics},
    VOLUME = {15},
      YEAR = {1975},
     PAGES = {63--92},
      ISSN = {0001-8708},
   MRCLASS = {20G15},
  MRNUMBER = {354892},
MRREVIEWER = {S.\ I.\ Gel\cprime fand},
       DOI = {10.1016/0001-8708(75)90125-5},
       URL = {https://doi.org/10.1016/0001-8708(75)90125-5},
}

@article {Ste08,
    AUTHOR = {Stevens, Shaun},
     TITLE = {The supercuspidal representations of {$p$}-adic classical
              groups},
   JOURNAL = {Invent. Math.},
  FJOURNAL = {Inventiones Mathematicae},
    VOLUME = {172},
      YEAR = {2008},
    NUMBER = {2},
     PAGES = {289--352},
      ISSN = {0020-9910,1432-1297},
   MRCLASS = {22E50},
  MRNUMBER = {2390287},
MRREVIEWER = {Anne-Marie\ H.\ Aubert},
       DOI = {10.1007/s00222-007-0099-1},
       URL = {https://doi.org/10.1007/s00222-007-0099-1},
}

@misc{Tak26,
      title={Characteristic-free approaches around {Y}u's construction}, 
      author={Yuta Takaya},
      year={2026},
      eprint={2605.03638},
      archivePrefix={arXiv},
      primaryClass={math.RT},
      url={https://arxiv.org/abs/2605.03638}, 
}

@misc{Tok23,
      title={Local {J}acquet-{L}anglands correspondence for regular supercuspidal representations}, 
      author={Kazuki Tokimoto},
      year={2023},
      eprint={2303.02354},
      archivePrefix={arXiv},
      primaryClass={math.NT},
      url={https://arxiv.org/abs/2303.02354}, 
}

@article {TV,
    AUTHOR = {Treumann, David and Venkatesh, Akshay},
     TITLE = {Functoriality, {S}mith theory, and the {B}rauer homomorphism},
   JOURNAL = {Ann. of Math. (2)},
  FJOURNAL = {Annals of Mathematics. Second Series},
    VOLUME = {183},
      YEAR = {2016},
    NUMBER = {1},
     PAGES = {177--228},
      ISSN = {0003-486X},
   MRCLASS = {11F70 (20G25)},
 MRNUMBER = {3432583},
MRREVIEWER = {B. Sury},
       DOI = {10.4007/annals.2016.183.1.4},
       URL = {https://doi.org/10.4007/annals.2016.183.1.4},
}

@misc{TVarxiv,
      title={Functoriality, {S}mith theory, and the {B}rauer homomorphism}, 
      author={David Treumann and Akshay Venkatesh},
      year={2014},
      eprint={1407.2346},
      archivePrefix={arXiv},
      primaryClass={math.NT},
      note = {ArXiv version 1, \url{https://arxiv.org/abs/1407.2346}}
}

@book {Vig96,
    AUTHOR = {Vign\'{e}ras, Marie-France},
     TITLE = {Repr\'{e}sentations {$l$}-modulaires d'un groupe r\'{e}ductif
              {$p$}-adique avec {$l\ne p$}},
    SERIES = {Progress in Mathematics},
    VOLUME = {137},
 PUBLISHER = {Birkh\"{a}user Boston, Inc., Boston, MA},
      YEAR = {1996},
     PAGES = {xviii and 233},
      ISBN = {0-8176-3929-2},
   MRCLASS = {22E35 (20G25 22-02 22E50)},
 MRNUMBER = {1395151},
MRREVIEWER = {Chon Hu Cheng},
}

@incollection {XZ19,
    AUTHOR = {Xiao, Liang and Zhu, Xinwen},
     TITLE = {On vector-valued twisted conjugation invariant functions on a
              group},
 BOOKTITLE = {Representations of reductive groups},
    SERIES = {Proc. Sympos. Pure Math.},
    VOLUME = {101},
     PAGES = {361--425},
      NOTE = {With an appendix by Stephen Donkin},
 PUBLISHER = {Amer. Math. Soc., Providence, RI},
      YEAR = {2019},
      ISBN = {978-1-4704-4284-2},
   MRCLASS = {20G05 (20G10)},
  MRNUMBER = {3930024},
MRREVIEWER = {Stuart\ Martin},
       DOI = {10.1090/pspum/101/14},
       URL = {https://doi.org/10.1090/pspum/101/14},
}

@article {Yu01,
    AUTHOR = {Yu, Jiu-Kang},
     TITLE = {Construction of tame supercuspidal representations},
   JOURNAL = {J. Amer. Math. Soc.},
  FJOURNAL = {Journal of the American Mathematical Society},
    VOLUME = {14},
      YEAR = {2001},
    NUMBER = {3},
     PAGES = {579--622},
      ISSN = {0894-0347,1088-6834},
   MRCLASS = {22E50},
  MRNUMBER = {1824988},
MRREVIEWER = {Bertrand\ Lemaire},
       DOI = {10.1090/S0894-0347-01-00363-0},
       URL = {https://doi.org/10.1090/S0894-0347-01-00363-0},
}

@incollection {Yu09,
    AUTHOR = {Yu, Jiu-Kang},
     TITLE = {On the local {L}anglands correspondence for tori},
 BOOKTITLE = {Ottawa lectures on admissible representations of reductive
              {$p$}-adic groups},
    SERIES = {Fields Inst. Monogr.},
    VOLUME = {26},
     PAGES = {177--183},
 PUBLISHER = {Amer. Math. Soc., Providence, RI},
      YEAR = {2009},
      ISBN = {978-0-8218-4493-9},
   MRCLASS = {11S37 (22E50)},
  MRNUMBER = {2508725},
MRREVIEWER = {Michael\ M.\ Schein},
       DOI = {10.1090/fim/026/07},
       URL = {https://doi.org/10.1090/fim/026/07},
}

@incollection {Zhu17a,
    AUTHOR = {Zhu, Xinwen},
     TITLE = {An introduction to affine {G}rassmannians and the geometric
              {S}atake equivalence},
 BOOKTITLE = {Geometry of moduli spaces and representation theory},
    SERIES = {IAS/Park City Math. Ser.},
    VOLUME = {24},
     PAGES = {59--154},
 PUBLISHER = {Amer. Math. Soc., Providence, RI},
      YEAR = {2017},
   MRCLASS = {14M15 (14D24 20F65 22E57)},
  MRNUMBER = {3752460},
MRREVIEWER = {Felipe Zald\'{\i}var},
}

@misc{Zhu25,
      title={Tame categorical local {L}anglands correspondence}, 
      author={Xinwen Zhu},
      year={2025},
      eprint={2504.07482},
      archivePrefix={arXiv},
      primaryClass={math.RT},
      url={https://arxiv.org/abs/2504.07482}, 
}

@book{Art13,
    AUTHOR = {Arthur, James},
     TITLE = {The endoscopic classification of representations: orthogonal and symplectic groups},
    SERIES = {American Mathematical Society Colloquium Publications},
    VOLUME = {61},
 PUBLISHER = {American Mathematical Society, Providence, RI},
      YEAR = {2013},
      ISBN = {978-0-8218-4990-3},
       URL = {https://bookstore.ams.org/COLL/61},
}

@article{CZ21,
    AUTHOR = {Chen, Rui and Zou, Jialiang},
     TITLE = {Local {L}anglands correspondence for even orthogonal groups via theta lifts},
   JOURNAL = {Selecta Math. (N.S.)},
    VOLUME = {27},
    NUMBER = {5},
      YEAR = {2021},
     PAGES = {Paper No. 88, 71},
       DOI = {10.1007/s00029-021-00704-8},
       URL = {https://doi.org/10.1007/s00029-021-00704-8},
}

@article{GT11,
    AUTHOR = {Gan, Wee Teck and Takeda, Shuichiro},
     TITLE = {The local {L}anglands conjecture for {$\mathrm{GSp}(4)$}},
   JOURNAL = {Ann. of Math. (2)},
    VOLUME = {173},
    NUMBER = {3},
      YEAR = {2011},
     PAGES = {1841--1882},
       DOI = {10.4007/annals.2011.173.3.12},
       URL = {https://doi.org/10.4007/annals.2011.173.3.12},
}

@article{GT14,
    AUTHOR = {Gan, Wee Teck and Tantono, Welly},
     TITLE = {The local {L}anglands conjecture for {$\mathrm{GSp}(4)$}, {II}: the case of inner forms},
   JOURNAL = {Amer. J. Math.},
    VOLUME = {136},
    NUMBER = {3},
      YEAR = {2014},
     PAGES = {761--805},
       DOI = {10.1353/ajm.2014.0016},
       URL = {https://doi.org/10.1353/ajm.2014.0016},
}

@article{Ish24,
    AUTHOR = {Ishimoto, Hiroshi},
     TITLE = {The endoscopic classification of representations of non-quasi-split odd special orthogonal groups},
   JOURNAL = {Int. Math. Res. Not. IMRN},
    NUMBER = {14},
      YEAR = {2024},
     PAGES = {10939--11012},
       DOI = {10.1093/imrn/rnae113},
       URL = {https://doi.org/10.1093/imrn/rnae113},
}

@misc{KMSW14,
    AUTHOR = {Kaletha, Tasho and M\'{i}nguez, Alberto and Shin, Sug Woo and White, Paul-James},
     TITLE = {Endoscopic classification of representations: inner forms of unitary groups},
      YEAR = {2014},
    EPRINT = {1409.3731},
    ARCHIVEPREFIX = {arXiv},
    PRIMARYCLASS = {math.NT},
       URL = {https://arxiv.org/abs/1409.3731},
}

@article{GT10,
    AUTHOR = {Gan, Wee Teck and Takeda, Shuichiro},
     TITLE = {The local {L}anglands conjecture for {$\mathrm{Sp}_4$}},
   JOURNAL = {Int. Math. Res. Not. IMRN},
      YEAR = {2010},
    NUMBER = {15},
     PAGES = {2987--3038},
       DOI = {10.1093/imrn/rnp203},
       URL = {https://doi.org/10.1093/imrn/rnp203},
}

@article{Cho17,
    AUTHOR = {Choiy, Kwangho},
     TITLE = {The local {L}anglands conjecture for the {$p$}-adic inner form of {$\mathrm{Sp}_4$}},
   JOURNAL = {Int. Math. Res. Not. IMRN},
      YEAR = {2017},
    NUMBER = {6},
     PAGES = {1830--1889},
       DOI = {10.1093/imrn/rnw043},
       URL = {https://doi.org/10.1093/imrn/rnw043},
}

@misc{LH23,
    AUTHOR = {Li-Huerta, Siyan Daniel},
     TITLE = {Local-global compatibility over function fields},
      YEAR = {2023},
    EPRINT = {2301.09711},
    ARCHIVEPREFIX = {arXiv},
    PRIMARYCLASS = {math.NT},
       URL = {https://arxiv.org/abs/2301.09711},
}

\end{document}